\documentclass{amsart}
\usepackage[curve,matrix,arrow,frame,tips]{xy}
\usepackage{enumerate}
\usepackage{amssymb}
\usepackage{manfnt}
\usepackage{bbm}
\usepackage[pdfborder={0 0 0}, pdftitle={Equivariant K-theory as a global localization}, pdfauthor={Stefan Schwede},pdfstartview={FitH}]{hyperref}

\DeclareMathOperator{\Fun}{Fun}
\DeclareMathOperator{\Alg}{Alg}
\DeclareMathOperator{\CAlg}{CAlg}
\DeclareMathOperator{\UR}{URing}
\DeclareMathOperator{\rAb}{Ab}
\DeclareMathOperator{\cof}{cof}
\DeclareMathOperator{\colim}{colim}
\DeclareMathOperator{\gl}{gl}
\DeclareMathOperator{\dgl}{\!\text{-}gl}
\DeclareMathOperator{\lax}{lax}
\DeclareMathOperator{\cl}{cl}
\DeclareMathOperator{\nv}{nv}
\DeclareMathOperator{\st}{st}
\DeclareMathOperator{\uc}{uc}
\DeclareMathOperator{\mon}{mon}
\DeclareMathOperator{\Gal}{Gal}
\DeclareMathOperator{\Hom}{Hom}
\DeclareMathOperator{\Aut}{Aut}
\DeclareMathOperator{\Sym}{Sym}
\DeclareMathOperator{\map}{map}
\DeclareMathOperator{\LMod}{LMod}
\DeclareMathOperator{\Mod}{Mod}
\DeclareMathOperator{\Id}{Id}
\DeclareMathOperator{\ind}{ind}
\DeclareMathOperator{\ev}{ev}
\DeclareMathOperator{\ab}{ab}
\DeclareMathOperator{\op}{op}
\DeclareMathOperator{\Sp}{Sp}
\DeclareMathOperator{\tr}{tr}
\DeclareMathOperator{\res}{res}
\DeclareMathOperator{\rFin}{\mathcal{F}in}
\DeclareMathOperator{\GrAb}{GrAb}

\newcommand{\dbs}{\backslash\!\!\backslash}
\newcommand{\mone}{{\mathbbm 1}}
\newcommand{\mC}{{\mathbb C}}
\newcommand{\mE}{{\mathbb E}}
\newcommand{\mF}{{\mathbb F}}
\newcommand{\mL}{{\mathbb L}}
\newcommand{\mN}{{\mathbb N}}
\newcommand{\mP}{{\mathbb P}}
\newcommand{\mQ}{{\mathbb Q}}
\newcommand{\mR}{{\mathbb R}}
\newcommand{\mS}{{\mathbb S}}
\newcommand{\mZ}{{\mathbb Z}}
\newcommand{\Ab}{{\mathcal A}b}
\newcommand{\Ac}{{\mathcal A}}
\newcommand{\Bc}{{\mathcal B}}
\newcommand{\Cc}{{\mathcal C}}
\newcommand{\Dc}{{\mathcal D}}
\newcommand{\Fc}{{\mathcal F}}
\newcommand{\Lc}{{\mathcal L}}
\newcommand{\M}{{\mathcal M}}
\newcommand{\Oc}{{\mathcal O}}
\newcommand{\Pc}{{\mathcal P}}
\newcommand{\Sc}{{\mathcal S}}
\newcommand{\Uc}{{\mathcal U}}
\newcommand{\Wc}{{\mathcal W}}
\newcommand{\Xc}{{\mathcal X}}
\newcommand{\Yc}{{\mathcal Y}}
\newcommand{\bL}{{\mathbf L}}
\newcommand{\bP}{{\mathbf P}}
\newcommand{\bT}{{\mathbf T}}
\newcommand{\bku}{{\mathbf{ku}}}
\newcommand{\bKU}{{\mathbf{KU}}}
\newcommand{\bKR}{{\mathbf{KR}}}
\newcommand{\iso}{\cong}
\newcommand{\sm}{\wedge}
\newcommand{\tensor}{\otimes}
\newcommand{\xra}{\xrightarrow}
\newcommand{\bs}{\backslash}
\newcommand{\upi}{{\underline \pi}}
\newcommand{\td}[1]{\langle #1\rangle}
\renewcommand{\to}{\longrightarrow}

\numberwithin{equation}{section}
\newtheorem{theorem}[equation]{Theorem}
\newtheorem{lemma}[equation]{Lemma}
\newtheorem{cor}[equation]{Corollary}
\newtheorem{prop}[equation]{Proposition}
 
\theoremstyle{definition}
\newtheorem{defn}[equation]{Definition}
\newtheorem{rk}[equation]{Remark}
\newtheorem{eg}[equation]{Example}
\newtheorem{construction}[equation]{Construction}

\newcommand{\Danger}{%
   \par\penalty-500\begingroup\parindent=0.0pt\clubpenalty=10000%
   \def\par{\endgraf\endgroup}\leavevmode
   \hangindent=2.4em\relax
   \hangafter=-2\relax
   \hbox to 0pt{\hss\relax
      \vbox to 0pt{\vskip-8pt\relax
        \hbox
          {\relax\dbend}
\vss}\kern.70em}}

\begin{document}

\title[Equivariant \texorpdfstring{$K$}{K}-theory as a global localization]
{Equivariant \texorpdfstring{$K$}{K}-theory as a global localization}

\author{Stefan Schwede}


\begin{abstract}
  We establish a global equivariant refinement of the theorem of Snaith that
  presents the complex $K$-theory spectrum as a localization
  of the suspension spectrum of $\mathbb C P^\infty$.
  We prove that the morphism of ultra-commutative global ring
  spectra from $\Sigma^\infty_+ B_{\gl} U(1)$ 
  that classifies the tautological $U(1)$-representation
  exhibits the global $K$-theory spectrum
  as a specific localization, in both the $\infty$-categories of
  ultra-commutative ring spectra and of commutative global ring spectra.
  The classes to be inverted are certain representation-graded
  $U(n)$-equivariant homotopy classes of $\Sigma^\infty_+ B_{\gl} U(1)$,
  the {\em pre-Bott classes}. 
  We also show that for every compact Lie group $G$, the $G$-equivariant
  $K$-theory spectrum is a telescopic localization of
  $\Sigma^\infty_+ B_G U(1)$ by the pre-Bott classes of all unitary $G$-representations. 
\end{abstract}

\maketitle

\tableofcontents

\section*{Introduction}

A celebrated theorem of Snaith \cite{snaith:algebraic cobordism, snaith:localized stable}
shows that complex $K$-theory can be obtained
from the unreduced suspension spectrum of $\mC P^\infty$ by inverting the Bott class.
A contemporary formulation of this result is that a specific morphism
of $E_\infty$-ring spectra $\Sigma^\infty_+ \mC P^\infty\to K U$
that classifies the tautological complex line bundle over $\mC P^\infty$
is initial, in the higher categorical sense, among morphisms
that take the Bott class in $\pi_2(\Sigma^\infty_+\mC P^\infty)$
to an invertible homotopy class, see \cite[Theorem 6.5.1]{lurie:ellipticII}.
Lurie reformulates this fact
as the statement that complex $K$-theory
classifies orientations of the multiplicative group in higher algebra,
compare \cite[Section 3.1]{lurie:elliptic_survey} and \cite[Corollary 4.3.27]{lurie:ellipticII}.
In this paper we prove a highly structured,
globally-equivariant refinement of Snaith's localization result.
Our protagonist is the morphism of ultra-commutative global ring spectra
\[  \eta\ :\ \Sigma^\infty_+ \bP \ \to\ \bKU  \]
introduced in \cite[Remark 6.4.20]{schwede:global}.
Here $\bP$ is a multiplicative model of the global classifying space of $U(1)$,
made from complex projective spaces, 
and $\Sigma^\infty_+ \bP$ is its unreduced global suspension spectrum.
The target $\bKU$ is Joachim's ultra-commutative ring spectrum model
for global equivariant $K$-theory \cite[Definition 4.3]{joachim:coherences},
see also \cite[Construction 6.4.9]{schwede:global}.

In Construction \ref{con:pre-Bott} we introduce the {\em pre-Bott classes}
$b_m$ in $\pi_{\nu_m}^{U(m)}(\Sigma^\infty_+\bP)$, for $m\geq 1$,
where $\nu_m$ is the tautological $U(m)$-representation on $\mC^m$.
The definition of these classes involves the global structure and power operations.
The morphism $\eta\colon \Sigma^\infty_+ \bP \to\bKU$ takes the pre-Bott class $b_m$
to the $U(m)$-equivariant Bott class of the representation $\nu_m$, hence the name.
In Construction \ref{con:pre-Bott_of_V}
we extend these classes by naturality in group homomorphisms to {\em pre-Bott classes}
$b_{G,V}$ in $\pi_V^G(\Sigma^\infty_+\bP)$, for all unitary $G$-representations $V$.
By Bott periodicity in equivariant $K$-theory,
the morphism $\eta$ in particular inverts all these pre-Bott classes.
We show that the morphism $\eta\colon \Sigma^\infty_+ \bP \to\bKU$
is in fact a localization by the pre-Bott classes in several different senses:
in the $\infty$-categories of ultra-commutative ring spectra
and of commutative global ring spectra;
for every compact Lie group $G$, in the $\infty$-category
 of commutative $G$-ring spectra;
for every compact Lie group $G$, at the level of $G$-equivariant cohomology
theories; and in the `Real-global' context.

We expand on these localization results.
The global formulation starts from the fact that
$\eta\colon \Sigma^\infty_+ \bP\allowbreak\to\bKU$ inverts the pre-Bott classes $b_m$ for $m\geq 1$,
see Theorem \ref{thm:pre-Bott}.
We show in Theorem \ref{thm:main}
that $\eta$ is in fact initial with this property,
even with two different levels of coherence in the multiplication:

\medskip
\noindent
{\bf Theorem A.}
{\em The morphism of ultra-commutative ring spectra $\eta\colon \Sigma^\infty_+ \bP \to\bKU$
is a localization by the pre-Bott classes $b_m$ for $m\geq 1$,
both in the $\infty$-category of ultra-commutative ring spectra,
and in the $\infty$-category of commutative global ring spectra.}
\medskip

In Theorem A, `localization' is to be understood in the coherent and higher categorical
sense, i.e., as an initial example of a morphism with source $\Sigma^\infty_+ \bP$
that sends all $b_m$ to invertible classes,
see Definitions \ref{def:global_loc} and \ref{def:ucom_loc}.
Readers so inclined may interpret Theorem A
as the statement that global $K$-theory classifies
ultra-commutative orientations of the global multiplicative group.
The main ingredient for the proof of Theorem A is a characterization of
the effect of $\eta\colon \Sigma^\infty_+ \bP \to\bKU$ on the geometric fixed point spectra
as a localization, see Theorem \ref{thm:geomfix}.

We emphasize that Theorem A is a fully global result, encompassing all compact Lie groups;
previous work on the subject \cite{snaith:model_equivariant, spitzweck-oestvaer:equivariant}
had restricted attention to {\em abelian} compact Lie groups.
Because $\bP$ is a global classifying space of $U(1)$,
the spectrum $\Sigma^\infty_+\bP$ is {\em left} induced from the global family
of abelian compact Lie groups; 
in contrast, $\bKU$ is {\em right} induced from the same global family.
Being left and right induced from some proper subclass of compact Lie groups
are antithetical properties; so it might come as a surprise that inverting
the pre-Bott classes mediates the gap between these.

In a forthcoming paper, we shall use the localization property of
the morphism $\eta\colon \Sigma^\infty_+ \bP \to\bKU$
to promote the $k$-th power endomorphism \eqref{eq:define_power_P} of $\bP$
to the {\em stable global Adams operation},
a morphism of ultra-commutative ring spectra
\[ \Psi^k \ : \ \bKU \ \to \ \bKU[t_k^{-1}]\]
which realizes the $k$-th Adams operation in equivariant $K$-theory.
Here
\[ t_k\ =\ \tr_{\{1\}}^{C_k}(1)\ \in \  \pi_0^{C_k}(\bKU) \]
is the transfer of the multiplicative unit to the cyclic group of order $k$,
and $\bKU[t_k^{-1}]$ denotes the ultra-commutative localization of $\bKU$ by $t_k$.
\medskip

In view of Theorem A, one might wonder about localizations that only invert some
of the pre-Bott classes.
The pre-Bott class $b_m$ is a factor of a restriction of $b_{m+1}$,
see Theorem \ref{thm:fix of pre-Bott} (i).
So inverting $b_{m+1}$ also inverts all the previous pre-Bott classes $b_1,\dots,b_m$.
We show in Theorem \ref{thm:all_different} that inverting $b_m$
in the $\infty$-category of ultra-commutative ring spectra
does {\em not} make $b_{m+1}$ invertible; our witness is the group $S U(m+1)$.
Hence the localization in the $\infty$-category of commutative global ring spectra
does not make $b_{m+1}$ invertible, either.
So the localizations by the various pre-Bott classes are all distinct,
and one needs to invert all $b_m$ to obtain global $K$-theory
from $\Sigma^\infty_+\bP$ on all compact Lie groups.

On the family of {\em abelian} compact Lie groups, 
inverting only the first pre-Bott class $b_1$ already turns 
$\Sigma^\infty_+\bP$ into global $K$-theory. However, inverting only $b_1$
leads to two different answers, depending on the level of coherence in the
commutativity, i.e., global-commutative versus ultra-commutative:
\begin{itemize}
\item 
As we show in Theorem \ref{thm:glcom b_1 loc}, the localization
of $\Sigma^\infty_+\bP$ by $b_1$
in the $\infty$-category of commutative global ring spectra
is left induced from the global family
$\Ab$ of abelian compact Lie groups, and $\Ab$-globally equivalent to $\bKU$.
In other words: the global localization by $b_1$ is equivalent to
the result of first restricting $\bKU$ to the global family $\Ab$, and then left inducing
to full global equivariance.
The ultimate reason why inverting only $b_1$ already yields $\bKU$ on abelian compact Lie groups
is that it inverts the pre-Bott classes of all 1-dimensional representations,
and of all sums of such.
For abelian compact Lie groups, this accounts for all unitary representations,
so on this global family inverting $b_1$ makes all other pre-Bott classes invertible as well.
\item
As we show in Theorem \ref{thm:ucom b_1 loc}, the localization
of $\Sigma^\infty_+\bP$ by $b_1$
in the $\infty$-category of ultra-commutative ring spectra
is left induced from the larger global family $\Ab^\circ$ of compact Lie groups
whose identity component is abelian, and $\Ab^\circ$-globally equivalent to $\bKU$.
In other words: the ultra-commutative localization by $b_1$ is equivalent to
the result of first restricting $\bKU$ to the global family $\Ab^\circ$, and then left inducing
to full global equivariance.
\end{itemize}
The global family $\Ab^\circ$ is substantially larger than $\Ab$;
for example, it also contains all finite groups.\medskip

In addition to our global results, we also prove localization properties
for the morphism of genuine $G$-ring spectra underlying $\eta$, for every compact Lie group $G$.
We write $\bP_G=P(\Uc_G^\mC)$ for the $G$-space of complex lines in a complete
complex $G$-universe $\Uc_G^\mC$,
which is also the underlying $G$-space of the global space $\bP$, and a classifying space
for $G$-equivariant complex line bundles.
When restricting to a fixed compact Lie group,
we can describe $G$-equivariant $K$-theory as an algebraic
$RO(G)$-graded localization of $G$-equivariant $\Sigma^\infty_+\bP_G$-cohomology
by inverting the pre-Bott classes of all unitary $G$-representations.
The key ingredient for this is Theorem \ref{thm:telescopic localization},
identifying $\bKU_G$ as a telescopic localization by the set of these pre-Bott classes.
The following result is a combination of Theorem \ref{thm:main-G-structured}
and Theorem \ref{thm:main-G-RO-version}.

\medskip
\noindent
{\bf Theorem B.}
{\em
    Let $G$ be a compact Lie group.
  \begin{enumerate}[\em (i)]
\item  The morphism of commutative $G$-ring spectra
  $\eta_G\colon\allowbreak \Sigma^\infty_+ \bP_G\to \bKU_G$
    is a localization by the pre-Bott classes $b_{G,V}$
    for all unitary $G$-representations $V$,
    in the $\infty$-category of commutative $G$-ring spectra.
  \item For every $\Sigma^\infty_+ \bP_G$-module $M$ in $\Sc p_G^\sm$,
    the morphism
    \[ \pi_\star^G(\eta_G\sm_{\Sigma^\infty_+\bP_G} M)\ : \  \pi_\star^G(M)  \ \to \ \pi_\star^G(\bKU_G\sm_{\Sigma^\infty_+\bP_G} M) \]
    exhibits the target as a localization,
    in $RO(G)$-graded  $\pi_\star^G(\Sigma^\infty_+ \bP_G)$-modules,
    of the source by the set of pre-Bott classes of all unitary $G$-representations.
   \item For every finite based $G$-CW-complex $A$, the morphism 
    \[  \eta^\star_G(A)\ : \  (\Sigma^\infty_+ \bP_G)^\star(A)  \ \to \ \bKU_G^\star(A) \]
    exhibits the target as a localization,
    in $RO(G)$-graded  $(\Sigma^\infty_+ \bP_G)^\star$-modules,
    of the source by the set of pre-Bott classes of all unitary $G$-representations.
  \item The morphism of graded-commutative $RO(G)$-graded rings
    \[  \pi_\star^G(\eta_G)\ : \  \pi_\star^G(\Sigma^\infty_+ \bP_G) \ \to \ \pi_\star^G(\bKU_G) \]
    exhibits the target as an $RO(G)$-graded localization of the source by the set
    of pre-Bott classes of all unitary $G$-representations.
  \end{enumerate}
  }
\medskip

When the group $G$ is finite, then in Theorem B one can obtain $\bKU_G$ by telescopically
inverting a single pre-Bott class; any unitary representation that contains all irreducibles
does the job, for example the complex regular representation.

Given Theorem B, one may wonder if $RO(G)$-graded localization is really necessary,
or whether it is also possible to obtain $\bKU_G$ from $\Sigma^\infty_+ \bP_G$
by telescopic localization by some integer-graded $G$-equivariant homotopy class.
The paper \cite{spitzweck-oestvaer:equivariant} claims this for
abelian compact Lie groups, by inverting a single class in the integer-graded
homotopy group $\pi_2^G(\Sigma^\infty_+\bP_G)$.
However, we show in Corollary \ref{cor:no_Z_localization} that for $G\ne\{1\}$,
one {\em cannot} obtain $\bKU_G$ from $\Sigma^\infty_+ \bP_G$ by inverting any
set of integer-graded classes.
In particular, the main result of \cite{spitzweck-oestvaer:equivariant} is incorrect;
we refer to Remark \ref{rk:spitz-oest} for more details.\medskip

The global $K$-theory spectrum underlies a `Real-global' refinement $\bKR$,
which also knows about conjugate-linear phenomena, and represents
Real-equivariant $K$-theory for compact Lie groups augmented to
the Galois group of $\mC$ over $\mR$.
In Section \ref{sec:Real} we extend our localization results to this Real-global context.
This might sound like another vast generalization, but the main work
is already done in proving Theorem A, and only one additional step is needed.
In Construction \ref{con:Real pre-Bott} we introduce
the {\em Real-global pre-Bott class} $\tilde b_1$ in the group
$\pi^{\tilde T}_{\nu_1}(\Sigma^\infty_+\bP)$, where $\tilde T=T\rtimes C$
is the extended circle group, the semidirect product of $C=\Gal(\mC/\mR)$ acting on $T=U(1)$
by complex conjugation.
As the name suggests, the Real-global pre-Bott class restricts to the pre-Bott class $b_1$,
and it maps to the Bott class for the tautological Real $\tilde T$-representation,
see Proposition \ref{prop:properties_tilde b}.
As we explain in the proof of Theorem \ref{thm:Realglobal},
the key feature of the refined pre-Bott class $\tilde b_1$ is that inverting it
kills the geometric fixed point spectra of all surjectively augmented Lie groups $G\to C$.
Consequently, as far as localization is concerned, all new `Real' features
are already taken care of by $\tilde b_1$, and the pre-Bott classes $b_m$ for $m\geq 2$
do not have to be modified, or extended, in the `Real' direction.
The final result, to be proved as Theorem \ref{thm:Realglobal},
is the following.

\medskip
\noindent
{\bf Theorem C.}
{\em The morphism  $\eta\colon \Sigma^\infty_+ \bP \to\bKR$
  is a localization in the $\infty$-category of commutative Real-global ring spectra
  by the Real-global pre-Bott class $\tilde b_1$ and the pre-Bott classes $b_m$ for $m\geq 2$.}
\medskip

As I explain in Remark \ref{rk:C_global_UCom}, I am convinced that
$\eta\colon \Sigma^\infty_+ \bP \to\bKR$
is also a localization by the classes $\tilde b_1$ and $\{b_m\mid m\geq 2\}$
in the $\infty$-category of ultra-commutative Real-global ring spectra.
However, adapting the non-Real proof to the Real-global situation requires
some foundational properties of ultra-commutative $C$-global ring spectra
that have not yet been established at the time of this writing.\medskip

This paper includes several appendices.
In Appendix \ref{app:equivariant_spectra} we collect the background material on
the $\infty$-categories of genuine $G$-spectra
that we need in the body of the paper; it does not contain any new mathematics.
In Appendix \ref{app:global}
we review global spectra and global ring spectra from the higher categorical perspective.
This appendix also contains a rectification result, and the bimonadicity
of ultra-commutative ring spectra.
In more detail: we show in Theorem \ref{thm:monoidal_loc}
that for every commutative orthogonal ring spectrum $A$,
the symmetric monoidal 1-category of $A$-modules in orthogonal spectra
presents the symmetric monoidal $\infty$-category of modules over the
commutative global ring spectrum presented by $A$.
And in Theorem \ref{thm:UR_monadic}
we show that the $\infty$-category of ultra-commutative ring spectra
is both monadic and comonadic over the $\infty$-category of
commutative global ring spectra.

Appendix \ref{app:global_classes} develops global stable homotopy theory
relative to a {\em global class}, i.e.,
a non-empty class of compact Lie groups that is closed under isomorphism
and passage to closed subgroups;
the emphasis is on multiplicative and ultra-commutative features.
Global classes are more general than global families
in the sense of \cite[Definition 1.4.1]{schwede:global}, in that they
need not be closed under quotient groups.
Two key results of this appendix are coreflectivity statements for multiplicative structures.
Theorem \ref{thm:multiplicative_class}
includes the statement that for every {\em multiplicative} global class,
the left induced commutative global ring spectra
are coreflective inside all commutative global ring spectra.
And Theorem \ref{thm:wreath_closed} shows, among other things, that
for multiplicative global classes which are additionally closed under wreath products,
the left induced ultra-commutative ring spectra
are coreflective inside all ultra-commutative ring spectra.\medskip

{\bf History and related results.}
The first published version
of Snaith's theorem \cite[II Theorem 9.1.1]{snaith:algebraic cobordism} 
amounts to the equivalence of underlying infinite loop spaces
obtained from the equivalence $(\Sigma^\infty_+ \mC P^\infty)[b^{-1}]\sim K U$
of spectra. In \cite[Theorem 2.12]{snaith:localized stable},
Snaith states the result as a natural isomorphism
of cohomology theories, with a different proof.
A translation of Snaith's work into contemporary
language as a localization of $E_\infty$-ring spectra
can be found in \cite[Theorem 6.5.1]{lurie:ellipticII};
Lurie gives an independent proof from the perspective of 
orientations of the multiplicative group in higher algebra.
Gepner--Snaith \cite{gepner-snaith} and
Spitzweck--{\O}stv{\ae}r \cite{spitzweck-oestvaer:motivic} independently proved
analogues of Snaith's theorem in the context of motivic homotopy theory.

There are also two previous papers
\cite{snaith:model_equivariant, spitzweck-oestvaer:equivariant} devoted to describing
$G$-equivariant $K$-theory, for abelian compact Lie groups $G$, as the result of
`inverting the Bott class on $\Sigma^\infty_+ \bP_G$'.
As we already mentioned above, the main result of \cite{spitzweck-oestvaer:equivariant}
is incorrect, as our Corollary \ref{cor:no_Z_localization} shows;
we explain in Remark \ref{rk:spitz-oest} below where the argument
of \cite{spitzweck-oestvaer:equivariant} breaks.
In \cite[Theorem 3.1]{snaith:model_equivariant}, Snaith establishes an equivariant version of
his non-equivariant theorem for {\em abelian} compact Lie groups.
The result is different from our Theorem B, however,
in that Snaith implicitly works with naive-equivariant (as opposed to genuine-equivariant)
$G$-spectra. Indeed, the definition  \cite[(2.6)]{snaith:model_equivariant}
of the cohomology theory $h_G(A)$ for $G$-spaces $A$ is by taking equivariant maps
into $\bP_G$ while simultaneously stabilizing by spheres with trivial $G$-action
and multiplication by the `naive' equivariant Bott class, 
i.e., the non-equivariant Bott class inflated to $G$-spectra.
If we write $\{-,-\}_G^{\text{naive}}$ for the morphism group in the homotopy category
of $G$-spectra indexed on a trivial universe, then
Snaith's main result \cite[Theorem 3.1]{snaith:model_equivariant} is a natural isomorphism
\[ h_G(A) = \{\Sigma^\infty_+ A,\Sigma^\infty_+ \bP_G\}_G^{\text{naive}}[\beta^{-1}]\ \iso \ K_G(A) \ ,\]
for abelian compact Lie groups $G$, and finite $G$-CW-complexes $A$,
between a naive equivariant cohomology theory
on the one side and a genuine equivariant cohomology theory on the other side.\medskip

{\bf Conventions.}
We state our main results in higher categorical terms in the $\infty$-categories
of commutative or ultra-commutative global or $G$-equivariant ring spectra.
For definiteness, `$\infty$-category' shall in this paper mean `quasi-category',
and the more technical appendices draw heavily from Lurie's writings
\cite{lurie:htt, lurie:higher_algebra}. The exposition and proofs in the body
of the paper are largely agnostic to the particular model for higher categories,
and mostly use naturality and universal property arguments.
To get to $\infty$-categories, we use the 1-categorical models
for unstable and stable global homotopy theory provided in \cite{schwede:global}.
In other words: we model global spaces by orthogonal spaces,
and we model global spectra by orthogonal spectra,
relative to the respective classes of global equivalences.
In fact, I would not know how to construct the morphism
$\eta\colon \Sigma^\infty_+\bP\to\bKU$ with the full coherence that it supports
(i.e., as a morphism of Real-global ultra-commutative ring spectra) in model-independent terms.

{\bf Acknowledgments.}
I would like to thank William Balderrama, Dave Benson, Jack Davies,
Markus Hausmann, Branko Juran, Dominik Kirstein, Tobias Lenz and Sil Linskens
for several enlightening discussions about the contents of this paper,
and for providing help with and references for the higher categorical foundations
and manipulations needed in this paper.
The author is a member of the Hausdorff Center for Mathematics
at the University of Bonn (DFG GZ 2047/1, project ID 390685813).

{\bf AI disclosure.}
I obtained the main ideas and results of this paper without consulting AI.
If evidence for this claim is desired, I offer the 
recorded talk that I gave at the Isaac Newton Institute on February 25, 2025,
where I presented most of the results in this paper for the first time outside of Bonn.
AI became very helpful during the final stage of proof-reading, polishing, and optimizing
the written paper in the summer and fall of 2026.
I have used ChatGPT 5.5 and 5.6 (OpenAI) and Claude Opus 5.5 (Anthropic)
to search for literature, to locate precise references,
to discuss various pieces of mathematics related to this paper,
and for proof-reading several previous versions of this document.
They found embarrassingly many typos and inaccuracies,
and several smaller gaps or oversights in mathematical arguments.
I believe that I would have caught many of these issues myself, but with much more
time and effort; and some of them would probably have slipped past me.
I also made Claude scrutinize existing arguments for simplifications;
it suggested simplifications
of my original proofs of Proposition \ref{prop:nearcyclic},
of Remark \ref{rk:spitz-oest}, of Example \ref{eg:Sigma_3_witnesses},
of Theorem \ref{thm:monoidal_loc},
and of parts (iii) and (iv) of Theorem \ref{thm:wreath_closed},
and a reorganization that resulted in the present formulation of
the localization recognition in Theorem \ref{thm:loc_reg}.
The final text was written, reviewed and approved by me.

\section{Localizations of commutative global ring spectra}

In this section we introduce localizations
of commutative and ultra-commutative global ring spectra,
see Definitions \ref{def:global_loc} and \ref{def:ucom_loc}, respectively.
In each case, a localization is an initial example of a morphism that inverts
the given set of representation-graded equivariant homotopy classes.
The {\em geometric fixed point criterion}, established in Theorem \ref{thm:fix_criterion},
is our main tool to recognize global localizations; as the name suggests,
it provides a sufficient criterion
at the level of geometric fixed point rings.\medskip

The pre-Bott classes do not live in integer gradings, but rather they
are {\em representation-graded} equivariant homotopy classes.
We briefly recall what we mean by this concept.

\begin{defn}[Representation-graded homotopy groups]
  We let $G$ be a compact Lie group, $E$ an orthogonal $G$-spectrum, 
  and $V$ an orthogonal $G$-representation.
  The $V$-graded $G$-equivariant homotopy group of $E$ is
  \[ \pi_V^G(E)\ = \ \colim_{U\in s(\Uc_G)} [S^{U\oplus V},E(U)]^G_*\ .  \]
  The colimit is taken over the poset $s(\Uc_G)$, under inclusion,
  of finite-dimensional $G$-subrepresentations of a complete $G$-universe $\Uc_G$,
  along stabilization maps much like in the integer-graded case.
\end{defn}

We emphasize that we grade by actual representations,
as opposed to isomorphism classes of such, or elements in the representation ring.
Invertibility in the representation-graded sense is nevertheless
closely related to invertibility in the $RO(G)$-graded homotopy ring,
see Remark \ref{rk:2invertibles} for details.
One can imagine grading more generally by pairs of representations
that represent formal differences `$V-W$'.
Since the pre-Bott classes live in `effective' representation degrees,
we will not need the additional generality in the present paper,
and we stick to the simpler situation.

The group $\pi_V^G(E)$ is naturally isomorphic to $\pi_0^G(\Omega^V E)$,
so the functor $\pi_V^G$ takes $\upi_*$-isomorphisms of orthogonal $G$-spectra
to isomorphisms, see \cite[Proposition 3.1.40]{schwede:global}.
So whenever convenient we can---and will---consider $\pi_V^G$ as a functor on
the $\infty$-category $\Sc p_G$ of genuine $G$-spectra,
or on the homotopy category of $G$-spectra $h(\Sc p_G)$.
On the homotopy category, the functor $\pi_V^G$ is represented by the
suspension spectrum of $S^V$: the identity of $S^V=(\Sigma^\infty S^V)(0)$
represents a class
\[ e_{G,V}\ \in \ \pi_V^G(\Sigma^\infty S^V) \ ,\]
and the pair $(\Sigma^\infty S^V\!, e_{G,V})$ represents the functor
$\pi_V^G\colon h(\Sc p_G)\to\rAb$.

If $H$ is a closed subgroup of $G$, $E$ an orthogonal $G$-spectrum,
and $V$ a $G$-representation, then the restriction and transfer homomorphisms
easily generalize from integer-graded
to representation-graded homotopy groups to homomorphisms
\[ \res^G_H \ : \ \pi_V^G(E)\ \to \ \pi_V^H(E) \text{\qquad and\qquad}
 \tr^G_H \ : \ \pi_V^H(E)\ \to \ \pi_V^G(E) \ ;\]
here $\pi^H_V$ refers to the homotopy group graded by the underlying $H$-representation of $V$.
Indeed, we can simply define these representation-graded restriction and transfer operations
as the degree 0 restrictions and transfers for the orthogonal $G$-spectrum $\Omega^V E$.\medskip

Now we let $R$ be a {\em homotopy $G$-ring spectrum}, i.e., 
an algebra in the symmetric monoidal stable $G$-homotopy category $h(\Sc p_G^\sm)$.
Since the suspension $G$-spectrum $\Sigma^\infty S^V$
represents the functor $\pi_V^G$,
a representation-graded homotopy class $v\in \pi_V^G(R)$
yields a unique morphism $v^\circ\colon\Sigma^\infty S^V\to R$ in $h(\Sc p_G)$
such that $v^\circ_*(e_{G,V})=v$.
We denote by $v^\flat\colon R\sm S^V\to R$ the composite
\begin{equation}\label{eq:define_v_flat}
  R\sm S^V\iso R\sm (\Sigma^\infty S^V)\ \xra{R\sm v^\circ}\ R\sm R \ \xra{\ \mu\ }\ R
\end{equation}
in $h(\Sc p_G)$, where $\mu\colon R\sm R\to R$ is the multiplication morphism of $R$.

\begin{defn}[Invertibility]\label{def:RO-invertible}
  Let $G$ be a compact Lie group, let $R$ be a homotopy $G$-ring spectrum,
  and let $V$ be a $G$-representation.
  A representation-graded homotopy class $v\in \pi_V^G(R)$
  is {\em invertible} if the morphism 
  $v^\flat \colon R\sm S^V\to R$  is an equivalence of $G$-spectra.
\end{defn}

\begin{rk}\label{rk:2invertibles}
In Construction \ref{con:pi_star^G} we recall the $RO(G)$-graded
equivariant homotopy groups of a $G$-spectrum.
Invertibility in the sense of Definition \ref{def:RO-invertible}
is then essentially the same as invertibility in the $RO(G)$-graded homotopy ring $\pi_\star^G(R)$.
More precisely, a class $v\in \pi_V^G(R)$ is invertible in the representation-graded sense
just defined if and only if for some (hence any)
equivalence of $G$-spectra $\kappa\colon \mS^{[V]}_G\sim \Sigma^\infty S^V$,
the class $\kappa^*(v)\in\pi_{[V]}^G(R)$ is a unit in the 
$RO(G)$-graded homotopy ring $\pi_\star^G(R)$.
The caveat is that while
the representation-graded homotopy group $\pi_V^G(R)$
is isomorphic to the $RO(G)$-graded homotopy group $\pi_{[V]}^G(R)$,
there is no preferred or `canonical' way to identify them.
\end{rk}

The next proposition characterizes invertibility of representation-graded homotopy classes
in terms of geometric fixed points.
We recall the geometric fixed point group $\Phi^G_k(X)$
of an orthogonal $G$-spectrum $X$ in \eqref{eq:define_geom_fix}.
A biadditive pairing
\[  \times\ : \ \Phi_k^G(X)\times \Phi_l^G(Y)\ \to \ \Phi_{k+l}^G(X\sm Y) \]
is defined in \eqref{eq:cdot_Phi^G}, providing a lax symmetric monoidal structure
on the $\mZ$-graded geometric fixed point groups $\Phi_*^G(X)$.
Since the functor $\Phi^G_*$ inverts $\upi_*$-isomorphisms of orthogonal
$G$-spectra, it descends to a lax symmetric monoidal functor
$\bar\Phi^G_*\colon\Sc p_G^\sm\to \GrAb^\tensor$
on the $\infty$-category of $G$-spectra;
see Construction \ref{con:descend_pi_and_Phi} for more details.
For homotopy $G$-ring spectra, the geometric fixed point groups inherit a graded multiplication,
compare \eqref{eq:concrete_product}.

Now we let $V$ be a $G$-representation.
We write $k=\dim(V^G)$ for the dimension of the $G$-fixed subspace of $V$.
A choice of orientation of $V^G$ determines a homotopy class of
linear isometry between $\mR^k$ and $V^G$, and thus a
{\em geometric fixed point homomorphism}:
\begin{align}\label{eq:geom_fix_map}
  \Phi^G\ : \ \pi_V^G (X) \quad &\to \qquad \Phi_k^G (X) \\
  [f\colon S^{U\oplus V}\to X(U)]\ &\longmapsto \ [f^G\colon S^{U^G\oplus\mR^k}\iso S^{(U\oplus V)^G}\to X(U)^G]  \nonumber
\end{align}
For $V=\mR^k$ with trivial $G$-action this fixed point homomorphism
specializes to the one defined in \eqref{eq:define_Phi^G_map}.
The geometric fixed point homomorphism \eqref{eq:geom_fix_map} is additive and multiplicative,
provided we orient $(V\oplus W)^G$ by concatenating the
chosen orientations of $V^G$ and $W^G$.
Choosing the opposite orientation changes this homomorphism
into its negative, which does not influence whether or not a class is invertible. 

From now on, we shall employ a standard abuse of notation.
When $H$ is a closed subgroup of a compact Lie group $G$,
and $X$ is an orthogonal $G$-spectrum, 
we shall write $\pi_V^H(X)$ and $\Phi^H_k(X)$ for the representation-graded
homotopy groups and the geometric fixed point homotopy groups, respectively,
of the underlying orthogonal $H$-spectrum.
In other words, we leave the restriction functor $\Sp^O_G\to\Sp^O_H$
implicit. The same goes for geometric fixed point homomorphisms:
whenever confusion is unlikely,
we shall write $\Phi^H\colon \pi_V^G(X)\to \Phi_l^H(X)$
for what rigorously ought to be $\Phi^H\circ \res^G_H$, for $l=\dim(V^H)$.

\begin{prop}\label{prop:invert_criterion}
  Let $V$ be a representation of a compact Lie group $G$,
  and let $R$ be a commutative homotopy $G$-ring spectrum.
  Then for a representation-graded homotopy class $v\in\pi_V^G(R)$,
  the following conditions are equivalent.
  \begin{enumerate}[\em (a)]
  \item The class $v$ is invertible.
  \item There exists an element $u\in\pi_0^G(R\sm S^V)$
    such that $v^\flat_*(u)=1$ in $\pi_0^G(R)$.
    \item For every closed subgroup $H$ of $G$, the class $\Phi^H(v)$
      is a homogeneous graded unit in the geometric fixed point ring $\Phi_*^H(R)$.
    \end{enumerate}
\end{prop}
\begin{proof}
  (a)$\Longrightarrow$(b)
  Since the morphism $v^\flat\colon R\sm S^V\to R$ is an equivalence of $G$-spectra,
  the induced map of homotopy groups
  $v^\flat_*\colon\pi_0^G(R\sm S^V)\to\pi_0^G(R)$
  is bijective. In particular, the class $1\in \pi_0^G(R)$ has a preimage.

  (b)$\Longrightarrow$(c)
  We let $H$ be a closed subgroup of $G$.
  Applying the geometric fixed point homomorphism $\Phi^H\colon \pi_0^G(R)\to\Phi^H_0(R)$
  to the relation $v^\flat_*(u)=1$ yields
  \begin{equation}\label{eq:Phi^H_of_rel}
      \Phi^H_0(v^\flat)(\Phi^H(u))\ =\ \Phi^H(v^\flat_*(u))\ =\  \Phi^H(1)\ = \ 1   
  \end{equation}
   in $\Phi_0^H(R)$.
  A choice of orientation of $V^H$ induces a well-defined homotopy
  class of equivalence $S^l\sim S^{V^H}$, for $l=\dim(V^H)$. This, in turn, induces 
 isomorphisms
 \begin{equation}\label{eq:induced_fix_iso}
   \Phi^H_{m-l}(R) \
   \iso \ \Phi^H_m(R\sm S^l) \
   \iso \ \Phi^H_m(R\sm S^{V^H})\ \iso \ \Phi^H_m(R\sm S^V)\ . 
 \end{equation}
 Moreover, composition of this isomorphism with the map
 $\Phi^H_m(v^\flat)\colon\Phi^H_m(R\sm S^V)\to \Phi^H_m(R)$ is multiplication
 by the element $\Phi^H(v)$.

 For $m=0$, we use this isomorphism to turn the class $\Phi^H(u)\in\Phi^H_0(R\sm S^V)$
 into a class $\bar u\in\Phi_{-l}^H(R)$. The relation \eqref{eq:Phi^H_of_rel}
 then becomes the relation $\bar u\cdot\Phi^H(v)=1$.
 So the class $\Phi^H(v)$ has a left inverse.
 Since the homotopy $G$-ring spectrum $R$ is commutative,
 the graded ring $\Phi^H_*(R)$ is graded-commutative, so $\Phi^H(v)$ also
 has a right inverse, and is thus a graded unit.

  (c)$\Longrightarrow$(a) 
  For every closed subgroup $H$ of $G$ and choice of orientation of $V^H$,
  the composite of the isomorphism \eqref{eq:induced_fix_iso}
  with $\Phi^H_m(v^\flat)\colon \Phi^H_m(R\sm S^V) \to\Phi^H_m(R)$
  is multiplication by the class $\Phi^H(v)$.
  Because $\Phi^H(v)$ is a graded unit, multiplication by $\Phi^H(v)$ is bijective,
  hence so is the map $\Phi^H_m(v^\flat)$.
  So the morphism $v^\flat$ induces isomorphisms of geometric fixed point groups
  for all closed subgroups of $G$. Hence $v^\flat$ is an equivalence of $G$-spectra,
  which means that $v$ is invertible.
\end{proof}

The next proposition establishes some expected properties of invertible classes.
For an orthogonal $G$-ring spectrum $R$, we can construct
multiplication maps on representation-graded homotopy groups
\begin{equation}\label{eq:ROpairing}
    \cdot \ : \ \pi_V^G(R)\times \pi_W^G(R)\ \to \ \pi_{V\oplus W}^G(R) 
\end{equation}
by multiplying representatives. More precisely,
if $f\colon S^{U\oplus V}\to R(U)$ and  $g\colon S^{U'\oplus W}\to R(U')$
are $G$-maps that represent classes $x=[f]\in \pi_V^G(R)$ and $y=[g]\in\pi_W^G(R)$,
respectively, then $x\cdot y$ is the class represented by
the following composite:
\[  S^{U\oplus U'\oplus V\oplus W}\ \xra{(U\oplus\tau_{U',V}\oplus W)_*}\
  S^{U\oplus V\oplus U'\oplus W}\
  \xra{f\sm g}\ R(U)\sm R(U')\ \xra{\mu_{U,U'}} \ R(U\oplus U')   \]
Here $\mu_{U,U'}$ is a component of the multiplication morphism,
in external form.

\begin{prop}\label{prop:invert_inherit}
  Let $V$ be a representation of a compact Lie group $G$,
  let $R$ be a commutative homotopy $G$-ring spectrum,
  and let $v\in\pi_V^G(R)$ be a representation-graded homotopy class.
  \begin{enumerate}[\em (i)]
  \item If the class $v$ is invertible, then for every
    morphism $i\colon R\to S$ of commutative homotopy $G$-ring spectra,
    the class $i_*(v)\in\pi_V^G(S)$ is invertible.
  \item If the class $v$ is invertible, and $H$ is a closed subgroup of $G$,
    then the class $\res^G_H(v)\in \pi^H_V(R)$ is invertible.
  \item Let $w\in\pi_W^G(R)$ be another representation-graded homotopy class.
    Then the product $v\cdot w$ is invertible if and only if both $v$ and $w$ are invertible.
  \end{enumerate}
\end{prop}
\begin{proof}
  (i) For every closed subgroup $H$ of $G$, the morphism $i$ induces
  a homomorphism of graded rings $\Phi^H_*(i)\colon \Phi^H_*(R)\to\Phi^H_*(S)$.
  Since $\Phi^H(v)$ is a graded unit by Proposition \ref{prop:invert_criterion},
  the class 
  \[ \Phi^H_*(i)(\Phi^H(v)) \ = \ \Phi^H(i_*(v)) \]
  is a graded unit in the graded ring $\Phi^H_*(S)$.
  Since this holds for all closed subgroups $H$ of $G$,
  Proposition \ref{prop:invert_criterion}
  shows that the class $i_*(v)$ is invertible. 

  (ii) The morphism $\res^G_H(v)^\flat\colon R\sm S^V\to R$
  is obtained from $v^\flat\colon R\sm S^V\to R$ by applying
  the restriction functor $\res^G_H\colon \Sc p_G\to\Sc p_H$.
  So if $v^\flat$ is an equivalence of $G$-spectra, then
  $\res^G_H(v)^\flat$ is an equivalence of $H$-spectra.

  (iii) We let $H$ be a closed subgroup of $G$.
  We choose orientations of $V^H$ and $W^H$, and then orient $(V\oplus W)^H$ by concatenating the
  chosen orientations.
  Then 
  \[ \Phi^H(v\cdot w)\ = \ \Phi^H(v)\cdot \Phi^H(w)  \]
  in $\Phi^H_{k+l}(R)$, where $k=\dim(V^H)$ and $l=\dim(W^H)$.
  So $\Phi^H(v\cdot w)$ is invertible in the graded ring $\Phi^H_*(R)$
  if and only if $\Phi^H(v)$ and $\Phi^H(w)$ are invertible.
  Since this holds for all closed subgroups $H$ of $G$,
  Proposition \ref{prop:invert_criterion} shows that
  $v\cdot w$ is invertible if and only if both $v$ and $w$ are invertible.
\end{proof}

\begin{construction}[Restriction homomorphisms]
  We let $X$ be an orthogonal spectrum, and we let $\gamma\colon K\to G$ be a continuous
  homomorphism of compact Lie groups. We define the {\em restriction homomorphism}
  \begin{equation}\label{eq:define_restriction}
    \gamma^*\ \colon\ \pi_V^G(X) \ \to \ \pi^K_{\gamma^*(V)}(X)    
  \end{equation}
  by sending the class of a $G$-map $f\colon S^{U\oplus V}\to X(U)$
  to the class of the $K$-map
  \[ S^{\gamma^*(U)\oplus \gamma^*(V)}\ = \
    \gamma^*(S^{U\oplus V})\ \to \ \gamma^*(X(U)) \ = \ X(\gamma^*(U))\ . \]
  As in the special case $V=0$ in \cite[Construction 3.1.15]{schwede:global},
  we exploit here that elements of $\pi_{\gamma^*(V)}^K(X)$
  can be unambiguously represented by $K$-maps $S^{W\oplus\gamma^*(V)}\to X(W)$
  for any $K$-representation $W$, not necessarily a subrepresentation of the chosen
  complete universe, see \cite[Construction 3.1.13]{schwede:global}.
  The restriction homomorphisms $\gamma^*$ are clearly contravariantly functorial
  for composition of group homomorphisms. When $\gamma$ is the inclusion of
  a subgroup $H\leq G$, then $\gamma^*$ specializes to the restriction homomorphism
  $\res^G_H$ considered above. The restriction homomorphism $\gamma^*$
  is a natural transformation between two functors from orthogonal spectra
  to abelian groups that both invert global equivalences.
  So $\gamma^*$ descends uniquely to a natural transformation
  between the descended functors from the $\infty$-category of global spectra
  to abelian groups; we abuse notation and denote the descended transformation
  by the same symbol.

  Whenever the continuous homomorphism $\gamma\colon K\to G$ is surjective,
  the geometric fixed point homomorphism \eqref{eq:geom_fix_map}
  matches the restriction homomorphism \eqref{eq:define_restriction}
  with the inflation homomorphism \eqref{eq:define_inflation}
  on geometric fixed point homotopy groups, in the sense of the commutative square:
  \begin{equation}\begin{aligned}\label{eq:infl_and_fix-ROG}
    \xymatrix{ \pi_V^G(X)\ar[r]^-{\gamma^*}\ar[d]_{\Phi^G} & \pi_{\gamma^*(V)}^K(X)\ar[d]^{\Phi^K}\\
      \Phi^G_k(X)\ar[r]_-{\gamma^*} &\Phi^K_k(X)    }      
    \end{aligned}
  \end{equation}
  It is implicit that we use the same orientation for $V^G=(\gamma^*(V))^K$.
  The square commutes at the level of representatives by direct inspection.
\end{construction}

By a {\em global homotopy ring spectrum} we mean
an algebra in the symmetric monoidal global stable homotopy category $h(\Sc p_{\gl}^\sm)$.
The strong symmetric monoidal functor
$\res_G^\sm\colon \Sc p_{\gl}^\sm\to \Sc p_G^\sm$
from \eqref{eq:define_res^sm_G}
that extracts the $G$-homotopy type of a global spectrum
induces a strong symmetric monoidal functor of homotopy categories
$h(\res_G^\sm)\colon h(\Sc p_{\gl}^\sm)\to h(\Sc p_G^\sm)$.
The functor $\Alg(h(\res^\sm_G))$ thus takes 
global homotopy ring spectra to homotopy $G$-ring spectra.
We define the representation-graded homotopy groups and the notion
of `invertibility' for  global homotopy ring spectra on the
underlying homotopy $G$-ring spectra.
We shall now show that invertible classes are then stable under restriction
along continuous homomorphisms of compact Lie groups.

\begin{prop}\label{prop:restriction_and_invert}
  Let $\gamma\colon K\to G$ be a continuous homomorphism of compact Lie groups.
  Let $R$ be a commutative global homotopy ring spectrum,
  and let $v\in \pi_V^G(R)$ be an invertible representation-graded homotopy class.
  Then the class $\gamma^*(v)$ in $\pi_{\gamma^*(V)}^K(R)$ is invertible.
\end{prop}
\begin{proof}
  We let $L$ be a closed subgroup of $K$,
  and we set   $H=\gamma(L)$, a closed subgroup of $G$.
  The restriction of $\gamma$ to $L$ factors as a continuous epimorphism
  $\beta\colon L\to H$,  followed by the inclusion of $H$ into $G$.
  Since $v$ is invertible, the class $\Phi^H(v)$ is a graded unit in the
  ring $\Phi^H_*(R)$, by Proposition \ref{prop:invert_criterion}.
  Since the inflation homomorphism $\beta^*\colon\Phi^H_*(R)\to\Phi^L_*(R)$
  is a morphism of graded rings, the class
  \[ \beta^*(\Phi^H(v)) \ = \ \Phi^L(\beta^*(\res^G_H(v)))
    \ = \ \Phi^L(\gamma^*(v)) \]
  is invertible in the graded ring $\Phi^L_*(R)$.
  The first equation is the compatibility of inflation with geometric fixed point
  homomorphisms  \eqref{eq:infl_and_fix-ROG}, applied to $\beta\colon L\to H$.
  Since this holds for all closed subgroups of $K$, the class
  $\gamma^*(v)$ is invertible by Proposition \ref{prop:invert_criterion}.
\end{proof}

By a {\em commutative global ring spectrum} we mean
a commutative algebra in the symmetric monoidal $\infty$-category $\Sc p_{\gl}^\sm$
of global spectra.
As we recall in Construction \ref{con:res_G},
every global spectrum has an underlying genuine $G$-spectrum,
and this process is implemented by a strong symmetric monoidal functor
$\res_G^{\sm} \colon \Sc p_{\gl}^{\sm}  \to \Sc p_G^{\sm}$
whose underlying functor $\res_G\colon \Sc p_{\gl}\to \Sc p_G$
has adjoints on either side, and thus preserves limits and colimits,
see Theorem \ref{thm:gl2G}.
We obtain an induced functor of commutative algebras
\[ \CAlg(\res_G^{\sm}) \ : \ \CAlg(\Sc p^\sm_{\gl})\ \to \ \CAlg(\Sc p_G^{\sm}) \]
from commutative global ring spectra to commutative $G$-ring spectra.
We define the representation-graded homotopy groups and the pairing 
\eqref{eq:ROpairing} as those for the underlying $G$-spectrum.

In the following a {\em set of representation-graded equivariant homotopy classes}
of a global spectrum $R$ will mean a set $\Xc=\{v\}$ of
classes in $\pi_V^G(R)$ for varying pairs $(G,V)$ consisting
of a compact Lie group $G$ and a $G$-representation $V$.
We will say that the class $v\in \pi_V^G(R)$ is {\em supported}
at the group $G$, or at the pair $(G,V)$.
The main example for this paper will be the set $\{b_m\}_{m\geq 1}$
of {\em pre-Bott classes} of the ultra-commutative ring spectrum $\Sigma^\infty_+\bP$
defined in Construction \ref{con:pre-Bott} below; the class $b_m$ is supported
at the unitary group $U(m)$ and its tautological representation $\nu_m$ on $\mC^m$.

\begin{defn}[Global localization]\label{def:global_loc}
  Let $R$ be a commutative global ring spectrum.
  Let $\Xc$ be a set of representation-graded equivariant homotopy classes of $R$.
  \begin{enumerate}[(i)]
  \item A morphism $i\colon R\to S$ of commutative global ring spectra
    {\em inverts $\Xc$} if for every
    element $v\in\Xc$ supported at $(G,V)$, the class $i_*(v)\in\pi_V^G(S)$
    is invertible in the sense of Definition \ref{def:RO-invertible}.
  \item 
  A {\em global localization by $\Xc$}
  is an initial object of the full subcategory of the
  slice category $R\bs\CAlg(\Sc p_{\gl}^\sm)$ spanned by the
  $\Xc$-inverting morphisms.
\end{enumerate}
\end{defn}

\begin{rk}
The kind of localizations of ring spectra we study first arose in
non-equivariant homotopy theory. A special feature in that context is that
the localization of $R$ by a homotopy class $u\in\pi_n(R)$,
or more precisely its underlying spectrum, can always be calculated as the colimit,
classically modeled by a mapping telescope, of the sequence
\[ R \ \xra{\ \cdot u}\ R\sm S^{-n}\ \xra{\ \cdot u}\ R\sm S^{-2n}\xra{\ \cdot u}\ \dots \]
of iterations of the graded endomorphism of  $R$ that lifts multiplication by $u$,
compare \cite[Proposition 4.3.17]{lurie:ellipticII}.
Consequently, the topological localization always realizes the algebraic localization
on graded homotopy rings.

\Danger We alert the reader that global localizations are {\em essentially never}
given by a sequential colimit over endomorphisms that resemble
`multiplication by the elements' to be inverted.
The issue is already present for localization at a class in
the equivariant homotopy rings of degree 0.
Much as in the $G$-equivariant context, the functor
$\pi_0^G\colon \Sc p_{\gl}\to\rAb$ is represented by a global spectrum.
However, a key difference is that the representing object $\Sigma^\infty_+ B_{\gl}G$,
the unreduced suspension spectrum of the global classifying space of $G$,
is {\em not} invertible in the global stable homotopy category (unless the group is trivial).
In fact, the Picard group of the global stable homotopy category is isomorphic
to $\mZ$, generated by the suspension of the global sphere spectrum,
see \cite[Theorem 4.5.5]{schwede:global}. In other words:
the only invertible global spectra are the `trivial spheres'. 
So if the class we wish to invert is supported at a nontrivial group,
there is not even any sequence of `graded endomorphisms' one might wish to take the colimit of.

The situation changes as soon as we fix a compact Lie group $G$
and pass to the underlying $G$-spectra.
As we show in Theorem \ref{thm:telescopic localization},
the spectrum $\bKU_G$ is a `Bott-inverted telescopic localization' of $\Sigma^\infty_+\bP_G$,
i.e., a sequential colimit of copies of $\Sigma^\infty_+\bP_G$ shifted by
the spheres of an exhaustive sequence of unitary representations,
along morphisms that represent multiplication by pre-Bott classes.
\end{rk}

The following theorem is our main tool for recognizing localizations.

\begin{theorem}[Localization recognition]\label{thm:loc_reg}
  Let $\Cc$ be a presentably symmetric monoidal $\infty$-category.
  Let $\{\Phi^j\colon\Cc\to\Sc p^\sm\}_{j\in J}$ be a set
  of strong symmetric monoidal left adjoints.
  Let $R$ be a commutative algebra in $\Cc$,
  and let $\Wc$ be a full subcategory of $R\bs\CAlg(\Cc)$.
  Let $i\colon R\to S$ be an object of $\Wc$.
  Suppose that the following conditions hold:
  \begin{enumerate}[\em (a)]
  \item The functors $\{\Phi^j\}_{j\in J}$ are jointly conservative.
  \item  The $\infty$-category $\Wc$ is closed under postcomposition with morphisms in $\CAlg(\Cc)$.
  \item For every $j\in J$, the morphism of graded-commutative rings
    $\pi_*(\Phi^j(i))\colon \pi_*(\Phi^j(R))\to \pi_*(\Phi^j(S))$
    is a localization by a set of homogeneous elements of the source,
    and for every object $f\colon R\to P$ of $\Wc$,
    the morphism $\pi_*(\Phi^j(f))\colon\allowbreak \pi_*(\Phi^j(R))\to \pi_*(\Phi^j(P))$
    factors through $\pi_*(\Phi^j(i))$.
  \end{enumerate}
  Then $i\colon R\to S$ is an initial object of $\Wc$.
\end{theorem}
\begin{proof}
  Throughout the proof we simplify notation by abbreviating $\pi_*(\Phi^j(R))$
  to $\Phi^j_*(R)$.
  In the examples where $\Phi^j$ is a geometric fixed point functor $\Phi^G$,
  the resulting notational ambiguity is resolved by Corollary \ref{cor:fix_ring_iso}.

  We let $f\colon R\to P$ be any object of $\Wc$.
  Since $\Cc$ is presentably symmetric monoidal, the $\infty$-category $\CAlg(\Cc)$
  has colimits.  We choose a pushout in $\CAlg(\Cc)$:
    \[  \xymatrix{ R\ar[r]^-i\ar[d]_f & S\ar[d]^g \\ 
P\ar[r]_-h & Q}       \]
  For every $j\in J$, the functor $\CAlg(\Phi^j)$ preserves colimits, so the square
  \[  \xymatrix@C=15mm{ \Phi^j(R)\ar[r]^-{\Phi^j(i)}\ar[d]_{\Phi^j(f)} & \Phi^j(S)\ar[d]^{\Phi^j(g)} \\
      \Phi^j(P)\ar[r]_-{\Phi^j(h)} & \Phi^j(Q)}       \]
  is a pushout of commutative ring spectra.
  By hypothesis, $\Phi_*^j(i)\colon \Phi_*^j(R)\to \Phi_*^j(S)$
  is a localization of graded-commutative rings by some multiplicative
  subset $T$ of homogeneous elements of $\Phi_*^j(R)$.
  So the morphism of commutative ring spectra
  $\Phi^j(i)\colon \Phi^j(R)\to \Phi^j(S)$ is also a localization by $T$,
  compare \cite[Proposition 4.3.17]{lurie:ellipticII}.
  Hence its cobase change
  $\Phi^j(h)\colon \Phi^j(P)\to \Phi^j(Q)$
  is a localization by the subset
  $\Phi^j_*(f)(T)=\{\Phi^j_*(f)(x)\mid x\in T\}$ of $\Phi^j_*(P)$.
  Since $\Phi_*^j(i)$ inverts the elements of $T$,
  and since $\Phi_*^j(f)$ factors through $\Phi_*^j(i)$,
  also $\Phi_*^j(f)$ inverts the elements of $T$.
  So $\Phi^j(h)$ is a localization by a set of elements that are already invertible,
  and hence an equivalence.
  So $h$ is an equivalence by the joint conservativity of the functors $\Phi^j$.
  
  Since $\Wc$ is closed under postcomposition, $q=h f\simeq g i\colon R\to Q$ lies in $\Wc$,
  and the pushout square exhibits $q$ as a coproduct of $f$ and $i$ in $\Wc$.
  As a coproduct, the map
  \[     \Wc(q,f) \ \to \ \Wc(f,f)\times\Wc(i,f)\]
  is an equivalence of animae.
  Since $h$ is an equivalence, the first projection is an equivalence.
  And since $\Wc(f,f)$ is non-empty, the anima $\Wc(i,f)$ must be contractible.
  Since $f$ was an arbitrary object of $\Wc$, this proves that $i\colon R\to S$
  is indeed an initial object of $\Wc$.
\end{proof}

We specialize the localization recognition of Theorem \ref{thm:loc_reg}
to our main case of interest, the presentably symmetric monoidal $\infty$-category
of global spectra.
The composite functor
\[  \Sc p_{\gl}^\sm \ \xra{\res^\sm_K}\  \Sc p^\sm_K \ \xra{\Phi^K}\ \Sc p^\sm \]
is a strong symmetric monoidal left adjoint.
To simplify the exposition we abusively abbreviate the induced functor
on $\infty$-categories of commutative algebras to 
\[ \Phi^K \ =  \ \CAlg(\Phi^K\circ \res^\sm_K)\ \colon\
  \CAlg(\Sc p_{\gl}^\sm)\ \to\ \CAlg(\Sc p^\sm)\ .\]
\Danger In the proof of the following theorem, and in many arguments that follow,
we shall systematically blur the distinction between two different, but naturally isomorphic,
geometric fixed point rings associated with a commutative $K$-ring spectrum $R$.
On the one hand, `geometric fixed point ring' could refer to the homotopy groups
of the geometric fixed point ring spectrum $\Phi^K(R)$ in $\CAlg(\Sc p^\sm)$.
On the other hand, it could refer to the explicit point-set definition
from \cite[Definition 3.3.1]{schwede:global},
which we recall in \eqref{eq:define_geom_fix} below,
descended from the 1-category of orthogonal $K$-spectra to the $\infty$-category of $K$-spectra.
These two interpretations are naturally and multiplicatively isomorphic,
by the isomorphism \eqref{eq:Phi_vs_Phi} that is characterized by
its effect on inflated $K$-spectra. So whenever we pretend that the
rings $\pi_*(\Phi^K(R))$ and $\Phi^K_*(R)$ were equal, we are implicitly applying
the preferred natural isomorphism of Corollary \ref{cor:fix_ring_iso}.

\begin{theorem}[Geometric fixed point criterion]\label{thm:fix_criterion}
  Let $i\colon R\to S$ be a morphism of commutative global ring spectra,
  and let $\Xc$ be a set of representation-graded equivariant homotopy classes of $R$.
  Suppose that for every compact Lie group $K$, the morphism
  of $K$-geometric fixed point rings
  $\Phi^K_*(i)\colon \Phi^K_*(R)\to \Phi^K_*(S)$
  is a localization by the set of classes $\Phi^K(\beta^*(x))$
  for all compact Lie groups $G$,
  all continuous homomorphisms $\beta\colon K\to G$,
  and all classes $x\in\Xc$ supported at $G$.
  Then the morphism $i\colon R\to S$ is a global localization
  by $\Xc$ in $\CAlg(\Sc p_{\gl}^\sm)$.
\end{theorem}
\begin{proof}
  We apply the localization recognition of Theorem \ref{thm:loc_reg}
  to the presentably symmetric monoidal $\infty$-category $\Sc p^\sm_{\gl}$
  of global spectra, and to the family of geometric fixed point functors
  $\Phi^K\circ\res_K^\sm\colon\Sc p^\sm_{\gl}\allowbreak\to \Sc p^\sm$,
  for $K$ ranging through representatives of all isomorphism classes of compact Lie groups.
  These functors are strong symmetric monoidal left adjoints, and jointly conservative.
  The full subcategory $\Wc$ of $R\bs\CAlg(\Sc p^\sm_{\gl})$
  is spanned by the $\Xc$-inverting morphisms; it is closed under postcomposition
  by Proposition \ref{prop:invert_inherit} (i).
  The morphism $i\colon R\to S$ lies in $\Wc$: for every $v\in \Xc$
  supported at $G$, and all closed subgroups $K$ of $G$,
  the class $\Phi^K(i_*(v))=\Phi^K_*(i)(\Phi^K(\res^G_K(v)))$ is invertible by hypothesis,
  and so the class $i_*(v)$ is invertible by Proposition \ref{prop:invert_criterion}.

  We claim that hypothesis (c) of Theorem \ref{thm:loc_reg} holds for every compact Lie group $K$.
  Indeed, $\Phi^K_*(i)\colon \Phi^K_*(R)\to \Phi^K_*(S)$
  is a localization by hypothesis, by the set of classes $\Phi^K(\beta^*(x))$
  as in the statement of the theorem.
  Moreover, if $f\colon R\to P$ belongs to $\Wc$, then it inverts $\Xc$,
  and hence it inverts all the classes $\Phi^K(\beta^*(x))$ by
  Propositions  \ref{prop:invert_criterion} and \ref{prop:restriction_and_invert}.
  Since  $\Phi^K_*(i)\colon \Phi^K_*(R)\to \Phi^K_*(S)$ is a localization by these
  classes, the morphism $\Phi^K_*(f)$ factors through $\Phi^K_*(i)$.
  So Theorem \ref{thm:loc_reg} applies and shows that $i\colon R\to S$
  is a global localization by $\Xc$ in $\CAlg(\Sc p^\sm_{\gl})$.
\end{proof}

The $\infty$-category of {\em ultra-commutative ring spectra} $\UR_{\gl}$
is the localization of the 1-category of commutative orthogonal ring spectra
at the class of global equivalences, see Definition \ref{def:com_global_ring} for more details.
In \eqref{eq:ucom2CR_gl} we exhibit a forgetful functor
\[ u_{\gl}\ :\ \UR_{\gl}\ \to\  \CAlg(\Sc p^\sm_{\gl})  \]
from the $\infty$-category of ultra-commutative ring spectra to the
$\infty$-category of commutative global ring spectra.
This functor is both monadic and comonadic,
by Theorem \ref{thm:UR_monadic} (iii).
We define the representation-graded homotopy groups 
for ultra-commutative ring spectra
as those of the underlying commutative global ring spectra,
and hence as the ones for the underlying $G$-spectra,
and similarly with the notion of invertibility.

The algebraic shadows of the additional structure that ultra-commutative
ring spectra enjoy are power operations and norm maps.
Power operations also enter into the construction of the higher pre-Bott classes,
see \eqref{eq:define_preBott}, so we review these now in the representation-graded context.
The following construction is much like in the special case  $V=0$ covered in
\cite[(5.1.2)]{schwede:global}.

\begin{construction}[Representation-graded power operations]
  We let $G$ be a compact Lie group.
  We recall that the wreath product $\Sigma_m\wr G$ with the symmetric group $\Sigma_m$ 
    is the semidirect product
  \[ \Sigma_m\wr G \ = \ \Sigma_m \ltimes G^m\] 
  formed with respect to the action of $\Sigma_m$ by permuting the
  factors of $G^m$. For every $G$-space $E$,
  the wreath product $\Sigma_m\wr G$ acts on the space $E^m$ by
  \[ (\sigma;\,g_1,\dots,g_m) \cdot (e_1,\dots, e_m) \ = \ 
    (g_{\sigma^{-1}(1)}e_{\sigma^{-1}(1)},\dots,g_{\sigma^{-1}(m)} e_{\sigma^{-1}(m)}) \ . \]
  For every $G$-representation $V$, this action even makes $V^m$ 
  a $(\Sigma_m\wr G)$-representation.
  
  For a commutative orthogonal ring spectrum $R$,
  and for a $G$-representation $U$, we let
  \[ \mu_{U,\dots,U}\ : \
    R(U)\sm\dots\sm R(U) \ \to \ R(U\oplus\dots\oplus U) \]
  denote the $(U,\dots,U)$-component of the iterated multiplication map of $R$.
  This map is $(\Sigma_m\wr G)$-equivariant because the multiplication
  of $R$ is commutative.
  The $m$-th power operation
  \begin{equation}\label{eq:define_P^m}
    P^m \ : \ \pi_V^G(R)\ \to \ \pi^{\Sigma_m\wr G}_{V^m}(R)    
  \end{equation}
  takes the class represented by a based $G$-map $f\colon S^{U\oplus V}\to R(U)$
  to the class of the composite $(\Sigma_m\wr G)$-map
  \begin{align*}
    S^{U^m\oplus V^m} \,  \iso \ S^{(U\oplus V)^m}\
    \xra{f^{\sm m}}\ R(U)^{\sm m} \  \xra{\mu_{U,\dots,U}} \ R(U^m) \ .
  \end{align*}
  The unnamed homeomorphism shuffles factors.
  Clearly, this power construction takes $G$-homotopic maps to 
  $(\Sigma_m\wr G)$-homotopic maps. And if we stabilize $f$ by a $G$-representation
  $W$, then the composite gets stabilized by the
  $(\Sigma_m\wr G)$-representation $W^m$. The upshot is thus a well-defined map
  \eqref{eq:define_P^m}.

  Much like for degree 0 homotopy classes in \cite[Remark 5.1.7]{schwede:global},
  the representation-graded power operations can be used to define
  representation-graded {\em norm maps}
  \[ N_H^G \ : \ \pi_W^H(R) \ \to \ \pi_{\ind_H^G(W)}^G(R)\ ,\]
  for any finite index subgroup $H$ of $G$, where $W$ is an $H$-representation,
  and $\ind_H^G(W)$ is the induced $G$-representation.
  We refrain from spelling out the details because we will not use
  these norm maps in this paper.
\end{construction}

The power operations \eqref{eq:define_P^m}
satisfy a whole list of relations that describe their interaction
with each other, as well as with addition, multiplication,
restriction homomorphisms and transfers.
For $V=0$, the relations can be summarized
by saying that the power operations make the ring-valued global functor $\upi_0(R)$
into a {\em global power functor} in the sense of \cite[Definition 5.1.6]{schwede:global},
see \cite[Theorem 5.1.11]{schwede:global}.
The power operations on representation-graded homotopy groups \eqref{eq:define_P^m}
satisfy suitable generalizations of the global power functor relations,
but we refrain from spelling these out.

For $m\geq 2$, the power operation \eqref{eq:define_P^m} is {\em not} additive,
but it is always multiplicative. It should thus not come as a surprise
that power operations preserve invertibility.
The next proposition formally proves that.

\begin{prop}\label{prop:invertibility_power}
  Let $R$ be an ultra-commutative ring spectrum,
  and let $V$ be a representation of a compact Lie group $G$.
  Let $v\in \pi_V^G(R)$ be an invertible representation-graded homotopy class.
  Then for all $m\geq 1$, the class $P^m(v)$ in $\pi_{V^m}^{\Sigma_m\wr G}(R)$ is invertible.
\end{prop}
\begin{proof}
  The morphism of $G$-spectra $v^\flat\colon R\sm S^V\to R$
  was defined in \eqref{eq:define_v_flat}.
  Proposition \ref{prop:invert_criterion} provides an element
  $u\in\pi_0^G(R\sm S^V)$ such that $v^\flat_*(u)=1$ in $\pi_0^G(R)$.
  We suppose that $u$ is represented by the 
  continuous based $G$-map $f\colon S^U\to R(U)\sm S^V$.
  We define $P^m(u)\in \pi_0^{\Sigma_m\wr G}(R\sm S^{V^m})$
  as the class represented by the $(\Sigma_m\wr G)$-map
  \[ 
    S^{U^m}\! =  (S^U)^{\sm m} \xra{\ f^{\sm m}} (R(U)\sm S^V)^{\sm m}
    \xra[\iso]{\text{shuffle}}  R(U)^{\sm m}\sm S^{V^m}
    \xra{\mu_{U,\dots,U}\sm S^{V^m}}\ R(U^m)\sm S^{V^m} .
  \]
  Then this class satisfies the relation
  \[  (P^m(v))^\flat_*(P^m(u)) \ = \ 1 \]
  in the group $\pi_0^{\Sigma_m\wr G}(R)$.
  So the class $P^m(u)$ witnesses criterion (b) of Proposition \ref{prop:invert_criterion}
  for the class $P^m(v)$, and thus $P^m(v)$ is invertible.
\end{proof}

\begin{defn}[Ultra-commutative localization]\label{def:ucom_loc}
  Let $R$ be an ultra-commutative ring spectrum, and
  let $\Xc$ be a set of representation-graded equivariant homotopy classes of $R$.
  An {\em ultra-commutative localization by $\Xc$}
  is an initial object of the full subcategory of the
  slice category $R\bs \UR_{\gl}$ spanned by the
  $\Xc$-inverting morphisms.
\end{defn}

\Danger  The localization process at a set of representation-graded homotopy classes
does {\em not} in general commute with forgetting from ultra-commutative
to commutative global ring spectra:

\begin{itemize}
\item if $i\colon R\to S$ is an ultra-commutative localization by $\Xc$,
  then the morphism of commutative global ring spectra $u_{\gl} i\colon u_{\gl} R\to u_{\gl} S$
  need not be a localization by $\Xc$ in $\CAlg(\Sc p^\sm_{\gl})$;
\item if $j\colon u_{\gl} R\to T$ is a global localization by $\Xc$ in $\CAlg(\Sc p^\sm_{\gl})$,
  then the localization $T$ need not admit any ultra-commutative structure.
\end{itemize}
In fact, the example  $R=\Sigma^\infty_+\bP$ and $\Xc=\{b_1\}$,
the set consisting only of the first pre-Bott class, illustrates both phenomena.
In Theorem \ref{thm:glcom b_1 loc} we identify the
commutative global localization of $\Sigma^\infty_+\bP$ by $b_1$
with $L_{\Ab}(\bKU)$, the left approximation of $\bKU$ by the global family
of abelian compact Lie groups.
We show in Example \ref{eg:Sigma_3_witnesses} that
$L_{\Ab}(\bKU)$ does not underlie any ultra-commutative ring spectrum,
witnessing the second item.
In particular, $L_{\Ab}(\bKU)$ is not equivalent to
the underlying global ring spectrum of the ultra-commutative
localization of  $\Sigma^\infty_+\bP$ by $b_1$, witnessing the first item.

The difference between the two kinds of localizations
is closely related to the additional operations:
inverting homotopy elements for ultra-commutative ring spectra
not only inverts all their restrictions along continuous homomorphisms of compact Lie groups,
but also all power operations of such restrictions, by Proposition \ref{prop:invertibility_power}.
We mention without proof that a morphism $i\colon R\to S$
is an ultra-commutative localization by $\Xc$ if and only if
the underlying morphism of commutative global ring spectra $u_{\gl} i\colon u_{\gl} R\to u_{\gl} S$
is a localization in $\CAlg(\Sc p^\sm_{\gl})$ by the set $\{P^m(x)\colon m\geq 1, x\in \Xc\}$
of all power operations of all elements in $\Xc$.\medskip

While in general the localizations as ultra-commutative ring spectra
and as commutative global ring spectra can differ,
there are examples that are simultaneously localizations in both worlds.
The next theorem describes a class of such examples.

A morphism $a\colon A\to B$ of animae is a {\em monomorphism}
if it is fully faithful, i.e., an inclusion of some of the path components.
Equivalently, the commutative square of animae
\[ \xymatrix{ A \ar[r]^{\Id_A} \ar[d]_{\Id_A} & A\ar[d]^-a\\ A\ar[r]_-a &  B } \]
is a pullback.
A morphism $i\colon R\to S$ in an $\infty$-category $\Cc$ is an {\em epimorphism}
if for every $\Cc$-object $T$, the map
$\Cc(i,T) \colon  \Cc(S,T) \to  \Cc(R,T)$ is a monomorphism of animae.
Equivalently, the commutative square
\[ \xymatrix{ R \ar[r]^-i\ar[d]_i & S\ar[d]^{\Id_S} \\
    S\ar[r]_-{\Id_S} & S} \]
is a pushout in $\Cc$.
For this paper, an important class of examples of epimorphisms
consists of all localizations of commutative global or of ultra-commutative
ring spectra at some set of representation-graded equivariant homotopy classes.

By Theorem \ref{thm:UR_monadic} (iii),
the forgetful functor $u_{\gl}\colon \UR_{\gl}\to\CAlg(\Sc p_{\gl}^\sm)$
has a left adjoint $\lambda_{\gl}\colon\allowbreak\CAlg(\Sc p_{\gl}^\sm)\to\UR_{\gl}$.
We write $\epsilon_R\colon \lambda_{\gl}(u_{\gl} R)\to R$ for the counit of the adjunction.
  
\begin{theorem}\label{thm:UR2CR}
  Let $i\colon R\to S$ be a morphism of ultra-commutative ring spectra.
  Let $\Xc$ be a set of representation-graded equivariant homotopy classes of $R$.
  \begin{enumerate}[\em (i)]
  \item If the underlying morphism of commutative global ring spectra
    $u_{\gl} i\colon u_{\gl} R\to u_{\gl} S$ is an epimorphism,
    then the following commutative square is a pushout in $\UR_{\gl}$:
    \[  \xymatrix@C=13mm{ \lambda_{\gl}(u_{\gl} R)\ar[r]^-{\lambda_{\gl}(u_{\gl} i)}\ar[d]_{\epsilon_R} &
      \lambda_{\gl}(u_{\gl} S)\ar[d]^{\epsilon_S}\\ R\ar[r]_-i & S}       \] 
\item Let $j\colon u_{\gl} R\to T$ be a global localization by $\Xc $ in $\CAlg(\Sc p_{\gl}^\sm)$,
  and let
  \[  \xymatrix@C=12mm{ \lambda_{\gl}(u_{\gl} R)\ar[r]^-{\lambda_{\gl}(j)}\ar[d]_{\epsilon_R} &
        \lambda_{\gl}(T)\ar[d]^g \\ R\ar[r]_-i & S}       \]
    be a pushout in $\UR_{\gl}$.
    Then $i\colon R\to S$ is an ultra-commutative  lo\-cal\-iza\-tion
    by $\Xc$ in $\UR_{\gl}$.
  \item If the underlying morphism of commutative global ring spectra
    $u_{\gl} i\colon u_{\gl} R\to u_{\gl} S$ is a global localization by $\Xc$ in $\CAlg(\Sc p_{\gl}^\sm)$,
    then $i$ is an ultra-commutative localization by $\Xc$ in $\UR_{\gl}$.
  \end{enumerate}
\end{theorem}
\begin{proof}
  (i)
  We choose a pushout of $\epsilon_R$ and $\lambda_{\gl}(u_{\gl} i)$ in $\UR_{\gl}$:
  \begin{equation}\begin{aligned}\label{eq:choose_po}
    \xymatrix@C=13mm{ \lambda_{\gl}(u_{\gl} R)\ar[r]^-{\lambda_{\gl}(u_{\gl} i)}\ar[d]_{\epsilon_R} &
      \lambda_{\gl}(u_{\gl} S)\ar[d]^g \\ R\ar[r]_-j & Q} 
    \end{aligned} \end{equation}
  Since $\lambda_{\gl}$ is a left adjoint, it preserves pushouts and thus epimorphisms,
  so $\lambda_{\gl}(u_{\gl} i)$ is an epimorphism in $\UR_{\gl}$.
  Hence also its cobase change $j$ is an epimorphism.
  Since the functor $u_{\gl}$ is a left adjoint, it preserves pushouts and thus epimorphisms,
  so $u_{\gl}(\lambda_{\gl}(u_{\gl} i))$ and $u_{\gl} j$ are epimorphisms in $\CAlg(\Sc p^\sm_{\gl})$.

  The universal property of the pushout \eqref{eq:choose_po} provides a unique morphism
  $h\colon Q\to S$ equipped with homotopies
  \[ h\circ j \ \sim \ i \text{\qquad and\qquad} h\circ g \ \sim \ \epsilon_S\ . \]
  Since $(u_{\gl} h)\circ (u_{\gl} j)\sim u_{\gl} i$ and both $u_{\gl} j$ and $u_{\gl} i$ are epimorphisms,
  also $u_{\gl} h\colon u_{\gl} Q\to u_{\gl} S$ is an epimorphism.
  We write $\eta_P\colon P\to u_{\gl}(\lambda_{\gl}(P))$ for the unit of the adjunction.
  The sequence of homotopies
  \begin{align*} (u_{\gl} g)\circ\eta_{u_{\gl} S}\circ(u_{\gl} h)\circ (u_{\gl} j)\
    &\sim \ (u_{\gl} g)\circ\eta_{u_{\gl} S}\circ(u_{\gl} i)\\
    &\sim \  (u_{\gl} g)\circ(u_{\gl}(\lambda_{\gl}(u_{\gl} i)))\circ\eta_{u_{\gl} R} \\
    &\sim \  (u_{\gl} j)\circ (u_{\gl}(\epsilon_R))\circ\eta_{u_{\gl} R}    \ \sim \ u_{\gl} j
  \end{align*}
  and the fact that $u_{\gl} j$ is an epimorphism show that
  the composite
  \[  u_{\gl} Q \xra{\ u_{\gl} h\ }\ u_{\gl} S \ \xra{\ \eta_{u_{\gl} S}\ }\ u_{\gl}(\lambda_{\gl}(u_{\gl} S))\
    \xra{u_{\gl} g}\ u_{\gl} Q \]
  is homotopic to the identity.
  So $u_{\gl} h$ is an epimorphism with a retraction, and thus an equivalence.
  Since $u_{\gl}$ is conservative, also $h\colon Q\to S$ is an equivalence.
  Since the square \eqref{eq:choose_po} is a pushout, so is the square of the theorem.

  (ii)
  For every ultra-commutative ring spectrum $E$, the upper rectangle
  in the commutative diagram of animae
  \[  \xymatrix@C=12mm{
      \UR_{\gl}(S,E)\ar[r]^-{i^*}\ar[d]_{g^*} & \UR_{\gl}(R,E)\ar[d]^{\epsilon_R^*}
      \ar@<4ex>@/^3pc/[dd]^(.3){u_{\gl}}\\
       \UR_{\gl}(\lambda_{\gl}(T),E)\ar[r]_-{\lambda_{\gl}(j)^*}\ar[d]_{\sim} &
       \UR_{\gl}(\lambda_{\gl}(u_{\gl} R),E)\ar[d]^\sim \\
              \CAlg(\Sc p_{\gl}^\sm)(T,u_{\gl} E)\ar[r]_-{j^*}&
        \CAlg(\Sc p_{\gl}^\sm)(u_{\gl} R,u_{\gl} E) }       \]
   is a pullback, by the universal property of pushouts.
   The two lower vertical maps are the adjunction equivalences.
   So the outer rectangle is a pullback, too.
   Since $j$ is a localization by $\Xc$ in $\CAlg(\Sc p_{\gl}^\sm)$,
   the lower horizontal map is an equivalence onto the path components
   of $\Xc$-inverting morphisms.
   A morphism $R\to E$ in $\UR_{\gl}$ is $\Xc$-inverting if and
   only if the underlying morphism in $\CAlg(\Sc p_{\gl}^\sm)$ is 
   $\Xc$-inverting.
   Hence the upper horizontal map is an equivalence onto the path components
   of $\Xc$-inverting morphisms. So $i$ has the defining property
   of an ultra-commutative localization by $\Xc$.

   (iii) Since localizations are epimorphisms, the square of part (i)
   is a pushout in $\UR_{\gl}$.
   Since $u_{\gl} i\colon u_{\gl} R\to u_{\gl} S$ is a global localization by $\Xc$ in $\CAlg(\Sc p_{\gl}^\sm)$
   by hypothesis, 
   $i$ is an ultra-commutative localization by $\Xc$ in $\UR_{\gl}$ by (ii).   
\end{proof}

\section{The pre-Bott classes}

To make sense of the slogan that global $K$-theory is obtained by `inverting the Bott classes',
we need to specify the classes we are inverting.
A naive attempt would be to invert the non-equivariant Bott class
in $\pi_2(\Sigma^\infty_+ \mC P^\infty)$ that features in Snaith's theorem.
However, inverting the classical Bott class does {\em not} yield the global $K$-theory spectrum,
not even at abelian compact Lie groups.
In this section we introduce the {\em pre-Bott classes}
$b_m\in \pi_{\nu_m}^{U(m)}(\Sigma^\infty_+\bP)$,
certain representation-graded equivariant homotopy classes
of $\Sigma^\infty_+ \bP$, see \eqref{eq:define_preBott}.  
The name is justified by the fact that
the morphism $\eta\colon \Sigma^\infty_+ \bP \to \bKU$ takes the pre-Bott class $b_m$
to the equivariant Bott class of the tautological $U(m)$-representation,
see Remark \ref{rk:preBott2Bott}.
One of our main results, Theorem \ref{thm:main},
is that the morphism $\eta$ is a localization,
both as commutative global ring spectra and as ultra-commutative ring spectra,
by the classes $b_m$ for $m\geq 1$.

The first pre-Bott class $b_1$ in $\pi_{\nu_1}^{T}(\Sigma^\infty_+\bP)$
is somewhat more distinguished than the higher pre-Bott classes.
On the one hand, the class $b_1$ is constructed first, and then
the class $b_m$ for $m\geq 2$ is defined from $b_1$ as a power operation
transferred up to $U(m)$, see \eqref{eq:define_preBott}.
On the other hand, $b_1$ is the unique class in
$\pi_{\nu_1}^T(\Sigma^\infty_+\bP)$ that maps to the equivariant Bott class
of the tautological $T$-representation, see Remark \ref{rk:b_1_unique};
for $m\geq 2$, the map $\eta_*\colon \pi_{\nu_m}^{U(m)}(\Sigma^\infty_+\bP)\to
\pi_{\nu_m}^{U(m)}(\bKU)$ has a nontrivial kernel, so the $U(m)$-equivariant
Bott class of $\nu_m$ has other preimages besides $b_m$.

Inverting only $b_1$ already goes quite a long way towards global $K$-theory:
as we show in Theorem \ref{thm:glcom b_1 loc},
the commutative global localization of $\Sigma^\infty_+\bP$ by $b_1$ is
$\Ab$-globally equivalent to $\bKU$,
where $\Ab$ is the global family of abelian compact Lie groups.
And we show in Theorem \ref{thm:ucom b_1 loc} that the ultra-commutative localization 
of $\Sigma^\infty_+\bP$ by $b_1$ is even $\Ab^\circ$-globally equivalent to $\bKU$,
where $\Ab^\circ$ is the global family of compact Lie groups
whose identity component is abelian; this family in particular includes
all finite groups and all abelian compact Lie groups.\medskip

A central player in this paper is the ultra-commutative monoid $\bP$;
we recall its definition in Construction \ref{con:P}.
We have settled on the notation $\bP$ because it is short and the letter `$P$'
also features in $\mC P^\infty$, which is its underlying non-equivariant homotopy type.
As $\bP$ is a global classifying space for the circle group
$T=U(1)$, other popular names for this global homotopy type are
$B_{\gl}T$, $\mathbb B T$ or  $*/\!/T$.
In \cite[(2.3.20)]{schwede:global}, we use the notation $\bP^\mC$ to distinguish this version
made from complex projective spaces from the real version.
In this paper, the real version plays no role, so we simplify notation
and drop the superscript `$\mC$'.
It is important for our purposes that the global classifying space
of the circle group (and in fact of any abelian compact Lie group)
can be endowed with an ultra-commutative global multiplication.
In informal terms,
an ultra-commutative multiplication consists of compatible $E_\infty$-multiplications
on the underlying $G$-spaces, natural for restriction along continuous
homomorphisms of compact Lie groups, together with
transfer morphisms for finite index inclusions of compact Lie groups.

An {\em ultra-commutative monoid} is a commutative orthogonal monoid space
in the sense of \cite[Definition 2.1.1]{schwede:global}.
In other words, it is a commutative monoid
in the category of orthogonal spaces under the box product
(convolution product). By the universal property of convolution products,
this data is equivalent to a lax symmetric monoidal functor from
the linear isometries category (under orthogonal direct sum)
to the category of spaces (under cartesian product).
We use the term `ultra-commutative monoid' to emphasize that we
consider objects in this 1-category through the eyes of global equivalences
in the sense of \cite[Definition 1.1.1]{schwede:global}.
Consequently, the $\infty$-category of {\em ultra-commutative monoids}
is the higher categorical localization of the 1-category of
commutative orthogonal monoid spaces at the class of global equivalences.

\begin{construction}[The global ultra-commutative monoid $\bP$]\label{con:P}
  We let 
  \[ \Sym(V_\mC) \ = \ {\bigoplus}_{i\geq 0}\,\Sym^i(V_\mC) \]
  denote the symmetric algebra of the complexification of an inner product space $V$.
  The value of the ultra-commutative monoid $\bP$ at $V$ is then
  \[ \bP(V)\ = \ P(\Sym(V_\mC)) \ ,\]
  the complex projective space of the symmetric algebra.
  The structure map $\bP(\varphi)\colon \bP(V)\to\bP(W)$
  induced by a linear isometric embedding $\varphi\colon V\to W$
  is given by  
  \[ \bP(\varphi)(L) \ = \ \Sym(\varphi_\mC)(L)\ , \]
  where $\Sym(\varphi_\mC)\colon \Sym(V_\mC)\to \Sym(W_\mC)$
  is the induced map of symmetric algebras.   
  The symmetric algebra $\Sym(V_\mC)$ is a commutative $\mC$-algebra without zero divisors.
  So the multiplication of $\Sym(V_\mC)$ descends to a well-defined abelian monoid structure 
  \begin{align*}
     \bP(V)\times \bP(V) = P(\Sym(V_\mC))\times P(\Sym(V_\mC))\ &\to \ P(\Sym(V_\mC)) =\bP(V)\\
    (\mC\!\cdot\! x,\mC\!\cdot\! y)\qquad &\longmapsto \qquad \mC\!\cdot\!x y
  \end{align*}
  on the projective space. The map $\bP(\varphi)$
  associated to a linear isometric embedding $\varphi$
  is a monoid homomorphism, so $\bP$ becomes a continuous functor
  from the linear isometries category to topological abelian monoids.
  As explained in \cite[Proposition 2.3.5]{schwede:global},
  this structure gives rise to an ultra-commutative multiplication on $\bP$,
  whose $(V,W)$-component is the composite
  \[ \mu_{V,W}\ : \ \bP(V)\times\bP(W)\ \xra{\bP(i_V)\times\bP(i_W)} \
    \bP(V\oplus W)\times\bP(V\oplus W)\ \xra{\text{multiply}} \
    \bP(V\oplus W)\ .  \]
  Here $i_V\colon V\to V\oplus W$ and $i_W\colon W\to V\oplus W$ are the direct summand embeddings.
  The multiplicative unit is the unique element of $\bP(0)=P(\Sym(0))=\{\mC\!\cdot\! 1\}$.
\end{construction}

The underlying global space of $\bP$ is a global classifying space,
in the sense of \cite[Definition 1.1.27]{schwede:global},
for the circle group $T=U(1)$; in particular, the associated Rep-functor
$\upi_0(\bP)$ is represented by $T$, see \cite[Proposition 1.5.12]{schwede:global}.
Since we will need the representability isomorphism, we review it in more detail.
We let $\nu_1$ denote the tautological unitary $T$-representation on $\mC$ by scalar multiplication.
We write $\mL$ for the complex line in the linear summand 
of $\Sym(u(\nu_1)_\mC)$ spanned by the vector $1\tensor 1 - i\tensor i$.
This line $\mL$ is $T$-invariant, and thus a $T$-fixed point of $\bP(u(\nu_1))$. We write
\[ u_T \ = \ [\mL]\ \in \ \pi_0^T(\bP) \]
for the class of this $T$-fixed point, the {\em unstable tautological class}.
Then for every compact Lie group $G$, the map
\begin{equation}\label{eq:Hom(G,T)}
\Hom(G,T)\ \to \  \pi_0^G(\bP) \ , \quad \alpha \ \longmapsto \ \alpha^*(u_T)  
\end{equation}
is bijective, where $\Hom(G,T)$ is the set of continuous homomorphisms.
The {\em stable tautological class}
\begin{equation} \label{eq:define_e_T}
 e_T\ = \ \sigma^T(u_T)\in\  \pi_0^T(\Sigma^\infty_+\bP)   
\end{equation}
is the stabilization \cite[(3.3.12)]{schwede:global}
of $u_T$, i.e., the one represented by the $T$-equivariant map
\[ S^{\nu_1}\ \xra{-\sm\mL}\ S^{\nu_1}\sm \bP(u(\nu_1))_+\ = \ (\Sigma^\infty_+\bP)(u(\nu_1))\ .\]
Because $\bP$ is a global classifying space for the circle group $T$,
the pair $(\Sigma^\infty_+\bP,e_T)$ represents the  functor $\pi_0^T$
on the global stable homotopy category, compare \cite[Theorem 4.4.3]{schwede:global}.

The unreduced suspension spectrum functor $\Sigma^\infty_+$ from
orthogonal spaces to orthogonal spectra \cite[Construction 4.1.7]{schwede:global}
is strong symmetric monoidal.
So the unreduced suspension spectrum of a commutative orthogonal monoid
space becomes a commutative orthogonal ring spectrum.
In particular, the orthogonal spectrum $\Sigma^\infty_+ \bP$ comes with
commutative multiplication, thereby modeling an ultra-commutative ring spectrum.

\begin{construction}[The pre-Bott classes]\label{con:pre-Bott}  
  We let $\nu_1$ denote the tautological representation of $T=U(1)$
  on $\mC$ by scalar multiplication.
  The cofiber sequence of based $T$-spaces
  \begin{equation} \label{eq:Tcof}
    T_+ \ \to \ S^0 \ \xra{\ i \ } \ S^{\nu_1} 
  \end{equation}
  induces a long exact sequence in any $T$-equivariant cohomology theory;
  here $i\colon S^0\to S^{\nu_1}$ is the inclusion of $T$-fixed points.
  For global spectra, the restriction homomorphism $\res^T_{\{1\}}$
  is split by the inflation homomorphism 
  $p_T^*\colon\pi_*(\Sigma^\infty_+\bP)\to\pi_*^T(\Sigma^\infty_+\bP)$
  along the unique homomorphism $p_T\colon T\to\{1\}$, see \eqref{eq:define_restriction}.
  So the long exact sequence splits off a short exact sequence:
  \[   \xymatrix@C=10mm{
      0\ar[r] & \pi_{\nu_1}^T(\Sigma^\infty_+\bP)\ar[r]^-{i^*} &
      \pi_0^T(\Sigma^\infty_+\bP)\ar[r]^-{\res^T_{\{1\}}}&
      \pi_0(\Sigma^\infty_+\bP)\ar[r] & 0 }   \]
  The class $1=p_T^*(\res^T_{\{1\}}(e_T))$ is the multiplicative
  unit of the ring $\pi_0^T(\Sigma^\infty_+\bP)$. Because $\res^T_{\{1\}}(1- e_T)=0$,
  there is a unique class, the first {\em pre-Bott class} $b_1\in \pi_{\nu_1}^T(\Sigma^\infty_+\bP)$
  such that
  \[     i^*(b_1)\ = \ 1- e_T\ \in \pi_0^T(\Sigma^\infty_+\bP)\ . \]
  For $m\geq 2$, we define the $m$-th pre-Bott class by transferring
  the $m$-th power operation of $b_1$ from $\Sigma_m\wr T$ to $U(m)$.
  To this end we identify $\Sigma_m\wr T$ with a maximal torus normalizer of $U(m)$
  via the preferred embedding
  \begin{equation}\label{eq:preferred_embed}
    \Sigma_m\wr T\ \to \ U(m)    
  \end{equation}
  that embeds $T^m$ as diagonal matrices,
  and that sends elements of $\Sigma_m$ to the associated permutation matrices.
  The image of this embedding then consists of the {\em monomial matrices},
  i.e., those unitary matrices
  that have precisely one nonzero entry in every row and in every column.
  The {\em $m$-th pre-Bott class} is then defined as
  \begin{equation}\label{eq:define_preBott}
    b_m \ = \ \tr_{\Sigma_m\wr T}^{U(m)}(P^m(b_1))\ \in \ \pi_{\nu_m}^{U(m)}(\Sigma^\infty_+\bP) \ .
  \end{equation}
  The power operation $P^m$ was defined in \eqref{eq:define_P^m}.
  The gradings work out because the restriction of the tautological $U(m)$-representation
  $\nu_m$ to $\Sigma_m\wr T$ is the representation $\nu_1^m$.
  For definiteness, we also set $b_0=1$ in $\pi_0^{\{1\}}(\Sigma^\infty_+\bP)$.
\end{construction}

\begin{rk}[Uniqueness of $b_1$]\label{rk:b_1_unique}
  Since $\eta\colon\Sigma^\infty_+\bP\to\bKU$ is a morphism of global spectra,
  it induces a commutative diagram of short exact sequences:
  \[ \xymatrix{
      0\ar[r] & \pi_{\nu_1}^T(\Sigma^\infty_+\bP)\ar[r]^-{i^*} \ar[d]_{\eta_*}&
      \pi_0^T(\Sigma^\infty_+\bP)\ar[r]^-{\res^T_{\{1\}}}\ar[d]_{\eta_*}&
      \pi_0(\Sigma^\infty_+\bP)\ar[r] \ar[d]^{\eta_*}& 0   \\
      0\ar[r] & \pi_{\nu_1}^T(\bKU)\ar[r]_-{i^*} &
      \pi_0^T(\bKU)\ar[r]_-{\res^T_{\{1\}}}&
      \pi_0(\bKU)\ar[r] & 0     } \]
  The rings $\pi_0(\Sigma^\infty_+\bP)$ and $\pi_0(\bKU)$
  are both isomorphic to $\mZ$, and the right vertical map is a ring homomorphism,
  and hence an isomorphism.
  Both the rings $\pi_0^T(\Sigma^\infty_+\bP)$ and $\pi_0^T(\bKU)$
  are Laurent polynomial rings over $\mZ$ in one variable,
  namely the stable tautological class $e_T$
  in the first case by \cite[Corollary 4.1.13]{schwede:global},
  and its image $\eta_*(e_T)=[\nu_1]$ in the second case.
  So also the middle vertical map is a ring isomorphism.
  The five lemma lets us conclude that the left vertical map is an isomorphism.
  We show in Remark \ref{rk:preBott2Bott} below that the map $\eta_*$
  takes the pre-Bott class $b_1$ to the $T$-equivariant Bott class $\beta_{T,\nu_1}$
  of the $T$-representation $\nu_1$.
  By injectivity, $b_1$ is the unique preimage of $\beta_{T,\nu_1}$.
\end{rk}

In the remainder of this section, we establish various fundamental properties
of the pre-Bott classes, such as their behavior under the power morphisms
of $\bP$ (Proposition \ref{prop:phi^m_on_b_1}),
multiplicativity properties (Theorem \ref{thm:fix of pre-Bott}),
and incompatibility with power operations (Example \ref{eg:norm_does_not_match}).
For every $k\geq 1$, the collection of maps
\begin{equation}\label{eq:define_power_P}
\phi^k(V)\ \colon \ \bP(V)\ \to \ \bP(V)\ , \quad \phi^k(V)(L)\ = \ L^k  
\end{equation}
forms a morphism of ultra-commutative monoids
$\phi^k\colon \bP\to\bP$, the {\em $k$-th power morphism}.
Hence $\Sigma^\infty_+\phi^k\colon\Sigma^\infty_+\bP\to \Sigma^\infty_+\bP$
is a morphism of ultra-commutative ring spectra.
  
\begin{prop}\label{prop:phi^m_on_b_1}
  For every $k\geq 1$, the relation
  \[ (\Sigma^\infty_+\phi^k)_*(e_T)\ = \ e_T^k \]
  holds in the ring $\pi_0^T(\Sigma^\infty_+\bP)$, and the relation
  \[ (\Sigma^\infty_+\phi^k)_*(b_1)\ = \ (1+e_T+\dots+e_T^{k-1})\cdot b_1 \]
  holds in the group $\pi_{\nu_1}^T(\Sigma^\infty_+\bP)$.
\end{prop}
\begin{proof}
   Since $\phi^k\colon \bP\to\bP$ is objectwise the $k$-th power map,
  its effect on $\pi_0^G(\bP)$ is the $k$-th power map in the abelian monoid structure,
  for every compact Lie group $G$. In particular, the effect of $\phi^k$
  on the unstable tautological class $u_T \in \pi_0^T(\bP)$ is given by
  \[ \phi^k_*(u_T)\ = \  u_T^k \ . \]
  The stabilization map $\sigma^T\colon \pi_0^T(\bP)\to \pi_0^T(\Sigma^\infty_+\bP)$
  is multiplicative, and the stable tautological class $e_T=\sigma^T(u_T)$
  is the stabilization of $u_T$. So the first relation
  $(\Sigma^\infty_+\phi^k)_*(e_T) =  e_T^k$ follows.

  The $T$-equivariant fixed point inclusion $i\colon S^0\to S^{\nu_1}$
  induces a commutative square:
  \[ \xymatrix@C=12mm{
      \pi_{\nu_1}^T(\Sigma^\infty_+\bP)\ar[r]^-{ (\Sigma^\infty_+\phi^k)_*}\ar[d]_{i^*} &
          \pi_{\nu_1}^T(\Sigma^\infty_+\bP)\ar[d]^{i^*}\\
      \pi_0^T(\Sigma^\infty_+\bP)\ar[r]_-{(\Sigma^\infty_+\phi^k)_*} &
          \pi_0^T(\Sigma^\infty_+\bP)
    }    \]
  We thus obtain the relation
  \begin{align*}
    i^*((\Sigma^\infty_+\phi^k)_*(b_1)) \
    &= \ (\Sigma^\infty_+\phi^k)_*(i^*(b_1)) \ = \ (\Sigma^\infty_+\phi^k)_*(1-e_T) \ = \ 1-e_T^k\\
    &= \ (1+e_T+\dots+e_T^{k-1}) \cdot (1-e_T) \\
    &= \  (1+e_T+\dots+e_T^{k-1}) \cdot i^*(b_1) \\
    &= \ i^*( (1+e_T+\dots+e_T^{k-1})\cdot b_1) \ .
  \end{align*}
  The second relation then follows because the restriction map $i^*$ is injective.
\end{proof}

The pre-Bott class $b_m$ for $m\geq 2$ is defined by a transfer; so identifying its restriction
to a subgroup of $U(m)$ involves a double coset formula.
The following double coset formula
ought to be well-known to experts, but I do not know a reference.
The case $m=n=1$ is classical because $T\times T$ is a maximal
torus of $U(2)$, and $\Sigma_2\wr T$ is the maximal torus normalizer;
so for $m=n=1$, the following double coset formula is a special case of
\cite[IV Corollary 6.7 (i)]{lms}.
In the following we write 
\[ U(m,n)\ = \ \{(\begin{smallmatrix}A & 0 \\ 0 & B \end{smallmatrix})\colon (A,B)\in U(m)\times U(n)\} \]
for the block subgroup of $U(m+n)$.
The embedding \eqref{eq:preferred_embed}
takes $(\Sigma_m\wr T)\times (\Sigma_n\wr T)$ into $U(m,n)$.

\begin{prop}\label{prop:double_coset}
  For all $m,n\geq 1$, the following double coset
  formula holds:
  \[ \res^{U(m+n)}_{U(m,n)}\circ\tr^{U(m+n)}_{\Sigma_{m+n}\wr T}\ = \
    \tr^{U(m,n)}_{(\Sigma_m\wr T)\times (\Sigma_n\wr T)}\circ
     \res^{\Sigma_{m+n}\wr T}_{(\Sigma_m\wr T)\times (\Sigma_n\wr T)}
   \]
\end{prop}
\begin{proof}
  We refer to \cite[IV Theorem 6.3]{lms} or \cite[Theorem 3.4.9]{schwede:global} for the general
  double coset formula for $\res^G_K\circ\tr_H^G$ for two closed subgroups $H$ and $K$
  of a compact Lie group $G$; we specialize it to the situation at hand.
  We write $\Lc_{m+n}$ for the space of sets $\{L_1,\dots,L_{m+n}\}$
  of $m+n$ pairwise orthogonal lines in $\mC^{m+n}$.
  The tautological $U(m+n)$-action on $\mC^{m+n}$ induces a transitive action on $\Lc_{m+n}$,
  and the stabilizer of the set of coordinate lines
  is the subgroup $\Sigma_{m+n}\wr T$, embedded into $U(m+n)$ by \eqref{eq:preferred_embed}.
  So the map
  \begin{equation}\label{eq:coset2L_n}
         U(m+n)/(\Sigma_{m+n}\wr T) \ \iso \ \Lc_{m+n}\ , \quad
    A\cdot(\Sigma_{m+n}\wr T)\ \longmapsto \ \{\mC\cdot A\cdot e_1,\dots,\mC\cdot A\cdot e_{m+n}\}
  \end{equation}
  that sends $A\cdot(\Sigma_{m+n}\wr T)$
  to the set of lines spanned by the columns of $A$
  is an equivariant homeomorphism.

  Now we let $\{L_1,\dots,L_{m+n}\}$ be any set of $m+n$ pairwise orthogonal lines in $\mC^{m+n}$.
  We consider the special case where at least one line $L_i$
  is neither contained in $\mC^m\oplus 0^n$ nor in $0^m\oplus \mC^n$.
  We claim that then the intersection of the stabilizer of the set $\{L_1,\dots,L_{m+n}\}$
  with the center of $U(m,n)$ is 1-dimensional.
  The center $Z(U(m,n))$ of $U(m,n)$ consists of
  the matrices  $(\lambda E_m)\oplus (\mu E_n)$ for all $\lambda,\mu\in T$.
  Since some line $L_i$ is neither contained
  in $\mC^m\oplus 0^n$ nor in $0^m\oplus \mC^n$,
  it is not stabilized by the matrix $(\lambda E_m)\oplus (\mu E_n)$
  for any $\lambda,\mu\in T$ with $\lambda\ne \mu$.
  So the intersection of the stabilizer of $L_i$ with $Z(U(m,n))$ 
  equals the center of $U(m+n)$.
  Since any line is stabilized by the center of $U(m+n)$,
  also the intersection of the stabilizer of the tuple $(L_1,\dots,L_{m+n})$
  of lines with $Z(U(m,n))$ equals the center of $U(m+n)$,
  and is thus 1-dimensional.
  The stabilizer of the tuple $(L_1,\dots,L_{m+n})$
  has finite index in the stabilizer of the set $\{L_1,\dots,L_{m+n}\}$.
  So the intersection of the stabilizer of the set $\{L_1,\dots,L_{m+n}\}$
  with $Z(U(m,n))$ is also 1-dimensional, as claimed.

  For every compact Lie group $G$,
  the center $Z(G)$ of $G$ normalizes any given closed subgroup $H$,
  so the group $Z(G)/(H\cap Z(G))$ injects into the Weyl group $W_G H=(N_G H)/H$.
  The center of $G=U(m,n)$ is isomorphic to $T\times T$, and hence 2-dimensional.
  We take $H=\text{Stab}\{L_1,\dots,L_{m+n}\}\cap U(m,n)$
  as the intersection of the stabilizer of the set $\{L_1,\dots,L_{m+n}\}$
  with $U(m,n)$; the claim we just established says that $H\cap Z(U(m,n))$
  is 1-dimensional. So the group $Z(U(m,n))/(H\cap Z(U(m,n)))$
  is 1-dimensional, and the Weyl group of $\text{Stab}\{L_1,\dots,L_{m+n}\}\cap U(m,n)$
  inside $U(m,n)$ is infinite.
  This means that the $U(m,n)$-orbit of the set $\{L_1,\dots,L_{m+n}\}$
  does not contribute to the double coset formula.
  
  We conclude that all contributions to the double coset formula
  come from sets $\{L_1,\dots,L_{m+n}\}$ of pairwise orthogonal lines
  for which each line is either contained in $\mC^m\oplus 0^n$ or in
  $0^m\oplus\mC^n$. The group $U(m,n)$ acts transitively on these,
  so they only contribute the principal summand to the double coset formula.
\end{proof}

\begin{construction}[Pre-Bott classes of unitary representations]\label{con:pre-Bott_of_V}  
  We let $V$ be a unitary representation of a compact Lie group $G$,
  of complex dimension $m$.
  We let $\alpha\colon G\to U(m)$ classify $V$, i.e., $V$ is isomorphic to $\alpha^*(\nu_m)$. 
  We let $\psi\colon V\iso\alpha^*(\nu_m)$ be a $G$-equivariant $\mC$-linear isometry.
  We write $\psi^\flat\colon \pi_{\alpha^*(\nu_m)}^G(\Sigma^\infty_+\bP)\to \pi_V^G(\Sigma^\infty_+\bP)$
  for the isomorphism given by precomposition with
  the equivariant homeomorphism $S^\psi\colon S^V\iso S^{\alpha^*(\nu_m)}$.
  The {\em pre-Bott class} of $V$ is
  \[  b_{G,V}\ = \ \psi^\flat(\alpha^*(b_m))\ \in \ \pi_V^G(\Sigma^\infty_+\bP)\ .     \]
  We claim that the class $b_{G,V}$ is independent of the choice of the pair $(\alpha,\psi)$.
  Indeed, if $\bar\alpha\colon G\to U(m)$ is another continuous homomorphism
  with $\bar\alpha^*(\nu_m)\iso V$,
  then $\bar\alpha=c_g\circ\alpha$ for some $g\in U(m)$.
  The group of equivariant self-isometries of a unitary representation
  is path connected.
  So if $\bar\psi\colon V\iso \bar\alpha^*(\nu_m)$ is an equivariant isometry,
  then the composite
  \[ V \ \xra[\iso]{\ \bar\psi\ }\ \bar\alpha^*(\nu_m)\ = \ \alpha^*(c_g^*(\nu_m))
    \ \xra[\iso]{\alpha^*(l_g)\ }\ \alpha^*(\nu_m) \]
  is homotopic to $\psi$ through $G$-equivariant isometries,
  where $l_g\colon c_g^*(\nu_m)\iso \nu_m$ denotes translation by the element $g$.
  Hence $(\alpha^*(l_g)\circ\bar\psi)^\flat=\psi^\flat$.
  Inner automorphisms act as the identity on representation-graded homotopy groups
  of a global spectrum $X$, in the sense that
  \[ c_g^*\ = \ l_g^\flat\ \colon \ \pi_{\nu_m}^{U(m)}(X)\ \xra{\ c_g^*\ } \  \pi_{c_g^*(\nu_m)}^{U(m)}(X)
  \ ,\]
  by the same kind of argument that proves \cite[Proposition 1.5.10]{schwede:global}.
  Hence
  \begin{align*}
    \bar\psi^\flat(\bar\alpha^*(b_m))\
    &= \  \bar\psi^\flat(\alpha^*(c_g^*(b_m))) \
    = \  \bar\psi^\flat(\alpha^*(l_g^\flat(b_m))) \\
    &= \  \bar\psi^\flat((\alpha^*(l_g))^\flat(\alpha^*(b_m)))\
    = \  (\alpha^*(l_g)\circ\bar\psi)^\flat(\alpha^*(b_m))
    \ = \ \psi^\flat(\alpha^*(b_m)) \ ,
  \end{align*}
  as claimed.   
\end{construction}

For an orthogonal ring spectrum $R$, representing a global ring spectrum,
we can `externalize' the homotopy group pairing \eqref{eq:ROpairing}
in the group variable by using the global structure. Given two compact Lie groups $G$ and $K$,
a $G$-representation $V$, and a $K$-representation $W$, we define a
pairing
\[  \times \ : \ \pi_V^G(R)\times \pi_W^K(R)\ \to \ \pi_{V\times W}^{G\times K}(R)   \]
as the composite
\[  \pi_V^G(R)\times \pi_W^K(R)\
  \xra{p_G^*\times p_K^*} \ \pi_{p_G^*(V)}^{G\times K}(R)\times \pi_{p_K^*(W)}^{G\times K}(R)\
\xra{\ \cdot \ } \ \pi_{V\times W}^{G\times K}(R) \ .\]
Here $p_G\colon G\times K\to G$ and $p_K\colon G\times K\to K$ are the projections,
and $V\times W=p_G^*(V)\oplus p_K^*(W)$ is the external product,
a representation of $G\times K$.
For $G=K$, the `internal' product \eqref{eq:ROpairing}
can be obtained from the external product by restricting to the diagonal in $G\times G$.

By the following theorem, the class $b_m$ is a factor of a restriction of $b_{m+1}$.
So every morphism of commutative global ring
spectra from $\Sigma^\infty_+\bP$ that inverts $b_{m+1}$ also inverts $b_1,b_2,\dots,b_m$.
We shall show in Theorem \ref{thm:all_different} below that the ultra-commutative localization
by $b_m$ does {\em not} invert $b_{m+1}$.
Hence the localizations by the classes $b_m$
are all distinct for different values of $m$,
and all distinct from the localization by the entire set $\{b_m\}_{m\geq 1}$.

\begin{theorem}[Multiplicativity of pre-Bott classes]\label{thm:fix of pre-Bott}\ 
  \begin{enumerate}[\em (i)]
    \item For all $m,n\geq 1$, the relation
      \[ 
        \res^{U(m+n)}_{U(m,n)}(b_{m+n}) \ = \  b_m\times b_n
      \]
  holds in the group $\pi_{\nu_m\times\nu_n}^{U(m)\times U(n)}(\Sigma^\infty_+\bP)$.
  \item
    For all unitary representations $V$ and $W$ of a compact Lie group $G$,
    the relation $b_{G,V}\cdot b_{G,W}=b_{G,V\oplus W}$
    holds in $\pi_{V\oplus W}^G(\Sigma^\infty_+\bP)$.
    \end{enumerate}
\end{theorem}
\begin{proof}
  (i) We need a property of the power operations \eqref{eq:define_P^m}
  that specializes to the `restriction relation' in the definition
  of a global power functor for $V=0$, compare \cite[Definition 5.1.6]{schwede:global}.
  The concatenation group monomorphism $+\colon \Sigma_m\times \Sigma_n \to \Sigma_{m+n}$
  is defined by
  \[ (\sigma+\sigma')(i) \ = \ 
    \begin{cases}
      \ \sigma(i) & \text{ for $1\leq i\leq m$, and}\\
      \sigma'(i-m)+m & \text{ for $m+1\leq i\leq m+n$.}
    \end{cases}
  \]
  We embed a product of wreath products by the monomorphism
  \begin{align*}
    (\Sigma_m\wr G)\ \times\ (\Sigma_n\wr G)\hspace*{1.3cm}
    &\to\quad \Sigma_{m+n}\wr G \\
    ((\sigma;\, g_1,\dots,g_m),\, (\sigma';\, g_{m+1},\dots,g_{m+n})) \
    &\longmapsto \ (\sigma+\sigma';\, g_1,\dots,g_{m+n})\ .
  \end{align*}
  Then for all $x\in \pi_V^G(R)$, the power operations satisfy the relation
  \begin{equation}\label{eq:power_restriction}
    \res^{\Sigma_{m+n}\wr G}_{ (\Sigma_m\wr G)\ \times\ (\Sigma_n\wr G)}( P^{m+n}(x) )\ = \
    P^m(x)\times P^n(x)      
  \end{equation}
  in $\pi_{V^m\times V^n}^{(\Sigma_m\wr G)\times(\Sigma_n\wr G)}(R)$,
  exploiting that $ \res^{\Sigma_{m+n}\wr G}_{ (\Sigma_m\wr G)\ \times\ (\Sigma_n\wr G)}(V^{m+n})=V^m\times V^n$
  as representations of the group $(\Sigma_m\wr G)\times(\Sigma_n\wr G)$.
  This relation is proved by direct comparison of representatives.
  Thus
  \begin{align*}
    \res^{U(m+n)}_{U(m,n)}(b_{m+n}) \
    &= \   \res^{U(m+n)}_{U(m,n)}(\tr_{\Sigma_{m+n}\wr T}^{U(m+n)}(P^{m+n}(b_1))) \\
    &= \ \tr_{(\Sigma_m\wr T)\times (\Sigma_n\wr T)}^{U(m,n)}(\res^{\Sigma_{m+n}\wr T}_{(\Sigma_m\wr T)\times (\Sigma_n\wr T)}(P^{m+n}(b_1))) \\
_{\eqref{eq:power_restriction}}  &= \ \tr_{(\Sigma_m\wr T)\times (\Sigma_n\wr T)}^{U(m,n)}(P^m(b_1)\times P^n(b_1)) \\
    &= \ \tr_{\Sigma_m\wr T}^{U(m)}(P^m(b_1)) \times
             \tr_{\Sigma_n\wr T}^{U(n)}(P^n(b_1)) \ =\ b_m\times b_n \ .
  \end{align*}
  The second relation is the double coset formula of Proposition \ref{prop:double_coset}.
  
  (ii)
  We may assume that $V=\alpha^*(\nu_m)$ and $W=\beta^*(\nu_n)$ for continuous homomorphisms
  $\alpha\colon G\to U(m)$ and $\beta\colon G\to U(n)$.
  Then $V\oplus W$ is the restriction of the tautological $U(m+n)$-representation $\nu_{m+n}$
  along the composite
  \[ G\ \xra{(\alpha,\beta)}\ U(m, n)\ \xra{\text{incl}} \ U(m+n)\ . \]
  Hence
  \begin{align*}
    b_{G,V}\cdot b_{G,W}
    &=     \alpha^*(b_m)\cdot \beta^*(b_n)  =  (\alpha,\beta)^*(b_m\times b_n)
    =  (\alpha,\beta)^*(\res^{U(m+n)}_{U(m,n)}(b_{m+n})) 
    =  b_{G,V\oplus W}\ .\qedhere
  \end{align*}
\end{proof}

In equivariant $K$-theory, the Thom classes of the Atiyah--Bott--Shapiro orientation
are compatible with power operations, see \cite{berwick-evans-guo}.
The Bott classes of unitary representations are particular cases of such Thom classes,
so in equivariant $K$-theory, for a compact Lie group $G$,
the power operation $P^m$ takes the $G$-equivariant Bott class of a unitary $G$-representation
$V$ to the $(\Sigma_m\wr G)$-equivariant Bott class of the representation $V^m$.
Since norms are defined from restriction homomorphisms and power operations
as in \cite[Remark 5.1.7]{schwede:global}, 
the norm operation $N_H^G$, for a finite index subgroup $H$ of $G$,
takes the $H$-equivariant Bott class of a unitary $H$-representation
$W$ to the $G$-equivariant Bott class of the induced representation $\ind_H^G(W)$.

\Danger
As we show in the remainder of this section, our pre-Bott classes are {\em not} compatible
in the same way with power operations:
for a unitary representation $V$ of a compact Lie group $G$,
the class $P^m(b_{G,V})$ is typically different from the class $b_{\Sigma_m\wr G,V^m}$
for $m\geq 2$.
In fact, this feature is already visible in the simplest possible
case, for $m=2$ and the pre-Bott class of the 1-dimensional representation
of the trivial group, see Example \ref{eg:norm_does_not_match}.
Since power operations can be derived from restriction homomorphisms and norms
as in \cite[(5.1.8)]{schwede:global},
the pre-Bott classes are not compatible with norms, either.\medskip
  
For an inner product space $V$, we write $\lambda_V\colon V\to V^k$ for the linear isometry
\[ \lambda_V(v)\ = \ 1/\sqrt{k}\cdot (v,\dots,v)\ ; \]
the image of $\lambda_V$ coincides with the $\Sigma_k$-fixed subspace of $V^k$.
In the geometric fixed point homomorphism appearing in the next proposition,
we orient the $\Sigma_k$-fixed points of $(\mR^n)^k$
by the isometry $\lambda_{\mR^n}\colon\mR^n\iso ((\mR^n)^k)^{\Sigma_k}$.
The power endomorphism $\phi^k\colon\bP\to\bP$ was defined in \eqref{eq:define_power_P},
and $p_{\Sigma_k}^*$ is the inflation homomorphism \eqref{eq:define_inflation}
along the unique homomorphism $p_{\Sigma_k}\colon\Sigma_k\to\{1\}$.
In part (ii), we write $b_{\cl}=\res^T_{\{1\}}(b_1)$ for the classical, non-equivariant
pre-Bott class in $\pi_2(\Sigma^\infty_+\bP)$.

\begin{prop}\label{prop:Phi_after_P^m}
  Let $k\geq 2$.
  \begin{enumerate}[\em (i)]
  \item 
  The following diagram commutes for all $n\geq 0$:
  \[ \xymatrix@C=6mm{
      \pi_n(\Sigma^\infty_+ \bP)\ar[rrr]^-{P^k}\ar[d]_{(\Sigma^\infty_+ \phi^k)_*} &&&
      \pi_{(\mR^n)^k}^{\Sigma_k}(\Sigma^\infty_+\bP)\ar[d]^-{\Phi^{\Sigma_k}} \\
      \pi_n(\Sigma^\infty_+\bP)\ar@{=}[r] &
      \Phi_n^{\{1\}}(\Sigma^\infty_+\bP) \ar[rr]_{p_{\Sigma_k}^*} &&
      \Phi^{\Sigma_k}_n(\Sigma^\infty_+\bP)  } \]
\item
  The relation
  \[ \Phi^{\Sigma_k}(P^k(b_{\cl}))\ = \ k\cdot p_{\Sigma_k}^*(b_{\cl}) \]
  holds in the group $\Phi^{\Sigma_k}_2(\Sigma^\infty_+\bP)$.
\end{enumerate}
\end{prop}
\begin{proof}
  (i)
  The counterclockwise composite sends the class of
  a continuous map $f\colon S^{V\oplus\mR^n}\to S^V\sm \bP(V)_+=(\Sigma^\infty_+\bP)(V)$
  to the class of the composite
  \[ S^{V\oplus\mR^n}\ \xra{\ f\ }\ S^V\sm\bP(V)_+ \ \xra{S^V\sm\phi^k(V)_+}\
    S^V\sm\bP(V)_+ \ ,\]
  supported in the colimit system \eqref{eq:define_geom_fix}
  for $\Phi^{\Sigma_k}_n(\Sigma^\infty_+ \bP)$
  at the trivial $\Sigma_k$-representation $V$.
  We get another representative for the same class in $\Phi^{\Sigma_k}_n(\Sigma^\infty_+\bP)$
  by stabilizing this representative along the $\Sigma_k$-equivariant
  linear isometric embedding $\lambda_V\colon V\to V^k$.
  By expansion of definitions, this stabilized representative
  features as the counterclockwise composite in the following diagram:
  \[ \xymatrix@C=10mm@R=6mm{
      S^{(V^k)^{\Sigma_k}\oplus\mR^n}\ar[r]^-\iso \ar[d]_{(\lambda_V\oplus\mR^n)^{-1}}^\iso &
       ((S^{V\oplus\mR^n})^{\sm k})^{\Sigma_k}\ar[r]^-{(f^{\sm k})^{\Sigma_k}} &
       ((S^V\sm \bP(V)_+)^{\sm k})^{\Sigma_k}\ar[d]^\iso\\
       S^{V\oplus\mR^n}  \ar[d]_f\ar[ur]_{\quad x\mapsto(x\sm\dots\sm x)}
       && (S^{V^k}\sm \bP(V)^k_+)^{\Sigma_k}\ar[dd]^{(S^{V^k}\sm(\mu_{V,\dots,V})_+)^{\Sigma_k}}\\
       S^V\sm \bP(V)_+ \ar[d]_{S^V\sm \phi^k(V)_+}
       \ar[urr]_{\qquad v\sm L\mapsto  (v\sm\dots\sm v)\sm (L,\dots,L)}
       && \\
       S^V\sm \bP(V)_+\ar[rr]_-{S^{\lambda_V}\sm \bP(\lambda_V)_+} &&
       \left(S^{V^k}\sm \bP(V^k)_+\right)^{\Sigma_k} = ((\Sigma^\infty_+\bP)(V^k))^{\Sigma_k}\quad
    } \]
  The unnamed isomorphisms consist of shuffling of factors and instances of $\lambda_V$.
  The clockwise composite around the diagram sends the class of
  $f$ to the class of $\Phi^{\Sigma_k}(P^k[f])$,
  so the claim comes down to the commutativity of the outer diagram up to based homotopy.
  The proof of that homotopy commutativity has two ingredients:
  \begin{enumerate}[(a)]
  \item
    a path in $\mR_{>0}$ from $1/\sqrt{k}$ to 1 to eliminate the scaling factor
    that is needed to make $\lambda_V$ an isometry;
    this rescaling path is used in the upper left triangle and in the lower trapezoid;
  \item
    the homotopy commutativity of the following diagram of spaces:
    \[ \xymatrix@C=25mm{
        \bP(V)\ar[d]_{\phi^k(V)}\ar[r]^-{\text{diag}} &
        ( \bP(V)^k)^{\Sigma_k} \ar[d]^{(\mu_{V,\dots,V})^{\Sigma_k}}\\
        \bP(V) \ar[r]_-{\bP(\lambda_V)} &  \bP(V^k)^{\Sigma_k}
      } \]
    We write  $i_1,\dots,i_k\colon V\to V^k$ for the embeddings of the summands.
    The clockwise composite sends $L\in \bP(V)$ to $\bP(i_1)(L)\cdot\ldots\cdot \bP(i_k)(L)$,
    whereas the counterclockwise composite sends it to $\bP(\lambda_V)(L^k)=(\bP(\lambda_V)(L))^k$.
    We set $\theta=\arccos(1/\sqrt{k})$.
    For $1\leq j\leq k$, 
    we consider the path of linear isometries $H_j\colon [0,1]\to \bL(V,V^k)$
    from $i_j$ to $\lambda_V$ defined by
    \[ H_j(t)(v)\ = \ \cos(t\theta)\cdot (0,\dots,0,v,0,\dots,0)\ +\ \frac{\sin(t\theta)}{\sqrt{k-1}}\cdot (v,\dots,v,0,v,\dots,v)\ , \]
    with the distinguished entries in this formula occurring in the $j$-th coordinate.
    Then 
    \[ \bP(V)\times [0,1]\ \to \ \bP(V^k)^{\Sigma_k}\ , \quad
    (L,t)\ \longmapsto \ \bP(H_1(t))(L)\cdot\ldots\cdot \bP(H_k(t))(L)\]
    is the desired homotopy for the above diagram.
  \end{enumerate}
  This concludes the proof that 
  $p_{\Sigma_k}^*((\Sigma^\infty_+\phi^k)_*[f])=\Phi^{\Sigma_k}(P^k[f])$.

  (ii) The restriction homomorphism $\res^T_{\{1\}}\colon\pi_0^T(\Sigma^\infty_+\bP)\to \pi_0(\Sigma^\infty_+\bP)$ takes the stable tautological class $e_T$ to the multiplicative unit $1$.
  So
  \begin{align*}
    (\Sigma^\infty_+\phi^k)_*( b_{\cl})\
    &= \ (\Sigma^\infty_+\phi^k)_*( \res^T_{\{1\}}(b_1))\
    = \      \res^T_{\{1\}}((\Sigma^\infty_+\phi^k)_*(b_1))\\
    & = \ \res^T_{\{1\}}((1+e_T+\dots+e_T^{k-1})\cdot b_1) \
     = \ k\cdot \res^T_{\{1\}}(b_1)\ = \ k\cdot b_{\cl}\ . 
  \end{align*}
  The third equation is provided by Proposition \ref{prop:phi^m_on_b_1}.
  Since
  $\Phi^{\Sigma_k}(P^k(b_{\cl}))=p_{\Sigma_k}^*((\Sigma^\infty_+\phi^k)_*(b_{\cl}))$
  by part (i), this proves the claim.
\end{proof}

\begin{eg}[Incompatibility with power operations]\label{eg:norm_does_not_match}
  As we now explain, the pre-Bott classes are not compatible with power operations,
  in the sense that
  the class $P^m(b_{G,V})$ is typically different from the class $b_{\Sigma_m\wr G,V^m}$
  for $m\geq 2$ and $V\ne 0$.
  The classical pre-Bott class $b_{\cl}=\res^T_{\{1\}}(b_1)$ is also
  the pre-Bott class $b_{\{1\},\mC}$ of the 1-dimensional representation
  of the trivial group.
  The $\Sigma_2$-representation $\mC^2$ is isomorphic to $\sigma\oplus\mC$,
  the sum of the sign and the trivial $\Sigma_2$-representation.
  So on the one hand,
  \[ 
    b_{\Sigma_2,\mC^2}\
    = \ b_{\Sigma_2,\sigma\oplus\mC} \ = \  b_{\Sigma_2,\sigma} \cdot b_{\Sigma_2,\mC}\
    = \
    b_{\Sigma_2,\sigma}\cdot p_{\Sigma_2}^*(b_{\cl}) \ . 
  \]
  Hence
  \[ \Phi^{\Sigma_2}(b_{\Sigma_2,\mC^2}) \ = \ \Phi^{\Sigma_2}(b_{\Sigma_2,\sigma})\cdot p_{\Sigma_2}^*(b_{\cl})\ =_{\eqref{eq:Phi^G_of_char}} \
    (1-\sigma^\sharp)\cdot p_{\Sigma_2}^*(b_{\cl}) \ . \]
  On the other hand,
  \[ \Phi^{\Sigma_2}(P^2(b_{\{1\},\mC}))\ = \ \Phi^{\Sigma_2}(P^2(b_{\cl}))\ = \ 2\cdot p_{\Sigma_2}^*(b_{\cl})\ , \]
  by Proposition \ref{prop:Phi_after_P^m} (ii).
  Proposition \ref{prop:monomial_basis} (ii) below and the fact that $b_{\cl}\ne 0$
  show that $2\cdot p_{\Sigma_2}^*(b_{\cl})\ne  (1-\sigma^\sharp)\cdot p_{\Sigma_2}^*(b_{\cl})$
  in $\Phi^{\Sigma_2}_2(\Sigma^\infty_+\bP)$. So also $P^2(b_{\{1\},\mC})\ne b_{\Sigma_2,\mC^2}$.
\end{eg} 

\section{Geometric fixed points of pre-Bott classes}

In this section we shall acquire
a firm grip on the images of the pre-Bott classes
$b_{G,V}$ in the $G$-geometric fixed points of $\Sigma^\infty_+\bP$.
This is the technical heart of this paper that makes our main results work.
It is well-known that the $G$-geometric fixed points of the equivariant
$K$-theory spectrum are trivial for all non-cyclic compact Lie groups $G$.
What we need in the remainder of the paper is that our pre-Bott classes 
are sufficiently plentiful so that inverting them enforces the same property,
i.e., kills the geometric fixed points of $\Sigma^\infty_+\bP$
for all non-cyclic compact Lie groups.
This feature is ensured by the main result of this section,
Theorem \ref{thm:rep_with_trivial_Phi-Lie}: for
every non-cyclic compact Lie group $G$, it provides 
a unitary $G$-representation $V$ such that
$\Phi^G(b_{G,V})=0$ in $\Phi^G_0(\Sigma^\infty_+\bP)$.

A key ingredient for the proof of
Theorem \ref{thm:rep_with_trivial_Phi-Lie}
is a formula for $\Phi^G(b_{G,V})$ in the case when $V$ is irreducible.
Our formula in Theorem \ref{thm:Phi_on_irrep}
is a sum indexed by all ways to write $V$ as induced from
a 1-dimensional representation of some finite index subgroup of $G$;
it says in particular that if $V$ is irreducible and not monomial,
i.e., not induced from a 1-dimensional representation
of any finite index subgroup, then $\Phi^G(b_{G,V})=0$.
This takes care of all compact Lie groups that have non-monomial irreducible representations.
Additional arguments are needed to deal with monomial groups, i.e., ones where
every irreducible representation is monomial. For those we provide ad hoc arguments
involving some explicit group theoretic manipulations.
\medskip

Our main calculations live in the geometric fixed point ring $\Phi^G_0(\Sigma^\infty_+\bP)$,
for varying compact Lie groups $G$;
the next proposition exhibits a monomial basis of this ring.
The $G$-equivariant homotopy ring $\pi_0^G(\Sigma^\infty_+\bP)$ 
is free as an abelian group with basis the classes $\tr_H^G(\alpha^*(e_T))$
for all $G$-conjugacy classes of pairs $(H,\alpha)$ consisting of a
closed subgroup $H$ of $G$ with finite Weyl group, and a continuous homomorphism
$\alpha\colon H\to T$, see \cite[Corollary 4.1.13]{schwede:global}.
We  write
\[\alpha^\sharp \ = \ \Phi^G(\alpha^*(e_T)) \ \in \ \Phi_0^G(\Sigma^\infty_+\bP)\ ,    \]
where $\Phi^G\colon\pi_0^G(\Sigma^\infty_+\bP)\to \Phi_0^G(\Sigma^\infty_+\bP)$
is the geometric fixed point homomorphism \eqref{eq:geom_fix_map} or \eqref{eq:define_Phi^G_map}.

\begin{eg}
  Let $\alpha\colon G\to T$ be a nontrivial character of a compact Lie group $G$.
  Then the geometric fixed point homomorphism
  $\Phi^G\colon \pi^G_{\alpha}(\Sigma^\infty_+\bP)\to \Phi^G_0(\Sigma^\infty_+\bP)$
  from \eqref{eq:geom_fix_map} factors as the composite
  \[ \pi^G_{\alpha}(\Sigma^\infty_+\bP)\ \xra{\ i^* \ }\
    \pi^G_0(\Sigma^\infty_+\bP)\ \xra{\Phi^G}\ \Phi^G_0(\Sigma^\infty_+\bP)\ , \]
  where $i\colon S^0\to S^\alpha$ is the inclusion of $G$-fixed points. We record the relation
  \begin{equation}\label{eq:i^*_of_b}
    i^*(b_{G,\alpha})\ = \ i^*(\alpha^*(b_1))\ = \ \alpha^*(i^*(b_1))\
    = \ \alpha^*(1-e_T)\ = \ 1-\alpha^*(e_T)\ .
  \end{equation}
  Hence
  \begin{equation}\label{eq:Phi^G_of_char}
    \Phi^G(b_{G,\alpha})\ = \ \Phi^G(i^*(b_{G,\alpha}))\ = \
    1-\Phi^G(\alpha^*(e_T))\ = \ 1 -\alpha^\sharp\ .
  \end{equation}
\end{eg}

The structure of an ultra-commutative monoid on $\bP$ gives the
Rep-functor $\upi_0(\bP)$ the structure of a {\em global power monoid}
in the sense of \cite[Definition 2.2.8]{schwede:global}.
Roughly speaking, this means that $\upi_0(\bP)$ becomes a functor
to abelian monoids that is additionally equipped with additive power operations.
By \cite[Proposition 2.2.23]{schwede:global},
the represented functor $\Hom(-,T)$---and hence also the Rep-functor $\upi_0(\bP)$---in fact
has a {\em unique} structure of global power monoid.
Namely, the abelian monoid structure on $\Hom(G,T)$
is elementwise multiplication of homomorphisms;
and for $m\geq 1$, the power operation
\[ [m] \ : \ \Hom(G,T)\ \to \ \Hom(\Sigma_m\wr G,T)\]
sends $\alpha\colon G\to T$ to $[m](\alpha)\colon  \Sigma_m\wr G\to T$
defined by
\[     [m](\alpha)(\sigma;\lambda_1,\dots,\lambda_m) \ = \ \alpha(\lambda_1)\cdot\ldots\cdot\alpha(\lambda_m)\  .\]

\begin{construction}[Transfer]
  We recall the construction of the transfer
  \[  \tr_H^G \ \colon \ \Hom(H,T)\ \to \ \Hom(G,T)\ , \]
  where $H$ is a finite index subgroup of a compact Lie group $G$.
  The construction uses a transversal of $H$ in $G$, i.e., an ordered $m$-tuple 
  $\bar g=(g_1,\dots,g_m)$ of elements in disjoint right $H$-orbits such that
  \[ G \ = \ {\bigcup}_{i=1}^m \ g_i H \ ,\]
  where $m=[G:H]$ is the index.  The wreath product $\Sigma_m\wr H$ acts freely and transitively 
  from the right on the set of all such transversals, by the formula
  \[ (g_1,\dots,g_m)\cdot (\sigma;\,h_1,\dots,h_m)\ = \ 
    (g_{\sigma(1)} h_1,\dots,g_{\sigma(m)}h_m)\ .  \]
  We obtain a continuous monomorphism 
  \begin{equation}\label{eq:define_Psi}
    \Psi_{\bar g}\ \colon\  G\ \to\ \Sigma_m\wr H
  \end{equation}
  by requiring that
  \[  \gamma \cdot\bar g \ = \ \bar g\cdot\Psi_{\bar g}(\gamma)  \ .     \]
  In more detail: the image $\Psi_{\bar g}(\gamma)=(\sigma;\,h_1,\dots,h_m)$ is determined
  by the relations
  \[ \gamma\cdot g_i \ = \ g_{\sigma(i)}\cdot h_i \]
  for all $1\leq i\leq m$. The transfer of an $H$-character $\chi\colon H\to T$
  is then defined as the composite
  \[  G \ \xra{\Psi_{\bar g}}\ \Sigma_m\wr H \ \xra{[m](\chi)}\ T \ . \]
  The resulting $G$-character $\tr_H^G(\chi)$ is independent of the choice
  of $H$-transversal.
\end{construction}

The inflation homomorphism
\[ p_G^* \ \colon \ \pi_k(\Sigma^\infty_+ \bP)\ = \ \Phi_k^{\{1\}}(\Sigma^\infty_+\bP)\ \to\
  \Phi_k^G(\Sigma^\infty_+ \bP)    \]
associated with the unique homomorphism $p_G\colon G\to\{1\}$
is defined in \cite[Construction 3.3.4]{schwede:global}, see also \eqref{eq:define_inflation}. 

\begin{prop}\label{prop:monomial_basis}
  Let $G$ be a compact Lie group.
  \begin{enumerate}[\em (i)]
  \item The map
    \[ \Hom(G,T) \ \to \ \Phi_0^G(\Sigma^\infty_+\bP)\ , \quad
    \alpha \ \longmapsto \ \alpha^\sharp \ = \ \Phi^G(\alpha^*(e_T)) \]
    is multiplicative, and its extension to a ring homomorphism
    \[ \mZ[\Hom(G,T)] \ \to \ \Phi_0^G(\Sigma^\infty_+\bP)\]
    is an isomorphism of rings.
  \item The morphism of graded-commutative rings
    \[  \mZ[\Hom(G,T)] \tensor \pi_*(\Sigma^\infty_+ \bP)
      \ \to \ \Phi_*^G(\Sigma^\infty_+\bP)\ , \quad
      \alpha\tensor y\ \longmapsto \ \alpha^\sharp\cdot p_G^*(y)\]
    is an isomorphism.
  \item  Let $\alpha\colon G\to T$ be a  continuous homomorphism.
    Then for every  $m\geq 1$, the relation
    \[ P^m(\alpha^*(e_T))\ = \ ([m](\alpha))^*(e_T) \]
    holds in the ring $\pi^{\Sigma_m\wr G}_0(\Sigma^\infty_+\bP)$.
  \item  Let $H$ be a finite index subgroup of $G$,
    and let $\chi\colon H\to T$ be a nontrivial continuous homomorphism.
    Then the relation
    \[ \Phi^G(\Psi_{\bar g}^*(P^m(b_{H,\chi})))\ = \ 1- \tr_H^G(\chi)^\sharp \]
    holds in the ring $\Phi^G_0(\Sigma^\infty_+\bP)$,
    where $m=[G:H]$ and $\Psi_{\bar g}$ is the monomorphism \eqref{eq:define_Psi}
    based on some $H$-transversal $\bar g$ in $G$.
  \item Let $\alpha\colon G\to T$ be a nontrivial continuous homomorphism.
    Then for every  $m\geq 1$, the relation
    \[ \Phi^{\Sigma_m\wr G}(P^m(b_{G,\alpha}))\ = \ 1- ([m](\alpha))^\sharp \]
    holds in the ring $\Phi^{\Sigma_m\wr G}_0(\Sigma^\infty_+\bP)$.  
  \end{enumerate}
\end{prop}
\begin{proof}
  (i) By \cite[Proposition 2.2.23]{schwede:global},
  the bijection \eqref{eq:Hom(G,T)}  is multiplicative in the sense that
  \[ \alpha^*(u_T)\cdot \beta^*(u_T)\ = \ (\alpha\cdot\beta)^*(u_T)\]
  in $\pi_0^G(\bP)$, for all continuous homomorphisms $\alpha,\beta\colon G\to T$.
  The stabilization map \cite[(3.3.12)]{schwede:global}
  is multiplicative and natural for restriction homomorphisms, and thus
  \begin{align*}
    \alpha^*(e_T)\cdot \beta^*(e_T)\
   &=_{\eqref{eq:define_e_T}}  \  \alpha^*(\sigma^T(u_T))\cdot \beta^*(\sigma^T(u_T))\
      = \  \sigma^G(\alpha^*(u_T)\cdot \beta^*(u_T))\\
    &= \  \sigma^G( (\alpha\cdot \beta)^*(u_T))\
      = \   (\alpha\cdot \beta)^*(\sigma^T(u_T))\
  =_{\eqref{eq:define_e_T}} \ (\alpha\cdot \beta)^*(e_T)\ .
  \end{align*}
  This relation and the multiplicativity of the geometric fixed point map
  yield
  \[ \alpha^\sharp\cdot\beta^\sharp\ =\ \Phi^G(\alpha^*(e_T))\cdot \Phi^G(\beta^*(e_T))
    \ = \  \Phi^G(\alpha^*(e_T)\cdot \beta^*(e_T))
    \ = \  \Phi^G( (\alpha\cdot \beta)^*(e_T))\ = \ (\alpha\cdot\beta)^\sharp\ . \]
  For every compact Lie group $G$ and every $G$-space $A$,
  the abelian group $\Phi_0^G(\Sigma^\infty_+ A)$ is free, with a basis
  indexed by the set $\pi_0(A^G)$. More precisely, the composite
  \[  \pi_0(A^G)\ \xra{\sigma^G} \ \pi_0^G(\Sigma^\infty_+ A)
    \ \xra{\Phi^G} \ \Phi_0^G(\Sigma^\infty_+ A) \]
  of the stabilization map
  and the geometric fixed point map identifies $\pi_0(A^G)$ with a $\mZ$-basis of
  $\Phi_0^G(\Sigma^\infty_+ A)$. For the underlying $G$-space of the global space $\bP$,
  we have $\pi_0(\bP^G)=\pi_0^G(\bP)$, and   \eqref{eq:Hom(G,T)} identifies this set with
  $\Hom(G,T)$.
  Moreover,
  \begin{align*}
    \Phi^G(\sigma^G(\alpha^*(u_T)))\ = \
    \Phi^G(\alpha^*(\sigma^T(u_T)))\ = \
    \Phi^G(\alpha^*(e_T))\ = \ \alpha^\sharp
  \end{align*}
  for every continuous homomorphism $\alpha\colon G\to T$,
  and hence these classes form a $\mZ$-basis of $\Phi_0^G(\Sigma^\infty_+ \bP)$.

  (ii) Part (i) shows that the graded ring
  $\Phi_0^G(\Sigma^\infty_+\bP)\tensor \pi_*(\Sigma^\infty_+ \bP)$
  is a free module over $\pi_*(\Sigma^\infty_+ \bP)$ with basis the classes
  $\alpha^\sharp\tensor 1$ for all $\alpha\in \Hom(G,T)$.
  Suspension spectra and geometric fixed points commute, in the sense that the composite
  \[ \pi_*(\Sigma^\infty_+ \bP^G)\ \xra{\ \sigma^G\ } \ \pi_*^G(\Sigma^\infty_+ \bP )\ \xra{\Phi^G}
    \Phi_*^G(\Sigma^\infty_+\bP) \]
  of the stabilization map and the geometric fixed point map is an isomorphism. 
  The fixed point space $\bP^G=P(\Sym(\Uc_G^\mC))^G$ is the space of $G$-invariant
  lines in a complete complex $G$-universe, and hence the disjoint union over all
  $\alpha\in \Hom(G,T)$ of copies of $\mC P^\infty$.
  More precisely, choices of linear isometries
  $\alpha_\flat\colon \mC^\infty\iso \Hom^G(\alpha,\Sym(\Uc_G^\mC))$
  to the $\alpha$-isotypical summands in the complete universe $\Sym(\Uc_G^\mC)$
  assemble into a homeomorphism
  \[ \Hom(G,T)\times \mC P^\infty\ \iso \ \bP^G\ . \]
  Stable homotopy groups take disjoint unions to direct sums, so the induced map 
  \[ \mZ[\Hom(G,T)]\tensor\pi_*(\Sigma^\infty_+\bP)\ \to \ 
    \pi_*(\Sigma^\infty_+ \bP^G)\]
  is an isomorphism.

  (iii)
  Like for any ultra-commutative monoid, the stabilization maps
  $\sigma^G\colon\pi_0^G(\bP)\to\pi_0^G(\Sigma^\infty_+\bP)$
  send power operations in the global power monoid $\upi_0(\bP)$
  to power operations in the global power functor $\upi_0(\Sigma^\infty_+ \bP)$,
  by direct inspection.
  We abbreviate
  \[ p_m=[m](\Id_T)\ \colon \ \Sigma_m\wr T \ \to \ T \text{,\quad so that\quad}
    p_m(\sigma;\lambda_1,\dots,\lambda_m)\ = \ \lambda_1\cdot\ldots\cdot \lambda_m \ .  \]
  Then
    \begin{align*}
    P^m(\alpha^*(e_T))\
      &= \  (\Sigma_m\wr \alpha)^*(P^m(e_T))\\
    _{\eqref{eq:define_e_T}}   &= \  (\Sigma_m\wr \alpha)^*(P^m(\sigma^T(u_T)))\
      = \  (\Sigma_m\wr \alpha)^*(\sigma^{\Sigma_m\wr T}([m](u_T)))\\
      &= \  (\Sigma_m\wr \alpha)^*(\sigma^{\Sigma_m\wr T}(p_m^*(u_T)))\
      = \  (\Sigma_m\wr \alpha)^*( p_m^*(\sigma^T(u_T)))\\
   _{\eqref{eq:define_e_T}} &= \  (p_m\circ(\Sigma_m\wr \alpha))^*(e_T)\  = \  ([m](\alpha))^*(e_T) \ .
  \end{align*}
  
  (iv)
  We show first that for every ultra-commutative ring spectrum $R$,
  the composite
  $\Phi^G\circ\Psi_{\bar g}^*\circ P^m\colon \pi_0^H(R)\to\Phi_0^G(R)$  is additive.
  Indeed, the composite
  \[ G \ \xra{\ \Psi_{\bar g}\ }\ \Sigma_m\wr H \ \xra{\text{proj}} \ \Sigma_m \]
  classifies the action of $G$ on $G/H$ by left translation.
  Since this action is transitive, the image of the composite is a transitive subgroup
  of $\Sigma_m$, and hence {\em not} subconjugate to $\Sigma_k\times \Sigma_{m-k}$
  for any $1\leq k\leq m-1$. Consequently, the image of $\Psi_{\bar g}\colon G\to\Sigma_m\wr H$
  is not subconjugate to $(\Sigma_k\wr H)\times(\Sigma_{m-k}\wr H)$ for any $1\leq k\leq m-1$. 
  Thus the composite
  \[ \Phi^G\circ \Psi_{\bar g}^*\circ \tr_{(\Sigma_k\wr H)\times (\Sigma_{m-k}\wr H)}^{\Sigma_m\wr H}\]
  vanishes for $1\leq k\leq m-1$, by \cite[Proposition 3.4.2 (ii)]{schwede:global}.
  So the desired additivity follows:
  \begin{align*}
   \Phi^G(\Psi_{\bar g}^*(P^m(x+y)))\
    &= \ \sum_{k=0}^m \Phi^G(\Psi_{\bar g}^*(\tr_{(\Sigma_k\wr H)\times (\Sigma_{m-k}\wr H)}^{\Sigma_m\wr H}(P^k(x)\times P^{m-k}(y))))\\
    &=\ \Phi^G(\Psi_{\bar g}^*(P^m(x)))\ +\ \Phi^G(\Psi_{\bar g}^*(P^m(y)))\ .
  \end{align*}
  The first equation is the generic formula for a power operation on a sum in a global power
  functor, proved for $\upi_0(R)$ in \cite[Theorem 5.1.11]{schwede:global}.

  Now we turn to $R=\Sigma^\infty_+\bP$.
  The class $\Psi_{\bar g}^*(P^m(b_{H,\chi}))$ has representation-grading $\Psi_{\bar g}^*(\chi^m)$.
  Since $\chi\colon H\to T$ is nontrivial, the $G$-fixed points
  of $\Psi_{\bar g}^*(\chi^m)$ are also trivial. So the geometric fixed point homomorphism
  $\Phi^G\colon \pi^G_{\Psi_{\bar g}^*(\chi^m)}(\Sigma^\infty_+\bP)\allowbreak\to \Phi^G_0(\Sigma^\infty_+\bP)$
  factors as the composite
  \[ \pi^G_{\Psi_{\bar g}^*(\chi^m)}(\Sigma^\infty_+\bP)\ \xra{\ i^* \ }\
    \pi^G_0(\Sigma^\infty_+\bP)\ \xra{\Phi^G}\ \Phi^G_0(\Sigma^\infty_+\bP)\ . \]
  Altogether we can now conclude
  \begin{align*}    \Phi^G( \Psi_{\bar g}^*(P^m(b_{H,\chi})))\
    &= \  \Phi^G(i^*(\Psi_{\bar g}^*(P^m(b_{H,\chi}))))\\
   &= \  \Phi^G(\Psi_{\bar g}^*(P^m(i^*(b_{H,\chi}))))\\
_{\eqref{eq:i^*_of_b}}    &= \  \Phi^G(\Psi_{\bar g}^*(P^m(1-\chi^*(e_T))))\\
    &= \  1- \Phi^G(\Psi_{\bar g}^*(P^m(\chi^*(e_T))))\\
_{\text{(iii)}}    &= \  1- \Phi^G(\Psi_{\bar g}^*(([m](\chi))^*(e_T)))\\
   &= \  1- \Phi^G( (\tr_H^G(\chi))^*(e_T))\
      = \  1- \tr_H^G(\chi)^\sharp\ .
    \end{align*}

  (v)
  We apply part (iv) to the ambient group $\Sigma_m\wr G$, its subgroup
  $H=G\times(\Sigma_{m-1}\wr G)$ of index $m$, and with $\chi=\alpha\circ q$,
  for $q\colon G\times(\Sigma_{m-1}\wr G)\to G$ the projection to the first coordinate.
  We use the transversal $\bar g$ of $G\times(\Sigma_{m-1}\wr G)$ in
  $\Sigma_m\wr G$ consisting of the elements $((1 i);1,\dots,1)$ for $1\leq i\leq m$.
  Then the associated monomorphism \eqref{eq:define_Psi} 
  \[ \Psi_{\bar g}\  \colon \ \Sigma_m\wr G \ \to \ \Sigma_m \wr(G\times(\Sigma_{m-1}\wr G)) \]
  is a section to the homomorphism
  $\Sigma_m\wr q\colon\Sigma_m \wr(G\times(\Sigma_{m-1}\wr G))\to\Sigma_m\wr G$.
  So we obtain
  \begin{align*}
    \Phi^{\Sigma_m\wr G}(P^m(b_{G,\alpha}))\
    &= \ \Phi^{\Sigma_m\wr G}(\Psi_{\bar g}^*((\Sigma_m\wr q)^*(P^m(b_{G,\alpha}))))\\
    &= \ \Phi^{\Sigma_m\wr G}(\Psi_{\bar g}^*(P^m( q^*(b_{G,\alpha}))))\\
    &= \
      \Phi^{\Sigma_m\wr G}(\Psi_{\bar g}^*(P^m(b_{G\times(\Sigma_{m-1}\wr G) ,\alpha\circ q})))\\
   _{\text{(iv)}} &= \ 1- (\tr_{G\times(\Sigma_{m-1}\wr G)}^{\Sigma_m\wr G}(\alpha\circ q))^\sharp
      \ = \ 1- ([m](\alpha))^\sharp \ .
  \end{align*}
  The final equation uses that
  $\tr_{G\times(\Sigma_{m-1}\wr G)}^{\Sigma_m\wr G}(\alpha\circ q)=[m](\alpha)$,
  see \cite[Proposition 2.2.30 (iv)]{schwede:global}.
\end{proof}

Now we embark on the journey towards Theorem \ref{thm:Phi_on_irrep}
which provides a formula for $\Phi^G(b_{G,V})$ when $V$ is irreducible.
The theorem refers to the ways to write an irreducible $G$-representation
as induced from a 1-dimensional representation of a finite index subgroup;
we set up some terminology and notation for dealing with this.

\begin{construction}[Monomial structures]
  We let $V$ be a unitary representation of a compact Lie group $G$.
  A {\em monomial datum} for $V$ is a pair $(H,\chi)$
  consisting of a finite index subgroup $H$ of $G$,
  and a continuous homomorphism $\chi\colon H\to T$,
  such that $V$ is isomorphic to the induced representation $\ind_H^G(\chi)$.
  For $g\in G$, let $^g H=g H g^{-1}$ be the conjugate subgroup,
  and let $c_g\colon {^g H}\to H$ be the conjugation isomorphism defined by $c_g(h)=g^{-1}h g$.
  Then the $G$-representation $\ind_{^g H}^G(\chi\circ c_g)$ is isomorphic to $\ind_H^G(\chi)$.
  So the group $G$ acts on the set of monomial data for $V$ by 
  \[ {^g(H,\chi)}\ = \ (^g H, \chi\circ c_g) \ .\]
  A {\em monomial structure} on $V$ is a $G$-orbit of monomial data.
  We write $\mon(V)$ for the set of monomial structures on $V$,
  and $[H,\chi]\in \mon(V)$ for the monomial structure represented
  by a monomial datum $(H,\chi)$.
\end{construction}

\begin{rk}[Finiteness of $\mon(V)$]
  If $(H,\chi)$ is a monomial datum for $V$, then in particular, the following
  two properties must be satisfied:
  \begin{itemize}
  \item Since $\dim(\ind_H^G(\chi))=[G:H]$,
    the index of $H$ in $G$ equals the dimension of $V$.
  \item Since the character $\chi$ is a summand in $\res^G_H(\ind_H^G(\chi))$,
    it is also a summand in $\res^G_H(V)$.
  \end{itemize}
  A compact Lie group $G$ has only finitely many subgroups of finite index,
  and $\res^G_H(V)$ has only finitely many non-zero isotypical summands.
  So $V$ has only finitely many monomial data.
  Hence also the set $\mon(V)$ of monomial structures is finite.
\end{rk}

Every 1-dimensional representation has a unique monomial datum,
and a unique monomial structure.
Tautologically, the pair $(H,\chi)$ is a monomial datum for
the induced representation $\ind_H^G(\chi)$.
We alert the reader that even irreducible representations can have
several monomial structures, see Example \ref{eg:several_monomial_structures}.
The next proposition exhibits a situation where monomial structures are unique.

\begin{prop}\label{prop:unique_mon}
  Let $V$ be a unitary representation of a compact Lie group $G$.
  Suppose that $G$ has a unique subgroup of index $\dim(V)$.
  Then $V$ has at most one monomial structure.
\end{prop}
\begin{proof}
  Every monomial datum is supported on the unique subgroup $H$ of $G$
  of index $\dim(V)$.
  By uniqueness, the group $H$ is in particular normal in $G$,
  so if $(H,\chi)$ is a monomial datum, then 
  \[ \res^G_H(V)\ \iso \
    \res^G_H(\ind_H^G(\chi))\ \iso \ \bigoplus_{g H\in G/H}  \chi\circ c_g\ .\]
  So any other monomial datum equals $(H,\chi\circ c_g)$
  for some $g\in G$, and is thus conjugate to $(H,\chi)$.
\end{proof}

\begin{eg}[Inflation of monomial structures]\label{eg:inflate_mon}
  We let $\epsilon\colon G\to K$ be a continuous epimorphism of compact Lie groups.
  If $(L,\lambda)$ is a monomial datum for a $K$-representation $W$, then
  \[ \epsilon^*(W)\ \iso \ \epsilon^*(\ind_L^K(\lambda))\ \iso \
    \ind_H^G(\lambda\circ \epsilon|_H)\ , \]
  where $H=\epsilon^{-1}(L)$. So $(H,\lambda\circ\epsilon|_H)$ is a monomial datum
  for the inflated $G$-representation $\epsilon^*(W)$.
  We claim that the resulting {\em inflation map}
  \begin{equation}\label{eq:inflate_mon}
    \epsilon^* \ \colon \ \mon(W)\ \to \ \mon(\epsilon^*(W))\ , \quad
    [L,\lambda]\ \longmapsto \ [\epsilon^{-1}(L),\lambda\circ \epsilon|_{\epsilon^{-1}(L)}]
  \end{equation}
  is bijective.
  For injectivity we let $(L,\lambda)$ and $(L',\lambda')$ be monomial data for $W$ 
  that become conjugate after inflation. So ${^g}(\epsilon^{-1}(L))=\epsilon^{-1}(L')$
  and $(\lambda\circ\epsilon|_{\epsilon^{-1}(L)})\circ c_g=\lambda'\circ\epsilon|_{{\epsilon^{-1}}(L')}$ for some $g\in G$.
  Then $k=\epsilon(g)$ satisfies ${^k}L=L'$
  and $\lambda\circ c_k=\lambda'$, and thus $[L,\lambda]=[L',\lambda']$.
  For surjectivity we let $(H,\chi)$ be a monomial datum for $\epsilon^*(W)$.
  Then the kernel $N=\ker(\epsilon)$ of the epimorphism acts trivially
  on $\epsilon^*(W)\iso\ind_H^G(\chi)$.
  Since $\res_N^G(\ind_H^G(\chi))$ contains $\ind_{N\cap H}^N(\chi|_{N\cap H})$
  as a direct summand, this forces $N\subseteq H$, and thus $H=\epsilon^{-1}(\epsilon(H))$.
  Since $\chi$ is a summand in $\res^G_H(\ind_H^G(\chi))$,
  and hence in $\res^G_H(\epsilon^*(W))$, the character $\chi$ is trivial on $H\cap N$.
  Hence $\chi=\lambda\circ \epsilon|_H$ for a unique continuous
  homomorphism $\lambda\colon\epsilon(H)\to T$.
  Because
  \[ \epsilon^*(\ind_{\epsilon(H)}^K(\lambda)) \ \iso \ \ind_{\epsilon^{-1}(\epsilon(H))}^G(\lambda\circ\epsilon|_H)\ = \ \ind_H^G(\chi) \ \iso \ \epsilon^*(W) \]
  as $G$-representations, we conclude that $\ind_{\epsilon(H)}^K(\lambda)$ is $K$-isomorphic to $W$.
  So $[\epsilon(H),\lambda]$ is a  monomial structure on $W$, and it inflates
  to the given monomial structure $[H,\chi]$.
\end{eg}

Now we can establish our formula for the geometric fixed points
of the pre-Bott class of an irreducible representation.
For irreducible $G$-representations $V$ without monomial structure,
the following theorem says that $\Phi^G(b_{G,V})=0$.

\begin{theorem}\label{thm:Phi_on_irrep}
  For every nontrivial irreducible unitary representation $V$ of a compact Lie group $G$,
  the relation
  \[ \Phi^G(b_{G,V})\ = \ \sum_{[H,\chi]\in \mon(V)}( 1-\tr_H^G(\chi)^\sharp )\]
  holds in the ring $\Phi_0^G(\Sigma^\infty_+\bP)$.
\end{theorem}
\begin{proof}
  We start with the special case where $G$ is a closed subgroup of $U(m)$,
  and  $V=\res^{U(m)}_G(\nu_m)$ is the restriction of the tautological representation.
  The proof of Proposition \ref{prop:double_coset}
  provides a  $U(m)$-equivariant homeomorphism
  from  $U(m)/(\Sigma_m\wr T)$ to
  the space $\Lc_m$ of sets $\ell=\{L_1,\dots,L_m\}$ of
  $m$ pairwise orthogonal 1-dimensional subspaces of $\mC^m$,
  by taking the columns of a matrix, see \eqref{eq:coset2L_n}.
  Hence the induced map on $G$-fixed points
  \[    (U(m)/(\Sigma_m\wr T))^G \ \iso \ (\Lc_m)^G\]
  is a homeomorphism, too.
  
  An element $\ell$ of $\Lc_m$ is $G$-fixed precisely if
  $G$ permutes the $m$ lines in $\ell$. Since $G$ acts irreducibly
  on $\nu_m$, the action of $G$ on such a set $\ell\in(\Lc_m)^G$ must be transitive.
  Hence the stabilizer $H$ of every element $L\in\ell$ has index $m$ in $G$. 
  We let $\chi\colon H\to T$ be the character that encodes the $H$-action on $L$, i.e.,
  \[ h\cdot x \ = \ \chi(h)\cdot x \]
  for all $(h,x)\in H\times L$.
  We let $x\in L$ be any unit vector.
  Then the map 
  \begin{equation}\label{eq:pref_iso}
    V\ \xra{\ \iso \ }\  \map^H(G,\chi^*(\nu_1)) =  \ind_H^G(\chi)\ , \quad
    v \ \longmapsto \ \{ g\mapsto \td{g\cdot v,x}\}  
  \end{equation}
  is an isomorphism of $G$-representations,
  so $(H,\chi)$ is a monomial datum for $V$.
  If $L'$ is another element of $\ell$, 
  then $L'=g L$ for some $g\in G$, so the stabilizer of $L'$ is ${^g H}$, and $^g H$ acts on $L'$
  through the character $\chi\circ c_g$. So the $G$-conjugacy class
  of the monomial datum is independent of the choice of line in $\ell\in (\Lc_m)^G$.
  In other words, the map
  \begin{equation}\label{eq:map_back}
    (\Lc_m)^G\ \to \ \mon(V) \ , \quad \ell \ \longmapsto [H,\chi]
  \end{equation}
  is well-defined.

  Now we let $\ell$, $\ell'$ be two elements of $(\Lc_m)^G$ such that $[H,\chi]=[H',\chi']$
  in $\mon(V)$. Then there are lines $L\in\ell$ and $L'\in \ell'$
  such that the associated monomial data $(H,\chi)$ and $(H',\chi')$ are $G$-conjugate.
  So after replacing $L'$ by a suitable $G$-translate, if necessary,
  we can assume that $L$ and $L'$ have the same stabilizer group $H$,
  and the $H$-action on $L$ and $L'$ is given by the same character $\chi$.
  After choosing unit vectors $x\in L$ and $x'\in L'$,
  each of them defines an isomorphism of $G$-representations \eqref{eq:pref_iso}.
  By direct inspection, these send $x$ and $x'$, respectively,
  to the distinguished function $\chi^\dagger\colon G\to\mC$ defined by
  \begin{equation}\label{eq:chi^dagger}
    \chi^\dagger(\gamma)\ = \
    \begin{cases}
      \chi(\gamma) & \text{ for $\gamma\in H$, and}\\
    \  0 & \text{ for $\gamma\not\in H$.}
    \end{cases}
   \end{equation}
  Since $V$ is irreducible, the two isomorphisms
  differ by multiplication by a scalar in $T$, so we must have $x'=\lambda x$
  for some $\lambda\in T$. In particular $L'=\mC\cdot x'=\mC\cdot x=L$.
  Hence also the $G$-orbits $\ell=G\cdot L$ and $\ell'=G\cdot L'$
  are equal in $(\Lc_m)^G$. In other words, the map \eqref{eq:map_back}
  is injective.

  If $(H,\chi)$ is any monomial datum for $V$, then the induced representation
  $\ind_H^G(\chi)$ comes with a preferred $G$-invariant set of pairwise orthogonal lines,
  namely the set
  \[ \{ \mC\cdot g\cdot \chi^\dagger\}_{g\in G} \]
  of lines spanned by the $G$-translates of the distinguished function \eqref{eq:chi^dagger}.
  Moreover, $H$ is the stabilizer group of the line $\mC\cdot \chi^\dagger$,
  and $H$ acts on this line by the character $\chi$.
  So the image of this set under any isomorphism of $G$-representations
  \[    \ind_H^G(\chi)\ \iso \ V  \]
  is an element of $(\Lc_m)^G$, and the image of $\mC\cdot \chi^\dagger$
  is an element with stabilizer group $H$ and character $\chi$.
  So the map \eqref{eq:map_back} is surjective, and hence bijective.
  In sum, we have now established a bijection
  \begin{equation}\label{eq:fix2Ind}
    (U(m)/(\Sigma_m\wr T))^G \ \xra{\ \iso\ } \ \mon(V)
  \end{equation}
  that sends a $G$-fixed coset $A\cdot(\Sigma_m\wr T)$ to the
  monomial structure $[H,\chi]$, where
  $H$ is the stabilizer group of the line $\mC\cdot A\cdot e_1$,
  and $\chi\colon H\to T$ encodes the $H$-action on this line.
  
  Now we prove the claim in the special case.
  Proposition 3.4.2 (ii) of \cite{schwede:global} shows that
  \[ \Phi^G\circ \res^{U(m)}_G\circ \tr_{\Sigma_m\wr T}^{U(m)}\ = \ \sum_{A(\Sigma_m\wr T)\in (U(m)/(\Sigma_m\wr T))^G} \Phi^G\circ A_\star\circ \res^{\Sigma_m\wr T}_{G^A}\ . \]
  Here $G^A=A^{-1} G A$, and $A_\star$ is restriction along the conjugation
  homomorphism $c_A\colon G\to G^A$.
  We have used that the space $(U(m)/(\Sigma_m\wr T))^G$
  is finite by the bijection \eqref{eq:fix2Ind}
  and the finiteness of $\mon(V)$, so that its path components
  are points and have Euler characteristic equal to 1. Then
  \begin{align*}
    \Phi^G(b_{G,V})\
    &= \ \Phi^G(\res^{U(m)}_G(b_m))\ = \ \Phi^G(\res^{U(m)}_G(\tr_{\Sigma_m\wr T}^{U(m)}(P^m(b_1))))\\
    &=\  \sum_{A (\Sigma_m\wr T)\in  (U(m)/(\Sigma_m\wr T))^G}
      \Phi^G(A_\star(\res^{\Sigma_m\wr T}_{G^A}(P^m(b_1))))\ .
  \end{align*}
  We use the bijection \eqref{eq:fix2Ind} to replace the indexing set
  of the sum by the set of monomial structures on $V$,
  and we rewrite the corresponding summand.

  Given any $G$-fixed point $A (\Sigma_m\wr T)\in  (U(m)/(\Sigma_m\wr T))^G$,
  we set $v_i=A\cdot e_i$, where $e_i$ is the $i$th vector of the standard basis
  of $\mC^m$. Then $(v_1,\dots,v_m)$ is a monomial basis for $V$, i.e.,
  an orthonormal basis in which the $G$-action is given entirely
  by monomial matrices in $\Sigma_m\wr T$.
  We let $H$ be the stabilizer group of the line $\mC\cdot v_1$,
  and we let $\chi\colon H\to T$ encode the $H$-action on this line,
  so that $[H,\chi]$ is the monomial structure corresponding to
  $A (\Sigma_m\wr T)$ under the bijection \eqref{eq:fix2Ind}.
  
  Because $G$ acts irreducibly on $V$,
  it acts transitively on the set of lines $\{\mC\cdot v_1,\dots,\mC\cdot v_m\}$.
  So we can choose an $H$-transversal $\bar g=(g_1,\dots,g_m)$ such that
  $g_i\cdot v_1=\lambda_i\cdot v_i$ for some $\lambda_i\in T$.
  The matrix $B=A\cdot(1;\lambda_1,\dots,\lambda_m)$ is then another representative
  of the same coset $A (\Sigma_m\wr T)=B (\Sigma_m\wr T)$.
  Moreover, $B \cdot e_i=A \cdot \lambda_i\cdot e_i=\lambda_i\cdot v_i$.
  We claim that  the following diagram commutes:
  \[ \xymatrix{ G \ar[r]^-{\Psi_{\bar g}}\ar[d]_{c_B} & \Sigma_m\wr H\ar[d]^{\Sigma_m\wr\chi} \\
      G^B\ar[r]_-{\text{incl}} & \Sigma_m\wr T } \]
  To this end we consider $\gamma\in G$ and set $\Psi_{\bar g}(\gamma)=(\sigma;h_1,\dots,h_m)$,
  which means that
  \[ \gamma\cdot g_i \ = \ g_{\sigma(i)}\cdot h_i \]
  for all $1\leq i\leq m$. Then
  \begin{align*}
    c_B(\gamma)\cdot e_i\
    &= \ B^{-1}\cdot\gamma\cdot B \cdot e_i\ = \ B^{-1} \cdot \gamma\cdot \lambda_i\cdot v_i\
    = \  B^{-1} \cdot \gamma\cdot g_i\cdot v_1\\
    &= \  B^{-1} \cdot g_{\sigma(i)}\cdot h_i\cdot v_1\
    = \  B^{-1} \cdot g_{\sigma(i)}\cdot \chi(h_i)\cdot v_1\
    = \ \chi(h_i)\cdot  B^{-1} \cdot \lambda_{\sigma(i)}\cdot v_{\sigma(i)}\\
    &= \ \chi(h_i)\cdot  e_{\sigma(i)}\ =\
      (\sigma;\chi(h_1),\dots,\chi(h_m))\cdot e_i\ = \ (\Sigma_m\wr \chi)(\Psi_{\bar g}(\gamma))\cdot e_i\ ,
  \end{align*}
  also for all $1\leq i\leq m$. Since the matrices $c_B(\gamma)$
  and $((\Sigma_m\wr \chi)\circ\Psi_{\bar g})(\gamma)$ have the same effect on a basis of $\mC^m$,
  they agree.
  The commutative square then yields
  \begin{align*}
    \Phi^G(B_\star(\res^{\Sigma_m\wr T}_{G^B}(P^m(b_1))))\
    &= \ \Phi^G( \Psi_{\bar g}^*((\Sigma_m\wr \chi)^*(P^m(b_1))))\\
    &= \ \Phi^G( \Psi_{\bar g}^*(P^m(\chi^*(b_1))))\
    = \ \Phi^G( \Psi_{\bar g}^*(P^m(b_{H,\chi})))\ = \  1- \tr_H^G(\chi)^\sharp\ .
  \end{align*}
  The second equation is the naturality of power operations
  for restriction along continuous homomorphisms.
  The final equation is Proposition \ref{prop:monomial_basis} (iv);
  we use here that $\chi$ is nontrivial because $V$ is nontrivial irreducible.
  This completes the proof in the special case.
  
  Now we consider a general nontrivial and irreducible $G$-representation $V$.
  It is isomorphic to $\alpha^*(\nu_m)$
  for some continuous homomorphism $\alpha\colon G\to U(m)$,
  so we may assume that $V=\alpha^*(\nu_m)$.
  We let $K=\alpha(G)$ be the image of $\alpha$, a closed subgroup of $U(m)$.
  We let $\epsilon\colon G\to K$ be the continuous epimorphism given by $\alpha$,
  with restricted codomain.
  Then $\res^{U(m)}_K(\nu_m)$ is nontrivial irreducible, and the special case applies to it.
  As we explained in Example \ref{eg:inflate_mon}, the inflation map \eqref{eq:inflate_mon}
  that sends a monomial structure $[L,\lambda]$ on $\res^{U(m)}_K(\nu_m)$
  to the monomial structure $[\epsilon^{-1}(L),\lambda\circ \epsilon|_{\epsilon^{-1}(L)}]$
  on $\epsilon^*(\res^{U(m)}_K(\nu_m))=V$ is bijective.
  Transfers commute with inflation, in the sense that
  \[ \epsilon^*(\tr_L^K(\lambda)) \ = \
    \tr_{\epsilon^{-1}(L)}^G(\epsilon^*(\lambda)) \ = \
    \tr_{\epsilon^{-1}(L)}^G(\lambda\circ\epsilon|_{\epsilon^{-1}(L)}) \ .
  \]
 Hence
  \begin{align*}
    \Phi^G(b_{G,V})\
    &= \ \Phi^G(\epsilon^*(b_{K,\res^{U(m)}_K(\nu_m)})) 
      \ = \ \epsilon^*(\Phi^K(b_{K,\res^{U(m)}_K(\nu_m)})) \\
    &= \  \sum_{[L,\lambda]\in\mon(\res^{U(m)}_K(\nu_m))} ( 1-\epsilon^*(\tr_L^K(\lambda)^\sharp ))\\
    &= \  \sum_{[L,\lambda]\in\mon(\res^{U(m)}_K(\nu_m))} ( 1-\tr_{\epsilon^{-1}(L)}^G(\lambda\circ \epsilon|_{\epsilon^{-1}(L)})^\sharp )\
    = \  \sum_{[H,\chi]\in\mon(V)} ( 1-\tr_H^G(\chi)^\sharp )\ .
  \end{align*}
  So the theorem holds for the $G$-representation $V$.
\end{proof}

\begin{eg}\label{eg:several_monomial_structures}
  An irreducible $G$-representation can have several monomial structures,
  so the right hand side of the formula in Theorem \ref{thm:Phi_on_irrep}
  may have more than one summand.
  An example is the quaternion group $G=Q=\{\pm 1,\pm i, \pm j, \pm k\}$,
  with $i^2=j^2=k^2=i j k=-1$.
  The group $Q$ has four 1-dimensional representations,
  and one other irreducible representation of dimension 2.
  The cyclic subgroups
  $I=\{\pm 1,\pm i\}$, $J=\{\pm 1,\pm j\}$ and $K=\{\pm 1,\pm k\}$
  of order~4 are normal in $Q$, but they are not conjugate.
  We let $\alpha\colon I\to T$,
  $\beta\colon J\to T$ and $\gamma\colon K\to T$  be any monomorphisms.
  Then the induced 2-dimensional representations
  $\ind_I^Q(\alpha)$, $\ind_J^Q(\beta)$ and $\ind_K^Q(\gamma)$
  are irreducible, and hence isomorphic.
  So $[I,\alpha]$, $[J,\beta]$ and $[K,\gamma]$ are three distinct
  monomial structures on the 2-dimensional irreducible $Q$-representation,
  and in fact all such monomial structures.
\end{eg}

We recall that a compact Lie group $G$ is called {\em cyclic}
if it has an element whose powers are dense in $G$.
Then $G$ is isomorphic to the product of a torus and a finite cyclic group,
see for example \cite[I Corollary 4.14]{broecker-tomdieck}.
Now we attack the main goal of this section:
every compact Lie group $G$ that is not cyclic
has a unitary representation $V$ such that $\Phi^G(b_{G,V})=0$.
This implies that for such groups,
inverting the pre-Bott classes kills the $G$-geometric fixed points of $\Sigma^\infty_+\bP$.
The easy case is when $G$ has an irreducible
representation without monomial structure: any such
representation serves the purpose by Theorem \ref{thm:Phi_on_irrep}.
So a large part of the upcoming work goes into dealing with monomial groups,
i.e., compact Lie groups all of whose irreducible representations are monomial.

Taketa showed in \cite{taketa:monomial} that
all finite monomial groups are solvable, see also
\cite[Corollary 5.13]{isaacs:character}; the same arguments also show
that monomial compact Lie groups are solvable.
The converse is not true, and the smallest witness is
the solvable group $S L_2(\mF_3)$ of order 24.
Every abelian compact Lie group is monomial,
because all irreducible representations are 1-dimensional.
Examples of nonabelian monomial groups are the symmetric group $\Sigma_3$,
the alternating group $A_4$, and the orthogonal group $O(2)$.
Further examples of monomial groups are those in class (b)
of Proposition \ref{prop:nearcyclic}.

The following definition will be convenient for our purposes.

\begin{defn}\label{def:near-cyclic}
  A finite group $G$ is {\em near-cyclic} if it is not cyclic,
  and for every normal subgroup $N$ of $G$ with $N\ne\{1\}$,
  the quotient group $G/N$ is cyclic.  
\end{defn}

The point of the previous definition is that every non-cyclic finite group admits
an epimorphism onto a near-cyclic group.
The next proposition separates the near-cyclic groups into three disjoint classes.
Of these, all groups in classes (a) and (b) are monomial, and those in class (c) are not.

\begin{prop}\label{prop:nearcyclic}
  Every finite near-cyclic group $G$ has exactly one of the following three properties.
  \begin{enumerate}[\em (a)]
  \item The group $G$ is an elementary abelian $p$-group of rank~2, for some prime $p$.
  \item The commutator subgroup $G'=[G,G]$ is nontrivial and abelian,
    and the conjugation action of the abelianization $G_{\ab}=G/G'$
    on $G'\setminus\{1\}$ is free.
  \item The commutator subgroup $G'$ is nontrivial and perfect.
    In particular, $G$ is not solvable.
  \end{enumerate}
  Moreover, if $G$ has property {\em (b)} or {\em (c)},
  then there is a subgroup $H$ of $G$ and a nontrivial character $\chi\colon H\to T$
  such that $\tr_H^G(\chi)=1$.
\end{prop}
\begin{proof}
  If $G$ is abelian and not cyclic, then it admits an epimorphism
  onto $C_p\times C_p$, for some prime $p$. So if $G$ is abelian and near-cyclic,
  it belongs to class (a).

  Now we suppose that $G$ is nonabelian, i.e., its commutator subgroup $G'$ is nontrivial.
  If the commutator subgroup is perfect, then $G$ belongs to class (c).
  We let $H$ be a nontrivial cyclic subgroup of $G'$, and we let $\chi\colon H\to T$
  be any nontrivial character. Since $G'$ is perfect, every $G'$-character is trivial,
  so in particular $\tr_H^{G'}(\chi)=1$. Transitivity of transfers then yields
  \[ \tr_H^G(\chi)\ = \ \tr_{G'}^G(\tr_H^{G'}(\chi))\ = \ \tr_{G'}^G(1)\ = \ 1 \ .  \]
  
  For the rest of the argument we assume that $G$ is near-cyclic, nonabelian
  and the commutator subgroup $G'$ is not perfect.
  We will argue in several steps that then $G$ has property (b).
  Since the commutator subgroup $G'$ is nontrivial and normal,
  and since $G$ is near-cyclic, the abelianization $G_{\ab}=G/G'$ is cyclic.
  The iterated commutator subgroup $G''=[G',G']$ is a normal subgroup of $G$.
  Since $G'$ is not perfect, $G''$ is strictly contained in $G'$. So $G/G''$ is
  a nonabelian quotient of $G$. Since $G$ is near-cyclic, $G''$ must be trivial,
  i.e., $G'$ is abelian. 

  Now we show that the center $Z(G)$ of $G$ is trivial.
  Otherwise, $G/Z(G)$ would be cyclic, by near-cyclicity of $G$.
  Then $G$ is generated by $Z(G)$ and any lift of a generator of $G/Z(G)$;
  since these elements commute with themselves and each other, $G$ would be abelian,
  contradicting the hypothesis.
  
  Now we show that $G$ belongs to class (b).
  For any element $g\in G\setminus G'$, we consider
  the centralizer of $g$ inside the commutator subgroup $G'$,
  \[ N\ =\ C_{G'}(g)\ = \ C_G(g)\cap G'\ . \]
  We show that $N$ is normal in $G$.
  Since $G'$  is abelian and normal, the conjugation action of 
  $G$ on $G'$ factors through $G_{\ab}=G/G'$, making 
  $G'$  a $\mZ[G_{\ab}]$-module. For any 
  $g\in G$ with image $\bar g=g G'$ in $G_{\ab}$, we have 
  $N=C_{G'}(g)=\{a\in G'\colon {^{\bar g} a}=a\}$, i.e., 
  $N$ is the fixed subgroup of $\bar g$ acting on 
  $G'$. Since $G_{\ab}$ is cyclic, hence abelian, $\bar g$
  commutes with every element of $G_{\ab}$, so the fixed-point subgroup 
  $N$ is $G_{\ab}$-invariant, and hence normal in $G$, as claimed.
  
  Now we show that $N$ is strictly smaller than $G'$.
  To prove that we argue by contradiction.
  If $N=G'$, then $g$ would centralize $G'$.
  We let $t\in G$ be an element whose coset $t G'\in G/G'=G_{\ab}$
  generates the abelianization.
  We write $g=t^i\cdot \gamma$ for some $i\geq 0$ and some $\gamma\in G'$.
  Because $G'$ is abelian, $\gamma$ centralizes $G'$.
  Hence $t^i=g\gamma^{-1}$ centralizes $G'$. Since $t^i$ also commutes with $t$,
  the element $t^i$ is central in $G$. Since $G$ is centerless,
  we conclude that $t^i=1$, and thus $g=\gamma\in G'$, contradicting the hypothesis on $g$.
  This proves that $N$ is a proper subgroup of $G'$.
  So $G/N$ is a nonabelian quotient of $G$.
  Since $G$ is near-cyclic, we must have $N=\{1\}$.
  So conjugation by $g$ does not fix any element of $G'\setminus\{1\}$,
  i.e., $G_{\ab}$ acts freely on $G'\setminus\{1\}$. So $G$ belongs to class (b).

  Since $G_{\ab}$ acts freely on $G'\setminus\{1\}$,
  the order of $G'$ is congruent to 1 modulo the order of $G_{\ab}$;
  in particular, the orders of $G'$ and $G_{\ab}$ are coprime.
  We let $\chi\colon G'\to T$ be any nontrivial character,
  which exists because $G'$ is abelian and nontrivial.
  The transfer $\tr_{G'}^G\colon\Hom(G',T)\to \Hom(G,T)$
  is a group homomorphism; its source has the same order as $G'$,
  and its target has the same order as $\Hom(G_{\ab},T)$, and hence as $G_{\ab}$.
  Since the orders of $G'$ and $G_{\ab}$ are coprime, this transfer homomorphism
  must be trivial. Hence $\tr_{G'}^G(\chi)=1$, and the addendum holds for $H=G'$.
\end{proof}

We can now prove the main result of this section for finite groups.

\begin{theorem}\label{thm:rep_with_trivial_Phi-finite}
  Let $G$ be a non-cyclic finite group.
  Then there is a unitary $G$-representation $V$ with $V^G=0$
  such that $\Phi^G(b_{G,V})=0$ in $\Phi^G_0(\Sigma^\infty_+\bP)$.
\end{theorem}
\begin{proof}
  We start with the case where $G$ is near-cyclic.
  Then $G$ has one of the three properties (a), (b) or (c)
  specified in Proposition \ref{prop:nearcyclic}. We give separate constructions
  of the desired $G$-representation.

  (a) For $G=C_p\times C_p$ for some prime $p$,
  we can let $V$ be the reduced regular representation,
  which is isomorphic to the direct sum of all nontrivial
  irreducible $G$-representations;
  hence 
  \[ b_{G,V}\ =\  \prod_{\chi\ne 1} b_{C_p\times C_p,\chi}\ , \]
  where the product is over all nontrivial homomorphisms $\chi\colon C_p\times C_p\to T$.
  The relation
  \[ \prod_{\chi\ne 1} ( 1-\chi)\  =\  0\]
  holds in the group ring $\mZ[\Hom(C_p\times C_p,T)]$. 
  Multiplicativity of the pre-Bott classes then yields
  \[ \Phi^{C_p\times C_p}(b_{C_p\times C_p,V})\ = \
    \prod_{\chi\ne 1} \Phi^{C_p\times C_p}(b_{C_p\times C_p,\chi})\ =_{\eqref{eq:Phi^G_of_char}} \
    \prod_{\chi\ne 1} (1-\chi^\sharp)\ = \  0 \ . \]

  (b) We suppose that $G$ is nonabelian, the commutator subgroup $G'$ is abelian,
  and the conjugation action of $G_{\ab}=G/G'$ on $G'\setminus\{1\}$ is free.
  In this case the order of the commutator subgroup $G'$
  is congruent to 1 modulo the order of $G_{\ab}$,
  so the orders of $G'$ and $G_{\ab}$ are coprime.
  Because $G_{\ab}$ acts freely on $G'\setminus\{1\}$, it also acts freely
  on the set of nontrivial homomorphisms from $G'$ to $T$.
  We let $\chi\colon G'\to T$ be any nontrivial homomorphism.
  Then $\chi\circ c_g\ne \chi$ for all $g\in G\setminus G'$.
  The induced $G$-representation $\ind_{G'}^G(\chi)$ is thus irreducible.
  The transfer $\tr_{G'}^G\colon\Hom(G',T)\to \Hom(G,T)$
  is a group homomorphism; its source has the same order as $G'$,
  and its target has the same order as $\Hom(G_{\ab},T)$, and hence as $G_{\ab}$.
  Since the orders of $G'$ and $G_{\ab}$ are coprime, this transfer homomorphism
  must be trivial. In particular, $\tr_{G'}^G(\chi)=1$.
  
  Since the orders of $G'$ and $G_{\ab}=G/G'$ are coprime, $G'$ is the only
  subgroup of $G$ of index $[G:G']$.
  So $[G',\chi]$ is the only monomial structure on $\ind_{G'}^G(\chi)$,
  by Proposition \ref{prop:unique_mon}.
  Hence for $V=\ind_{G'}^G(\chi)$, the formula of Theorem \ref{thm:Phi_on_irrep}
  has only one summand, and it yields the relation
  \[ \Phi^G(b_{G,V})\ = \  1-\tr_{G'}^G(\chi)^\sharp \ = 0 \ . \]

  (c) If $G$ is not solvable, then $G$ cannot be a monomial group,
  see  \cite{taketa:monomial} or \cite[Corollary 5.13]{isaacs:character}.
  Hence $G$ has an irreducible representation $V$ that is not
  induced from a 1-dimensional representation of any subgroup of $G$.
  Then $\Phi^G(b_{G,V})=0$ by Theorem \ref{thm:Phi_on_irrep}.

  It remains to treat the case of a general non-cyclic finite group.
  Any non-cyclic group $G$ admits an epimorphism $\epsilon\colon G\to K$
  onto a near-cyclic group. By the above, there is a $K$-representation $W$
  with $W^K=0$ and such that $\Phi^K(b_{K,W})=0$.
  Then $V=\epsilon^*(W)$ is the desired $G$-representation, because
  the geometric fixed point homomorphisms commute with inflation, see \eqref{eq:infl_and_fix-ROG},
  and thus
 \[   \Phi^G(b_{G,V})\ = \ \Phi^G( \epsilon^*(b_{K,W}))\
   =  \  \epsilon^*(\Phi^K(b_{K,W}))\ =\ 0 \ .\qedhere \]
\end{proof}

Now that we took care of finite groups, we turn to compact Lie groups
of positive dimension.
For those, most of the work goes into the case where the identity
component is abelian, and the component group is cyclic.

\begin{theorem}\label{thm:G^circ_image}
  Let $G$ be a nonabelian compact Lie group whose identity component $G^\circ$ is abelian,
  and whose component group $\pi_0(G)$ is cyclic.
  \begin{enumerate}[\em (i)]
  \item 
    There is a unitary $G$-representation $V$ with $V^G=0$ such that
    $\Phi^G(b_{G,V})=0$ in $\Phi^G_0(\Sigma^\infty_+\bP)$.
  \item  There is a finite index subgroup $H$ of $G$ and
    a nontrivial continuous homomorphism $\chi\colon H\to T$
    such that $\tr_H^G(\chi)=1$.
\end{enumerate}
\end{theorem}
\begin{proof}
  We first show a reduction scheme for epimorphisms.
  Let $\epsilon\colon K\to G$ be a continuous epimorphism, and suppose that the
  conclusion of the theorem holds for $G$. Then the conclusion of
  the theorem also holds for $K$.
  Indeed, if $V$ is a unitary $G$-representation with $V^G=0$ and 
  $\Phi^G(b_{G,V})=0$, the inflated $K$-representation $W=\epsilon^*(V)$
  satisfies $W^K=V^G=0$ and 
  \[ \Phi^K(b_{K,W})\ = \ \Phi^K(\epsilon^*(b_{G,V}))\ =_{\eqref{eq:infl_and_fix-ROG}} \
    \epsilon^*(\Phi^G(b_{G,V}))\ = \ 0\ . \]
  And if $\chi\colon H\to T$ is a nontrivial continuous homomorphism from a
  finite index subgroup of $G$   
  such that $\tr_H^G(\chi)=1$, then $L=\epsilon^{-1}(H)$
  is a finite index subgroup of $K$ and  $\beta=\chi\circ \epsilon|_L\colon L\to T$
  is a nontrivial continuous homomorphism
  such that
  \[ \tr_L^K(\beta)\ = \ \tr_L^K(\epsilon|_L^*(\chi))\ = \
    \epsilon^*(\tr_H^G(\chi))\ = \ 1\ .\]  
  So whenever properties (i) and  (ii) hold for $G$, they also hold for $K$.

  Now we prove the theorem in three steps, for successively more general classes of groups.
  In the proof of the following first case, we shall need an algebraic fact.
  We let $A$ be a torsion-free abelian group
  equipped with a faithful action of a finite cyclic group $C$ of order $n$.
  Such an action has a free orbit, i.e., there is an element $a\in A$
  whose orbit $C a$ has $n$ elements.
  Now let $\tau\colon A\to B$ be a homomorphism of abelian groups that is $C$-equivariant
  for the trivial $C$-action on the target.
  We define $b\in A$ by
  \[ b\ = \ n\cdot a - (a + t a +\dots + t^{n-1} a)\ ,\]
  where $t$ is a generator of $C$.
  Then $b$ lies in the kernel of $\tau$,
  and the $C$-orbit of $b$ also has $n$ elements.
  Indeed, if $t^k b = b$ for some $0 < k < n$,
  then expanding gives $n(t^k a - a) = 0$, so $t^k a = a$
  by torsion-freeness, contradicting the freeness of the orbit $C a$.
  So the kernel of $\tau$ contains a free $C$-orbit.
  
  Case 1: The conjugation action of $\pi_0(G)$ on $G^\circ$ is faithful.
  Then also the induced $\pi_0(G)$-action on $\Hom(G^\circ,T)$ is faithful.
  Since $G^\circ$ is a torus, the abelian group $\Hom(G^\circ,T)$ is torsion-free.
  Each element of $\pi_0(G)$ acts on $G^\circ$ by the restriction of an 
  inner automorphism of $G$. So the transfer homomorphism 
  \[ \tr_{G^\circ}^G \ \colon \ \Hom(G^\circ,T)\ \to \ \Hom(G,T) \]
  is equivariant for the trivial $\pi_0(G)$-action on the target.
  Since the $\pi_0(G)$-action on $\Hom(G^\circ,T)$ is faithful, the kernel of this transfer
  homomorphism contains a free $\pi_0(G)$-orbit, by the comment above.
  In other words: there is 
  a continuous homomorphism $\chi\colon G^\circ\to T$ such that
  $\tr_{G^\circ}^G(\chi)=1$ and such that
  $\chi\circ c_g\ne \chi$ for all $g\in G\setminus G^\circ$.
  In particular, condition (ii) holds for $H=G^\circ$.

  The fact that $\chi\circ c_g\ne \chi$ for all $g\in G\setminus G^\circ$
  means that the induced $G$-representation $\ind_{G^\circ}^G(\chi)$
  is irreducible.
  Since $G^\circ$ is the only subgroup of $G$ whose index equals $[G:G^\circ]$,
  $[G^\circ,\chi]$ is the only monomial structure on $\ind_{G^\circ}^G(\chi)$,
  by Proposition \ref{prop:unique_mon}.
  Theorem  \ref{thm:Phi_on_irrep} then lets us conclude
  \[ \Phi^G(b_{G,\ind_{G^\circ}^G(\chi)}) \ = \ 1 -\tr_{G^\circ}^G(\chi)^\sharp\ = \ 0\ .\]
  So condition (i) holds for $V=\ind_{G^\circ}^G(\chi)$.
  
  Case 2: The projection $q\colon G\to G/G^\circ=\pi_0(G)$ has a multiplicative section.
  We can assume without loss of generality that $G=C\ltimes_\psi G^\circ$
  is the semidirect product of an action of a finite cyclic group $C$
  on a torus $G^\circ$, given by a homomorphism $\psi\colon C\to \Aut(G^\circ)$.
  We let $D$ be the kernel of $\psi$, i.e., the subgroup
  of $C$ consisting of the elements that act trivially on $G^\circ$.
  Then the $C$-action descends uniquely to a faithful action $\bar\psi\colon C/D\to\Aut(G^\circ)$
  of the quotient group,
  and the group $G$ surjects onto the semidirect product $\bar G=(C/D)\ltimes_{\bar\psi}G^\circ$.
  Since $G$ is nonabelian, the $C$-action on $G^\circ$ is nontrivial.
  Hence also the induced $(C/D)$-action on $G^\circ$ is nontrivial,
  and the group $\bar G$ is nonabelian.
  Moreover the component group $\pi_0(\bar G)$ is isomorphic to $C/D$,
  hence cyclic, in a way that identifies the conjugation action
  on $\bar G^\circ=G^\circ$ with the action $\bar\psi\colon C/D\to\Aut(G^\circ)$.
  Since $C/D$ acts faithfully on $G^\circ$ by design, the group $\bar G$
  satisfies the hypothesis of case 1.
  Since the theorem holds for $\bar G$ by case 1,
  the epimorphic reduction scheme proves the theorem for the group $G$.

  Case 3, the general case:
  We let $N=Z(G)\cap G^\circ$ be the intersection of the center of $G$
  with the identity component $G^\circ$.
  We claim that the quotient group $G/N$ is nonabelian.
  If $G/N$ were abelian, then the commutator subgroup $G'$ of $G$ would
  be contained in $N$, and hence central in $G$. Since the abelianization $G_{\ab}=G/G'$
  has a cyclic component group $\pi_0(G_{\ab})=(\pi_0(G))_{\ab}=\pi_0(G)$,
  the abelianization is itself cyclic.
  But any group that has a central subgroup with cyclic quotient is abelian,
  contradicting the hypothesis on $G$. This shows that $G/N$ is nonabelian.
  The identity component of $G/N$ is $G^\circ/N$, and thus abelian.
  Since $N$ is contained in $G^\circ$, the quotient map $G\to G/N$
  induces an isomorphism of component groups $\pi_0(G)\iso\pi_0(G/N)$.
  
  We claim that the projection $G/N\to\pi_0(G/N)$ has a multiplicative section.
  To this end we let $t\in G$ be an element whose class $t\cdot G^\circ$ in $G/G^\circ=\pi_0(G)$
  generates the cyclic component group.
  Then the image of $t$ also generates the cyclic group $\pi_0(G/N)$.
  If $n=[G:G^\circ]$ is the order of $\pi_0(G)$, and hence also of $\pi_0(G/N)$,
  then $t^n\in G^\circ$.
  Since $G^\circ$ is abelian, $t^n$ commutes with $t$ and with $G^\circ$.
  Since $G$ is generated by $t$ and $G^\circ$, the element $t^n$ is central,
  and thus $t^n\in N$.
  So the quotient map $G/N\to \pi_0(G/N)$
  has a multiplicative section, as claimed.
  Since the theorem holds for $G/N$ by case 2,
  the epimorphic reduction scheme proves the theorem for the group $G$.
\end{proof}

The following theorem is the main result of this section,
implying that inverting pre-Bott classes kills the $G$-geometric
fixed points for all non-cyclic compact Lie groups.
 
\begin{theorem}\label{thm:rep_with_trivial_Phi-Lie}
  Let $G$ be a non-cyclic compact Lie group.
  Then there is a unitary $G$-representation $V$ with $V^G=0$ such that
  $\Phi^G(b_{G,V})=0$ in $\Phi^G_0(\Sigma^\infty_+\bP)$.
\end{theorem}
\begin{proof}
  We distinguish three cases.
  
  (a) If the component group $\pi_0(G)$ is not cyclic,
  then Theorem \ref{thm:rep_with_trivial_Phi-finite}
  provides a $\pi_0(G)$-representation $W$ such that
  $W^{\pi_0(G)}=0$ and $\Phi^{\pi_0(G)}(b_{\pi_0(G),W})=0$.
  We let $q\colon G\to G/G^\circ=\pi_0(G)$ denote the projection.
  Then the inflation $q^*(W)$ has the desired property, because
   $(q^*(W))^G=W^{\pi_0(G)}=0$ and
  \[   \Phi^G( b_{G,q^*(W)})\ = \  \Phi^G( q^*(b_{\pi_0(G),W}))\
    =  \  q^*(\Phi^{\pi_0(G)}(b_{\pi_0(G),W}))\ =\ 0 \ .\]
  
  (b) If the identity component $G^\circ$ is nonabelian,
  then $G^\circ$ has irreducible representations
  of arbitrarily large dimensions, see for example
  \cite[Chapter VI, (1.20), Exercise 2]{broecker-tomdieck}.
  Since every irreducible $G^\circ$-representation is a
  $G^\circ$-equivariant summand of some irreducible $G$-representation,
  $G$ itself has irreducible representations of arbitrarily large dimensions.
  In particular, there is an irreducible $G$-representation $V$ whose dimension exceeds
  the cardinality of the finite group $\pi_0(G)$.
  Then $G$ has no closed subgroup of finite index equal to the dimension of $V$,
  so $V$ has no monomial structure.
  Hence $\Phi^G(b_{G,V})=0$ by Theorem \ref{thm:Phi_on_irrep}.
  
  (c) If the component group $\pi_0(G)$ is cyclic
  and the identity component $G^\circ$ is abelian,
  then $G$ itself must be nonabelian, for otherwise it would be cyclic.
  So Theorem \ref{thm:G^circ_image} (i) provides the desired $G$-representation.
\end{proof}

\section{Global \texorpdfstring{$K$}{K}-theory as a global localization}

In this section we show our main result, Theorem \ref{thm:main},
namely that the morphism 
$\eta\colon \Sigma^\infty_+ \bP \to \bKU$ is a global localization by
the set of pre-Bott classes $\{b_m\}_{m\geq 1}$, both in the $\infty$-category of
commutative global ring spectra 
and also  in the $\infty$-category of ultra-commutative ring spectra.
The key ingredient is Theorem \ref{thm:geomfix},
saying that for every compact Lie group $G$,
the morphism $\Phi^G_*(\eta)\colon\Phi^G_*(\Sigma^\infty_+\bP)\to\Phi^G_*(\bKU)$
of geometric fixed point rings is a localization of graded-commutative rings by the images
of the pre-Bott classes of all unitary $G$-representations.\medskip

To set the stage, we first review some facts about the global $K$-theory spectrum
and about the morphism $\eta\colon \Sigma^\infty_+ \bP  \to \bKU$.
The global $K$-theory spectrum $\bKU$ was introduced
by Joachim \cite[Definition 4.3]{joachim:coherences},
see also \cite[Construction 6.4.9]{schwede:global};
for every compact Lie group $G$, the underlying genuine $G$-spectrum of $\bKU$ represents
$G$-equivariant complex $K$-theory,
see \cite[Theorem 4.4]{joachim:coherences} or \cite[Corollary 6.4.23]{schwede:global}.
In particular, the equivariant homotopy group $\pi_0^G(\bKU)$
is isomorphic to the complex representation ring of $G$;
a specific ring isomorphism $RU(G)\iso \pi_0^G(\bKU)$
is described in \cite[Theorem 6.4.24]{schwede:global}.
As the group $G$ varies, these isomorphisms are compatible with all the additional structure
in sight, i.e., with restriction homomorphisms, transfers, power operations and norms,
see again \cite[Theorem 6.4.24]{schwede:global}.

\begin{construction}
  The morphism of non-equivariant spectra from $\Sigma^\infty_+ \mC P^\infty$ to $K U$ 
  that classifies the tautological complex line bundle over $\mC P^\infty$
  has a particularly nice and prominent global-equivariant refinement,
  a morphism of ultra-commutative ring spectra:
  \[ \eta\  : \ \Sigma^\infty_+ \bP\ \to\ \bKU\ . \]
  The morphism is defined 
  as a composite of two morphisms of ultra-commutative ring spectra
  \[ \Sigma^\infty_+ \bP\ \xra{\ \mu\ }\  \bku\ \xra{\ j \ }\ \bKU \ .\]
  The intermediate object is
  the connective global $K$-theory spectrum \cite[Construction 6.3.9]{schwede:global},
  and the first morphism $\mu$ is the inclusion of the `rank 1' part
  in the rank filtration, compare \cite[Construction 6.3.40]{schwede:global}.
  The second morphism $j$ is a `periodization morphism' $j\colon \bku\to\bKU$
  defined in \cite[Construction 6.4.13]{schwede:global}.

  The underlying morphism of global spectra classifies
  the tautological $T$-representation on $\mC$, in the following sense.
  As mentioned earlier, $\bP$ is a global classifying space for the group $T$,
  so the pair $(\Sigma^\infty_+ \bP,e_T)$  represents the functor $\pi_0^T$,
  where $e_T \in  \pi_0^T(\Sigma^\infty_+\bP)$
  is the stable tautological class \eqref{eq:define_e_T}.
  Under the preferred identification \cite[Theorem 6.4.24]{schwede:global}
  of $\pi_0^T(\bKU)$ with the complex representation ring $RU(T)$, 
  the element $e_T$ maps to the class of the
  tautological $T$-representation on $\mC$, i.e.,
  \[ \eta_*(e_T)\ = \ [\nu_1] \quad \text{in $\pi_0^T(\bKU)\iso RU(T)$.} \]
  Because $\eta$ is a morphism of ultra-commutative ring spectra,
  its effect on equivariant homotopy groups is not only compatible with restrictions,
  inflations and transfers, but also with multiplicative power operations and norms.
\end{construction}

\begin{rk}[Pre-Bott classes map to Bott classes]\label{rk:preBott2Bott}
  We will show in Theorem \ref{thm:pre-Bott} below
  that any morphism $\eta\colon \Sigma^\infty_+\bP\to\bKU$ of commutative global ring spectra
  satisfying $\eta_*(e_T)=[\nu_1]$  inverts all the pre-Bott classes $b_m$.
  An invertible class in the same representation-grading as $\eta_*(b_m)$
  is the $U(m)$-equivariant Bott class, also called the {\em Thom class},
  of the tautological $U(m)$-representation $\nu_m$.
  More generally, the equivariant Bott class of a unitary $G$-representation $V$
  is defined for example
  by Atiyah \cite[Section 4]{atiyah:elliptic},
  Segal \cite[\S 3]{segal:equivariant K-theory},
  or Karoubi \cite[Section III]{karoubi:sur la K-theorie}.  
  We write $\beta_{G,V}\in  \pi_V^G(\bKU)$
  for the image of this Bott class under the isomorphism
  \[  \widetilde{KU}_G(S^V)  \ \xra{\ \iso \ } \ \widetilde{\bKU}_G^0(S^V)  =  \pi_V^G(\bKU)      \]
  established in \cite[Corollary 6.4.23]{schwede:global}.
  We claim that $\eta_*(b_{G,V})=\beta_{G,V}$.
  
  We start with the case $G=T$ and $V=\nu_1$.
  The $T$-equivariant Bott class of $\nu_1$ is in essence defined as the reduced
  $K$-theory class of the tautological $T$-equivariant line bundle
  over $P(\mC\oplus\nu_1)$, pulled back along a specific equivariant
  homeomorphism to $S^{\nu_1}$. All we need to know in the following is that
  it restricts to the classical non-equivariant Bott class,
  and that its image under
  $i^*\colon\widetilde{KU}_T(S^{\nu_1})\to  \widetilde{KU}_T(S^0)= RU(T)$
  is the class $1-[\nu_1]$. Since the isomorphisms of
  \cite[Corollary 6.4.23]{schwede:global} are natural for equivariant maps,
  and compatible with $\pi_0^T(\bKU)=RU(T)$, also  $i^*(\beta_{T,\nu_1})=1-[\nu_1]$.
  Moreover
  \[ i^*(\eta_*(b_1))\ = \ \eta_*(i^*(b_1)) \ = \ \eta_*(1-e_T)\ = \  1-[\nu_1] \ .\]
  Like for any global spectrum, the restriction homomorphism
  $i^*\colon \pi_{\nu_1}^T(\bKU)\to\pi_0^T(\bKU)$ is injective,
  compare the short exact sequence in Remark \ref{rk:b_1_unique},
  which proves that $\eta_*(b_1)=\beta_{T,\nu_1}$.

  The equivariant Bott classes are multiplicative and natural for restriction.
  For $m\geq 2$, this yields
  \begin{align*}
    \res^{U(m)}_{T^m}(\eta_*(b_m))\
    &= \  \eta_*(\res^{U(m)}_{T^m}(b_m))\
      = \  \eta_*(b_1\times\dots\times b_1)\\
    &= \  \eta_*(b_1)\times\dots\times\eta_*(b_1)\ = \
      \beta_{T,\nu_1}\times\dots\times\beta_{T,\nu_1}\\ 
    &= \  \beta_{T^m,\nu_1^m}\ =\ \res^{U(m)}_{T^m}(\beta_{U(m),\nu_m})\ .
  \end{align*}
  The second equation is Theorem \ref{thm:fix of pre-Bott} (i). 
  Because the restriction map of representation rings $\res^{U(m)}_{T^m}\colon RU(U(m))\to RU(T^m)$
  is injective, so is the restriction map
  $\res^{U(m)}_{T^m}\colon \pi_{\nu_m}^{U(m)}(\bKU)\to\pi_{\nu_1^m}^{T^m}(\bKU)$,
  by Bott periodicity for unitary representations in equivariant $K$-theory.
  Hence the previous relation implies that $\eta_*(b_m)=\beta_{U(m),\nu_m}$.
  A general unitary representation is isomorphic to $\alpha^*(\nu_m)$
  for some continuous homomorphism $\alpha\colon G\to U(m)$.
  Since the pre-Bott classes and the equivariant Bott classes
  are all natural for restriction along continuous homomorphisms,
  the relation $\eta_*(b_{G,V})=\beta_{G,V}$ follows.
\end{rk}

Our main result in particular says that the morphism $\eta$ inverts the pre-Bott classes,
which we show now.

\begin{theorem}\label{thm:pre-Bott}
  Let $\eta\colon \Sigma^\infty_+ \bP \to \bKU$ be a morphism of commutative global ring spectra
  such that $\eta_*(e_T)=[\nu_1]$ in $\pi_0^T(\bKU)=RU(T)$.
    Then $\eta$ inverts the  pre-Bott classes $b_m$ for all $m\geq 1$.
\end{theorem}
\begin{proof}
  We start with the case $m=1$, showing that $\eta_*(b_1)$ is invertible.
  We use that $\eta_*(b_1)$ equals the equivariant
  Bott class $\beta=\beta_{T,\nu_1}$, see Remark \ref{rk:preBott2Bott}.
  While the Thom isomorphism says that exterior multiplication by $\beta$
  is an isomorphism of reduced equivariant $K$-groups, this is not quite
  the same statement as invertibility in the sense of Definition \ref{def:RO-invertible},
  so we need an extra argument.
  The right square in the following diagram
  commutes because $\beta^\flat\colon\bKU\sm S^{\nu_1}\to \bKU$ is a morphism of $T$-spectra:
  \[ \xymatrix@C=12mm{
      0\ar[r] & \pi_0^T(\bKU)\ar[r]^-{(\bKU\sm i)_*}\ar@{=}[d] &  \pi_0^T(\bKU\sm S^{\nu_1})\ar[r]^-{\res^T_{\{1\}}} \ar[d]^{\beta^\flat_*} &
      \pi_0(\bKU\sm S^2) \ar[d]^{\beta^\flat_*} \ar[r]& 0\\
      0\ar[r] &\pi_0^T(\bKU)\ar[r]_-{(1-[\nu_1])\cdot -} &    \pi_0^T(\bKU)\ar[r]_-{\res^T_{\{1\}}} &
      \pi_0(\bKU)\ar[r] & 0 
    } \]
  The left square commutes because
  \[ \beta^\flat_*((\bKU\sm i)_*(x)) \ = \  i^*(\beta)\cdot x \ = \ (1-[\nu_1])\cdot x\]
  for all $x\in\pi_0^T(\bKU)$.
  We claim that both rows are exact.
  The upper row is part of the long exact homology sequence of
  the cofiber sequence of based $T$-spaces \eqref{eq:Tcof}, combined with
  the Wirthm{\"u}ller isomorphism $\pi_0^T(\bKU\sm T_+\sm S^1)\iso\pi_0(\bKU\sm S^2)$.
  It is short exact because the adjacent terms are non-equivariant and $T$-equivariant
  homotopy groups of $\bKU$ of odd degrees, and thus trivial.
  The lower row is exact because $\pi_0^T(\bKU)$
  is a Laurent polynomial ring on the class $[\nu_1]$, and hence a domain
  whose augmentation ideal is generated by $1-[\nu_1]$.
  
  The $T$-equivariant Bott class restricts
  to the classical non-equivariant Bott class in $\pi_2(\bKU)$,
  which is invertible in the graded ring $\pi_*(\bKU)$;
  in the model for complex $K$-theory that we work with, this is
  proved in \cite[Theorem 6.4.29]{schwede:global}.
  So the right vertical map in the above commutative diagram is bijective.
  The middle vertical map is thus bijective by the five lemma.
  In particular, there is a class $u\in\pi_0^T(\bKU\sm S^{\nu_1})$
  such that $\beta^\flat_*(u)=1$, and so the class $\beta=\eta_*(b_1)$
  is invertible by Proposition \ref{prop:invert_criterion}.

  Now we consider $m\geq 2$.
  Since the class $\eta_*(b_1)=\beta$ is invertible, so is its
  power $P^m(\beta)$, by Proposition \ref{prop:invertibility_power}.
  We let $\lambda\in \pi_0^{\Sigma_m\wr T}(\bKU\sm S^{\nu_1^m})$ be the inverse of $P^m(\beta)$,
  i.e., the unique class such that $(P^m(\beta))^\flat_*(\lambda)=1$.
  We claim that
  \[ \tr_{\Sigma_m\wr T}^{U(m)}(\lambda) \ \in \ \pi_0^{U(m)}(\bKU\sm S^{\nu_m})\]
  is inverse to $\eta_*(b_m)$.
  To this end we calculate
  \begin{align*}
    \res^{U(m)}_{T^m}((\eta_*(b_m))^\flat_*(\tr_{\Sigma_m\wr T}^{U(m)}(\lambda)))\
    &= \   (\eta_*(\res^{U(m)}_{T^m}(b_m)))^\flat_*(\res^{U(m)}_{T^m}(\tr_{\Sigma_m\wr T}^{U(m)}(\lambda)))\\
    &= \   (\eta_*(b_1\times\dots\times b_1))^\flat_*(\res_{T^m}^{\Sigma_m\wr T}(\lambda))\\
  &= \   (\beta\times\dots\times\beta)^\flat_*(\res_{T^m}^{\Sigma_m\wr T}(\lambda))\\
    &= \   \res^{\Sigma_m\wr T}_{T^m}( (P^m(\beta))^\flat_*(\lambda))\
      = \   \res^{\Sigma_m\wr T}_{T^m}(1)\ =\ 1\ .
  \end{align*}
  The second equation uses Theorem \ref{thm:fix of pre-Bott} (i) and the double coset formula
  $\res^{U(m)}_{T^m}\circ \tr_{\Sigma_m\wr T}^{U(m)}=\res^{\Sigma_m\wr T}_{T^m}$, compare
  \cite[IV Corollary 6.7 (i)]{lms}.
  
  The restriction homomorphism $\res^{U(m)}_{T^m}\colon \pi_0^{U(m)}(\bKU)\to
  \pi_0^{T^m}(\bKU)$ is isomorphic to the
  restriction homomorphism of representation rings $\res^{U(m)}_{T^m}\colon RU(U(m))\to RU(T^m)$,
  which is injective. So we can conclude that
  $(\eta_*(b_m))^\flat_*\allowbreak(\tr_{\Sigma_m\wr T}^{U(m)}(\lambda))=1$, as claimed.
  In particular, the class $\eta_*(b_m)$ is invertible.
\end{proof}

\begin{rk}[Geometric fixed points of $\bKU$]\label{rk:fix_of_KU}
  As already mentioned, a special feature of the global $K$-theory spectrum $\bKU$
  is that its geometric fixed points vanish
  for all non-cyclic compact Lie groups.
  The classical folklore proof proceeds via the Euler classes of unitary representations,
  but I do not know a reference. Here is an independent proof that
  uses our results on the pre-Bott classes.

  We let $G$ be a non-cyclic compact Lie group. 
  Theorem \ref{thm:rep_with_trivial_Phi-Lie} provides a unitary $G$-representation $V$
  with $V^G=0$ such that $\Phi^G(b_{G,V})=0$ in $\Phi^G_0(\Sigma^\infty_+\bP)$.  
  Because $\eta\colon\Sigma^\infty_+\bP\to \bKU$ inverts all pre-Bott classes $b_m$
  by Theorem \ref{thm:pre-Bott},
  and because $b_{G,V}$ is a restriction of some $b_m$ along a continuous homomorphism,
  $\eta$ also inverts the class  $b_{G,V}$ by Proposition \ref{prop:restriction_and_invert}.
  Hence the class $\Phi^G(\eta_*(b_{G,V}))=\Phi_*^G(\eta)(\Phi^G(b_{G,V})) $ is simultaneously
  invertible and zero. So the ring $\Phi^G_0(\bKU)$ is trivial.
\end{rk}

The next theorem identifies the morphism induced by 
$\eta\colon \Sigma^\infty_+\bP\to \bKU$ on $G$-geometric fixed points
as a specific localization; this property is the
key calculational input for the eventual proof that $\eta$ is
a localization by the pre-Bott classes.
We will make our statements about the morphism
$\eta\colon \Sigma^\infty_+\bP\to \bKU$ independent of the specific construction,
thereby making them future-proof in the following sense.
At the time of this writing, there does not yet seem to be
a model-independent definition of ultra-commutative ring spectra,
nor of the morphism $\eta$ in this category. However, I expect it to just be a matter
of time until the technology will catch up and provide  
an intrinsically higher categorical construction
of the $\infty$-category of ultra-commutative ring spectra.
For me, a litmus test for such a construction is whether it accommodates $\bKU$
as an ultra-commutative ring spectrum, as well as the morphism
$\eta\colon \Sigma^\infty_+\bP\to \bKU$, without reference to models.
If so, then our localization results will automatically apply to any such incarnation of $\eta$.

\begin{theorem}\label{thm:geomfix}
  Let $\eta\colon \Sigma^\infty_+ \bP \to \bKU$
  be a morphism of commutative global ring spectra
  such that $\eta_*(e_T)=[\nu_1]$ in $\pi_0^T(\bKU)=RU(T)$.
  Then for every compact Lie group $G$,
  the morphism of geometric fixed point rings
  \[ \Phi^G_*(\eta)\ : \ \Phi^G_*(\Sigma^\infty_+\bP)\ \to \ \Phi^G_*(\bKU) \]
  is a localization by the set of classes $\Phi^G(b_{G,V})$
  for all unitary $G$-representations $V$.
\end{theorem}
\begin{proof}
  If the group $G$ is not cyclic,
  then Theorem \ref{thm:rep_with_trivial_Phi-Lie} provides a $G$-representation $V$ with $V^G=0$
  such that $\Phi^G(b_{G,V})=0$
  in the ring $\Phi^G_0(\Sigma^\infty_+\bP)$.
  So the localization of $\Phi^G_*(\Sigma^\infty_+\bP)$ by $\Phi^G(b_{G,V})$ is trivial.
  Since $\Phi^G_*(\bKU)$ is also trivial, for example by Remark \ref{rk:fix_of_KU},
  the claim holds for all non-cyclic groups.

  Now we treat the case where $G$ is cyclic, hence abelian.
  Then all irreducible $G$-representations are 1-dimensional,
  so we must show that $\Phi^G_*(\eta)$ is a localization away
  from the classes $\Phi^G(b_{G,\lambda})$ for all continuous homomorphisms $\lambda\colon G\to T$.
  The following diagram of graded rings commutes because
  \[ \Phi^G_*(\eta)(\alpha^\sharp)\ = \ \Phi^G_*(\eta)(\Phi^G(\alpha^*(e_T))) \ = \  \Phi^G(\alpha^*(\eta_*(e_T))) \ = \ \Phi^G(\alpha^*[\nu_1])\ ,  \]
  and by naturality and multiplicativity of the geometric fixed point maps:
    \[ \xymatrix@C=30mm{
      \mZ[\Hom(G,T)]\tensor \pi_*(\Sigma^\infty_+\bP)\ar[d]_{\mZ[\Hom(G,T)]\tensor\pi_*(\eta)}
      \ar[rr]_-\iso^-{\alpha\tensor x\mapsto \alpha^\sharp\cdot p_G^*(x)} &&
      \Phi^G_*(\Sigma^\infty_+\bP)\ar[d]^{\Phi^G_*(\eta)}\\
      \mZ[\Hom(G,T)]\tensor K U_* \ar[r]_-\iso^-{\alpha\tensor x\mapsto \alpha^*[\nu_1]\cdot p_G^*(x)} \ar@<-1ex>@/_1pc/[rr]_(.2){\alpha\tensor x\mapsto \Phi^G(\alpha^*[\nu_1])\cdot p_G^*(x)}
      & \pi_*^G(\bKU) \ar[r]^-{\Phi^G}    &\Phi^G_*(\bKU)
    } \]
  The upper horizontal map is an isomorphism by Proposition \ref{prop:monomial_basis} (ii);
  the left lower horizontal map is an isomorphism because $G$ is abelian,
  and hence all irreducible $G$-representations are 1-dimensional.
  We showed in Remark \ref{rk:preBott2Bott} that $\eta$ sends the pre-Bott class $b_1$
  to the $T$-equivariant Bott class $\beta_{T,\nu_1}$.
  So the underlying morphism of non-equivariant ring spectra
  sends the classical non-equivariant pre-Bott class $b_{\cl}=\res^T_{\{1\}}(b_1)$
  in $\pi_2(\Sigma^\infty_+\bP)$  to the classical Bott class in $\pi_2(\bKU)$.
  By Snaith's theorem, the morphism
  $\eta_*\colon \pi_*(\Sigma^\infty_+ \bP)\to K U_*$ of non-equivariant
  homotopy rings is a localization by $b_{\cl}$.
  So the left vertical morphism in the diagram is a localization by $1\tensor b_{\cl}$.

  If $\lambda\colon G\to T$ is nontrivial, then
  $\Phi^G(b_{G,\lambda})=1-\lambda^\sharp$ by \eqref{eq:Phi^G_of_char};
  hence the isomorphism of Proposition \ref{prop:monomial_basis} (i)
  identifies $\Phi^G(b_{G,\lambda})$  with $1-\lambda$.  
  The first two maps in the above composite send $1-\lambda$
  to $1-[\lambda]\in\pi_0^G(\bKU)$.
  It is a well-known fact that the geometric fixed point map
  $\Phi^G\colon \pi_*^G(\bKU)\to \Phi_*^G(\bKU)$ is localization away
  from the set of classes $1-[\lambda]$ for all nontrivial
  continuous homomorphisms $\lambda\colon G\to T$; we recall the classical argument.
  The global $K$-theory spectrum is $G$-equivariantly complex oriented, with the
  Bott classes as Thom classes.
  So the geometric fixed point morphism $\Phi^G\colon \pi_*^G(\bKU)\to \Phi^G_*(\bKU)$
  presents the target as the localization of the source by
  the Euler classes of all nontrivial irreducible complex $G$-representations,
  see \cite[Proposition 7.4.1]{tomDieck:and representation theory} or
  \cite[Proposition 3.20]{greenlees-may:equivariant_stable}.
  Moreover, the Euler class of $\lambda$ is precisely
  \[ i^*(\beta_{G,\lambda}) \ = \  i^*(\lambda^*(\beta_{T,\nu_1})) \ = \
    \lambda^*(i^*(\beta_{T,\nu_1})) \ = \     \lambda^*(1-[\nu_1]) \ = \ 1-[\lambda]\ . \]
  After inverting the upper horizontal isomorphism in the diagram,
  we have factored the morphism
  $\Phi^G_*(\eta)\colon\allowbreak \Phi^G_*(\Sigma^\infty_+ \bP)\to\Phi^G_*(\bKU)$
  as the composite of two localizations:
  \begin{itemize}
  \item a morphism
    $\Phi^G_*(\Sigma^\infty_+\bP)\to\pi_*^G(\bKU)$ that inverts the class
    $p_G^*(b_{\cl})=\Phi^G(b_{G,\mC})$ in $\Phi^G_2(\Sigma^\infty_+\bP)$
    arising from the trivial  homomorphism $G\to T$;
  \item
    followed by the morphism $\Phi^G\colon \pi_*^G(\bKU)\to\Phi_*^G(\bKU)$
    that inverts the images of the classes $\Phi^G(b_{G,\lambda})$
    for all nontrivial homomorphisms $\lambda\colon G\to T$.
  \end{itemize}
  So the composite $\Phi^G_*(\eta)\colon \Phi^G_*(\Sigma^\infty_+\bP)\to\Phi^G_*(\bKU) $
  is a localization by the classes
  $\Phi^G(b_{G,\lambda})$ for all continuous homomorphisms $\lambda\colon G\to T$, trivial or not.
\end{proof}

Now we have all tools at hand to prove our main result.

\begin{theorem}\label{thm:main}\
  \begin{enumerate}[\em (i)]
  \item 
  Let $\eta\colon \Sigma^\infty_+ \bP \to \bKU$ be a morphism of commutative global ring spectra
  such that $\eta_*(e_T)=[\nu_1]$ in $\pi_0^T(\bKU)$.
  Then $\eta$ is a global localization in $\CAlg(\Sc p_{\gl}^\sm)$
  by the pre-Bott classes $\{b_m\}_{m\geq 1}$.
\item
  Let $\eta\colon \Sigma^\infty_+ \bP \to \bKU$ be a morphism of ultra-commutative ring spectra
  such that $\eta_*(e_T)=[\nu_1]$ in $\pi_0^T(\bKU)$.
  Then $\eta$ is  an ultra-commutative localization in $\UR_{\gl}$
  by the pre-Bott classes $\{b_m\}_{m\geq 1}$.
  \end{enumerate}
\end{theorem}
\begin{proof}
  Theorem \ref{thm:geomfix} witnesses the geometric fixed point criterion
  of Theorem \ref{thm:fix_criterion} for the pre-Bott classes $\{b_m\}_{m\geq 1}$,
  and for the morphism $\eta\colon \Sigma^\infty_+ \bP \to \bKU$, proving (i).
  In the situation of (ii), the underlying morphism of commutative global ring spectra
  is a global localization in  $\CAlg(\Sc p_{\gl}^\sm)$ by $\{b_m\}_{m\geq 1}$, by (i).
  Theorem \ref{thm:UR2CR} (iii) then shows that $\eta$ is also
  an ultra-commutative localization in $\UR_{\gl}$ by $\{b_m\}_{m\geq 1}$.
\end{proof}

\section{Equivariant \texorpdfstring{$K$}{K}-theory as a \texorpdfstring{$G$}{G}-equivariant localization}

In this section we fix a compact Lie group $G$.
Restricting the global morphism $\eta$ to underlying $G$-spectra
yields a morphism of ultra-commutative $G$-ring spectra 
\[ \eta_G\ :\ \Sigma^\infty_+\bP_G\ \to\  \bKU_G\ . \]
We will show in this section that also $\eta_G$ is a localization---in
several different ways---by the set of pre-Bott classes
of all unitary $G$-representations; or, equivalently,
by the pre-Bott classes of all irreducible unitary $G$-representations.
One of these localization properties is directly analogous to the properties
in the global context in Theorem \ref{thm:main}: $\eta_G$ is a localization
in the  $\infty$-category of commutative $G$-ring spectra,
see Theorem \ref{thm:main-G-structured}.
In contrast to the global situation, the morphism $\eta_G$ is also a
telescopic localization (sequential colimit)
in the $\infty$-category of $G$-spectra, along morphisms that realize multiplication
by the pre-Bott classes of all unitary $G$-representations,
see Theorem \ref{thm:telescopic localization}.
And finally, $\eta_G$ induces algebraic, $RO(G)$-graded localizations,
again by the pre-Bott classes of all unitary  $G$-representations,
for $\Sigma^\infty_+\bP_G$-modules, at the level of $RO(G)$-graded
homotopy rings, and on equivariant cohomology theories,
see Theorem \ref{thm:main-G-RO-version}.

In Corollary \ref{cor:no_Z_localization} we show a complementary result,
namely that for $G\ne\{1\}$,
one {\em cannot} obtain $\bKU_G$ from $\Sigma^\infty_+ \bP_G$ by telescopically
inverting any set of integer-graded $G$-equivariant homotopy classes;
so $RO(G)$-grading is essential here.
In particular, the main result of \cite{spitzweck-oestvaer:equivariant},
claiming the contrary for abelian compact Lie groups, is incorrect;
see Remark \ref{rk:spitz-oest} for more details.
\medskip

In the following a {\em set of representation-graded $G$-equivariant homotopy classes}
of a $G$-spectrum $R$ will mean a set $\Xc$ of
classes in $\pi_V^G(R)$ for varying $G$-representations $V$.
We will say that $v\in \pi_V^G(R)$ is {\em supported} at $V$.
The main example for this paper is the set $\{b_{G,V}\}$
of pre-Bott classes of the $G$-ring spectrum $\Sigma^\infty_+\bP_G$,
for unitary $G$-representations $V$. 

\begin{defn}[Localization in $\CAlg(\Sc p_G^\sm)$]
  Let $\Xc$ be a set of representation-graded $G$-equivariant homotopy classes of
  a commutative $G$-ring spectrum $R$.
  A {\em localization by $\Xc$}
  is an initial object of the full subcategory of the
  slice category $R\bs\CAlg(\Sc p_G^\sm)$ spanned by the
  $\Xc$-inverting  morphisms.
\end{defn}

The next result is a $G$-equivariant analog of
the global localization result in Theorem \ref{thm:main} (i).

\begin{theorem}\label{thm:main-G-structured}
  Let $G$ be a compact Lie group.
  Let $\eta\colon \Sigma^\infty_+ \bP \to \bKU$ be a morphism of commutative global ring spectra
  such that $\eta_*(e_T)=[\nu_1]$ in $\pi_0^T(\bKU)=RU(T)$.
  Then the morphism $\eta_G\colon\Sigma^\infty_+ \bP_G\to\bKU_G$ 
  of commutative $G$-ring spectra is a localization in $\CAlg(\Sc p_G^\sm)$
  by the pre-Bott classes $b_{G,V}$ for all unitary $G$-representations $V$.
\end{theorem}
\begin{proof}
  We apply the localization recognition of Theorem \ref{thm:loc_reg}
  to the presentably symmetric monoidal $\infty$-category $\Sc p^\sm_G$
  of $G$-spectra, to the morphism $\eta_G\colon\Sigma^\infty_+ \bP_G\to\bKU_G$,
  and to the family of geometric fixed point functors
  $\Phi^H\circ\res^G_H\colon\Sc p^\sm_G\to \Sc p^\sm$,
  for $H$ ranging over all closed subgroups of $G$.
  These functors are strong symmetric monoidal left adjoints, and jointly conservative.
  The full subcategory $\Wc$ of $(\Sigma^\infty_+\bP_G)\bs\CAlg(\Sc p^\sm_G)$
  is spanned by the morphisms that invert the pre-Bott classes $b_{G,V}$
  for all unitary $G$-representations $V$; it is closed under postcomposition
  by Proposition \ref{prop:invert_inherit} (i),
  and it contains $\eta_G$ by Theorem \ref{thm:pre-Bott}
  and Proposition \ref{prop:restriction_and_invert}.
 
  We claim that hypothesis (c) of Theorem \ref{thm:loc_reg} holds for every
  closed subgroup $H$ of $G$.
  Since $\bKU$ is a global ring spectrum, $\res^G_H(\bKU_G)$ is equivalent to $\bKU_H$.
  By Theorem \ref{thm:geomfix},
  the morphism $\Phi^H_*(\eta_G)\colon \Phi^H_*(\Sigma^\infty_+\bP_G)\to \Phi^H_*(\bKU_G)$
  is a localization by the set of classes
  $\Phi^H(b_{H,W})$ for all unitary $H$-representations $W$.
  While not every $H$-representation underlies a $G$-representation,
  every $H$-representation is a direct summand in a $G$-representation.
  Since the pre-Bott classes are multiplicative (Theorem \ref{thm:fix of pre-Bott} (ii)),
  and a product in a graded-commutative ring is invertible
  if and only if its factors are, we deduce that
  $\Phi^H_*(\eta_G)$ is also a localization by the set of classes
  $\Phi^H(b_{G,V})$ for all unitary  $G$-representations $V$.

  Any $f\colon \Sigma^\infty_+\bP_G\to P$ in $\Wc$
  inverts $b_{G,V}$ for all unitary $G$-representations $V$,
  and hence it inverts all the classes $\Phi^H(b_{G,V})$ by
  Propositions \ref{prop:invert_criterion} and \ref{prop:invert_inherit} (ii).
  Since  $\Phi^H_*(\eta_G)\colon\allowbreak\Phi^H_*(\Sigma^\infty_+\bP_G)\to \Phi^H_*(\bKU_G)$
  is a localization by these classes, the morphism $\Phi^H_*(f)$ factors through $\Phi^H_*(\eta_G)$.
  So Theorem \ref{thm:loc_reg} applies and shows that $\eta_G\colon \Sigma^\infty_+\bP_G\to\bKU_G$
  is a localization by $\{b_{G,V}\}$ in $\CAlg(\Sc p^\sm_G)$.
\end{proof}

\begin{rk}[$\eta_G\colon\Sigma^\infty_+ \bP_G\to\bKU_G$
  as an ultra-commutative $G$-localization]
  Besides commutative $G$-ring spectra, there are even more highly
  structured kinds of genuine equivariant ring spectra:
  the {\em $\infty$-category of ultra-commutative $G$-ring spectra}  $\UR_G$
  is the localization of the 1-category of commutative orthogonal $G$-ring spectra
  at the class of those morphisms that are $\upi_*$-isomorphisms
  of underlying orthogonal $G$-spectra.
  A {\em localization} of an ultra-commutative $G$-ring spectrum $R$
  by a set of representation-graded $G$-equivariant homotopy classes
  $\Xc$ is then an initial object of the full subcategory of the
  slice category $R\bs\UR_G$ spanned by the $\Xc$-inverting morphisms.

  Much like in the global context in \eqref{eq:ucom2CR_gl}, one can define a `forgetful' functor
  \[ u_G\ :\ \UR_G\ \to\  \CAlg(\Sc p_G^\sm)  \]
  from the $\infty$-category of ultra-commutative $G$-ring spectra to the
  $\infty$-category of commutative $G$-ring spectra.
  Also like in the global context,
  this forgetful functor is {\em not} an equivalence
  (except when the group $G$ is trivial),
  and one can show that it exhibits the $\infty$-category of ultra-commutative $G$-ring spectra
  as bimonadic over the $\infty$-category of commutative $G$-ring spectra.
   In the $G$-equivariant context, the analogous warning as in the global context applies:
  for ultra-commutative $G$-ring spectra,
  the two localizations in $\UR_G$ and in $\CAlg(\Sc p_G^\sm)$ need not coincide
  in general.

  One can adapt the proof of Theorem \ref{thm:UR2CR} (iii)
  from the global to the $G$-equivariant context to show the following:
  Let $i\colon R\to S$ be a morphism of ultra-commutative $G$-ring spectra
  such that the underlying morphism $u_G i\colon u_G R\to u_G S$
  is a localization by $\Xc$ in $\CAlg(\Sc p_G^\sm)$.
  Then $i$ is a localization by $\Xc$ in $\UR_G$.
  In fact, the proof of Theorem \ref{thm:UR2CR} (iii) literally goes through
  in the $G$-equivariant context
  if one uses that the forgetful functor $u_G$ is bimonadic.
  This fact, in turn, can be proved in much the same way as its global analog in
  Theorem \ref{thm:UR_monadic} (iii).
  However, as this paper is already quite long, we will not go
  into a detailed proof.  

  Assuming all this, we obtain the $G$-ultra-commutative analog of
  Theorem \ref{thm:main-G-structured}:
  Let $\eta\colon \Sigma^\infty_+ \bP \to \bKU$ be a morphism of ultra-commutative ring spectra
  such that $\eta_*(e_T)=[\nu_1]$ in $\pi_0^T(\bKU)=RU(T)$.
  Then the morphism of ultra-commutative $G$-ring spectra
  $\eta_G\colon\Sigma^\infty_+ \bP_G\to\bKU_G$ is a localization in $\UR_G$ 
  by the pre-Bott classes $b_{G,V}$ for all unitary $G$-representations $V$.
\end{rk}

Our next task is to prove that the $RO(G)$-graded homotopy ring of $\bKU_G$
is an algebraic localization of the $RO(G)$-graded homotopy ring of $\Sigma^\infty_+\bP_G$
at the pre-Bott classes of the unitary $G$-representations, see Theorem \ref{thm:main-G-RO-version}.
The topic of `$RO(G)$-gradings'
has a somewhat colorful track record, and provides some pitfalls to be avoided.
Readers unfamiliar with the subtleties of $RO(G)$-gradings are encouraged to consult Adams' 
entertaining introduction to \cite[\S 6]{adams:prerequisites}.
Since we want to use $RO(G)$-graded multiplicative structure,
we give a brief review of how we handle these issues.

\begin{construction}
  Lewis and Mandell \cite[Theorem A.1]{lewis-mandell}
  and Dugger \cite{dugger:multigrade} show that one can make consistent choices of
  invertible $G$-spectra for virtual representations, in the following sense.
  We view the representation ring $RO(G)$ as a
  discrete symmetric monoidal category, i.e., with only identity morphisms,
  and with monoidal product the addition in $RO(G)$.
  This symmetric monoidal structure is extremely strict, in that all coherence isomorphisms,
  for unitality, associativity, and even commutativity, are identities.
  The construction of \cite{dugger:multigrade, lewis-mandell} provides a strong monoidal functor
  \begin{equation}\label{eq:duggerlewismandell}
  \mS^\bullet_G\  : \  RO(G)\ \to\ h(\Sc p_G^\sm)   
  \end{equation}
  with values in smash-invertible objects. Moreover, for every orthogonal $G$-representation $V$,
  the invertible object $\mS^{[V]}_G=\mS^\bullet_G([V])$ is isomorphic in 
  $h(\Sc p_G^\sm)$ to the suspension spectrum $\Sigma^\infty S^V$,
  but not in any preferred or specified way.   
  The construction proceeds from enumerated choices of representatives for the isomorphism
  classes of irreducible orthogonal $G$-representations $\lambda$,
  plus choices of inverses of $\Sigma^\infty S^\lambda$, including the data
  witnessing the inverse property.
  There is no intrinsic way to make `canonical' choices of these data,
  and the choices affect the resulting $RO(G)$-graded rings.
  For definiteness, we shall work with Dugger's construction,
  because it also enjoys a specific symmetry feature discussed below.
  Other than that, we do not need to know the precise construction
  of \eqref{eq:duggerlewismandell}.
   
  Part of the data of the strong monoidal functor \eqref{eq:duggerlewismandell}
  are isomorphisms
  \begin{equation}\label{eq:define_omega}
    \omega_{\alpha,\beta}\ : \ \mS^{\alpha+\beta}_G\ \xra{\ \iso\ }\ \mS^\alpha_G \sm \mS^\beta_G
  \end{equation}
  for all $\alpha,\beta\in RO(G)$, which satisfy the associativity constraints.
  We emphasize that the functor $\mS^\bullet_G$ is not {\em symmetric} monoidal;
  this is already visible for $G=\{1\}$, where $RO(\{1\})=\mZ$,
  simply because the twist isomorphism of $S^1\sm S^1$ has degree $-1$. 
  But the functor $\mS_G^\bullet$ has a particular kind of symmetry behavior:
  since $\mS_G^{\alpha+\beta}=\mS_G^{\beta+\alpha}$ is invertible,
  there is a unique invertible element $t_{\alpha,\beta}\in \pi_0^G(\mS)$
  that measures the deviation from symmetry, in the sense that
  the following diagram commutes:
  \[   \xymatrix@C=18mm@R=6mm{
      \mS_G^{\alpha+\beta}\ar[r]^-{t_{\alpha,\beta}\cdot -}  \ar[d]_{\omega_{\alpha,\beta}}  &
      \mS_G^{\beta+\alpha}\ar[d]^{\omega_{\beta,\alpha}}\\
      \mS_G^\alpha\sm \mS_G^\beta \ar[r]_-{\text{symmetry}}  &
      \mS_G^\beta\sm \mS_G^\alpha  } \]
  Then clearly $t_{\alpha,\beta}=t_{\beta,\alpha}^{-1}$,
  and $t_{\alpha,0}=1$ by the unitality condition.
  The hexagon coherence conditions between the
  associativity and symmetry isomorphisms in $h(\Sc p_G^\sm)$,
  and the fact that the isomorphisms $\omega_{\alpha,\beta}$ are associative yield
  the relations
  \begin{equation}\label{eq:t_biadd}
    t_{\alpha,\gamma}\cdot t_{\beta,\gamma}\ = \ t_{\alpha+\beta,\gamma}   \text{\quad and \quad}
    t_{\alpha,\beta}\cdot t_{\alpha,\gamma}\ = \ t_{\alpha,\beta+\gamma}
  \end{equation}
  for all $\alpha,\beta,\gamma$ in $RO(G)$.
  Moreover, the element $\varepsilon_\alpha=t_{\alpha,\alpha}$ in $\pi_0^G(\mS)$ 
  satisfies $\varepsilon_\alpha^2=1$, and it equals the trace of the identity of $\mS^\alpha_G$,
  and hence the Euler characteristic of the invertible object $\mS^\alpha_G$.
  The relations \eqref{eq:t_biadd}
  imply that the units $t_{\alpha,\beta}$ are determined by the special cases
  where $\alpha=[\lambda]$ and $\beta=[\mu]$ are the classes of two irreducible
  representations $\lambda$ and $\mu$. A property of Dugger's construction is
  that $t_{[\lambda],[\mu]}=1$ whenever $\lambda$ and $\mu$ are non-isomorphic irreducibles.
  So in fact
  \[ t_{\alpha,\beta}\ = \ \prod_{\lambda} \varepsilon_{[\lambda]}^{m(\lambda,\alpha)\cdot m(\lambda,\beta)}\ ,\]
  where the product is over a set of representatives of the isomorphism
  classes of irreducible representations, and $m(\lambda,\alpha)\in\mZ$
  is the multiplicity of $\lambda$ in $\alpha$, compare \cite[Proposition 1.2]{dugger:multigrade}.
  
  For later use in Construction \ref{con:RO(G) pre-Bott} we note that
  $\varepsilon_{[\lambda]}=1$ whenever the orthogonal representation $\lambda$
  underlies a unitary $G$-representation, simply because for every
  unitary $G$-representation $W$, the twist self-isometry of $S^{W\oplus W}$
  is equivariantly homotopic to the identity.
  Moreover, if $\lambda$ is an irreducible orthogonal representation of real type,
  then it occurs in any unitary representation with even multiplicity.
  Consequently, the symmetry markers $t_{\alpha,\beta}$ equal $1$
  whenever $\alpha$ and $\beta$ are in the image of the forgetful homomorphism $RU(G)\to RO(G)$.
\end{construction}

\begin{construction}[$RO(G)$-graded homotopy groups]\label{con:pi_star^G}
  The strong monoidal functor $\mS^\bullet_G$ from \eqref{eq:duggerlewismandell}
  allows us to introduce $RO(G)$-graded $G$-equivariant homotopy groups,
  as follows. For a $G$-spectrum $X$ and a virtual representation $\alpha\in RO(G)$,
  we set
  \[ \pi_\alpha^G(X)\ = \ h(\Sc p_G)(\mS_G^\alpha,X)\ ,\]
  the abelian group of morphisms, in the homotopy category of $G$-spectra,
  from $\mS_G^\alpha$ to $X$.
  Since the object  $\mS_G^\alpha$ is invertible, it is in particular compact
  in $h(\Sc p_G)$, so the functor $\pi_\alpha^G$ sends arbitrary coproducts
  to direct sums.  We shall write
  \[ \pi_\star^G(X)\ = \ \{\pi_\alpha^G(X)\}_{\alpha\in RO(G)} \]
  for the resulting $RO(G)$-graded abelian group.

  If $Y$ is another $G$-spectrum and $\beta\in RO(G)$, we obtain a biadditive pairing
  \begin{align}\label{eq:RO-sm}
  \sm \ : \   \pi_\alpha^G(X)\times \pi_\beta^G(Y)\quad
    = \quad  &h(\Sc p_G)(\mS_G^\alpha,X)\times h(\Sc p_G)(\mS_G^\beta,Y)\\
        \xra{\quad \sm \quad}\  &h(\Sc p_G)(\mS_G^\alpha\sm \mS_G^\beta,X\sm Y)\nonumber\\
    \xra[\eqref{eq:define_omega}]{-\circ \omega_{\alpha,\beta}}\  &h(\Sc p_G)(\mS_G^{\alpha+\beta},X\sm Y)\ =\ \pi_{\alpha+\beta}^G(X\sm Y)\ .\nonumber
  \end{align}
   It is a routine check, using that the functor $\mS_G^\bullet$ is monoidal,
   that these pairings are associative and unital.
   Since the functor $\mS_G^\bullet$ is not symmetric monoidal,
   the pairing is not strictly commutative, which is already the case for
   the usual $\mZ$-grading in stable homotopy theory.
   The deviation from strict commutativity is measured by the units $t_{\alpha,\beta}$,
   in the sense that
   \[ x\sm y \ =\ t_{\alpha,\beta}\cdot (\tau_{Y,X})_*(y\sm x)
     \text{\quad   in $\pi_{\alpha+\beta}^G(X\sm Y)$,} \]
   for all $x\in \pi_\alpha^G(X)$ and $y\in\pi_\beta^G(Y)$.
   Here $\tau_{Y,X}\colon Y\sm X\to X\sm Y$ denotes the symmetry equivalence for
   the smash product of $G$-spectra.
   
   Now we let $R$ be a homotopy $G$-ring spectrum
   with multiplication morphism $\mu\colon R\sm R\to R$ in $h(\Sc p_G^\sm)$.
   Then the composite maps
   \begin{equation}\label{eq:RO_product}
 \cdot \ : \ \pi_\alpha^G(R)\times \pi_\beta^G(R)\ \xra{\ \sm \ }\
    \pi_{\alpha+\beta}^G(R\sm R)\ \xra{\ \mu_*\ }\
    \pi_{\alpha+\beta}^G(R)       
   \end{equation}
  enhance $\pi_\star^G(R)$ to an $RO(G)$-graded ring.
  And for every left homotopy $R$-module $M$
  in $h(\Sc p_G^\sm)$, the composite maps
  \[ \pi_\alpha^G(R)\times \pi_\beta^G(M)\ \xra{\ \sm \ }\
    \pi_{\alpha+\beta}^G(R\sm M)\ \xra{\ \text{act}_*\ }\
    \pi_{\alpha+\beta}^G(M)   \]
  enhance $\pi_\star^G(M)$ to an $RO(G)$-graded left $\pi_\star^G(R)$-module.
  If the multiplication of $R$ is homotopy-commutative,
  then
  $\pi_\star^G(R)$ is a `graded-commutative' $RO(G)$-graded ring,
  in the sense that
  \[ y\cdot x \ =\ t_{\alpha,\beta}\cdot (x\cdot y)
     \text{\quad   in $\pi_{\beta+\alpha}^G(R)$,} \]
   for all $x\in \pi_\alpha^G(R)$ and $y\in\pi_\beta^G(R)$.
\end{construction}

We can descend the representation-graded pre-Bott classes
to $RO(G)$-graded equivariant homotopy classes, as follows.

\begin{construction}[$RO(G)$-graded pre-Bott classes]\label{con:RO(G) pre-Bott}
  We let $G$ be a compact Lie group.
  As we shall now explain, we can make `consistent' choices
  of pre-Bott classes $\mathbbm{b}_{G,\td{V}}\in \pi_{[V]}^G(\Sigma^\infty_+\bP_G)$
  in the $RO(G)$-graded homotopy ring.
  These are indexed by effective classes in $RU(G)$,
  i.e., ones represented by actual unitary $G$-representations;
  as the notation suggests, the class $\mathbbm{b}_{G,\td{V}}$ only depends
  on the isomorphism class $\td{V}$ of $V$ in $RU(G)$,
  but not on the specific representative.
  One should mind the appearance of both $RU(G)$ and $RO(G)$:
  the classes $\mathbbm{b}_{G,\td{V}}$ are defined for effective elements
  in the complex representation ring $RU(G)$, but they live in the
  $RO(G)$-graded homotopy ring, in the degree of the underlying
  orthogonal $G$-representation.
  
  For the construction, we choose representatives $\{\mu_i\}_{i\in I}$
  for the isomorphism classes of irreducible unitary $G$-representations.
  For every chosen representative, we choose an equivalence of $G$-spectra
  $\kappa_i\colon \mS_G^{[\mu_i]}\iso\Sigma^\infty S^{\mu_i}$,
  where $[\mu_i]$ is the class in $RO(G)$ of the underlying orthogonal
  $G$-representation of $\mu_i$. Then we define
  \[ \mathbbm{b}_{G,\td{\mu_i}}\ = \
    \kappa^*_i(b_{G,\mu_i})\ \in \ \pi_{[\mu_i]}^G(\Sigma^\infty_+\bP_G)\ .\]
  A general unitary $G$-representation $V$ is isomorphic to
  $\mu_{i_1}\oplus\dots\oplus\mu_{i_k}$
  for suitable $i_1,\dots,i_k\in I$, and we set
  \begin{equation}\label{eq:RO-pre-Bott}
    \mathbbm{b}_{G,\td{V}}\ = \
    \mathbbm{b}_{G,\td{\mu_{i_1}}}\cdot\ldots\cdot \mathbbm{b}_{G,\td{\mu_{i_k}}}
    \ \in \ \pi_{[V]}^G(\Sigma^\infty_+\bP_G)  \ ,  
  \end{equation}
  where the product is taken in the $RO(G)$-graded homotopy ring as in \eqref{eq:RO_product}.
  As mentioned above, the symmetry markers $t_{\alpha,\beta}$ equal $1$
  whenever $\alpha$ and $\beta$ are in the image of the forgetful homomorphism $RU(G)\to RO(G)$.
  Hence the class \eqref{eq:RO-pre-Bott} does not depend on the ordering of the
  irreducible summands of $V$, and moreover
  \[ \mathbbm{b}_{G,\td{V}}\cdot \mathbbm{b}_{G,\td{W}}\ = \ \mathbbm{b}_{G,\td{V\oplus W}}\ .\]
  A final property is the following consistency: for every unitary $G$-representation $V$,
  there is an equivalence $\kappa_V\colon \mS_G^{[V]}\xra{\sim}\Sigma^\infty S^V$
  such that $\mathbbm{b}_{G,\td{V}}=\kappa_V^*(b_{G,V})$.
  Indeed, we may choose a $\mC$-linear $G$-equivariant isometry
  $\psi\colon\mu_{i_1}\oplus\dots\oplus\mu_{i_k}\iso V$, and then the composite equivalence
  \begin{align*}
    \mS_G^{[V]}\ = \ \mS_G^{[\mu_{i_1}]+\dots+[\mu_{i_k}]} \ \xra{\omega_{[\mu_{i_1}],\dots,[\mu_{i_k}]}}\
    &\mS_G^{[\mu_{i_1}]}\sm \dots\sm\mS_G^{[\mu_{i_k}]} \\
    \xra{\kappa_{i_1}\sm\dots\sm\kappa_{i_k}}\ &(\Sigma^\infty S^{\mu_{i_1}})\sm\dots\sm(\Sigma^\infty S^{\mu_{i_k}})\\
    \iso \quad &\Sigma^\infty S^{\mu_{i_1}\oplus\dots\oplus\mu_{i_k}}\ \xra{\ \Sigma^\infty S^\psi\ } \
                 \Sigma^\infty S^V
  \end{align*}
  has the desired property.
\end{construction}

\begin{construction}[Telescopic localization]
  We choose an exhaustive sequence
  \begin{equation}\label{eq:exhaust}
    U_1 \ , \ U_2 \ , \ U_3\  , \ \dots    
  \end{equation}
  of finite-dimensional unitary $G$-representations,
  i.e., such that every finite-dimensional unitary $G$-re\-presen\-ta\-tion
  embeds, equivariantly and $\mC$-linearly, into 
  \[ V_n\ =  \ U_1\oplus\dots\oplus U_n\]
  for sufficiently large $n$.
  We set
  \[ \alpha_n\ =  \ [V_n]\ = \ [U_1\oplus\dots\oplus U_n] \]
  in $RO(G)$, including $\alpha_0=0$.
  We represent the pre-Bott class $\mathbbm{b}_{G,\td{U_n}}\in\pi_{[U_n]}^G(\Sigma^\infty_+\bP_G)$
  by a morphism
  \[ h_n\ : \ \Sigma^\infty_+\bP_G\sm \mS_G^{-\alpha_{n-1}}\ \to \
    \Sigma^\infty_+\bP_G \sm \mS_G^{-\alpha_n} \]
  in $\Mod_{\Sigma^\infty_+\bP_G}(\Sc p_G^\sm)$, i.e.,
  such that 
  \[ h_n(1\sm \mS_G^{-\alpha_{n-1}})\ =\ \mathbbm{b}_{G,\td{U_n}}\sm \mS_G^{-\alpha_n} \ ,\]
  where the pairing on the right hand side is
  the $RO(G)$-graded external multiplication  
  \[ \sm \ : \ \pi_{[U_n]}^G(\Sigma^\infty_+\bP_G)\times \pi^G_{-\alpha_n}(\mS_G^{-\alpha_n})\
    \to \
    \pi_{-\alpha_{n-1}}^G(\Sigma^\infty_+\bP_G\sm\mS_G^{-\alpha_n})  \]
  defined in \eqref{eq:RO-sm}; this defines the morphism $h_n$ uniquely up to homotopy.
  The sequence of morphisms
  \begin{equation}\label{eq:KU_G-sequence}
    \Sigma^\infty_+ \bP_G\xra{\ h_1\ } 
    \Sigma^\infty_+ \bP_G\sm \mS^{-\alpha_1}_G \xra{\ h_2\ }
    \dots \xra{\ h_n\ } \Sigma^\infty_+ \bP_G\sm \mS_G^{-\alpha_n} \dots
  \end{equation}
  then determines a functor
  $F\colon \mN\to\Mod_{\Sigma^\infty_+\bP_G}(\Sc p_G^\sm)$.

 We let $\Sigma^\infty_+\bP_G$ act on $\bKU_G$ through the morphism
 $\eta_G$ of commutative $G$-ring spectra, making $\bKU_G$ an object of
 $\Mod_{\Sigma^\infty_+\bP_G}(\Sc p_G^\sm)$. We define 
 \[ k_n\ :\  \Sigma^\infty_+ \bP_G\sm \mS^{-\alpha_n}_G \ \to \ \bKU_G \]
 as a morphism in $\Mod_{\Sigma^\infty_+\bP_G}(\Sc p_G^\sm)$,
 again unique up to homotopy, such that
 \[  k_n(1\sm \mS_G^{-\alpha_n})\ =\ (\eta_*(\mathbbm{b}_{G,\td{V_n}}))^{-1} \ , \]
 where we exploit that the morphism $\eta_G$ inverts all pre-Bott classes.
 The relations
  \[ \eta_*(\mathbbm{b}_{G,\td{U_n}})\cdot  (\eta_*(\mathbbm{b}_{G,\td{V_n}}))^{-1} =
    \eta_*(\mathbbm{b}_{G,\td{U_n}})\cdot  (  \eta_*(\mathbbm{b}_{G,\td{U_n}})
   \cdot\eta_*(\mathbbm{b}_{G,\td{V_{n-1}}}) )^{-1} = (\eta_*(\mathbbm{b}_{G,\td{V_{n-1}}}))^{-1} \]
  in $\pi_{-\alpha_{n-1}}^G(\bKU_G)$ show that
  $k_n\circ h_n$ equals $k_{n-1}$ in the homotopy category
  $h(\Mod_{\Sigma^\infty_+\bP_G}(\Sc p_G^\sm))$.
  We choose homotopies witnessing this, and they 
  extend the functor  $F\colon \mN\to\Mod_{\Sigma^\infty_+\bP_G}(\Sc p_G^\sm)$ to a cocone diagram
  $\bar F\colon \mN^\triangleright\to\Mod_{\Sigma^\infty_+\bP_G}(\Sc p_G^\sm)$.
\end{construction}

\begin{theorem}[$\bKU_G$ as a telescopic localization]\label{thm:telescopic localization}
  Let $G$ be a compact Lie group.
  Let $\eta\colon \Sigma^\infty_+ \bP \to \bKU$ be a morphism of commutative global ring spectra
  such that $\eta_*(e_T)=[\nu_1]$ in $\pi_0^T(\bKU)=RU(T)$.
  Then the cocone diagram  $\bar F\colon \mN^\triangleright\to\Mod_{\Sigma^\infty_+\bP_G}(\Sc p_G^\sm)$
  presents $\bKU_G$ as a colimit in $\Mod_{\Sigma^\infty_+\bP_G}(\Sc p_G^\sm)$
  of the sequence \eqref{eq:KU_G-sequence}.
\end{theorem}
\begin{proof}
  The forgetful functor $u\colon\Mod_{\Sigma^\infty_+\bP_G}(\Sc p_G^\sm)\to\Sc p_G$
  preserves colimits and is conservative.
  The geometric fixed point functors $\Phi^H\circ\res^G_H\colon \Sc p_G\to \Sc p$
  preserve colimits, and are jointly conservative for all closed subgroups of $G$.
  And the functor $\pi_*\colon \Sc p\to \GrAb$ is conservative and preserves sequential colimits.
  So it suffices to show that
  $\pi_*\circ \Phi^H\circ\res^G_H\circ u\circ\bar F\colon \mN^\triangleright\to\GrAb$ is a colimit,
  for all $H\leq G$.
  Since the $\{U_i\}$ are an exhaustive sequence of $G$-representations,
  the underlying $H$-representations are also exhaustive,
  and thus contain every irreducible $H$-representation infinitely often.
  So the colimit of $\pi_*\circ \Phi^H\circ\res^G_H\circ u\circ F\colon \mN\to\GrAb$
  calculates the localization of the geometric fixed point ring $\Phi^H_*(\Sigma^\infty_+\bP_G)$
  obtained by inverting the classes  $\Phi^H(b_{H,\lambda})$ for all irreducible
  $H$-representations $\lambda$. 
  By Theorem \ref{thm:geomfix}, the morphism $\Phi^H_*(\eta_G)$ exhibits $\Phi^H_*(\bKU_G)$
  as such a localization.
  Hence $\pi_*\circ \Phi^H\circ\res^G_H\circ u\circ \bar F$ is a colimit
  for all $H\leq G$, and so
  $\bar F\colon \mN^\triangleright\to\Mod_{\Sigma^\infty_+\bP_G}(\Sc p_G^\sm)$
  is a colimit, too.
\end{proof}

\Danger We alert the reader that in general,
the underlying $G$-spectrum of a global localization by $\Xc$
in $\CAlg(\Sc p^\sm_{\gl})$ need not be a telescopic localization
in $\Sc p_G$ by the `obvious' $G$-equivariant classes, i.e.,
the ones obtained from $\Xc$ by restriction along all homomorphisms from $G$.

\begin{theorem}[$\pi_\star^G(\bKU_G)$ as an $RO(G)$-graded localization]\label{thm:main-G-RO-version}
  Let $G$ be a compact Lie group.
  Let $\eta\colon \Sigma^\infty_+ \bP \to \bKU$ be a morphism of commutative global ring spectra
  such that $\eta_*(e_T)=[\nu_1]$ in $\pi_0^T(\bKU)=RU(T)$.
  \begin{enumerate}[\em (i)]
  \item For every $\Sigma^\infty_+ \bP_G$-module $M$ in $\Sc p_G^\sm$,
    the morphism
    \[ \pi_\star^G(\eta_G\sm_{\Sigma^\infty_+\bP_G} M)\ : \  \pi_\star^G(M)  \ \to \ \pi_\star^G(\bKU_G\sm_{\Sigma^\infty_+\bP_G} M) \]
    exhibits the target as a localization,
    in $RO(G)$-graded  $\pi_\star^G(\Sigma^\infty_+ \bP_G)$-modules,
    of the source by the set of pre-Bott classes of all unitary $G$-representations.
   \item For every finite based $G$-CW-complex $A$, the morphism 
    \[  \eta^\star_G(A)\ : \  (\Sigma^\infty_+ \bP_G)^\star(A)  \ \to \ \bKU_G^\star(A) \]
    exhibits the target as a localization,
    in $RO(G)$-graded  $(\Sigma^\infty_+ \bP_G)^\star$-modules,
    of the source by the set of pre-Bott classes of all unitary $G$-representations.
  \item The morphism of graded-commutative $RO(G)$-graded rings
    \[  \pi_\star^G(\eta_G)\ : \  \pi_\star^G(\Sigma^\infty_+ \bP_G) \ \to \ \pi_\star^G(\bKU_G) \]
    exhibits the target as an $RO(G)$-graded localization of the source by the set
    of pre-Bott classes of all unitary $G$-representations.
  \end{enumerate}
\end{theorem}
\begin{proof}
  (i)  The functor
  $-\sm_{\Sigma^\infty_+\bP_G}M\colon\Mod_{\Sigma^\infty_+\bP_G}(\Sc p_G^\sm)\to\Sc p_G$
  preserves all colimits. 
  So Theorem \ref{thm:telescopic localization}
  shows that the cocone diagram
  $\bar F\sm_{\Sigma^\infty_+\bP_G}M\colon \mN^\triangleright\to\Sc p_G$
  presents $\bKU_G\sm_{\Sigma^\infty_+\bP_G}M$
  as a colimit in $\Sc p_G$ of the sequence
\[ 
    M\xra{M\sm h_1}  M\sm \mS_G^{-\alpha_1} \xra{M\sm h_2}
    \dots \xra{M\sm h_n}  M\sm \mS_G^{-\alpha_n} \dots\ .
 \]
  We have exploited
  \[  (\Sigma^\infty_+ \bP_G\sm \mS_G^{-\alpha_n})\sm_{\Sigma^\infty_+\bP_G}M
    \ \iso \ M\sm \mS_G^{-\alpha_n}\ . \]
  By design, the morphism $M\sm h_n$ induces multiplication
  by the pre-Bott class $\mathbbm{b}_{G,\td{U_n}}$ on $RO(G)$-graded homotopy groups.

  For every $\alpha\in RO(G)$, the functor $\pi_\alpha^G\colon \Sc p_G\to\rAb$
  preserves sequential colimits. 
  So $\pi_\star^G(\bKU_G\allowbreak\sm_{\Sigma^\infty_+\bP_G}M)$ is exhibited
  as a sequential colimit, over multiplications by pre-Bott classes,
  of appropriate $RO(G)$-graded shifts of instances of $\pi_\star^G(M)$.
  Since the sequence of $G$-representations \eqref{eq:exhaust} is exhaustive,
  the algebraic effect of the sequential colimit
  is precisely $RO(G)$-graded localization by the pre-Bott classes
  of all unitary $G$-representations.

  (ii) By the finiteness of $A$, the assembly morphism
  \[ \map_*(A,\Sigma^\infty_+\bP_G)\sm_{\Sigma^\infty_+\bP_G}\bKU_G \ \to \
    \map_*(A,\bKU_G) \]
  is an equivalence of $\bKU_G$-modules. So part (ii) follows from part (i),
  applied to the $\Sigma^\infty_+\bP_G$-module $\map_*(A,\Sigma^\infty_+\bP_G)$.

  (iii) The localization property in the category of graded-commutative $RO(G)$-graded rings
  is detected by the analogous localization property of the underlying
  $RO(G)$-graded  $\pi_\star^G(\Sigma^\infty_+ \bP_G)$-modules.
  So property (iii) follows from (i) for $M=\Sigma^\infty_+\bP_G$.
\end{proof}

\begin{rk}[Segal--Becker splitting]
  The effect of the morphism $\eta_G\colon\Sigma^\infty_+\bP_G\to \bKU_G$
  on equivariant cohomology theories has previously received some attention
  under the heading `Segal--Becker splitting'.
  Segal \cite{segal:CP infty} and Becker \cite{becker:characteristic_and_K}
  showed in the non-equivariant context that the transformation of
  cohomology theories induced by $\eta$ is split surjective; in fact,
  the induced map of infinite loop spaces
  $\Omega^\infty\eta\colon \Omega^\infty(\Sigma^\infty_+\mC P^\infty)\to \Omega^\infty K U\sim \mZ\times BU$
  admits a section by a loop map.
  The Segal--Becker splitting was generalized to the equivariant context by
  Iriye and Kono  \cite{iriye-kono:Segal-Becker KG} for finite groups,
  and by Crabb \cite{crabb:brauer} for general compact Lie groups.
  It says that for every finite $G$-CW-complex $A$, the map
  \begin{equation}\label{eq:split_epi}
    \eta_G^0(A)\ :\ (\Sigma^\infty_+ \bP)^0_G(A)\ \to\ \bKU_G^0(A)
  \end{equation}
  is a naturally split epimorphism.
  Since the map \eqref{eq:split_epi} is surjective already before any localization,
  its extension to the algebraic $RO(G)$-graded localization
  by the pre-Bott classes is also surjective.
  Hence the new insight provided by Theorem \ref{thm:main-G-RO-version} (ii) is the fact that
  every element in the kernel of \eqref{eq:split_epi} is `pre-Bott torsion',
  i.e., annihilated by the pre-Bott class of some unitary $G$-representation.
\end{rk}

As we shall explain now, it is necessary to invert {\em representation-graded}
(as opposed to integer-graded) equivariant homotopy classes to obtain
a generalization of Snaith's theorem to genuine equivariant spectra.
To show this we need some control over the integer-graded ring $\pi_*^G(\Sigma^\infty_+\bP)$,
and the following proposition determines it modulo torsion for cyclic groups $C$ of prime order.
We write
\begin{equation}\label{eq:naive_pre-Bott}
 b_{\nv}\ =\ b_{C,\mC}\ = \ p_C^*(\res^T_{\{1\}}(b_1))\ \in \ \pi_2^C(\Sigma^\infty_+\bP)  
\end{equation}
for the {\em naive pre-Bott class},
the pre-Bott class of the trivial 1-dimensional $C$-representation.
In the next proposition and below, we write `torsion' for the torsion subgroup
of an abelian group.

\begin{prop}\label{prop:modulo_torsion}
  Let $C$ be a cyclic group of prime order $p$.
  \begin{enumerate}[\em (i)]
  \item 
    The composite
    \[ \pi_0^C(\Sigma^\infty_+\bP)[b_{\nv}]\ \to \
      \pi_*^C(\Sigma^\infty_+\bP)\ \xra{\text{\em proj}} \
      \pi_*^C(\Sigma^\infty_+\bP)/\text{\em (torsion)}  \]
    is an isomorphism of graded rings.
  \item The morphism $\pi_*^C(\eta)\colon\pi_*^C(\Sigma^\infty_+\bP)\to\pi_*^C(\bKU)$
    is not a localization of graded-commutative rings by any subset
    of homogeneous elements of the source.
\end{enumerate}
\end{prop}
\begin{proof}
  (i) The abelian group $\pi_0^C(\Sigma^\infty_+\bP)$ is torsion-free,
  see for example \cite[Corollary 4.1.13]{schwede:global};
  so the result is tautologically true in dimension zero.
  By design, the graded ring $\pi_0^C(\Sigma^\infty_+\bP)[b_{\nv}]$ is concentrated
  in even degrees, and the class $b_{\nv}$ acts bijectively in non-negative degrees.
  Since the claim holds in dimension 0, it remains to show that 
  the graded ring $\pi_*^C(\Sigma^\infty_+\bP)$ is torsion in all odd degrees,
  and the naive pre-Bott class acts bijectively modulo torsion, i.e.,
  \[ -\cdot b_{\nv}\ \colon \ \pi_k^C(\Sigma^\infty_+\bP)/\text{(torsion)}
    \ \to \ \pi_{k+2}^C(\Sigma^\infty_+\bP)/\text{(torsion)}  \]
  is bijective for all $k\geq 0$.
  
  We use the classical fact that the composite
  \[ \mZ[b_{\cl}] \ \to \ \pi_*(\Sigma^\infty_+ \mC P^\infty)\ \xra{\text{proj}} \
    \pi_*(\Sigma^\infty_+ \mC P^\infty)/\text{(torsion)}  \]
  is an isomorphism of graded rings, with $b_{\cl}=\res^T_{\{1\}}(b_1)$
  the classical pre-Bott class. Indeed, the quotient of
  $\pi_*(\Sigma^\infty_+ \mC P^\infty)$ by its torsion ideal
  embeds into $\mQ\tensor\pi_*(\Sigma^\infty_+ \mC P^\infty)\allowbreak\iso H_*(\mC P^\infty;\mQ)=\mQ[b_{\cl}]$,
  and the powers of $b_{\cl}$ cannot be divisible because their images in $\pi_*(KU)$
  are not divisible.
  In particular,
  $\pi_*(\Sigma^\infty_+ \mC P^\infty)$ is torsion in odd degrees, and
  the classical pre-Bott class acts bijectively modulo torsion.

  For the suspension spectrum of any $C$-space $X$, the geometric fixed point
  homomorphism is a split epimorphism.
  So the isotropy separation sequence becomes a split short exact sequence:
  \[ 0 \to \pi_*^C(\Sigma^\infty_+ (EC \times X)) \xra{\ q_*\ }
    \pi_*^C(\Sigma^\infty_+ X) \xra{\ \Phi^C\ }  \pi_*(\Sigma^\infty_+ X^C) \to 0     \]
  Here $q\colon EC\times X\to X$ is the projection.
  Since the sequence is split exact, it remains exact after dividing out the torsion subgroups
  everywhere.
  
  Now we specialize to $X=\bP_C$, the underlying $C$-space of the global space $\bP$.
  We let $b_{\nv}$ act on $\pi_*(\Sigma^\infty_+ (\bP_C)^C)\iso\Phi_*^C(\Sigma^\infty_+ \bP)$
  by multiplication by the class $\Phi^C(b_{\nv})$.
  The torsion subgroups are invariant under the $b_{\nv}$-action,
  so we arrive at a short exact sequence of graded $\mZ[b_{\nv}]$-modules:
  \[      0 \to \pi_*^C(\Sigma^\infty_+ (EC \times \bP_C))/(\text{tors.}) \xra{\ q_*\ }
    \pi_*^C(\Sigma^\infty_+ \bP)/(\text{tors.}) \xra{\ \Phi^C\ } 
    \Phi_*^C(\Sigma^\infty_+ \bP)/(\text{tors.}) \to 0 \]
  The above classical fact and the description of
  Proposition \ref{prop:monomial_basis} (ii) show that the term
  $\Phi_*^C(\Sigma^\infty_+ \bP)/\allowbreak(\text{torsion})$  is trivial in all odd degrees,
  and that $b_{\nv}$ acts bijectively on it.
  
  We claim that also
  $\pi_*^C(\Sigma^\infty_+ (EC\times \bP_C))/(\text{torsion})$ is trivial in odd degrees,
  and $b_{\nv}$ acts bijectively on it.
  Granted this, the five lemma shows that $b_{\nv}$ acts bijectively on
  $\pi_*^C(\Sigma^\infty_+ \bP)/(\text{torsion})$, which is also concentrated in even degrees.

  Since $\bP$ is a global classifying space for the circle group $T$,
  the $C$-space $\bP_C$ is $C$-homotopy equivalent to $(E_C T)/T$,
  with $E_C T$ a universal $(C\times T)$-space for the family of graph subgroups
  of $C\times T$, see \cite[Proposition 1.1.26 (i)]{schwede:global}.
  The $(C\times T)$-space $EC\times (E_C T)$ is free and underlying contractible,
  so
  \[ EC\times_C\bP_C\ \simeq \  (EC\times (E_C T))/(C\times T)\ \simeq \
    B(C\times T) \ \simeq\ BC \times \mC P^\infty \]
  is a non-equivariant classifying space for $C\times T$.
  The Adams isomorphism thus identifies $\pi_*^C(\Sigma^\infty_+ (EC \times \bP_C))$
  with the non-equivariant stable homotopy groups $\pi_*(\Sigma^\infty_+ (BC \times \mC P^\infty))$.
  A choice of basepoint of $B C$ splits this graded group as the direct sum of
  \[ \pi_*(\Sigma^\infty_+ \mC P^\infty) \text{\quad and\quad}
    \pi_*((\Sigma^\infty  B C)\sm \mC P^\infty_+)\ .\]
  The second of these summands is $p$-power torsion because the reduced suspension spectrum
  $\Sigma^\infty BC$ becomes trivial in the stable homotopy category after inverting $p$;
  so it contributes a torsion direct summand to $\pi_*^C(\Sigma^\infty_+ (EC \times \bP_C))$.
  The torsion subgroup is $b_{\nv}$-invariant, so 
  $\pi_*^C(\Sigma^\infty_+ (EC\times \bP_C))/(\text{torsion})$ 
  is $\mZ[b_{\nv}]$-linearly isomorphic to
  $\pi_*(\Sigma^\infty_+ \mC P^\infty)/(\text{torsion})$,
  with $b_{\nv}$ acting as the classical pre-Bott class $b_{\cl}$.
  One more time by the above classical fact,
  $b_{\nv}$ indeed acts bijectively on $\pi_*(\Sigma^\infty_+ \mC P^\infty)/(\text{torsion})$,
  and hence also on $\pi_*^C(\Sigma^\infty_+ (EC\times \bP_C))/(\text{torsion})$, as claimed. 
  This completes the proof of (i).
  
  (ii)   We show first that the ring homomorphism
  $\pi_0^C(\eta)\colon  \pi_0^C(\Sigma^\infty_+\bP)\to\pi_0^C(\bKU)$
  is not a localization by any multiplicative subset of the source.
  We argue by contradiction and suppose that $\pi_0^C(\eta)$ presents $\pi_0^C(\bKU)$
  as a localization by a multiplicative subset $S$ of $\pi_0^C(\Sigma^\infty_+\bP)$.
  We set $\xi=\res^T_C(e_T)\in \pi_0^C(\Sigma^\infty_+\bP)$, for some embedding $C\to T$.
  The abelian group $\pi_0^C(\Sigma^\infty_+\bP)$ is free of rank $p+1$,
  with basis the classes $\xi^k$ for $0\leq k\leq p-1$,
  and the class $t=\tr_{\{1\}}^C(1)$, see for example \cite[Corollary 4.1.13]{schwede:global}.
  The ring structure is determined by the relations
  \[ \xi^p \ = \ 1\ ,\quad   \xi\cdot t\  = \ t  \text{\quad and\quad} t^2 \ = \ p\cdot t\  . \]  
  The kernel of $\pi_0^C(\eta)$ is free of rank 1 as an abelian group,  generated by the element 
  \[ \kappa\ =\ (1+\xi+\dots+\xi^{p-1})- t \ . \]
  The localization property provides an element $s\in S$ such that $s\cdot \kappa=0$.
  The ring homomorphism $\pi_0^C(\eta)$
  has a section $\sigma\colon\pi_0^C(\bKU)\to \pi_0^C(\Sigma^\infty_+\bP)$
  by a ring homomorphism, determined by sending $\res^T_C[\nu_1]=\pi_0^C(\eta)(\xi)$
  to the element $\xi$ in $\pi_0^C(\Sigma^\infty_+\bP)$.
  The element $\pi_0^C(\eta)(s)$ is a unit in $\pi_0^C(\bKU)$ by the localization property.
  So $u=\sigma(\pi_0^C(\eta)(s))$ is a unit in $\pi_0^C(\Sigma^\infty_+\bP)$,
  and $\pi_0^C(\eta)(s\cdot u^{-1})=1$.
  Hence
  \[ s \ = \ u\cdot(1+ a\kappa) \]
  for some $a\in\mZ$, because $\kappa$ generates the kernel of $\pi_0^C(\eta)$.
  The relation $\kappa^2=p\cdot\kappa$ yields
  \[ s\cdot \kappa \ = \ u\cdot(1+ a\kappa)\cdot \kappa
    \ = \ u\cdot(1+ p a)\cdot \kappa\ .  \]
  Since $\kappa$ is not a torsion element, $(1+pa)\cdot\kappa$ is non-zero,
  and hence $s\cdot\kappa$ is non-zero.
  We have arrived at a contradiction, so $\pi_0^C(\eta)$
  is not a localization by any multiplicative subset, as claimed.

  Now we show that the composite
  \begin{equation}\label{eq:2nd_compo}
    \pi_0^C(\Sigma^\infty_+\bP)[b_{\nv}]\ \to \
    \pi_*^C(\Sigma^\infty_+\bP)\ \xra{\pi_*^C(\eta)} \
    \pi_*^C(\bKU) = \pi_0^C(\bKU)[\beta_{\nv}^\pm]    
  \end{equation}
  is not a localization of graded-commutative rings by any set of homogeneous elements of
  the source, where $\beta_{\nv}=\pi_2^C(\eta)(b_{\nv})=p_C^*(\res^T_{\{1\}}(\beta_{T,\nu_1}))\in \pi_2^C(\bKU)$.
  Again we argue by contradiction and suppose that  \eqref{eq:2nd_compo}
  were a localization by such a set $\Yc$.
  We write each element of this set as $y=w\cdot b_{\nv}^k$,
  with $w\in \pi_0^C(\Sigma^\infty_+\bP)$ and $k\geq 0$.
  Then $k$ must be positive for at least one element of $\Yc$,
  for else the localization would be concentrated in non-negative degrees.
  Inverting $y$ can be achieved by first inverting $b_{\nv}$, if $k>0$, and then inverting $w$.
  Hence the extension of \eqref{eq:2nd_compo} to a homomorphism
  \[ \pi_0^C(\Sigma^\infty_+\bP)[b_{\nv}^\pm]\ \to \ \pi_0^C(\bKU)[\beta_{\nv}^\pm]\]
  is a localization by the set of factors $w$ of degree 0 in the above products.
  Thus $\pi_0^C(\eta)\colon\pi_0^C(\Sigma^\infty_+\bP)\to\pi_0^C(\bKU)$
  is a localization by the set of these $w$, contradicting the previous paragraph.

  Now we prove part (ii) by contradiction and suppose that
  $\pi_*^C(\eta)\colon \pi_*^C(\Sigma^\infty_+\bP)\allowbreak\to\pi_*^C(\bKU)$
  were a localization by a set $\Yc$ of homogeneous elements of the source.
  Since $\pi_*^C(\bKU)$ is torsion-free, every torsion element of
  $\pi_*^C(\Sigma^\infty_+\bP)$ lies in the kernel of $\pi_*^C(\eta)$.
  Since localization is exact, any localization of $\pi_*^C(\Sigma^\infty_+\bP)$ by $\Yc$
  is thus also a localization of
  $\pi_*^C(\Sigma^\infty_+\bP)/(\text{torsion})$ by the image of $\Yc$.
  Part (i) identifies $\pi_*^C(\Sigma^\infty_+\bP)/(\text{torsion})$
  with $\pi_0^C(\Sigma^\infty_+\bP)[b_{\nv}]$, so we deduce that the composite \eqref{eq:2nd_compo}
  is a localization, contradicting the previous paragraph.
\end{proof}

Now we are ready to prove that inverting integer-graded classes
cannot turn $\Sigma^\infty_+\bP_G$ into $\bKU_G$, unless the group $G$ is trivial.

\begin{cor}\label{cor:no_Z_localization}
  Let $G$ be a nontrivial compact Lie group.  
  There is no set $\Yc$ of homogeneous elements of
  the integer-graded ring $\pi_*^G(\Sigma^\infty_+\bP)$
  such that the morphism $\eta_G\colon\Sigma^\infty_+\bP_G\to\bKU_G$
  extends to an equivalence of $G$-spectra from the telescopic localization
  of $\Sigma^\infty_+\bP_G$ by $\Yc$.
\end{cor}
\begin{proof}
  Every nontrivial compact Lie group $G$
  has a closed cyclic subgroup $C\leq G$ of prime order.
  If $\bKU_G$ were a telescopic localization by a set $\Yc$
  of homogeneous elements of $\pi_*^G(\Sigma^\infty_+\bP)$,
  then the underlying $C$-spectrum $\bKU_C$ would be
  a telescopic localization of $\Sigma^\infty_+\bP_C$ by the set
  $\{\res^G_C(y)\mid y\in\Yc\}$.
  Telescopically inverting a set of integer-graded equivariant homotopy
  classes induces the algebraic localization on equivariant homotopy groups,
  so this would contradict Proposition \ref{prop:modulo_torsion} (ii).
\end{proof}

\begin{rk}\label{rk:spitz-oest}
  As already mentioned in the introduction, the paper 
  \cite{spitzweck-oestvaer:equivariant}
  by Spitzweck and {\O}stv{\ae}r claims that for abelian compact Lie groups $A$,
  the spectrum $\bKU_A$ can be obtained from $\Sigma^\infty_+ \bP_A$
  by telescopically inverting a single class in the integer-graded
  homotopy group $\pi_2^A(\Sigma^\infty_+\bP_A)$.
  Their class is the analog of our naive pre-Bott class from \eqref{eq:naive_pre-Bott},
  i.e., the inflation  $b_{\nv}=p_A^*(\res^T_{\{1\}}(b_1))$ of the non-equivariant Bott class.
  However, the main result \cite[Theorem 1.1]{spitzweck-oestvaer:equivariant}
  contradicts Corollary \ref{cor:no_Z_localization}, and is incorrect.

  The argument of \cite{spitzweck-oestvaer:equivariant} breaks down in their Lemma 3.2,
  which in our notation asserts that $b_{\nv}^{-1}(1-\xi_\infty)$
  is an equivariant orientation of $(\Sigma^\infty_+\bP_A)[b_{\nv}^{-1}]$,
  for $\xi_\infty\in (\Sigma^\infty_+\bP)^0_A(\bP_A)$
  the tautological class, represented by the identity of $\bP_A$.
  In other words, the claim is that
  for every $A$-character $\alpha\colon A\to T$,
  the restriction of $b_{\nv}^{-1}(1-\xi_\infty)$
  along $P(\mC\oplus\alpha)\to \bP_A$
  is a generator of the reduced $A$-equivariant $(\Sigma^\infty_+\bP_A)[b_{\nv}^{-1}]$-cohomology
  of $P(\mC\oplus\alpha)$, based at $P(\mC\oplus 0)$,
  as a graded module over the $\mZ$-graded ring
  $(\Sigma^\infty_+\bP_A)[b_{\nv}^{-1}]^*_A$.
  This orientation property is not checked in the proof of 
  \cite[Lemma 3.2]{spitzweck-oestvaer:equivariant},
  and it fails for example when $A$ is cyclic of prime order and $\alpha$ is non-trivial.
  Consequently, in the first step of the proof of
  \cite[Theorem 1.1]{spitzweck-oestvaer:equivariant},
  the appeal to the universality of $MU_A$ is unjustified.

  We claim that the restriction of $1-\xi_\infty$ along 
  \[ S^\alpha\ \xra{\lambda\mapsto[\bar\lambda:1]}\ P(\mC\oplus\alpha)\ \to\ \bP_A \]
  equals $b_{A,\alpha}$ in $\pi_\alpha^A(\Sigma^\infty_+\bP_A)$.
  Indeed, this relation holds after restricting further along
  the fixed point inclusion $i\colon S^0\to S^\alpha$, by \eqref{eq:i^*_of_b}.
  In the special case $A=T$ and $\alpha=\Id_T$, restriction along $i$ is injective,
  so the claim holds for $(T,\Id_T)$. Naturality for
  restriction along continuous homomorphisms then proves the claim for general $(A,\alpha)$.
  Consequently, $b_{\nv}^{-1}(1-\xi_\infty)$ restricts to $b_{\nv}^{-1}\cdot b_{A,\alpha}$ in
  $\pi_{\alpha-2}^A(\Sigma^\infty_+\bP_A)[b_{\nv}^{-1}]$.
  So inverting the pre-Bott classes of all unitary representations,
  as opposed to only inverting $b_{\nv}$,
  is what turns the class $b_{\nv}^{-1}(1-\xi_\infty)$ into an equivariant orientation.
 \end{rk}

\section{Localization by \texorpdfstring{$b_1$}{b1}}
In this section we investigate the localizations of $\Sigma^\infty_+\bP$
by an individual class $b_m$ for some fixed $m\geq 1$,
with particular emphasis on the case $m=1$.
By Theorem \ref{thm:fix of pre-Bott} (i), the class $b_m$
is a factor of a restriction of $b_{m+1}$. So every morphism of commutative global ring
spectra from $\Sigma^\infty_+\bP$ that inverts $b_{m+1}$ also inverts $b_1,b_2,\dots,b_m$.
Hence each localization is potentially `more drastic' than the previous ones;
and Theorem \ref{thm:all_different} below shows that these localizations
are indeed all distinct.

The other results of this section concern localizations
of $\Sigma^\infty_+\bP$ by the first pre-Bott class $b_1$.
We show in Theorem \ref{thm:glcom b_1 loc} that
this localization in $\CAlg(\Sc p_{\gl}^\sm)$ is the left $\Ab$-approximation
of $\bKU$, where $\Ab$ is the global family of abelian compact Lie groups.
And we show in Theorem \ref{thm:ucom b_1 loc} that the
ultra-commutative localization by $b_1$ is the left $\Ab^\circ$-approximation
of $\bKU$, where $\Ab^\circ$ is the global family of compact Lie groups
whose identity component is abelian.
The family $\Ab^\circ$ is much larger than $\Ab$; for example,
$\Ab^\circ$ includes all finite groups. In particular,
the two localizations 
of $\Sigma^\infty_+\bP$ by $b_1$ in  $\CAlg(\Sc p_{\gl}^\sm)$ and in $\UR_{\gl}$
are distinct.
\medskip

The next theorem shows that inverting $b_m$ in $\UR_{\gl}$ does not invert $b_{m+1}$. 
Because a localization in the less structured category $\CAlg(\Sc p_{\gl}^\sm)$
maps to the localization in the more highly structured category $\UR_{\gl}$,
inverting $b_m$ in $\CAlg(\Sc p_{\gl}^\sm)$ does not invert $b_{m+1}$, either. 

\begin{theorem}\label{thm:all_different}
  Let $i\colon \Sigma^\infty_+\bP\to R$ be an ultra-commutative localization in $\UR_{\gl}$
  by the pre-Bott class $b_m$ for some $m\geq 1$.
  Then $i$ does not invert $b_{m+1}$.
\end{theorem}
\begin{proof}
  It suffices to exhibit any morphism of ultra-commutative ring spectra from
  $ \Sigma^\infty_+\bP$ that inverts $b_m$  but not $b_{m+1}$.
  We write $\Fc_m$ for the global class of compact Lie groups
  that are isomorphic to a subgroup of $\Sigma_k\wr U(m)$ for some $k\geq 1$.
  This global class is multiplicative because $(\Sigma_k\wr U(m))\times (\Sigma_l\wr U(m))$
  embeds into  $\Sigma_{k+l}\wr U(m)$.
  And $\Fc_m$ is closed under wreath products because
  $\Sigma_k\wr(\Sigma_l\wr U(m))$ embeds into  $\Sigma_{k l}\wr U(m)$.
  By design, $\Fc_m$ is the minimal multiplicative global class that contains $U(m)$
  and is closed under wreath products.

  We let $\delta\colon L_{\Fc_m}^{\uc}(\bKU)\to\bKU$ be an $\Fc_m$-approximation
  of the global $K$-theory spectrum in the $\infty$-category $\UR_{\gl}$
  of ultra-commutative ring spectra, in the sense of Remark \ref{rk:approxFucom}.
  So this morphism is terminal in the full subcategory of $\UR_{\gl}/\bKU$
  spanned by morphisms whose source is $\Fc_m$-left induced as a global spectrum.
  Such a morphism exists by Theorem \ref{thm:wreath_closed} (iv).
  
  Since $\bP$ is a global classifying space of
  the abelian compact Lie group $U(1)$,
  the global spectrum $\Sigma^\infty_+\bP$ is left induced from the
  global family $\Ab$ of abelian compact Lie groups.
  Hence $\Sigma^\infty_+\bP$ is also left induced from the larger global class $\Fc_m$.
  So the morphism $\eta\colon\Sigma^\infty_+\bP\to\bKU$ factors uniquely
  as a morphism of ultra-commutative ring spectra $\psi\colon\Sigma^\infty_+\bP\to L_{\Fc_m}^{\uc}(\bKU)$
  followed by $\delta\colon L_{\Fc_m}^{\uc}(\bKU)\to\bKU$.
  Since $\delta$ is an $\Fc_m$-equivalence, since $U(m)\in\Fc_m$,
  and since $\eta$ inverts $b_m$, the morphism
  $\psi\colon\Sigma^\infty_+\bP\to L_{\Fc_m}^{\uc}(\bKU)$ also inverts $b_m$.
  We will show that $\psi$ does {\em not} invert $b_{m+1}$.
  Incidentally, for $m=1$ the family $\Fc_1$ is the family $\Ab^\circ$
  of compact Lie groups with abelian identity component, and we show in
  Theorem \ref{thm:ucom b_1 loc} below that the morphism
  $\psi\colon\Sigma^\infty_+\bP\to L_{\Ab^\circ}^{\uc}(\bKU)$ is an ultra-commutative
  localization by $b_1$. For $m\geq 2$, I do not know whether
  $\psi\colon\Sigma^\infty_+\bP\to L_{\Fc_m}^{\uc}(\bKU)$ is an ultra-commutative
  localization by $b_m$.
  
  To show that $\psi\colon\Sigma^\infty_+\bP\to L_{\Fc_m}^{\uc}(\bKU)$ does not invert $b_{m+1}$
  we consider the group $S U(m+1)$. 
  Since $S U(m+1)$ is connected, every continuous homomorphism
  from $S U(m+1)$ to $\Sigma_k\wr U(m)$ takes values in the identity component
  $U(m)^k$ of the target.
  Since every nontrivial irreducible $S U(m+1)$-representation
  has dimension at least $m+1$, any such homomorphism is trivial.
  Hence the unique homomorphism $p_{SU(m+1)}\colon S U(m+1)\to\{1\}$
  is initial among morphisms from $S U(m+1)$ to groups in $\Fc_m$.
  Since $L_{\Fc_m}^{\uc}(\bKU)$ is  $\Fc_m$-left induced,
  and since this global class is multiplicative,
  the inflation homomorphism
  \[ p_{S U(m+1)}^*\ : \ \pi_0(L_{\Fc_m}^{\uc}(\bKU)) =
    \Phi_0^{\{1\}}(L_{\Fc_m}^{\uc}(\bKU)) \ \to \ \Phi^{S U(m+1)}_0(L_{\Fc_m}^{\uc}(\bKU)) \]
  is an isomorphism by Proposition \ref{prop:reflexive} (ii).
  Because $L_{\Fc_m}^{\uc}(\bKU)$ is non-equivariantly equivalent to $\bKU$,
  the source ring $\pi_0(L_{\Fc_m}^{\uc}(\bKU))$ is isomorphic to $\pi_0(\bKU)=\mZ$.
  We conclude that $\Phi^{S U(m+1)}_0(L_{\Fc_m}^{\uc}(\bKU))\allowbreak\iso\mZ$.
  Since $SU(m+1)$ is connected, it has no proper closed subgroups of finite index. 
  So the tautological $S U(m+1)$-representation $\nu_{m+1}$
  is not induced from any 1-dimensional representation of a closed subgroup of $S U(m+1)$,
  and we obtain the relation
  \[ \Phi^{S U(m+1)}(\res^{U(m+1)}_{S U(m+1)}(b_{m+1}))\ = \
    \Phi^{S U(m+1)}(b_{S U(m+1),\nu_{m+1}})\ =\ 0\]
  in $\Phi^{S U(m+1)}_0(\Sigma^\infty_+ \bP)$, by Theorem \ref{thm:Phi_on_irrep}.
  If $\psi_*(b_{m+1})$ were invertible in $L_{\Fc_m}^{\uc}(\bKU)$,
  then  $\psi_*(\Phi^{S U(m+1)}\allowbreak(b_{S U(m+1),\nu_{m+1}}))=0$
  would be invertible in the ring $\Phi_0^{S U(m+1)}(L_{\Fc_m}^{\uc}(\bKU))$,
  contradicting the fact that this ring is nontrivial.
\end{proof}

Now we look more closely at the localization of $\Sigma^\infty_+\bP$
by the first pre-Bott class;
as we shall see, inverting $b_1$ already yields the $K$-theory spectrum
for a large class of compact Lie groups, including all abelian ones.
When inverting $b_1$, the level of coherence in the multiplication
makes a difference: the localizations as commutative global ring spectra
and as ultra-commutative ring spectra are {\em not} equivalent,
as a comparison of Theorems \ref{thm:glcom b_1 loc} and \ref{thm:ucom b_1 loc} shows.
In fact, the localization in $\CAlg(\Sc p^\sm_{\gl})$
does not even admit an ultra-commutative lift, see Example \ref{eg:Sigma_3_witnesses}.
This is in contrast to the localization by $\{b_m\}_{m\geq 1}$,
where the ultra-commutative localization is also a localization
in the less highly structured category, by Theorem \ref{thm:main}.\medskip

We let $\Ab$ denote the global family of abelian compact Lie groups.
This global family is multiplicative, i.e., closed under products.
This closure property implies that the full subcategory of $\Sc p_{\gl}$
spanned by the $\Ab$-left induced global spectra
is closed under the global smash product,
see Theorem \ref{thm:multiplicative_class} (i).
And it implies that the full subcategory of $\CAlg(\Sc p_{\gl}^\sm)$
spanned by those commutative global ring spectra
whose underlying global spectra are $\Ab$-left induced
is coreflective in  $\CAlg(\Sc p_{\gl}^\sm)$,
i.e., the inclusion has a right adjoint, see Theorem \ref{thm:multiplicative_class} (vi).
In particular, for every commutative global ring spectrum $R$, the full subcategory
of the slice category $\CAlg(\Sc p_{\gl}^\sm)/R$ spanned by those morphisms
$A\to R$ such that $A$ is $\Ab$-left induced has a terminal object
 $L_{\Ab}(R)\to R$, see Remark \ref{rk:comFapp}.

\Danger We alert the reader that if $R$ in $\CAlg(\Sc p_{\gl}^\sm)$
underlies an ultra-commutative ring spectrum,
then $L_{\Ab}(R)$ need not underlie any ultra-commutative ring spectrum;
the commutative global ring spectrum $L_{\Ab}(\bKU)$ is a witness to this,
see Example \ref{eg:Sigma_3_witnesses} below.\medskip

In the remainder of this section, we abuse notation
by omitting the forgetful functor $u_{\gl}$, i.e.,
we denote the underlying commutative global ring spectrum of an ultra-commutative ring spectrum
by the same name. 
Since $\bP$ is a global classifying space of an abelian compact Lie group,
the underlying global spectrum of $\Sigma^\infty_+\bP$ is left induced from the
family $\Ab$. The morphism $\eta\colon \Sigma^\infty_+\bP\to\bKU$
thus uniquely factors as a morphism of commutative global ring spectra
$\eta^\natural\colon\Sigma^\infty_+\bP\to L_{\Ab}(\bKU)$, followed by the universal
$\Ab$-approximation $\epsilon\colon L_{\Ab}(\bKU)\to\bKU$ in $\CAlg(\Sc p_{\gl}^\sm)$.

\begin{theorem}[Commutative global localization by $b_1$]\label{thm:glcom b_1 loc}
    The morphism of commutative global ring spectra
    $\eta^\natural\colon \Sigma^\infty_+\bP\to L_{\Ab}(\bKU)$
    is a global localization in $\CAlg(\Sc p_{\gl}^\sm)$
    by the pre-Bott class $b_1$.
\end{theorem}
\begin{proof}
  We let $G$ be a compact Lie group.
  The projection $q_G\colon G\to G_{\ab}$ to the abelianization is initial among continuous
  homomorphisms to abelian compact Lie groups;
  so the global family $\Ab$ is reflexive in the sense of \cite[Definition 4.5.7]{schwede:global}.
  The underlying global spectra of both $\Sigma^\infty_+\bP$ and $L_{\Ab}(\bKU)$
  are left induced from the global family $\Ab$.
  So \cite[Proposition 4.5.8]{schwede:global} shows that both vertical inflation morphisms
  in the commutative square of graded-commutative rings
  \[ \xymatrix@C=18mm{\Phi_*^{G_{\ab}}(\Sigma^\infty_+\bP)\ar[d]^\iso_{q_G^*}
      \ar[r]_-{\Phi^{G_{\ab}}_*(\eta^\natural)}\ar@/^2pc/[rr]^(.2){\Phi^{G_{\ab}}_*(\eta)}&
      \Phi^{G_{\ab}}_*(L_{\Ab}(\bKU))\ar[d]^{q_G^*}_\iso \ar[r]^-{\iso}_-{\Phi^{G_{\ab}}_*(\epsilon)}&
      \Phi^{G_{\ab}}_*(\bKU)\\
      \Phi^G_*(\Sigma^\infty_+\bP)\ar[r]_-{\Phi^G_*(\eta^\natural)} &
      \Phi^G_*(L_{\Ab}(\bKU))
    } \]
  are isomorphisms, see also Proposition \ref{prop:reflexive}.
  By Theorem \ref{thm:geomfix}, the morphism $\Phi_*^{G_{\ab}}(\eta)$
  is a localization by the classes $\Phi^{G_{\ab}}(b_{G_{\ab},V})$
  for all unitary $G_{\ab}$-representations $V$.
  Since  $\epsilon\colon L_{\Ab}(\bKU)\to\bKU$ is an $\Ab$-equivalence,
  it induces an isomorphism of $G_{\ab}$-geometric fixed points, and so
  also the morphism $\Phi_*^{G_{\ab}}(\eta^\natural)$
  is a localization by the same classes.
  Since $G_{\ab}$ is abelian, every $G_{\ab}$-representation is a direct sum
  of 1-dimensional representations. By the multiplicativity of the pre-Bott classes
  (Theorem \ref{thm:fix of pre-Bott} (ii)),
  the morphism $\Phi^{G_{\ab}}_*(\eta^\natural)$ is also 
  a localization by the classes $\Phi^{G_{\ab}}(b_{G_{\ab},\chi})$
  for all continuous homomorphisms $\chi\colon G_{\ab}\to T$.
  Since $q_G^*(\Phi^{G_{\ab}}(b_{G_{\ab},\chi}))= \Phi^G(b_{G,\chi\circ q_G})$,
  and since every continuous homomorphism $G\to T$ factors uniquely through
  $q_G\colon G\to G_{\ab}$, this shows that the lower horizontal morphism  $\Phi^G_*(\eta^\natural)$
  is a localization by the classes $\Phi^G(b_{G,\alpha})=\Phi^G(\alpha^*(b_1))$
  for all continuous homomorphisms $\alpha\colon G\to T$.
  The geometric fixed point criterion of Theorem \ref{thm:fix_criterion}
  applied to the single class $b_1$ thus shows that
  $\eta^\natural\colon\Sigma^\infty_+\bP\to L_{\Ab}(\bKU)$
  is a global localization by $b_1$ in $\CAlg(\Sc p_{\gl}^\sm)$.
\end{proof}

Our next aim is to prove a result analogous to Theorem \ref{thm:glcom b_1 loc}
for the ultra-commutative localization of $\Sigma^\infty_+\bP$ by $b_1$,
see Theorem \ref{thm:ucom b_1 loc} below.
We write $\Ab^\circ$ for the multiplicative global family of those compact Lie groups
$G$ whose identity component $G^\circ$ is abelian.
This family in particular contains all finite groups and all abelian compact Lie groups.
In contrast to the global family $\Ab$, the family $\Ab^\circ$ is not only multiplicative,
but also closed under wreath products. Indeed, if $G$ is any compact Lie group in $\Ab^\circ$,
then 
\[ (\Sigma_m\wr G)^\circ \ = \ (G^\circ)^m\]
is again abelian, so the wreath product $\Sigma_m\wr G$ also belongs to $\Ab^\circ$.

Because $\Ab^\circ$ is multiplicative and closed under wreath products,
the full subcategory of $\UR_{\gl}$
spanned by those ultra-commutative ring spectra
whose underlying global spectra are $\Ab^\circ$-left induced
is coreflective in  $\UR_{\gl}$, see Theorem \ref{thm:wreath_closed} (iv).
Hence for every ultra-commutative  ring spectrum $R$, the full subcategory
of the slice category $\UR_{\gl}/R$ spanned by those morphisms
$S\to R$ such that $S$ is $\Ab^\circ$-left induced has a terminal object
$\delta\colon L_{\Ab^\circ}^{\uc}(R)\to R$;
informally speaking $L_{\Ab^\circ}^{\uc}(R)$ is closest from the left to $R$
among all $\Ab^\circ$-left induced ultra-commutative ring spectra.
Another reassuring property is that the forgetful functor
$\UR_{\gl}\to\CAlg(\Sc p_{\gl}^\sm)$
commutes with left $\Ab^\circ$-approximations,
by the horizontal right adjointability of Theorem \ref{thm:wreath_closed} (iv).
More precisely, the morphism of commutative global ring spectra underlying 
$\delta\colon L_{\Ab^\circ}^{\uc}(R)\to R$ is a left $\Ab^\circ$-approximation
of the underlying commutative global ring spectrum of $R$.

Since the underlying global spectrum of $\Sigma^\infty_+\bP$
is left induced from $\Ab$, it is also left induced from the larger global
family $\Ab^\circ$. So the morphism $\eta\colon \Sigma^\infty_+\bP\to\bKU$
factors uniquely as a morphism of ultra-commutative ring spectra
$\eta^\flat\colon \Sigma^\infty_+\bP\to L_{\Ab^\circ}^{\uc}(\bKU)$, followed by the universal
$\Ab^\circ$-approximation $\delta\colon L_{\Ab^\circ}^{\uc}(\bKU)\to\bKU$.
The reader might want to contrast our next result with Theorem \ref{thm:glcom b_1 loc},
which identifies the localization in $\CAlg(\Sc p_{\gl}^\sm)$, as opposed to $\UR_{\gl}$.

\begin{theorem}[Ultra-commutative localization by $b_1$]\label{thm:ucom b_1 loc}
    The morphism
    $\eta^\flat\colon \Sigma^\infty_+\bP \to L_{\Ab^\circ}^{\uc}(\bKU)$
    is an ultra-commutative localization by the pre-Bott class $b_1$.
\end{theorem}
\begin{proof}
  In a first step we show the following claim:
  For every group $G$ in the global family $\Ab^\circ$,
  the morphism of geometric fixed point rings
  \[ \Phi^G_*(\eta)\ : \ \Phi^G_*(\Sigma^\infty_+\bP)\ \to \ \Phi^G_*(\bKU) \]
  is a localization by the set of classes $\Phi^G(\beta^*(P^m(b_1)))$
  for all $m\geq 1$ and all continuous homomorphisms $\beta\colon G\to\Sigma_m\wr T$.

  We start with the case where $G$ is cyclic.
  Then all irreducible unitary $G$-representations are 1-dimensional.
  Hence in this case the morphism
  $\Phi^G_*(\eta)\colon\Phi^G_*(\Sigma^\infty_+\bP)\to\Phi^G_*(\bKU)$
  is a localization by the set of classes $\Phi^G(b_{G,\chi})=\Phi^G(\chi^*(b_1))$
  for all continuous homomorphisms $\chi\colon G\to T$,
  by Theorem \ref{thm:geomfix}. Since $\eta$ is a morphism of ultra-commutative ring spectra,
  it is compatible with geometric fixed points, restrictions, and power operations,
  and each of these operations preserves invertibility
  by Propositions  \ref{prop:invert_criterion},  \ref{prop:restriction_and_invert},
  and \ref{prop:invertibility_power}, respectively.
  Since $\eta_*(b_1)$ is invertible by Theorem \ref{thm:pre-Bott}, the classes 
  \[  \eta_*(\Phi^G(\beta^*(P^m(b_1)))) \ = \ \Phi^G(\beta^*(P^m(\eta_*(b_1)))) \]
  are all invertible in $\Phi^G_*(\bKU)$ for all $m\geq 2$.
  Hence these classes are already invertible after inverting the
  classes for $m=1$, and thus $\Phi^G_*(\eta)$ is also a localization by the larger set.

  Now we assume that $G$ is in the family $\Ab^\circ$, but it is not cyclic.
  We claim that then $G$ has at least one of the following two properties:
  \begin{itemize}
  \item[(A)] The group $G$ admits a continuous epimorphism onto $C_p\times C_p$ for some prime $p$.
  \item[(B)] There is a finite index subgroup $H$ of $G$ and a nontrivial
    character $\chi\colon H\to T$ such that $\tr_H^G(\chi)=1$.
  \end{itemize}

  If the component group $\pi_0(G)$ is cyclic, then $G$ must be nonabelian,
  for else $G$ would be cyclic.
  Theorem \ref{thm:G^circ_image} (ii) then provides the data $(H,\chi)$
  witnessing that $G$ has property (B).

  If the component group $\pi_0(G)$ is not cyclic, then
  $\pi_0(G)$ admits a homomorphism onto a near-cyclic group in the sense of 
  Definition \ref{def:near-cyclic}.
  So there is a continuous epimorphism $\epsilon\colon G\to K$
  onto a finite near-cyclic group.
  This near-cyclic group $K$ has one of the properties (a), (b) or (c)
  listed in Proposition \ref{prop:nearcyclic}.
  The groups in class (a) are of the form $C_p\times C_p$ for some prime $p$;
  so if $K$ has property (a), then $G$ has property (A).

  If $K$ has property (b) or (c), then
  the addendum of Proposition \ref{prop:nearcyclic}
  provides a subgroup $L$ of $K$ and a nontrivial character $\lambda\colon L\to T$
  such that $\tr_L^K(\lambda)=1$.
  Then $H=\epsilon^{-1}(L)$ is a finite index subgroup of $G$, and
  $\chi=\lambda\circ \epsilon|_H\colon H\to T$ is a nontrivial character such that
  \[ \tr_H^G(\chi)\ = \
    \tr_H^G(\lambda\circ\epsilon|_H)\ = \ \tr_L^K(\lambda)\circ\epsilon \ = \ 1\ .\]
  So $G$ has property (B).

  If $G$ has property (A), we let $\epsilon\colon G\to C_p\times C_p$ be an epimorphism, 
  for some prime $p$. The relation
  \[ \prod_{\lambda\in\Hom(C_p\times C_p,T)\setminus\{1\}}(1-\lambda) \ = \ 0 \]
  holds in the integral group ring $\mZ[\Hom(C_p\times C_p,T)]$,
  where the product is over the $p^2-1$ nontrivial characters of $C_p\times C_p$.  
  Hence the relation
  \[ \prod_{\lambda\in\Hom(C_p\times C_p,T)\setminus\{1\}}(1-(\lambda\circ\epsilon)) \ = \ 0 \]
  holds in the group ring $\mZ[\Hom(G,T)]$, and the $G$-characters $\lambda\circ\epsilon$
  that occur are all nontrivial.
  In the ring $\Phi^G_0(\Sigma^\infty_+\bP)\iso \mZ[\Hom(G,T)]$,
  we are in particular inverting the classes
  \[ \Phi^G((\lambda\circ\epsilon)^*(b_1))\ = \ \Phi^G(b_{G,\lambda\circ\epsilon})\ = _{\eqref{eq:Phi^G_of_char}}\ 1 -(\lambda\circ\epsilon)^\sharp\ , \]
  and hence also their product.
  So the relevant localization of  $\Phi^G_0(\Sigma^\infty_+\bP)$
  is the zero ring, consistent with the fact that $\Phi^G_*(\bKU)=0$ for non-cyclic groups $G$.
  
  If $G$ has property (B), witnessed by $(H,\chi)$, then 
  we set $\beta=(\Sigma_m\wr\chi)\circ\Psi_{\bar g}\colon G\to\Sigma_m\wr T$,
  where $m=[G:H]$, and $\Psi_{\bar g}\colon G\to\Sigma_m\wr H$  
  is the monomorphism  \eqref{eq:define_Psi} arising from a choice of transversal
  of $H$ in $G$.
  Then
  \begin{align*}
    \Phi^G(\beta^*(P^m(b_1)))\
    &= \ \Phi^G(\Psi^*_{\bar g}((\Sigma_m\wr\chi)^*(P^m(b_1))))\
    = \ \Phi^G(\Psi^*_{\bar g}(P^m(\chi^*(b_1))))\\
    &= \ \Phi^G(\Psi^*_{\bar g}(P^m(b_{H,\chi})))\  = \ 1- \tr_H^G(\chi)^\sharp\ = \ 0 \ . 
  \end{align*}
  The second equality  $(\Sigma_m\wr\chi)^*\circ P^m=P^m\circ\chi^*$
  is naturality of the power operation with respect to restriction along group homomorphisms;
  the proof is straightforward from the definitions.
  The fourth equality is Proposition \ref{prop:monomial_basis} (iv).
  So the localization of $\Phi^G_*(\Sigma^\infty_+ \bP)$ by 
  $\Phi^G(\beta^*(P^m(b_1)))$ is the zero ring, consistent with the
  fact that $\Phi^G_*(\bKU)$ is trivial for non-cyclic groups $G$.

  With the claim established, we now show that
  the underlying morphism in $\CAlg(\Sc p_{\gl}^\sm)$
  of $\eta^\flat\colon \Sigma^\infty_+\bP\allowbreak\to L_{\Ab^\circ}^{\uc}(\bKU)$
  is a global localization by the set of classes $\{P^m(b_1)\}_{m\geq 1}$.
  The global family $\Ab^\circ$ is multiplicative, and hence reflexive
  by Proposition \ref{prop:mult=refl}, i.e., the inclusion of $\Ab^\circ$
  into all compact Lie groups has a left adjoint.
  For a compact Lie group $G$,
  the initial homomorphism $p\colon G\to G_\diamond$ to a group in $\Ab^\circ$ is the projection
  onto the quotient group  $G_\diamond=G/(G^\circ)'$, where $(G^\circ)'=[G^\circ,G^\circ]$
  is the commutator subgroup of the identity component $G^\circ$ of $G$. 

  Both $\Sigma^\infty_+\bP$ and $L_{\Ab^\circ}^{\uc}(\bKU)$
  are left induced from the global family $\Ab^\circ$.
  So \cite[Proposition 4.5.8]{schwede:global}
  or Proposition \ref{prop:reflexive} (ii) show that both vertical inflation morphisms
  in the commutative square of graded-commutative rings
  \begin{equation}\begin{aligned}\label{eq:G_diamond}
 \xymatrix@C=15mm{\Phi_*^{G_{\diamond}}(\Sigma^\infty_+\bP)
      \ar[d]^\iso_{p^*}\ar@/^2pc/[rr]^(.2){\Phi^{G_\diamond}_*(\eta)}
      \ar[r]_-{\Phi^{G_\diamond}_*(\eta^\flat)}&
      \Phi^{G_\diamond}_*(L_{\Ab^\circ}^{\uc}(\bKU))\ar[d]_\iso^{p^*}
      \ar[r]_-{\Phi^{G_\diamond}_*(\delta)}^\iso &      \Phi^{G_\diamond}_*(\bKU) \\
      \Phi^G_*(\Sigma^\infty_+\bP)\ar[r]_-{\Phi^G_*(\eta^\flat)}
      &  \Phi^G_*(L_{\Ab^\circ}^{\uc}(\bKU))
    }     \end{aligned}  \end{equation}
  are isomorphisms. The right upper horizontal map is an isomorphism because
  $\delta$ is an $\Ab^\circ$-equivalence.
  
  By the above claim applied to $G_\diamond$,
  the morphism $\Phi_*^{G_\diamond}(\eta)$ is a localization by the set of classes
  \[ \Phi^{G_\diamond}(\beta^*(P^m(b_1))) \]
  for all $m\geq 1$ and all continuous homomorphisms $\beta\colon G_\diamond\to\Sigma_m\wr T$.
  The commutative diagram \eqref{eq:G_diamond} thus shows that the lower horizontal homomorphism
  $\Phi^G_*(\eta^\flat)$ is a localization by the set of classes
  \[    p^*(\Phi^{G_\diamond}(\beta^*(P^m(b_1)))) \ = \
    \Phi^G(p^*(\beta^*(P^m(b_1))))\ = \   \Phi^G((\beta\circ p)^*(P^m(b_1)))
  \]
  for the same set of pairs $(m,\beta)$.
  Every continuous homomorphism $\alpha\colon G\to\Sigma_m\wr T$
  annihilates $(G^\circ)'$,
  and hence factors uniquely as $\alpha=\beta\circ p$ for
  a continuous homomorphism $\beta\colon G_\diamond\to \Sigma_m\wr T$.
  So  $\Phi^G_*(\eta^\flat)$ is also a localization by the set of classes
  $\Phi^G(\alpha^*(P^m(b_1)))$ for all $m\geq 1$ and all
  continuous homomorphisms $\alpha\colon G\to\Sigma_m\wr T$.
  The geometric fixed point criterion of Theorem \ref{thm:fix_criterion}
  thus proves the claim that $\eta^\flat$  is a global localization in
  $\CAlg(\Sc p^\sm_{\gl})$ by the set $\{P^m(b_1)\}_{m\geq 1}$.

  Since $\eta^\flat$ arises as a morphism of ultra-commutative ring spectra,
  and since the underlying morphism is a global localization in
  $\CAlg(\Sc p^\sm_{\gl})$ by $\{P^m(b_1)\}_{m\geq 1}$,
  Theorem \ref{thm:UR2CR} (iii) shows that $\eta^\flat$
  is also an ultra-commutative localization by $\{P^m(b_1)\}_{m\geq 1}$ in
  the more highly structured category $\UR_{\gl}$.
  A morphism of ultra-commutative ring spectra $j\colon\Sigma^\infty_+\bP\to S$
  is compatible with power operations, and thus inverts the class $b_1$
  if and only if it inverts $P^m(b_1)$ for all $m\geq 1$,
  compare Proposition \ref{prop:invertibility_power}.
  So every ultra-commutative localization in $\UR_{\gl}$ by $b_1$
  is also an ultra-commutative localization by $\{P^m(b_1)\}_{m\geq 1}$,
  and vice versa.
\end{proof}

\begin{eg}[$L_{\Ab}(\bKU)$ is not ultra-commutative]\label{eg:Sigma_3_witnesses}
  We will now show that the commutative global ring spectrum $L_{\Ab}(\bKU)$
  does not underlie any ultra-commutative ring spectrum.
  We argue by contradiction and suppose that there was an ultra-commutative ring spectrum $R$
  and an equivalence $s\colon L_{\Ab}(\bKU)\xra{\sim}u_{\gl}(R)$ in $\CAlg(\Sc p^\sm_{\gl})$.
  We let $\zeta\colon A_3\to T$ denote either of the two nontrivial characters of the
  alternating group $A_3$.
  If $\eta^\natural\colon \Sigma^\infty_+\bP\to L_{\Ab}(\bKU)$
  denotes the global localization by $b_1$ in $\CAlg(\Sc p_{\gl}^\sm)$
  discussed in Theorem \ref{thm:glcom b_1 loc},
  then the class 
  \[ B_{A_3,\zeta}\ =\ (s\circ \eta^\natural)_*(b_{A_3,\zeta}) \ = \
    s_*(\zeta^*(\eta^\natural_*(b_1)))\]
  in $\pi_\zeta^{A_3}(R)$ is invertible  by
  Propositions \ref{prop:invert_inherit} (i) and \ref{prop:restriction_and_invert}.
  Since $R$ is ultra-commutative, the class
  $\Phi^{\Sigma_3}(\Psi^*(P^2(B_{A_3,\zeta})))$ is then invertible by
  Propositions \ref{prop:invert_criterion},
  \ref{prop:restriction_and_invert} and \ref{prop:invertibility_power}.
  
  We set $x=(s\circ\eta^\natural)_*(\zeta^*(e_T))\in\pi_0^{A_3}(R)$;
  then $x^3=1$, because $\zeta^3=1$.
  Moreover,
  \[ i^*(B_{A_3,\zeta})\ = \ (s\circ\eta^\natural)_*(i^*(b_{A_3,\zeta}))\ =_{\eqref{eq:i^*_of_b}}  \
    (s\circ\eta^\natural)_*(1-\zeta^*(e_T))\ = \ 1-x \ , \]
  for $i\colon S^0\to S^{\zeta}$ the fixed point inclusion.
  
  We let $\Psi\colon\Sigma_3\to\Sigma_2\wr A_3$ be the monomorphism \eqref{eq:define_Psi}
  based on some transversal of the alternating group $A_3$ in $\Sigma_3$.
  We showed in the proof of Proposition \ref{prop:monomial_basis} (iv)
  that the map
  \[  N\ =\ \Phi^{\Sigma_3}\circ\Psi^*\circ P^2\ \colon\ \pi_0^{A_3}(R)\ \to \
    \Phi_0^{\Sigma_3}(R) \]
  is additive, and hence a unital ring homomorphism.
  The hypotheses provide ring isomorphisms
  \[ \Phi_0^{\Sigma_3}(R)\ \iso\ \Phi_0^{\Sigma_3}(L_{\Ab}(\bKU))\ \iso\ \Phi_0^{C_2}(L_{\Ab}(\bKU))\ \iso
    \  \Phi_0^{C_2}(\bKU) \ \iso  \   \mZ[1/2]\ .\]
  The second isomorphism is by inflation along the quotient map
  $\Sigma_3\to C_2$ to the abelianization, using Proposition \ref{prop:reflexive} (ii).
  The third uses that $C_2$ belongs to the family $\Ab$.
  The only cube root of unity in $\mZ[1/2]$, and hence in $\Phi^{\Sigma_3}_0(R)$,
  is 1, so $N(x)=1$.
  Then
  \begin{align*}
    \Phi^{\Sigma_3}(\Psi^*(P^2(B_{A_3,\zeta})))\
    &= \ \Phi^{\Sigma_3}(\bar i^*(\Psi^*(P^2(B_{A_3,\zeta}))))\
      = \ N(i^*(B_{A_3,\zeta}))\  = \ N(1-x)\  =\  0\ ,
  \end{align*}
  where now $\bar i$ is the inclusion of $\Sigma_3$-fixed points
  $S^0\to S^{\ind_{A_3}^{\Sigma_3}(\zeta)}$.
  We conclude that 0 is a unit in $\Phi_0^{\Sigma_3}(R)\iso\mZ[1/2]$, which is a contradiction.
  So $L_{\Ab}(\bKU)$ does not underlie any ultra-commutative ring spectrum.
  By Theorem \ref{thm:glcom b_1 loc},
  $L_{\Ab}(\bKU)$ is the commutative global localization of $\Sigma^\infty_+\bP$ by $b_1$;
  so in particular, the commutative global localization and
  the ultra-commutative localization of $\Sigma^\infty_+\bP$ by $b_1$ are not equivalent.
\end{eg}

\section{Real-global \texorpdfstring{$K$}{K}-theory as a Real-global localization}\label{sec:Real}

The global $K$-theory spectrum $\bKU$ underlies a Real-global refinement
$\bKR$ that also knows about all conjugate-linear phenomena related
to complex $K$-theory.
We shall explain now that, and how,
our localization results can be extended to this Real-global context.

In this section we write  $C=\Gal(\mC/\mR)$ for the Galois group of $\mC$ over $\mR$,
a group with two elements.
Since the non-identity element of $C$ always acts by some kind of complex
conjugation, and the $C$-actions encode conjugate-linear phenomena, we shall also use
the term `Real-global' as synonymous with `$C$-global'.
The complex-conjugation involution $\psi\colon\bKU\to\bKU$
was first investigated
by Halladay and Kamel \cite[Section 6]{halladay-kamel}, who show
that the resulting genuine $C$-spectrum that they denote $KU_{\mR}$
represents Atiyah's Real $K$-theory.
I show in \cite[Theorem B.58]{schwede:GSB} that $\bKR=(\bKU,\psi)$ deserves to be called
the {\em Real-global $K$-theory spectrum}:
for every augmented Lie group $\alpha\colon G\to C$,
the genuine $G$-spectrum  $\alpha^*(\bKR)$ represents
$\alpha$-equivariant Real $K$-theory $KR_\alpha$.
Complex conjugation induces a multiplicative involution on the orthogonal space $\bP$,
making it an ultra-commutative $C$-monoid,
compare \cite[Construction 4.1]{schwede:GSB}.
Moreover, the morphism  $\eta$
is $C$-equivariant, and hence a morphism of ultra-commutative $C$-ring spectra
 $\eta\colon \Sigma^\infty_+\bP\to \bKR$.
In this section we generalize the localization properties
of the morphism $\eta\colon \Sigma^\infty_+\bP\to \bKU$ from Theorem \ref{thm:main} (i)
to the Real-global context, see Theorem \ref{thm:Realglobal}.
And we sketch in Remark \ref{rk:C_global_UCom} how the
ultra-commutative localization property  from Theorem \ref{thm:main} (ii)
ought to generalize.
There is only one minor modification necessary: we refine the
first pre-Bott class $b_1$ to a specific `Real-global pre-Bott class'
in $\pi^{\tilde T}_{\nu_1}(\Sigma^\infty_+\bP)$, see
Construction \ref{con:Real pre-Bott},
where $\tilde T=T\rtimes C$ is the extended circle group,
the semidirect product of $C$ acting on $T$ by complex conjugation.
As we explain in the proof of Theorem \ref{thm:Realglobal},
the key feature of the refined pre-Bott class $\tilde b_1$ is that inverting it
kills the geometric fixed point spectra of all surjectively augmented Lie groups,
a consequence of the nilpotency of $\Phi^C(\res^{\tilde T}_C(\tilde b_1))$,
see Proposition \ref{prop:properties_tilde b} (iii).
\medskip

Before we can dive into the specifics of the $C$-global morphism
$\eta\colon \Sigma^\infty_+\bP\to\bKR$, we must give meaning to
`global localizations' in the more general context of Real-global homotopy
theory. The general theory would work perfectly well for 
augmentations to a fixed compact Lie group $K$.
Since we are only interested in the special case
where $K=C=\Gal(\mC/\mR)$ is the Galois group, we only spell out the
definitions in this case.
In the following, an {\em augmented Lie group}
is a continuous homomorphism $\alpha\colon G\to C$ from
a compact Lie group to the Galois group.

\begin{construction}[$C$-global spectra]
A morphism $f\colon X\to Y$ of orthogonal $C$-spectra is a {\em $C$-global equivalence}
if for every augmented Lie group $\alpha\colon G\to C$ and all integers $k$, the map
$\pi_k^G(\alpha^*(f))\colon\pi_k^G(\alpha^*(X))\to \pi_k^G(\alpha^*(Y))$ is an isomorphism.
The $C$-global equivalences are the weak equivalences
in a stable model structure on the category $\Sp^O_C$ of orthogonal $C$-spectra,
which is moreover symmetric monoidal with cofibrant unit for the smash product,
see \cite{grande:MA}, or Theorem 7.8 and Corollary 7.19 of \cite{brink-lenz}.
The monoidal localization thus yields a symmetric monoidal $\infty$-category
  \[ \Sc p^{\sm}_{C\dgl} \ = \ (\Sp^{O,\sm}_C)_\infty \]
  and a lax symmetric monoidal functor
  \[ \gamma_{C\dgl}^\sm \ : \ \Sp^{O,\sm}_C\  \to \ \Sc p^{\sm}_{C\dgl} \]
  whose underlying functor
  is a localization of the 1-category of orthogonal $C$-spectra at the class
  of $C$-global equivalences.
  We refer to the underlying $\infty$-category $\Sc p_{C\dgl}$
  as the {\em $\infty$-category of $C$-global spectra};
  it is stable and compactly generated,
  and hence presentable, by \cite[Corollary A.19]{schwede:stiefel}.
  The monoidal product on $\Sc p^{\sm}_{C\dgl}$ arises from a Quillen bifunctor
  at the level of the model, and so it preserves colimits in each variable;
  hence  $\Sc p^{\sm}_{C\dgl}$ is presentably symmetric monoidal.
  The $\infty$-category $\CAlg(\Sc p^\sm_{C\dgl})$ of {\em commutative $C$-global ring spectra}
  is thus presentable by \cite[Corollary 3.2.3.5]{lurie:higher_algebra}.  
\end{construction}

\begin{construction}[Underlying $G$-spectrum]
  In Construction \ref{con:res_G} we introduce the `underlying $G$-spectrum functor'
  from global spectra to $G$-spectra, for any compact Lie group $G$.
  In the $C$-global context, every augmented Lie group $\alpha\colon G\to C$
  gives rise to an analogous functor, as follows.
  The 1-categorical restriction functor
  \[ \alpha^* \ : \ \Sp^O_C\ \to \ \Sp_G^O \ , \quad X \ \longmapsto \ \alpha^*(X) \]
  from orthogonal $C$-spectra to orthogonal $G$-spectra
  takes $C$-global equivalences to $\upi_*$-isomorphisms, by definition.
  And this restriction functor is strong symmetric monoidal
  for the respective smash products, with diagonal action.
  The universal property of the monoidal localization
  $\gamma^\sm_{C\dgl}\colon \Sp_C^{O,\sm}\to\Sc p_{C\dgl}^\sm$
  yields a unique lax symmetric monoidal functor
  \[  \res_\alpha^{\sm} \ : \ \Sc p_{C\dgl}^{\sm} \ \to \ \Sc p_G^{\sm}     \]
  equipped with a monoidal equivalence
  \[ \gamma_G^\sm\circ \alpha^* \ \sim \ \res_\alpha^\sm \circ \gamma^\sm_{C\dgl}\ \colon\
   \Sp^{O,\sm}_C\ \to \  \Sc p_G^{\sm} \ . \]
 The same kind of argument as in the proof of Theorem \ref{thm:gl2G} (iv)
 shows that the a priori only lax symmetric monoidal functor $\res_\alpha^{\sm}$
 is in fact strong.
\end{construction}

We introduce some conventions to ease the notational burden.
For an augmented Lie group $\alpha\colon G\to C$, the symmetric monoidal functor $\res^\sm_\alpha$
induces a functor of $\infty$-categories of commutative algebras.
For a commutative $C$-global ring spectrum $R$, we shall abbreviate
the resulting commutative $G$-ring spectrum $\CAlg(\res_\alpha^\sm)(R)$ to $R_\alpha$.
And we shall abbreviate the $G$-equivariant homotopy and geometric fixed point groups,
and the geometric fixed point ring spectrum of $R_\alpha$ to 
\[ \pi_V^\alpha(R)\ = \ \pi_V^G(R_\alpha)\ ,\quad
  \Phi_k^\alpha(R)\ = \ \Phi_k^G(R_\alpha)\text{\quad and\quad}
  \Phi^\alpha(R)\ = \ \Phi^G(R_\alpha)  \ .\]
Finally, we shall implicitly augment any closed subgroup $H$ of $G$
by the restricted augmentation $\alpha|_H\colon H\to C$.

The following is the straightforward generalization
of global localizations of commutative global ring spectra
(Definition \ref{def:global_loc}) to the $C$-global context.

\begin{defn}[$C$-global localization]
  Let $R$ be a commutative $C$-global ring spectrum.
  A {\em set of representation-graded equivariant homotopy classes} of $R$
  is a set $\Xc=\{v\}$ of
  classes in $\pi_V^\alpha(R)$ for varying pairs $(\alpha,V)$ consisting
  of an augmented Lie group $\alpha\colon G\to C$ and a $G$-representation $V$.
  We will say that the class $v\in \pi_V^\alpha(R)$ is {\em supported}
  at the homomorphism $\alpha$, or at the pair $(\alpha,V)$.
  A morphism $i\colon R\to S$ of commutative $C$-global ring spectra
  {\em inverts $\Xc$} if for every
  element $v\in\Xc$ supported at $(\alpha,V)$, the class $\pi_V^\alpha(i)(v)\in\pi_V^\alpha(S)$
  is invertible in the sense of Definition \ref{def:RO-invertible}.
  A {\em $C$-global localization by $\Xc$}
  is an initial object of the full subcategory of the
  slice category $R\bs\CAlg(\Sc p_{C\dgl}^\sm)$ spanned by the
  $\Xc$-inverting morphisms.
\end{defn}

\begin{construction}[Real-global pre-Bott class]\label{con:Real pre-Bott}  
  We write $\tilde T=T\rtimes C$ for the {\em extended circle group},
  the semidirect product of the group $C=\Gal(\mC/\mR)$
  acting by complex conjugation on $T=U(1)$.
  The group $\tilde T$ is augmented to $C$ by the projection.
  We abuse notation and write $\nu_1$ for the tautological Real
  $\tilde T$-representation on $\mC$ by scalar multiplication and complex conjugation.
  The underlying unitary $T$-representation of $\nu_1$ is the tautological
  representation with the same name, justifying the abuse.
  We introduce a {\em Real-global pre-Bott class} $\tilde b_1$
  in the $\tilde T$-equivariant homotopy group of $\Sigma^\infty_+\bP$
  that restricts to the first pre-Bott class from Construction \ref{con:pre-Bott}  
  on the underlying `non-Real' global spectrum.

  We embed the group $C$ into $\tilde T$ as complex conjugation.
  The cofiber sequence of based $\tilde T$-spaces
\[  \tilde T/C_+ \ \to \ S^0 \ \xra{\ i \ } \ S^{\nu_1}  \]
  induces a long exact sequence in any $\tilde T$-equivariant cohomology theory.
  For $C$-global spectra, the restriction homomorphism $\res^{\tilde T}_C$
  is a split epimorphism, by inflation along the augmentation $\tilde T\to C$;
  so the long exact sequence splits off
  a short exact sequence:
  \[ \xymatrix@C=9mm{
      0\ar[r] & \pi_{\nu_1}^{\tilde T}(\Sigma^\infty_+\bP)\ar[r]^-{i^*} &
      \pi_0^{\tilde T}(\Sigma^\infty_+\bP)\ar[r]^-{\res^{\tilde T}_C}&
      \pi_0^C(\Sigma^\infty_+\bP)\ar[r] & 0     } \]
  We alert the reader that according to our present conventions,
  here and below the group $C$ is augmented by the identity,
  so it acts on $\Sigma^\infty_+\bP$ by complex conjugation.
  This is {\em different} from the trivial $C$-action,
  which occurs for example in Proposition \ref{prop:modulo_torsion}.
  We write $e_{\tilde T}\in \pi_0^{\tilde T}(\Sigma^\infty_+\bP)$ for the stable tautological class,
  see \cite[(4.2)]{schwede:GSB};
  its restriction to $T$ is the tautological class $e_T$ from \eqref{eq:define_e_T}.
  The ring $\pi_0^C(\Sigma^\infty_+\bP)$ is isomorphic to the Burnside ring of the group $C$,
  and additively generated by $1=\res^{\tilde T}_C(e_{\tilde T})$,
  the multiplicative unit, and the class $\tr_{\{1\}}^C(1)=\tr_{\{1\}}^C(\res^{\tilde T}_{\{1\}}(e_{\tilde T}))$.
  Because $\res^{\tilde T}_C(1- e_{\tilde T})=0$,
  there is a unique class, the {\em Real-global pre-Bott class}
  \[  \tilde b_1 \ \in \ \pi_{\nu_1}^{\tilde T}(\Sigma^\infty_+\bP) \text{\quad such that\quad}
    i^*(\tilde b_1)\ = \ 1- e_{\tilde T}\text{\quad in\quad} \pi_0^{\tilde T}(\Sigma^\infty_+\bP)\ . \]
\end{construction}

The following proposition shows that
the Real-global pre-Bott class is indeed a lift of the first pre-Bott class
to the $\tilde T$-equivariant homotopy group,
and that it maps to the Bott class in Real-equivariant  $K$-theory.

\begin{prop}\label{prop:properties_tilde b}\ 
  \begin{enumerate}[\em (i)]
  \item
    The relation $\res^{\tilde T}_T(\tilde b_1)=b_1$ holds in $\pi^T_{\nu_1}(\Sigma^\infty_+\bP)$.
  \item
    The morphism $\eta\colon \Sigma^\infty_+ \bP \to \bKR$
    sends the class $\tilde b_1$ to the $\tilde T$-equivariant Bott class of $\nu_1$.      
  \item
    The class $\Phi^C(\res^{\tilde T}_C(\tilde b_1))$ in $\Phi^C_1(\Sigma^\infty_+\bP)$
    is nilpotent, where $C$ sits inside $\tilde T$ as complex conjugation. 
    \end{enumerate}
  \end{prop}
  \begin{proof}
  (i) By design we have
  \[ i^*(\res^{\tilde T}_T(\tilde b_1))\ = \
    \res^{\tilde T}_T(i^*(\tilde b_1))\ = \
    \res^{\tilde T}_T(1-e_{\tilde T})\ = \   1-e_T \ = \ i^*(b_1)  \]
  in the group $\pi_0^T(\Sigma^\infty_+\bP)$. Since the map
  $i^*\colon\pi_{\nu_1}^T(\Sigma^\infty_+\bP)\to\pi_0^T(\Sigma^\infty_+\bP)$
  is injective, this proves the desired relation.

  (ii)
  The forgetful ring homomorphism $\res^{\tilde T}_T\colon RR(\tilde T)\to RU (T)$
  from the Real representation ring of $\tilde T$
  is injective by \cite[page 13]{atiyah-segal:completion}.
  By Real-equivariant Bott periodicity \cite[Theorem 5.1]{atiyah:elliptic},
  the restriction homomorphism
  $\res^{\tilde T}_T\colon \pi_{\nu_1}^{\tilde T}(\bKR)\to\pi_{\nu_1}^T(\bKR)$
  is then injective, too. 
  It also takes the $\tilde T$-equivariant Bott class $\beta_{\tilde T,\nu_1}$
  to the $T$-equivariant Bott class $\beta_{T,\nu_1}$.
  Since $\eta$ is a morphism of $C$-global spectra, we have
  \[ \res^{\tilde T}_T(\eta_*(\tilde b_1))\ = \
    \eta_*(\res^{\tilde T}_T(\tilde b_1))\ = \
    \eta_*(b_1)\ = \ \beta_{T,\nu_1} \ = \ \res^{\tilde T}_T(\beta_{\tilde T,\nu_1})\ , \]
  where the third equation is Remark \ref{rk:preBott2Bott}.
  Since $\res^{\tilde T}_T$ is injective,
  the desired relation $\eta_*(\tilde b_1)=\beta_{\tilde T,\nu_1}$ follows.

  (iii) The underlying $C$-space of the $C$-global space $\bP$
  is infinite complex projective space with $C$-action by complex conjugation.
  So the fixed point space $\bP^C$ is homeomorphic to $\mR P^\infty$.
  The geometric fixed point groups of an equivariant suspension spectrum
  are isomorphic to the stable homotopy groups of the fixed point space,
  compare \cite[Example 3.3.3]{schwede:global}. So the ring $\Phi^C_*(\Sigma^\infty_+\bP)$
  is isomorphic to $\pi_*(\Sigma^\infty_+\mR P^\infty)$.
  The non-equivariant stable homotopy groups of $\mR P^\infty$
  are explicitly known in small degrees, see for example \cite{liulevicius}.
  In particular, $\pi_5(\Sigma^\infty_+ \mR P^\infty)=0$,
  so the fifth power of every class in $\pi_1(\Sigma^\infty_+\mR P^\infty)$ is zero.
  Hence also the fifth power of the class $\Phi^C(\res^{\tilde T}_C(\tilde b_1))$ is zero.    
\end{proof}

The following is the main result of this section;
it extends Theorem \ref{thm:main} (i) to the Real-global context.
As I explain in Remark \ref{rk:C_global_UCom}, I am convinced that
part (ii) of Theorem \ref{thm:main} also generalizes,
but to adapt the proof we need certain foundational properties
of ultra-commutative $C$-global ring spectra that are not yet in the literature.
In the following theorem we consider the pre-Bott classes $b_m$
as supported at the trivially augmented group $U(m)$.

\begin{theorem}\label{thm:Realglobal}\
  The morphism $\eta\colon \Sigma^\infty_+ \bP \to \bKR$
  is a $C$-global localization in $\CAlg(\Sc p_{C\dgl}^\sm)$
  by the classes $\tilde b_1$ and $\{b_m\}_{m\geq 2}$.
\end{theorem}
\begin{proof}
  We apply the localization recognition of Theorem \ref{thm:loc_reg}
  to the presentably symmetric monoidal $\infty$-category $\Sc p^\sm_{C\dgl}$
  of $C$-global spectra, to the morphism $\eta\colon\Sigma^\infty_+ \bP\to\bKR$,
  and to the family of geometric fixed point functors
  $\Phi^G\circ\res^\sm_\alpha\colon\Sc p^\sm_{C\dgl}\to \Sc p^\sm$,
  for $\alpha\colon G\to C$ ranging
  through representatives of all isomorphism classes of augmented Lie groups.
  These functors are strong symmetric monoidal left adjoints, and jointly conservative,
  by the $C$-global analog of Theorem \ref{thm:gl2G}.
  The full subcategory $\Wc$ of $(\Sigma^\infty_+\bP)\bs\CAlg(\Sc p^\sm_{C\dgl})$
  is spanned by the morphisms that invert the pre-Bott classes $\tilde b_1$ and $b_m$ for $m\geq 2$.
  It is closed under postcomposition by the $C$-global analog of
  Proposition \ref{prop:invert_inherit} (i).
  The morphism $\eta\colon\Sigma^\infty_+ \bP\to\bKR$ inverts $\tilde b_1$ by
  Proposition \ref{prop:properties_tilde b} (ii) and Real-equivariant Bott periodicity
  \cite[Theorem 5.1]{atiyah:elliptic};
  and it inverts the classes $b_m$ by Theorem \ref{thm:pre-Bott}.
  
  Now we verify hypothesis (c) for an augmented Lie group $\alpha\colon G\to C$.
  If the augmentation $\alpha$ is trivial, then $\bKR_\alpha=\bKU_G$,
  and similarly for $\Sigma^\infty_+\bP$. So in this case, the morphism of graded rings
  $\Phi^\alpha_*(\eta)\colon\Phi^\alpha_*(\Sigma^\infty_+\bP)\to \Phi^\alpha_*(\bKR)$
  specializes to $\Phi^G_*(\eta)\colon\Phi^G_*(\Sigma^\infty_+\bP)\to \Phi^G_*(\bKU)$,
  with the understanding that now $\Sigma^\infty_+\bP$ and $\bKU$
  denote the underlying global spectra.
  This homomorphism is a localization by the set of classes $\Phi^G(b_{G,V})$
  for all unitary $G$-representations $V$, by Theorem \ref{thm:geomfix}.
  Moreover, if $f\colon \Sigma^\infty_+\bP\to P$ belongs to $\Wc$,
  then it inverts $b_m$ for all $m\geq 1$; for $m=1$, this is because $f$ inverts
  $\tilde b_1$, and hence also its restriction $b_1$.
  So $f$ inverts all the classes $\Phi^G(b_{G,V})$, by 
  Propositions \ref{prop:restriction_and_invert} and \ref{prop:invert_criterion}.
  Since $\Phi^G_*(\eta)$ is a localization by these classes, the morphism $\Phi^G_*(f)$
  factors through $\Phi^G_*(\eta)$. This verifies hypothesis (c) for trivial augmentations.

  If the augmentation $\alpha\colon G\to C$ is surjective, then the inflation homomorphism
  $\alpha^*\colon\Phi^C_*(\Sigma^\infty_+\bP)\to \Phi^\alpha_*(\Sigma^\infty_+\bP)$
  is a ring homomorphism. Since $\Phi^C(\res^{\tilde T}_C(\tilde b_1))$ is nilpotent
  by Proposition \ref{prop:properties_tilde b} (iii), also the class
  \[ \Phi^\alpha_*(f)(\alpha^*(\Phi^C(\res^{\tilde T}_C(\tilde b_1)))) \ =\
    \Phi^G(\alpha^*(\res^{\tilde T}_C(f_*(\tilde b_1)))) \]
  is nilpotent. Because $f\colon\Sigma^\infty_+\bP\to P$ inverts $\tilde b_1$,
  it also inverts $\alpha^*(\Phi^C(\res^{\tilde T}_C(\tilde b_1)))$ by 
  Propositions \ref{prop:invert_criterion} and \ref{prop:restriction_and_invert}. 
  So $\Phi^G(\alpha^*(\res^{\tilde T}_C(f_*(\tilde b_1))))$ is both nilpotent
  and invertible, and the commutative ring spectrum $\Phi^\alpha(P)$ is trivial.
  In particular, the ring spectrum $\Phi^\alpha(\bKR)$ is trivial.
  So $\Phi^\alpha_*(\eta)\colon\Phi^\alpha_*(\Sigma^\infty_+\bP)\to \Phi^\alpha_*(\bKR)$
  is a localization, namely the effect of inverting 0.
  And $\Phi^\alpha_*(f)\colon \Phi^\alpha_*(\Sigma^\infty_+\bP)\to\Phi^\alpha_*(P)$
  factors through $\Phi^\alpha_*(\eta)$, proving (c).  
  So altogether, Theorem \ref{thm:loc_reg} applies
  and shows that $\eta\colon \Sigma^\infty_+\bP\to\bKR$
  is a $C$-global localization by $\{\tilde b_1, b_m\mid m\geq 2\}$
  in $\CAlg(\Sc p^\sm_{C\dgl})$.
\end{proof}

\begin{rk}[Ultra-commutative $C$-global localization]\label{rk:C_global_UCom}
  The $\infty$-category of {\em ultra-commutative $C$-global ring spectra} $\UR_{C\dgl}$
  is the localization of the 1-category of commutative orthogonal $C$-ring spectra
  at the class of $C$-global equivalences.
  I expect that the $C$-global equivalences on the category of
  commutative orthogonal $C$-ring spectra are the weak equivalences in
  a model structure that generalizes the global model structure
  for ultra-commutative ring spectra from \cite[Theorem 5.4.3]{schwede:global}.
  Granted this, we could deduce that the $\infty$-category $\UR_{C\dgl}$ is
  complete and cocomplete.

  Let $R$ be an ultra-commutative $C$-global ring spectrum, and
  let $\Xc$ be a set of representation-graded equivariant homotopy classes of $R$.
  An {\em ultra-commutative $C$-global localization by $\Xc$}
  is an initial object of the full subcategory of the
  slice category $R\bs \UR_{C\dgl}$ spanned by the $\Xc$-inverting morphisms.
  I expect that the morphism $\eta\colon \Sigma^\infty_+ \bP \to \bKR$
  is an ultra-commutative $C$-global localization
  by the classes $\tilde b_1$ and $\{b_m\}_{m\geq 2}$.
  In fact, the entire proof of the non-Real counterpart in 
  Theorem \ref{thm:main} (ii) should generalize to the Real-global context, as follows.
 
  Much like in \eqref{eq:ucom2CR_gl}, we obtain a `forgetful' functor
  \[ u_{C\dgl}\ :\ \UR_{C\dgl}\ \to\  \CAlg(\Sc p^\sm_{C\dgl})  \]
  from the $\infty$-category of ultra-commutative $C$-global ring spectra to the
  $\infty$-category of commutative $C$-global ring spectra.
  Theorem \ref{thm:UR_monadic} should generalize, hopefully with
  essentially the same proof, to show that the $\infty$-category $\UR_{C\dgl}$
  is presentable and that the functor $u_{C\dgl}$ is bimonadic.
  Granted this, in the proof of Theorem \ref{thm:UR2CR},
  we could replace `global' by `$C$-global' throughout, and replace compact Lie groups
  by augmented compact Lie groups.
  We would then in particular obtain the following $C$-global generalization of
  Theorem \ref{thm:UR2CR} (iii):
  let $i\colon R\to S$ be a morphism of ultra-commutative $C$-global ring spectra, and
  let $\Xc$ be a set of representation-graded equivariant homotopy classes of $R$.
  Suppose that the underlying morphism of commutative $C$-global ring spectra
  $u_{C\dgl}(i)$ is a $C$-global localization by $\Xc$ in $\CAlg(\Sc p^{\sm}_{C\dgl})$.
  Then $i$ is an ultra-commutative $C$-global localization by $\Xc$ in $\UR_{C\dgl}$.
  If this all works out as expected, the localization property of 
  $\eta\colon \Sigma^\infty_+ \bP \to \bKR$ in $\CAlg(\Sc p^\sm_{C\dgl})$
  from Theorem \ref{thm:Realglobal}
  implies that this morphism is also an ultra-commutative $C$-global localization in $\UR_{C\dgl}$
  by the classes $\tilde b_1$ and $\{b_m\}_{m\geq 2}$.
\end{rk}

    \begin{appendix}
  \section{Equivariant spectra}\label{app:equivariant_spectra}
  
  In this appendix we collect the background material
  on the $\infty$-categories of genuine $G$-spectra
that we need in the body of the paper.
Throughout the following discussion, we let $G$ be a compact Lie group. For our purposes,
it is convenient to introduce the $\infty$-category of $G$-spectra
via the orthogonal spectrum model of Mandell and May \cite{mandell-may}.
A detailed discussion based on this model is provided by Mathew, Naumann and Noel
in \cite[Section 5]{mathew-naumann-noel:nilpotence}, and we recall some key aspects below.
An equivalent model-independent definition of genuine $G$-spectra
is via inverting linear representation spheres on genuine $G$-spaces, see  
\cite[Appendix C]{gepner-meier:equivariantTMF}.
For finite groups, yet another intrinsically higher categorical definition
is as spectral Mackey functors \cite{barwick:spectral_mackey},
see \cite[Appendix A]{clausen-mathew-naumann-noel:descent-vanishing} for a streamlined proof
of the equivalence to the other models.

Geometric fixed points of equivariant ring spectra are our main
tool in this paper to identify localizations, see the geometric
fixed point criterion from Theorem \ref{thm:fix_criterion}.
A substantial part of this appendix is thus devoted to geometric fixed point spectra,
which we recall in Construction \ref{con:geom_fix}.
Below we shall reconcile two different meanings of `geometric fixed point groups':
the non-equivariant homotopy groups of the geometric fixed point spectrum,
and the explicit 1-categorical definition for equivariant orthogonal spectra
that can be found for example in \cite[Definition 3.3.1]{schwede:global},
and that we recall in Construction \ref{con:geomfix} below.
In Proposition \ref{prop:identify_Phis} we show that
these two interpretations are naturally and multiplicatively isomorphic.
\smallskip

We will frequently use {\em monoidal localization}
to construct lax or strong symmetric monoidal functors,
and to construct monoidal transformations between such functors.
Given symmetric monoidal $\infty$-categories $\Cc$ and $\Dc$,
we write $\Fun_{\lax}(\Cc,\Dc)$ for the $\infty$-category
with objects the lax symmetric monoidal functors from $\Cc$ to $\Dc$,
and with morphisms the monoidal transformations.
We write $\Fun_\tensor(\Cc,\Dc)$ for the full subcategory
of $\Fun_{\lax}(\Cc,\Dc)$ spanned by the strong symmetric monoidal functors.

\begin{defn}[Monoidal localization]
  Let $\M^\tensor$ be a symmetric monoidal model category
  with cofibrant unit object. A {\em monoidal localization}
  is a lax symmetric monoidal functor
  \[ \gamma^\tensor_\M\ \colon \ N(\M^\tensor) \ \to \ \M^\tensor_\infty \]
  to a symmetric monoidal $\infty$-category $\M^\tensor_\infty$
  with the following universal property: for every symmetric monoidal $\infty$-category $\Dc$,
  the restriction functor
  \[ 
    \Fun_{\lax}(\gamma^\tensor_\M,\Dc)\ \colon \
    \Fun_{\lax}(\M^\tensor_\infty,\Dc)\ \to \ \Fun_{\lax}^W(N(\M^\tensor),\Dc)
  \]
  is an equivalence of $\infty$-categories. Here the target
  $\Fun_{\lax}^W(N(\M^\tensor),\Dc)$ is the full subcategory of
  $\Fun_{\lax}(N(\M^\tensor),\Dc)$ spanned by those lax symmetric monoidal
  functors whose underlying functors invert weak equivalences.
\end{defn}

It would be more precise to speak of {\em symmetric monoidal localizations}.
However, since we have no need to consider localizations of (non-symmetric)
monoidal categories, we simplify the terminology by dropping the adjective `symmetric'.

\begin{rk}\label{rk:monloc_addendum}
  Clearly, monoidal localizations are unique (in the sense of higher category theory)
  whenever they exist. And Nikolaus and Scholze show in \cite[Theorem A.7]{nikolaus-scholze:TC}
  that every symmetric monoidal model category admits a monoidal localization.
  In the same theorem, they show moreover that
  the monoidal localization has the following additional properties:
  \begin{enumerate}[(i)]
  \item The functor of underlying $\infty$-categories
    \[ (\gamma^\tensor_\M)_{\td{1}}\ \colon \ N \M = N(\M^\tensor)_{\td{1}} \ \to \ (\M^\tensor_\infty)_{\td{1}} \]
    is a localization of the 1-category $\M$ at the class of weak equivalences.
  \item Let $\M_{\cof}^\tensor$ denote the full symmetric monoidal subcategory
    of cofibrant objects in $\M^\tensor$. Then the composite
    \[ \gamma^\tensor_{\M,\cof}\ : \
      N(\M^\tensor_{\cof}) \  \ \xra{\text{incl}} \ N(\M^\tensor) \ \xra{\gamma^\tensor_\M} \ \M^\tensor_\infty \]
    is strong (as opposed to only lax) symmetric monoidal.
 Since $\gamma^\tensor_{\M,\cof}$ is both strong symmetric monoidal
    and essentially surjective on objects, precomposition with it `detects strongness',
    in the following sense:
    let $F\colon \M^\tensor_\infty\to\Dc$ be a lax symmetric monoidal functor
    such that the composite $F\circ\gamma^\tensor_{\M,\cof}\colon\allowbreak N(\M^\tensor_{\cof})\to\Dc$ is strong (as opposed to only lax) symmetric monoidal.
    Then $F$ is strong symmetric monoidal.
  \item
    For every symmetric monoidal $\infty$-category $\Dc$,
    restriction along $\gamma^\tensor_{\M,\cof}$ is an equivalence of $\infty$-categories
 \[    \Fun_\tensor(\gamma^\tensor_{\M,\cof},\Dc)\ \colon \
      \Fun_\tensor(\M^\tensor_\infty,\Dc)\ \xra{\ \sim \ } \
      \Fun_\tensor^W(N(\M^\tensor_{\cof}),\Dc)\ .  
 \]
  \end{enumerate}
\end{rk}

\begin{construction}[The symmetric monoidal $\infty$-category of genuine $G$-spectra]
  We write $\Sp_G^O$ for the 1-category of orthogonal $G$-spectra.
  A morphism $f\colon X\to Y$ of orthogonal $G$-spectra is a {\em $\upi_*$-isomorphism}
  if the induced map of equivariant homotopy groups $\pi_k^H(f)\colon \pi_k^H(X)\to\pi_k^H(Y)$
  is an isomorphism for all $k\in\mZ$ and all closed subgroups $H$ of $G$.
  Various stable model structures on the category of orthogonal $G$-spectra
  have been established  with $\upi_*$-isomorphisms as weak equivalences,
  the first one being due to Mandell--May \cite[III Theorem 4.2]{mandell-may}.
  By \cite[III Proposition 7.5]{mandell-may} the smash product of orthogonal $G$-spectra
  is a Quillen bifunctor for the Mandell--May model structure.
  Since the monoidal unit, the $G$-sphere spectrum $\mS_G$,
  is cofibrant in this model structure,
  we are dealing with a symmetric monoidal model category.
  The monoidal localization thus yields a symmetric monoidal $\infty$-category
  \[ \Sc p^{\sm}_G \ = \ (\Sp_G^{O,\sm})_\infty \]
  and a lax symmetric monoidal functor
  \[ \gamma_G^\sm \ : \Sp_G^{O,\sm}\ \to \ \Sc p^{\sm}_G \]
  whose underlying functor
  \[ \gamma_G = (\gamma_G^\sm)_{\td{1}}\ : \
    \Sp_G^O \ \to \ (\Sc p_G^{\sm})_{\td{1}} = \Sc p_G \]
  is a localization of the 1-category of orthogonal $G$-spectra
  at the class of $\upi_*$-isomorphisms.
  We refer to the underlying $\infty$-category $\Sc p_G$
  as the {\em $\infty$-category of genuine $G$-spectra}, and to the monoidal structure as the
  {\em equivariant smash product}.
\end{construction}

As the localization of a complete and cocomplete model category,
the $\infty$-category $\Sc p_G$ is complete and cocomplete,
see \cite[Theorem 2.5.9]{barnea-harpaz-horel}.
The {\em $G$-stable homotopy category} $h(\Sc p_G)$ is
the triangulated homotopy category of the stable $\infty$-category $\Sc p_G$.
For every closed subgroup $H$ of $G$, the  unreduced suspension spectrum
$\Sigma^\infty_+ G/H$ of the coset $G$-space $G/H$ represents the functor
$\pi_0^H\colon h(\Sc p_G)\to \rAb$, compare \cite[Proposition 3.1.46]{schwede:global}.
Since the functors $\pi_*^H$ preserve coproducts
and jointly detect isomorphisms in $h(\Sc p_G)$, the set $\{\Sigma^\infty_+ G/H\}_{H\leq G}$
compactly generates  $h(\Sc p_G)$ as a triangulated category.
Hence the stable $\infty$-category $\Sc p_G$ of genuine $G$-spectra
is also compactly generated by these objects, and thus presentable.

\begin{construction}[Geometric fixed points]\label{con:geom_fix}
  We let $\Sc_{G,\ast}^\sm$ denote the presentably symmetric monoidal $\infty$-category
  of based $G$-spaces, defined as the monoidal localization of
  the symmetric monoidal model category $G\bT_*$
  of based compactly generated topological $G$-spaces \cite[Proposition B.7]{schwede:global},
  with weak equivalences those morphisms that induce weak homotopy equivalences
  of fixed points for all closed subgroups of $G$. The monoidal structure is by smash product.

  The 1-categorical suspension spectrum  functor
  \[ \Sigma^\infty\ : \ G\bT_\ast \ \to \ \Sp^O_G \]
  from based $G$-spaces to orthogonal $G$-spectra is strong symmetric monoidal
  and a left Quillen functor for the model structure presently under consideration.
  So it descends to a symmetric monoidal left adjoint
  \[ \Sigma^\infty\ : \ \Sc_{G,\ast}^\sm \ \to \ \Sc p_G^\sm \]
  of presentably symmetric monoidal $\infty$-categories.
  Gepner--Meier \cite[Corollary C.7]{gepner-meier:equivariantTMF}
  and Clausen--Mathew--Naumann--Noel
  \cite[Theorem A.2]{clausen-mathew-naumann-noel:descent-vanishing}
  show that the suspension spectrum functor
  $\Sigma^\infty\colon \Sc_{G,\ast}^\sm\to \Sc p_G^\sm$
  is the universal example of a strong symmetric monoidal left adjoint
  that inverts all representation spheres.
  More precisely,  $\Sigma^\infty\colon \Sc_{G,\ast}^\sm\to \Sc p_G^\sm$
  is initial in the full subcategory of $\Sc_{G,\ast}^\sm\bs \CAlg(\Pr^L)$ spanned
  by those symmetric monoidal left adjoints $F\colon \Sc_{G,\ast}^\sm\to \Dc$ with
  the property that $F(S^V)$ is invertible in $\Dc$ for every orthogonal $G$-representation $V$.
  
  The 1-categorical fixed point functor
  \[ (-)^G\ : \ G\bT_\ast \ \to \ \bT_* \]
  from based $G$-spaces to based spaces is fully homotopical,
  preserves pushouts and filtered colimits along closed embeddings,
  and commutes with geometric realization and smash products,
  compare \cite[Proposition B.1]{schwede:global}.
  So it descends to a symmetric monoidal left adjoint
  \[ (-)^G\ : \ \Sc_{G,\ast}^\sm \ \to \ \Sc_*^\sm \]
  on the level of presentably symmetric monoidal $\infty$-categories.
  Then the composite
  \[  \Sc_{G,\ast}^\sm \ \xra{(-)^G} \ \Sc_*^\sm \ \xra{\Sigma^\infty}\ \Sc p^\sm \]
  is another strong symmetric monoidal left adjoint. For every $G$-representation
  $V$, the space $(S^V)^G$ is a sphere, and thus becomes invertible in $\Sc p^\sm$.
  So by the above universal property of $\Sc p_G^\sm$ of inverting representation spheres,
  there is a unique strong symmetric monoidal left adjoint
  \begin{equation}\label{eq:define_Phi^G}
    \Phi^G\ : \ \Sc p_G^\sm \ \to \ \Sc p^\sm \ ,    
  \end{equation}
  the {\em geometric fixed point functor}, equipped with a homotopy
  that witnesses the commutativity of the following square in $\CAlg(\Pr^L)$:
  \begin{equation}\begin{aligned}\label{eq:Phi^G_def_square}
   \xymatrix{
      \Sc^\sm_{G,\ast} \ar[r]^-{(-)^G}\ar[d]_{\Sigma^\infty} & \Sc_*^\sm\ar[d]^{\Sigma^\infty} \\
      \Sc p_G^\sm \ar[r]_-{\Phi^G} & \Sc p^\sm}        
    \end{aligned} \end{equation}
\end{construction}

\begin{rk}
  Mathew, Naumann and Noel introduce the geometric fixed point functor $\Phi^G$
  somewhat differently in \cite[Definition 6.12]{mathew-naumann-noel:nilpotence},
  as the composite of smashing with $\tilde E \Pc$, followed by
  the `categorical fixed point' functor $(-)^G\colon \Sc p_G^\sm\to \Sc p^\sm$,
  right adjoint to the inflation functor $i_G\colon \Sc p^\sm\to \Sc p_G^\sm$,
  see \eqref{eq:inflation_functor}.
  Here $E\Pc$ denotes a universal $G$-space for the family of proper closed
  subgroups, and $\tilde E\Pc$ its unreduced suspension.
  Their geometric fixed point functor is also a strong monoidal left adjoint,
  and it also makes the diagram \eqref{eq:Phi^G_def_square} commute.
  So the two constructions are equivalent by the universal property
  of $\Sigma^\infty\colon \Sc_{G,\ast}^\sm\to \Sc p_G^\sm$.
\end{rk}

\begin{construction}[Inflation]
  The symmetric monoidal $\infty$-category of spectra is initial
  in the $\infty$-category $\CAlg(\Pr^L_{\st})$ of presentably symmetric monoidal
  stable $\infty$-categories, see \cite[Corollary 4.8.2.19]{lurie:higher_algebra}.
  So there is a unique strong symmetric monoidal left adjoint
  \begin{equation} \label{eq:inflation_functor}
    i_G \ \colon \ \Sc p^\sm \ \to \ \Sc p^\sm_G\ ,   
  \end{equation}
  called the {\em inflation functor}.
  The composite of inflation and geometric fixed points \eqref{eq:define_Phi^G} is a 
  strong symmetric monoidal left adjoint endofunctor of $\Sc p^\sm$;
  by virtue of $\Sc p^\sm$ being initial in $\CAlg(\Pr^L_{\st})$,
  the composite is uniquely monoidally equivalent to the identity:
  \begin{equation}\label{eq:Phi_p_G}
    \Id_{\Sc p^\sm}\ \sim\   \Phi^G\circ i_G\ .
  \end{equation}
\end{construction}

We call a morphism  $f\colon X\to Y$ of $G$-spectra
a {\em geometric fixed point equivalence} if the morphism
$f\sm\tilde E\Pc\colon X\sm\tilde E\Pc\to Y\sm\tilde E\Pc$
is an equivalence in $\Sc p_G$. This happens if and only if
the morphism $\Phi^G(f)\colon\Phi^G(X)\to\Phi^G(Y)$ is an equivalence in $\Sc p$.
A result of Lewis and May \cite[II Theorem 6.9]{lms},
expressed in more classical language, says that the geometric fixed point
functor $\Phi^G\colon \Sc p_G\to\Sc p$ witnesses the $\infty$-category of
non-equivariant spectra as a localization of the
$\infty$-category of $G$-spectra by the class of geometric fixed point equivalences.
This result was revisited from the higher categorical perspective
by Mathew, Naumann and Noel in \cite[Theorem 6.11]{mathew-naumann-noel:nilpotence}.
We shall now reinterpret this result of Lewis and May as a monoidal
localization property, without any particular claim to originality.

For a symmetric monoidal $\infty$-category $\Dc$,
we let $\Fun_{\lax}^{\tilde E\Pc}(\Sc p^\sm_G,\Dc)$ denote the full subcategory
of the $\infty$-category $\Fun_{\lax}(\Sc p^\sm_G,\Dc)$ spanned by those
lax symmetric monoidal functors $F\colon \Sc p^\sm_G\to\Dc$
whose underlying functor $F_{\td{1}}$ inverts the morphism
$X\sm i\colon X\to X\sm\tilde E\Pc$ for every $G$-spectrum $X$,
where $i\colon S^0\to\tilde E\Pc$ is the inclusion of the $G$-fixed cone points.
Equivalently, $F_{\td{1}}$ sends all geometric fixed point equivalences to equivalences.
And we set
\[ \Fun_\tensor^{\tilde E\Pc}(\Sc p^\sm_G,\Dc)\ = \
   \Fun_\tensor(\Sc p^\sm_G,\Dc)\cap \Fun_{\lax}^{\tilde E\Pc}(\Sc p^\sm_G,\Dc)\ ,
\]
the full subcategory
of $\Fun_{\lax}^{\tilde E\Pc}(\Sc p^\sm_G,\Dc)$ spanned by the strong symmetric monoidal
functors.

The following theorem is a higher categorical version of 
a special case of \cite[II Theorem 6.9]{lms}, and closely related to
\cite[Theorem 6.11]{mathew-naumann-noel:nilpotence}.

\begin{theorem}\label{thm:Phi_mon_loc}
  For every symmetric monoidal $\infty$-category $\Dc$, the functors
\[
 \xymatrix{
    \Fun_{\lax}(\Phi^G,\Dc)\colon       \Fun_{\lax}(\Sc p^\sm,\Dc)\ \ar@<.4ex>[r]
      &\ \Fun_{\lax}^{\tilde E\Pc}(\Sc p^\sm_G,\Dc)\colon \Fun_{\lax}(i_G,\Dc) \ar@<.4ex>[l]  }  \]
  are mutually inverse equivalences of $\infty$-categories, and
  the functors
\[
 \xymatrix{
    \Fun_{\tensor}(\Phi^G,\Dc)\colon  \Fun_{\tensor}(\Sc p^\sm,\Dc)\ \ar@<.4ex>[r]
      &\ \Fun_{\tensor}^{\tilde E\Pc}(\Sc p^\sm_G,\Dc)\colon \Fun_{\tensor}(i_G,\Dc) \ar@<.4ex>[l]  }  \]
are mutually inverse equivalences of $\infty$-categories.
\end{theorem}
\begin{proof}
  The morphism of based $G$-spaces $\tilde E\Pc\sm i\colon \tilde E\Pc\to\tilde E\Pc\sm\tilde E\Pc$
  is an equivalence in $\Sc_{G,\ast}$, so the induced morphism of suspension
  spectra
  \[ (\Sigma^\infty \tilde E\Pc)\sm i\ \colon\
    \Sigma^\infty \tilde E\Pc\ \to\
    (\Sigma^\infty \tilde E\Pc)\sm (\Sigma^\infty\tilde E\Pc)
  \]
  is an equivalence in $\Sc p_G$; we have implicitly used the strong
  symmetric monoidal structure of the suspension spectrum functor.
  This means that the morphism $\Sigma^\infty i\colon\mS_G\to\Sigma^\infty \tilde E\Pc$
  is an idempotent object in the sense of \cite[Definition 4.8.2.1]{lurie:higher_algebra}.

  As in \cite[Theorem 6.11]{mathew-naumann-noel:nilpotence},
  we write $\Sc p_G[\Pc^{-1}]$ for the full subcategory of $\Sc p_G$
  spanned by those $G$-spectra $X$ such that the morphism 
  $X\sm i\colon X\to X\sm\tilde E\Pc$ is an equivalence.
  Because $\Sigma^\infty \tilde E\Pc$ is an idempotent algebra,
  this subcategory inherits a unique symmetric monoidal structure,
  denoted $\Sc p_G^\sm[\Pc^{-1}]$, such that the  Bousfield localization functor
  \[ (\Sigma^\infty \tilde E\Pc)\sm - \ \colon  \ \Sc p^\sm_G \ \to \ \Sc p^\sm_G[\Pc^{-1}] \]
  becomes strong symmetric monoidal, see \cite[Proposition 4.8.2.7]{lurie:higher_algebra}.
  In particular, this functor is a monoidal localization
  by the class of geometric fixed point equivalences, i.e., the functor
  \[  \Fun_{\lax}( (\Sigma^\infty \tilde E\Pc)\sm - ,\Dc)\colon  \Fun_{\lax}(\Sc p^\sm_G[\Pc^{-1}],\Dc)\ \to
    \ \Fun_{\lax}^{\tilde E\Pc}(\Sc p^\sm_G,\Dc)\]
  is an equivalence.

  Mathew, Naumann and Noel show in \cite[Theorem 6.11]{mathew-naumann-noel:nilpotence}
  that the composite strong symmetric monoidal functor
  \[ \Sc p^\sm \ \xra{\ i_G \ }\ \Sc p^\sm_G \ \xra{\Sigma^\infty\tilde E\Pc\sm-} \
    \Sc p^\sm_G[\Pc^{-1}]  \]
  is an equivalence of symmetric monoidal $\infty$-categories.
  So in the diagram      
  \[
   \Fun_{\lax}(\Sc p^\sm_G[\Pc^{-1}],\Dc)\xra{\Fun_{\lax}( (\Sigma^\infty \tilde E\Pc)\sm - ,\Dc)} 
    \Fun_{\lax}^{\tilde E\Pc}(\Sc p^\sm_G,\Dc)
   \xra{\Fun_{\lax}(i_G,\Dc)} \Fun_{\lax}(\Sc p^\sm,\Dc)
   \]
the first functor and the composite are equivalences, hence so is the second.
 Since $\Phi^G$ is monoidally left
 inverse to $i_G$ by \eqref{eq:Phi_p_G}, the functor
  \[     \Fun_{\lax}(\Phi^G,\Dc)\colon \Fun_{\lax}(\Sc p^\sm,\Dc)\ \to\
    \Fun_{\lax}^{\tilde E\Pc}(\Sc p^\sm_G,\Dc)  \]
  is an equivalence of $\infty$-categories, too.
  Since both  $i_G$ and $\Phi^G$ are strong (as opposed to just lax)
  symmetric monoidal, composition with them preserves the full subcategories
  $\Fun_{\tensor}^{\tilde E\Pc}(\Sc p^\sm_G,\Dc)$ and $\Fun_{\tensor}(\Sc p^\sm,\Dc)$,
  so the restrictions to these full subcategories are equivalences, too.
\end{proof}

There are two possible meanings of `geometric fixed point groups'
for a $G$-spectrum $X$.
On the one hand, `geometric fixed point groups' could refer to the homotopy groups
of the geometric fixed point spectrum $\Phi^G(X)$
in the $\infty$-category $\Sc p$.
On the other hand, it could refer to the explicit 1-categorical definition
that can be found for example in \cite[Definition 3.3.1]{schwede:global},
and that we recall now.
The former interpretation has convenient formal properties,
while we use the latter interpretation for explicit calculations.
We need to know that the two interpretations are naturally
and multiplicatively isomorphic, and the isomorphism \eqref{eq:Phi_vs_Phi}
constructed below witnesses precisely this.

\begin{construction}[Geometric fixed point homotopy groups]\label{con:geomfix}
  We let $G$ be a compact Lie group, $X$ an orthogonal $G$-spectrum,
  and $k\in\mZ$ an integer. The {\em $k$-th $G$-geometric fixed point group} is
  \begin{equation}\label{eq:define_geom_fix}
  \Phi_k^G(X)\ = \ \colim_{m+k\geq 0, V\in s(\Uc_G)} [S^{V^G\oplus \mR^{m+k}},X(V\oplus \mR^m)^G] \ .  
  \end{equation}
  Here $s(\Uc_G)$ is the poset, under inclusion, of finite-dimensional
  $G$-subrepresentations of the complete $G$-universe $\Uc_G$.
  And `$m+k\geq 0$' means that $m$ runs over the poset of those natural numbers
  such that $m+k$ is non-negative.
  For $m\leq n$ and $V\leq W$ in $s(\Uc_G)$, the transition maps
  $[S^{V^G\oplus \mR^{m+k}},X(V\oplus\mR^m)^G]\to [S^{W^G\oplus\mR^{n+k}},X(W\oplus\mR^n)^G]$
  in the colimit system generalize those in
  \cite[Definition 3.3.1]{schwede:global}.
  The monoidal natural isomorphism \eqref{eq:Phi_vs_Phi} below
  gives a different perspective on these 
  geometric fixed point groups, namely as the homotopy groups of
  a geometric fixed point spectrum $\Phi^G(X)$.
    
  The geometric fixed point groups are multiplicative, in the following sense.
  For orthogonal $G$-spectra $X$ and $Y$, we define external multiplication maps
  \begin{equation}\label{eq:cdot_Phi^G}
    \times\ : \ \Phi_k^G(X)\times \Phi_l^G(Y)\ \to \ \Phi_{k+l}^G(X\sm Y)\ .     
  \end{equation}
  We let $f\colon S^{V^G\oplus\mR^{m+k}}\to X(V\oplus\mR^m)^G$ and
  $g\colon S^{W^G\oplus\mR^{n+l}}\to Y(W\oplus\mR^n)^G$
  be continuous based maps that represent classes $x=[f]\in \Phi_k^G(X)$ and $y=[g]\in\Phi_l^G(Y)$,
  respectively. We define $x\times y$ as the class represented by
  the following composite:
  \begin{align*}
    S^{(V\oplus W)^G\oplus \mR^{m+n+k+l}}\
    &\xra{\text{shuffle}}\
    S^{V^G\oplus \mR^{m+k}\oplus  W^G\oplus\mR^{n+l}}\\
    &\xra{f\sm g}\ X(V\oplus\mR^m)^G\sm Y(W\oplus\mR^n)^G\ = \ (X(V\oplus\mR^m)\sm Y(W\oplus\mR^n))^G\\
    &\xra{(\iota_{V\oplus\mR^m,W\oplus\mR^n})^G} \ ((X\sm Y)(V\oplus\mR^m\oplus W\oplus\mR^n))^G \\
    &\xra{\text{shuffle}}\ ((X\sm Y)( V\oplus W\oplus\mR^{m+n}))^G 
  \end{align*}
  Here $\iota_{V\oplus\mR^m,W\oplus\mR^n}$ is a component of the universal bimorphism.

  Essentially the same arguments as for the pairing on equivariant homotopy groups
  in \cite[Theorem 3.5.14]{schwede:global} show that the pairing
  \eqref{eq:cdot_Phi^G} is biadditive, unital, associative and commutative,
  the latter with respect to the Koszul sign $(-1)^{k l}$ on pairs of elements
  of degree $(k,l)$. We omit the formal proof.
\end{construction}

\begin{construction}[Descending equivariant and geometric fixed point homotopy groups]\label{con:descend_pi_and_Phi}
  The universal property of the monoidal localization functor 
  $\gamma^\sm_G\colon\Sp^{O,\sm}_G\to\Sc p^\sm_G$ allows us to descend
  the multiplicative structure of the homotopy group functors $\pi_*^G$
  and the geometric fixed point functors $\Phi^G_*$ from the orthogonal spectrum model
  to the $\infty$-category of $G$-spectra, as follows.

  We write $\GrAb^\tensor$ for the symmetric monoidal 1-category of
  $\mZ$-graded abelian groups, under graded tensor product, 
  and with the Koszul sign built into the symmetry isomorphism.
  We write $\pi_*^G(X)=\{\pi_k^G(X)\}_{k\in\mZ}$ for the $\mZ$-graded $G$-equivariant
  homotopy groups of an orthogonal $G$-spectrum $X$.
  By \cite[Theorem 3.5.14]{schwede:global}, the maps
  \[ \times \ \colon \ \pi_k^G(X)\times\pi_l^G(Y)\ \to \ \pi_{k+l}^G(X\sm Y) \]
  defined in \cite[(3.5.13)]{schwede:global}
  endow the functor $\pi_*^G$ with a lax symmetric monoidal structure.
  By definition, this functor also inverts $\upi_*$-isomorphisms,
  so it descends uniquely through
  the monoidal localization $\gamma^\sm_G\colon\Sp^{O,\sm}_G\to\Sc p^\sm_G$.
  In other words, there is a unique lax symmetric monoidal functor
  \[ \bar\pi_*^G \ \colon\  \Sc p^\sm_G \ \to \ \GrAb^\tensor \]
  equipped with a monoidal natural isomorphism
  \[  \bar\pi_*^G\circ \gamma^\sm_G \ \iso \ \pi_*^G   \]
  between lax symmetric monoidal functors $\Sp^{O,\sm}_G\to\GrAb^\tensor$.
  
  Similarly, the maps \eqref{eq:cdot_Phi^G} endow the functor $\Phi_*^G$
  with a lax symmetric monoidal structure.
  This functor inverts $\upi_*$-isomorphisms by \cite[Proposition 3.3.10]{schwede:global},
  so it, too,  descends uniquely through
  the monoidal localization $\gamma^\sm_G\colon\Sp^{O,\sm}_G\to\Sc p^\sm_G$:
  there is a unique lax symmetric monoidal functor
  \[ \bar\Phi_*^G \ \colon\ \Sc p^\sm_G \ \to \ \GrAb^\tensor \]
  equipped with a monoidal natural isomorphism
    \[
  \bar\Phi_*^G\circ \gamma^\sm_G \ \iso \ \Phi_*^G\ .    
\]
  The {\em geometric fixed point homomorphism}
  \begin{equation}\label{eq:define_Phi^G_map}
    \Phi^G\ \colon\ \pi_k^G(X)\ \to\  \Phi^G_k(X)
  \end{equation}
  does what the name suggests:
  it sends the class of a $G$-map $f\colon S^{U\oplus\mR^{m+k}}\to X(U\oplus\mR^m)$
  to the class of the map
  \[ f^G \  \colon \ S^{U^G\oplus\mR^{m+k}} =  (S^{U\oplus\mR^{m+k}})^G\ \to\  X(U\oplus\mR^m)^G
    \ , \]
  the restriction to $G$-fixed points.
  As the integer $k$ varies, the geometric fixed point maps \eqref{eq:define_Phi^G_map}
  form a monoidal natural transformation $\Phi^G\colon \pi_*^G\to\Phi^G_*$
  between lax symmetric monoidal functors that descend through $\gamma^\sm_G$.
  So there is a unique monoidal natural transformation
  \[ \bar\Phi^G \ \colon \ \bar\pi_*^G \ \to \ \bar\Phi^G_* \]
  of lax symmetric monoidal functors $\Sc p^\sm_G\to\GrAb^\tensor$
  such that $\bar\Phi^G\circ \gamma^\sm_G=\Phi^G$.
\end{construction}

\section{Global spectra}\label{app:global}

In this appendix, we recall the necessary background material on global spectra
and global ring spectra.
We work in the orthogonal spectrum model of \cite[Chapter 4]{schwede:global}.
An equivalent and intrinsically higher categorical construction of global
spectra was provided by Linskens, Nardin and Pol in \cite{linskens-nardin-pol}.
The main qualitative features of global spectra
are summarized from a higher categorical perspective in Theorem \ref{thm:Sp_gl properties}.
And the main properties of the functor $\res_G$ that extracts the underlying $G$-homotopy
type of a global spectrum are summarized in Theorem \ref{thm:gl2G}.

We also include a rectification result that exhibits `global' module categories,
in the higher categorical sense, as localizations of 1-categorical
modules in the orthogonal spectrum model, including symmetric monoidal structures.
This symmetric monoidal rectification in Theorem \ref{thm:monoidal_loc}
will also be needed in the proof of Theorem \ref{thm:UR_monadic},
where we establish the
bimonadicity of the forgetful functor $u_{\gl}\colon \UR_{\gl}\to\CAlg(\Sc p_{\gl}^\sm)$
from ultra-commutative to commutative global ring spectra.

\begin{construction}[The symmetric monoidal $\infty$-category of global spectra]
  A morphism $f\colon X\to Y$ of orthogonal spectra is a {\em global equivalence}
  if the induced map of equivariant homotopy groups $\pi_k^G(f)\colon \pi_k^G(X)\to\pi_k^G(Y)$
  is an isomorphism for all $k\in\mZ$ and all compact Lie groups $G$.
  We show in \cite[Theorem 4.3.18]{schwede:global} that the global equivalences
  are part of a stable, cofibrantly generated, proper and topological model structure
  on the category of orthogonal spectra.
  By \cite[Proposition 4.3.24]{schwede:global},
  the smash product of orthogonal spectra is a Quillen bifunctor for the global model structure.
  Since the monoidal unit, the sphere spectrum $\mS$, is cofibrant in the global model structure,
  we are dealing with a symmetric monoidal model category.
  The monoidal localization thus yields a symmetric monoidal $\infty$-category
  \[ \Sc p^{\sm}_{\gl} \ = \ (\Sp^{O,\sm})_\infty \]
  and a lax symmetric monoidal functor
  \[ \gamma_{\gl}^\sm \ : \ \Sp^{O,\sm}\  \to \ \Sc p^{\sm}_{\gl} \]
  whose underlying functor
  \[ \gamma_{\gl} = (\gamma^\sm_{\gl})_{\td{1}}\ : \ \Sp^O \ \to \ (\Sc p_{\gl}^{\sm})_{\td{1}} = \Sc p_{\gl}\]
  is a localization of the 1-category of orthogonal spectra at the class of global equivalences.
  We refer to the underlying $\infty$-category $\Sc p_{\gl}$
  as the {\em $\infty$-category of global spectra}, and to the monoidal structure as the
  {\em global smash product}.
\end{construction}

The {\em $\infty$-category of commutative global ring spectra} $\CAlg(\Sc p^\sm_{\gl})$
is the $\infty$-category of commutative algebras in the symmetric monoidal $\infty$-category
$\Sc p^{\sm}_{\gl}$ of global spectra under the global smash product.

\begin{theorem}\label{thm:Sp_gl properties}\
  \begin{enumerate}[\em (i)]
  \item The $\infty$-category $\Sc p_{\gl}$ of global spectra 
    is stable and compactly generated, and hence presentable.
  \item 
    The symmetric monoidal $\infty$-category $\Sc p^{\sm}_{\gl}$ of global spectra
    is presentably symmetric monoidal.
  \item The $\infty$-category $\CAlg(\Sc p_{\gl}^\sm)$
    of commutative global ring spectra is presentable.
  \end{enumerate}
\end{theorem}
\begin{proof}
  (i)
  The $\infty$-category $\Sc p_{\gl}$ is complete and cocomplete
  as the localization of a complete and cocomplete model category,
  see \cite[Theorem 2.5.9]{barnea-harpaz-horel}.
  Since the global model structure of  \cite[Theorem 4.3.18]{schwede:global}
  is stable, the suspension functor on its homotopy category is invertible,
  so the $\infty$-category $\Sc p_{\gl}$ is stable,
  compare \cite[Corollary 1.4.2.27]{lurie:higher_algebra}.
  Moreover, by \cite[Theorem 4.4.3]{schwede:global} the triangulated homotopy category
  of the stable $\infty$-category of global spectra
  is compactly generated by the unreduced suspension spectra of the global classifying spaces
  of all compact Lie groups;  in particular, the $\infty$-category $\Sc p_{\gl}$
  of global spectra is presentable by \cite[Corollary 1.4.4.2]{lurie:higher_algebra}.

  (ii) Because the smash product is a Quillen bifunctor
  for the global model structure, the induced smash product
  on $\Sc p_{\gl}$ preserves colimits in both variables.
  Since the underlying $\infty$-category is presentable by (i),
  the symmetric monoidal $\infty$-category $\Sc p^{\sm}_{\gl}$
  is presentably symmetric monoidal.

  Since $\Sc p^{\sm}_{\gl}$ is presentably symmetric monoidal by (ii),
  the $\infty$-category $\CAlg(\Sc p^\sm_{\gl})$ of its commutative algebras
  is presentable by \cite[Corollary 3.2.3.5]{lurie:higher_algebra}, proving (iii).  
\end{proof}

As we now recall, every global spectrum has an underlying $G$-spectrum,
for every compact Lie group $G$.
This process is implemented by a strong symmetric monoidal left adjoint
from global spectra to $G$-spectra.

\begin{construction}[Underlying $G$-spectrum]\label{con:res_G}
  We consider the functor of 1-categories
  \[ (-)_G \ : \ \Sp^O\ \to \ \Sp_G^O \ , \quad X \ \longmapsto \ X_G\]
  from orthogonal spectra to orthogonal $G$-spectra
  that endows an orthogonal spectrum with the trivial $G$-action.
  This functor takes global equivalences to $\upi_*$-isomorphisms.
  The `equivariant' smash product of orthogonal $G$-spectra is simply the smash
  product of the underlying non-equivariant orthogonal spectra
  with diagonal $G$-action. So the trivial action functor $(-)_G$ is strong symmetric monoidal
  for the respective smash products.
  The universal property of the monoidal localization $\gamma^\sm_{\gl}\colon \Sp^{O,\sm}\to\Sc p_{\gl}^\sm$
  thus yields a unique lax symmetric monoidal functor of symmetric monoidal $\infty$-categories
  \begin{equation}\label{eq:define_res^sm_G}
    \res_G^{\sm} \ : \ \Sc p_{\gl}^{\sm} \ \to \ \Sc p_G^{\sm}    
  \end{equation}
  equipped with a homotopy witnessing the commutativity of the following square on the left:
  \[ \xymatrix@C=12mm{
      \Sp^{O,\sm}\ar[r]^-{\gamma^\sm_{\gl}}\ar[d]_{(-)_G^\sm} &  \Sc p_{\gl}^\sm \ar[d]^{\res_G^\sm}   &&
      \CAlg(\Sp^{O,\sm})\ar[r]^-{\CAlg(\gamma^\sm_{\gl})}\ar[d]_{\CAlg((-)_G^\sm)} &  \CAlg(\Sc p_{\gl}^\sm) \ar[d]^{\CAlg(\res_G^\sm)}      \\
       \Sp_G^{O,\sm} \ar[r]_-{\gamma_G^\sm} & \Sc p_G^\sm &&
       \CAlg(\Sp_G^{O,\sm}) \ar[r]_-{\CAlg(\gamma_G^\sm)} & \CAlg(\Sc p_G^\sm)
     } \]
   The underlying  functor
   \[ \res_G \ : \ \Sc p_{\gl}\ \to \ \Sc p_G \]
   of $\res_G^\sm$ is then descended from the trivial action functor
   by the universal property of the localization $\gamma_{\gl}\colon \Sp^O\to\Sc p_{\gl}$.
   While a priori only lax, we show in the next theorem that
   $\res_G^{\sm}$ is in fact strong symmetric monoidal.
   Lax symmetric monoidal functors induce
   functors of commutative algebras, yielding the
   commutative square of $\infty$-categories on the right.
   \end{construction}

A commutative square of $\infty$-categories
\[\xymatrix{ \Bc \ar[r]^-E \ar[d]_-v & \Dc\ar[d]^u\\ 
    \Ac \ar[r]_F &\Cc   }    \]
is {\em horizontally right adjointable} if the functors
$E$ and $F$ admit right adjoints $\rho\colon\Dc\to\Bc$ and $\omega\colon \Cc\to\Ac$,
respectively, and the {\em Beck--Chevalley transform}, i.e., the composite natural transformation
\[ v\circ\rho \ \xra{\eta\diamond(v\circ\rho)} \
  \omega\circ F \circ v\circ \rho \ \sim \
  \omega\circ u \circ E\circ \rho\ \ \xra{(\omega\circ u)\diamond\epsilon} \
  \omega\circ u
\]
of functors $\Dc\to\Ac$ is an equivalence.
Here $\eta\colon \Id_\Ac\to \omega\circ F$ is the unit of the adjunction
$(F,\omega)$, and $\epsilon\colon E\circ\rho\to\Id_\Dc$ is the counit
of the adjunction $(E,\rho)$.
There are analogous definitions of {\em horizontally left adjointable},
and of {\em vertically left/right adjointable}.

Every lax symmetric monoidal functor $F\colon \Cc\to\Dc$
between symmetric monoidal $\infty$-categories induces a
functor between the $\infty$-categories of commutative algebras
$\CAlg(F)\colon \CAlg(\Cc)\to\CAlg(\Dc)$, by postcomposition.

\begin{lemma}\label{lem:CAlg_left_adjoint}
  Let $F\colon \Cc\to\Dc$ be a strong symmetric monoidal functor
  between symmetric monoidal $\infty$-categories
  whose underlying functor $F_{\td{1}}\colon \Cc_{\td{1}}\to \Dc_{\td{1}}$ is a left adjoint.
  Then $F\colon \Cc\to\Dc$ is a left adjoint relative to $\rFin_*$,
  the functor  $\CAlg(F)\colon \CAlg(\Cc)\to\CAlg(\Dc)$
  has a right adjoint, and the commutative square
    \[ \xymatrix@C=12mm@R=7mm{
        \CAlg(\Cc)\ar[d]_{\ev_{\td{1}}} \ar[r]^-{\CAlg(F)} &
        \CAlg(\Dc)\ar[d]^{\ev_{\td{1}}}\\
        \Cc_{\td{1}}\ar[r]_-{F_{\td{1}}}  & \Dc_{\td{1}}
      } \]
    is horizontally right adjointable.
\end{lemma}
\begin{proof}
  By \cite[Corollary 7.3.2.7]{lurie:higher_algebra}, the functor $F$
  has a relative right adjoint $G\colon \Dc\to\Cc$ over $\rFin_*$,
  which is moreover a morphism of operads. So $G$ is a lax symmetric monoidal
  structure on a right adjoint $G_{\td{1}}\colon \Dc_{\td{1}}\to\Cc_{\td{1}}$ to $F_{\td{1}}$.
  The unit and counit of the relative adjunction then induce
  the unit and counit of an adjunction between
  $\CAlg(F)\colon \CAlg(\Cc)\to\CAlg(\Dc)$ and $\CAlg(G)$.
  The horizontal right adjointability of the square is simply the statement
  that the underlying functor of the right adjoint $\CAlg(G)$
  to $\CAlg(F)$ is the right adjoint $G_{\td{1}}$ to the functor $F_{\td{1}}$.
\end{proof}

\begin{theorem}\label{thm:gl2G}
  Let $G$ be a compact Lie group.
  \begin{enumerate}[\em (i)]
    \item The functor
    $\res_G\colon \Sc p_{\gl}\to\Sc p_G$ has a left adjoint and a right adjoint.
  \item The functors $\res_G\colon \Sc p_{\gl}\to \Sc p_G$ are jointly conservative
    as $G$ runs over representatives for the isomorphism classes of all compact Lie groups.
  \item The functors $\Phi^G\circ \res_G\colon \Sc p_{\gl}\to \Sc p$ are jointly conservative
    as $G$ runs over representatives for the isomorphism classes of all compact Lie groups.
  \item The  lax symmetric monoidal functor
    $\res_G^{\sm}\colon \Sc p_{\gl}^{\sm}\to\Sc p_G^{\sm}$ is strong symmetric monoidal.
    \item  The functor $\CAlg(\res^\sm_G)$ has a left adjoint and a right adjoint.
  \end{enumerate}
\end{theorem}
\begin{proof}
  (i) The functor $\res_G$ is an exact functor between stable $\infty$-categories,
  and the induced functor of triangulated homotopy categories
  $h(\res_G)\colon h(\Sc p_{\gl})\to h(\Sc p_G)$ has adjoints on either side
  by \cite[Theorem 4.5.24]{schwede:global}.
  So $h(\res_G)$ preserves sums and products.
  Hence the functor of $\infty$-categories $\res_G$ is an exact functor
  preserving sums and products, and hence all limits and colimits.
  Since the source and target of $\res_G$  are presentable,
  this functor has adjoints on either side, by Lurie's adjoint functor theorem
  \cite[Corollary 5.5.2.9]{lurie:htt}.
  
  Property (ii) holds by design, since a morphism
  $f\colon X\to Y$ of orthogonal spectra is defined to be a global equivalence
  if and only if the morphisms of orthogonal $G$-spectra
  $f_G\colon X_G\to Y_G$ are $\upi_*$-isomorphisms of orthogonal $G$-spectra
  for all compact Lie groups $G$.
  Property (iii) then follows from property (ii) and
  the fact that the geometric fixed point functors $\Phi^H\circ\res^G_H\colon \Sc p_G\to\Sc p$
  are jointly conservative for all closed subgroups $H$ of $G$,
  see \cite[Proposition 6.13]{mathew-naumann-noel:nilpotence} for an
  $\infty$-categorical treatment, and \cite[Proposition 3.3.10]{schwede:global}
  for a proof in the orthogonal spectrum model.
  We also use implicitly that the functor $\Phi^H\circ\res^G_H\circ \res_G\colon \Sc p_{\gl}\to \Sc p$
  is naturally equivalent to $\Phi^H\circ\res_H$.

  (iv) Strong (versus lax) monoidality can be tested on homotopy categories; 
  at the level of homotopy categories, the argument is explained
  immediately before \cite[Theorem 4.5.24]{schwede:global}.
  The ultimate reason is the fact that for every flat orthogonal spectrum $X$
  (i.e., cofibrant in the global model structure) and every compact Lie group $G$,
  the functor $X_G\sm-\colon \Sp_G^O\to\Sp_G^O$ preserves $\upi_*$-isomorphisms,
  see \cite[Theorem 3.5.10]{schwede:global}.

  (v) Since $\res_G^\sm$ is a strong symmetric monoidal left adjoint,
  $\CAlg(\res_G^\sm)$ has a right adjoint by Lemma \ref{lem:CAlg_left_adjoint}.
  For every presentably symmetric monoidal $\infty$-category $\Cc$,
  the forgetful functor $\ev_{\td{1}}\colon\allowbreak\CAlg(\Cc)\to \Cc_{\td{1}}$
  preserves and creates all limits and sifted colimits by
  Corollary 3.2.2.5 and Proposition 3.2.3.1 of  \cite{lurie:higher_algebra}.
  Since $\res_G$ has adjoints on either side, 
  $\ev_{\td{1}}\circ\CAlg(\res_G^\sm)\sim \res_G\circ\ev_{\td{1}}$ preserves limits
  and sifted colimits.
  Since  $\ev_{\td{1}}$ is conservative and creates limits and sifted colimits,
  we conclude that the functor $\CAlg(\res_G^\sm)$ preserves limits and sifted colimits.
  Since the source and target of $\CAlg(\res_G^\sm)$ are presentable,
  this functor has a left adjoint
  by the adjoint functor theorem \cite[Corollary 5.5.2.9]{lurie:htt}.
\end{proof}

\begin{construction}[Inflation homomorphisms]
  We let $X$ be an orthogonal spectrum, and we let $\gamma\colon K\to G$ be a continuous
  epimorphism of compact Lie groups.
  For every $G$-representation $V$, we have $\gamma^*(V)^K=V^G$ and
  \[    X(\gamma^*(V)\oplus\mR^m)^K\ = \ \gamma^*(X(V\oplus\mR^m))^K
    \ = \ X(V\oplus\mR^m)^G\ ,  \]
  because $\gamma$ is surjective.
  We define the {\em inflation homomorphism}
  \begin{equation}\label{eq:define_inflation}
  \gamma^*\ \colon\ \Phi_k^G(X) \ \to \ \Phi^K_k(X)   
  \end{equation}
  by sending the class of a based continuous map
  $f\colon S^{V^G\oplus\mR^{m+k}}\to X(V\oplus\mR^m)^G$
  to the class of the same map, but now viewed as a map
  \[ f \ \colon \ S^{\gamma^*(V)^K\oplus \mR^{m+k}}\  \to \  X(\gamma^*(V)\oplus\mR^m)^K\ . \]
  Much like in \cite[Construction 3.1.15]{schwede:global},
  we exploit here that elements of $\Phi_k^K(X)$
  can unambiguously be represented by based maps $S^{W^K\oplus\mR^{m+k}}\to X(W\oplus\mR^m)^K$,
  for any $K$-representation $W$, not necessarily a subrepresentation of the chosen
  complete universe, see \cite[Construction 3.1.13]{schwede:global}.
  The inflation homomorphisms $\gamma^*$ are clearly contravariantly functorial
  for composition of continuous epimorphisms.
  The inflation homomorphism $\gamma^*$
  is a monoidal natural transformation between two lax symmetric monoidal
  functors from orthogonal spectra to graded abelian groups that both invert global equivalences.
  So by monoidal localization, $\gamma^*$ descends uniquely to a
  monoidal natural transformation
  $\gamma^*\colon\bar\Phi_*^G\circ\res^\sm_G\to\bar\Phi^K_*\circ\res^\sm_K$
  between the descended functors $\Sc p^\sm_{\gl}\to\GrAb^\tensor$
  from the $\infty$-category of global spectra to graded abelian groups;
  we abuse notation and denote the descended transformation by the same symbol.
\end{construction}

\begin{construction}
  We let $i_{\gl}\colon \Sc p^\sm\to\Sc p^\sm_{\gl}$ be the unique
  strong symmetric monoidal left adjoint.
  It is fully faithful by \cite[Theorem 4.5.1]{schwede:global},
  and its right adjoint is the forgetful functor
  $\res^\sm_{\{1\}}\colon \Sc p^\sm_{\gl}\to \Sc p^\sm$.
  Then the composite
  \[ \Sc p^\sm\ \xra{\ i_{\gl}\ } \ \Sc p^\sm_{\gl}\ \xra{\res^\sm_G}\ \Sc p^\sm_G \]
  is a strong symmetric monoidal left adjoint;
  by virtue of $\Sc p^\sm$ being initial in $\CAlg(\Pr^L_{\st})$,
  the composite is uniquely monoidally equivalent
  to the inflation functor \eqref{eq:inflation_functor}:
  \begin{equation}\label{eq:i_gl2i_G}
    \res^\sm_G\circ i_{\gl}\ \sim\  i_G\ . 
  \end{equation}  
  A global spectrum $Y$ is {\em left induced} if it is in the essential
  image of the fully faithful functor $i_{\gl}$.
  In this case, the inflation homomorphism \eqref{eq:define_inflation}
  \[ p_G^* \ \colon \ \bar\pi_k(Y) = \bar\Phi^{\{1\}}_k(Y)\ \to\ \bar\Phi_k^G(Y) \]
  associated to the unique homomorphism $p_G\colon G\to\{1\}$
  is an isomorphism, see \cite[Example 4.5.10]{schwede:global}.
  In particular, for every non-equivariant spectrum $X$ in $\Sc p$,
  the inflation homomorphism
  \begin{equation}\label{eq:p_G_iso}
    p_G^* \ \colon \ \bar\pi_*(X)\ \iso_{\eqref{eq:i_gl2i_G}}\
    \bar\Phi^{\{1\}}_*(i_{\gl}(X))\ \to\ \bar\Phi_*^G(i_{\gl}(X))
   \ \iso_{\eqref{eq:i_gl2i_G}}\  \bar\Phi^G_*(i_G(X))   
  \end{equation}
  is a monoidal natural isomorphism between lax symmetric monoidal
  functors $\Sc p^\sm\to\GrAb^\tensor$. 
\end{construction}

\begin{prop}\label{prop:identify_Phis}
  There is a unique monoidal natural transformation
  \begin{equation}\label{eq:Phi_vs_Phi}
    \tau \ \colon \
    \bar\pi_*\circ \Phi^G \ \to\ \bar\Phi^G_*
  \end{equation}
of lax monoidal functors $\Sc p^\sm_G\to\GrAb^\tensor$  such that the composite
  \[  \bar\pi_* \ \xra[\iso]{\bar\pi_*\circ \eqref{eq:Phi_p_G}}\
    \bar\pi_*\circ \Phi^G \circ i_G\ \xra[\iso]{\tau\circ i_G}\
    \bar\Phi^G_*\circ i_G    \]
  equals the monoidal isomorphism $p_G^*$ from \eqref{eq:p_G_iso}.
  Moreover, $\tau$ is a monoidal natural isomorphism.
\end{prop}
\begin{proof}
  For every orthogonal $G$-spectrum $X$,
  the geometric fixed point functor $\Phi^G_k\colon\Sp^O_G\to\rAb$
  inverts the morphism $X\sm i\colon X\to X\sm\tilde E\Pc$, simply because
  the $G$-fixed point functor $(-)^G$ is strong monoidal for the smash product
  of based $G$-spaces, and the map $i^G\colon S^0=(S^0)^G\to (\tilde E\Pc)^G$ is a homeomorphism.
  So the descended functor $\bar\Phi^G_k\colon\Sc p_G\to\rAb$
  inverts geometric fixed point equivalences.
  The functor $\bar\pi_*\circ\Phi^G$ also 
  inverts geometric fixed point equivalences, so both $\bar\Phi^G_* $ and $\bar\pi_*\circ\Phi^G$
  are objects of the $\infty$-category $\Fun_{\lax}^{\tilde E\Pc}(\Sc p^\sm_G,\GrAb^\tensor)$.
  Theorem \ref{thm:Phi_mon_loc} shows that the functor
  \[
    \Fun_{\lax}(i_G,\GrAb^\tensor) \colon  \Fun_{\lax}^{\tilde E\Pc}(\Sc p^\sm_G,\GrAb^\tensor) \ \to \
    \Fun_{\lax}(\Sc p^\sm,\GrAb^\tensor) \]
  is an equivalence of $\infty$-categories, which proves the claim.
\end{proof}

An important consequence of the monoidal isomorphism \eqref{eq:Phi_vs_Phi}
is that for homotopy $G$-ring spectra $R$,
the two ways of assigning $\mZ$-graded geometric fixed point rings
are naturally isomorphic.
On the one hand, the composite
\[ \Phi^G(R)\sm\Phi^G(R)\ \xra{\ \sim \ } \ \Phi^G(R\sm R) \ \xra{\Phi^G(\mu)} \
\Phi^G(R)\]
turns a homotopy $G$-ring spectrum $R$ into a non-equivariant homotopy ring spectrum,
where the first map is the strong monoidal structure on the geometric fixed point functor.
This homotopy ring spectrum has a graded homotopy ring $\bar\pi_*(\Phi^G(R))$.
On the other hand, the lax symmetric monoidal structure
on the functor $\bar\Phi^G_*\colon \Sc p_G\to\GrAb$
descended from the point-set level pairing \eqref{eq:cdot_Phi^G}
provides the $\mZ$-graded abelian group $\bar\Phi^G_*(R)$ with a graded ring structure
via the composite
  \begin{equation}\label{eq:concrete_product}
    \bar\Phi_k^G(R)\times \bar\Phi_l^G(R)\ \xra{\ \times\ }\ \bar\Phi_{k+l}^G(R\sm R)\
    \xra{\bar\Phi_{k+l}^G(\mu)}  \ \bar\Phi_{k+l}^G(R)\ .  
  \end{equation}

\begin{cor}\label{cor:fix_ring_iso}
  Let $R$ be a homotopy $G$-ring spectrum.
  The value of the monoidal natural isomorphism \eqref{eq:Phi_vs_Phi}
  at the underlying $G$-spectrum of $R$ is an isomorphism of graded rings
  from $\bar\pi_*(\Phi^G(R))$ to $\bar\Phi^G_*(R)$.
\end{cor}

Our next result is a symmetric monoidal rectification result
for `global modules' over a commutative orthogonal ring spectrum $A$.
Theorem \ref{thm:monoidal_loc} says that
the symmetric monoidal model category of $A$-module spectra
presents the $\infty$-category
of modules over the commutative global ring spectrum represented by $A$.
This rectification will be used in the proof of Theorem \ref{thm:UR_monadic},
the bimonadicity of the functor $u_{\gl}\colon\UR_{\gl}\to\CAlg(\Sc p^\sm_{\gl})$.

A lax symmetric monoidal functor between symmetric monoidal 1-categories
$R\colon\Dc\to\Cc$ naturally refines to a lax symmetric monoidal functor
\[  \Mod(A|R) \ \colon \ \Mod^\tensor_A(\Dc)\ \to \ \Mod^\tensor_{R(A)}(\Cc) \ ,  \]
for every commutative algebra $A$ in $\Dc$, and assuming
that $\Dc$ and $\Cc$ have enough colimits to admit relative tensor products for modules.
The higher categorical analog of this lift will play an important role for us,
so we review the construction in some detail.

\begin{construction}\label{con:Mod(D)}
We let $\Dc$ be a symmetric monoidal $\infty$-category.
In \cite[Definition 3.3.3.8]{lurie:higher_algebra}, 
Lurie defines an $\infty$-operad $\Mod^\tensor(\Dc)$
that comes with a functor $\pi_\Dc\colon \Mod^\tensor(\Dc)\to \Dc\times \CAlg(\Dc)$;
Lurie uses the notation $\Mod^{\text{Comm}}(\Dc)^\tensor$
and works more generally over coherent $\infty$-operads;
we only need the case of the commutative operad.
Objects of $\Mod^\tensor(\Dc)$ can informally be thought of as pairs $(M,T)$
consisting of a commutative algebra $T$ in $\Dc$, and a left $T$-module $M$.
And morphisms are, informally speaking, pairs $(f,i)\colon (M,T)\to (M',T')$
consisting of a morphism $i\colon T\to T'$ of commutative algebras,
and a morphism of $T$-modules $f\colon M\to i^*(M')$.
Proposition 3.3.3.10 of \cite{lurie:higher_algebra} characterizes inert morphisms
in $\Mod^\tensor(\Dc)$: firstly, the image in $\rFin_*$ is inert; and
secondly, the classifying functor takes certain lifts to inert morphisms of $\Dc$.

For a commutative algebra $T$, Lurie then defines the $\infty$-operad $\Mod_T^\tensor(\Dc)$
as the fiber of $\pi_\Dc$ over $T$, i.e., by the pullback
\[ \xymatrix{ \Mod_T^\tensor(\Dc)\ar[r]\ar[d]_{u_T} & \Mod^\tensor(\Dc)\ar[d]^{\pi_\Dc}\\
\Dc\ar[r]_-{(-,T)} & \Dc\times \CAlg(\Dc),  } \]
see \cite[Definition 4.5.1.1]{lurie:higher_algebra};
Lurie uses the notation $\Mod_T(\Dc)^\tensor$.
The left vertical {\em forgetful functor} $u_T\colon \Mod_T^\tensor(\Dc)\to\Dc$ is a morphism
of $\infty$-operads, because $u_T$ preserves and detects inertness in the sense explained above.

Now we let $R\colon \Dc\to\Cc$ be a lax symmetric monoidal functor.
Postcomposition with $R$ is a functor
$\Mod^\tensor(R)\colon \Mod^\tensor(\Dc)\to \Mod^\tensor(\Cc)$
that participates in a commutative diagram:
\[ \xymatrix@C=25mm{
\Dc\ar[r]^-{(-,T)}\ar[d]_R &    \Dc\times \CAlg(\Dc) \ar[d]^{R\times\CAlg(R)}&
 \Mod^\tensor(\Dc)\ar[d]^{\Mod^\tensor(R)}\ar[l]_-{\pi_\Dc}\\
\Cc \ar[r]_-{(-,\CAlg(R)(T))}&  \Cc\times \CAlg(\Cc)   & \Mod^\tensor(\Cc)\ar[l]^-{\pi_\Cc}} \]
Passing to pullbacks yields a functor
\[ \Mod(T|R)\ \colon \ \Mod_T^\tensor(\Dc) \ \to \ \Mod_{\CAlg(R)(T)}^\tensor(\Cc) \]
that covers $R$ with respect to the forgetful functors.
Because $R$ preserves inert morphisms,
postcomposition with $R$ preserves the second condition that characterizes
inert morphisms in $\Mod^\tensor(\Dc)$.
So $\Mod^\tensor(R)$, and hence $\Mod(T|R)$, preserve inert morphisms,
and $\Mod(T|R)$ is a morphism of $\infty$-operads.

In \cite[Theorem 4.5.2.1]{lurie:higher_algebra}, Lurie shows that
if the underlying $\infty$-category $\Dc_{\td{1}}$
admits geometric realization of simplicial objects,
and if the tensor product functor $\tensor\colon\Dc_{\td{1}}\times\Dc_{\td{1}}\to\Dc_{\td{1}}$
preserves such geometric realizations in each variable separately,
then $\Mod_T^\tensor(\Dc)$ is in fact a symmetric monoidal $\infty$-category.
As a morphism of $\infty$-operads between symmetric monoidal $\infty$-categories,
the forgetful functor $u_T\colon \Mod_T^\tensor(\Dc)\to\Dc$ is then
lax symmetric monoidal.
Moreover $u_T$ is the right adjoint in a relative adjunction over $\rFin_*$.
The left adjoint $T\tensor-\colon\Dc\to\Mod_T^\tensor(\Dc)$,
the {\em free $T$-module functor}, is moreover strong symmetric monoidal.
To obtain this monoidal adjunction, we can appeal to \cite[Theorem 4.5.3.1]{lurie:higher_algebra}
for the unique morphism $\mone\to T$ in $\CAlg(\Dc)$
from an initial commutative algebra,
and then exploit \cite[Proposition 3.4.2.1]{lurie:higher_algebra}
to identify $\Mod_{\mone}^\tensor(\Dc)$ with $\Dc$.

If $\Cc$ and $\Dc$ are symmetric monoidal $\infty$-categories that both
satisfy the hypotheses of \cite[Theorem 4.5.2.1]{lurie:higher_algebra},
then for every lax symmetric monoidal functor $R\colon \Dc\to\Cc$, the functor
$\Mod(T|R)\colon\Mod_T^\tensor(\Dc)\allowbreak\to\Mod_{\CAlg(R)(T)}^\tensor(\Cc)$
is a morphism of $\infty$-operads between symmetric monoidal $\infty$-categories,
i.e., again a lax symmetric monoidal functor.
\end{construction}

In the next rectification
result we consider a commutative orthogonal ring spectrum $A$.
By \cite[Corollary 4.3.29]{schwede:global}, the global model
structure on the category of orthogonal spectra lifts to a 
model structure on the category $\Mod_A$ of left $A$-modules,
which is moreover a symmetric monoidal model structure
with respect to the relative smash product $\sm_A$ of $A$-modules.
We write $\Mod_{A,\cof}^\sm$ for the full symmetric monoidal subcategory of $\Mod_A^\sm$
spanned by the cofibrant $A$-modules.
When the $\Mod^\tensor$-construction is applied to the operadic nerve
of a symmetric monoidal 1-category with colimits,
then the relative tensor products are the classical coequalizers, so
Lurie's construction returns the operadic nerve of the 1-categorical module category.
In particular, $\Mod_A^\tensor(\Sp^{O,\sm})$ is equivalent
to the operadic nerve of the symmetric monoidal 1-category $(\Mod_A,\sm_A)$.
In the following we shall blur this distinction
and simply write $\Mod_A^\sm$ for the symmetric monoidal $\infty$-category
$\Mod^\tensor_A(\Sp^{O,\sm})$, and similarly for $\Mod_{A,\cof}^\sm$.

Commutative orthogonal ring spectra are precisely the
commutative algebras in the symmetric monoidal 1-category $\Sp^{O,\sm}$
of orthogonal spectra under smash product.
So for every commutative orthogonal ring spectrum $A$,
the lax symmetric monoidal localization functor
$\gamma^\sm_{\gl}\colon\Sp^{O,\sm}\to \Sc p^\sm_{\gl}$ induces
a lax symmetric monoidal functor on module categories
$\Mod(A|\gamma_{\gl}^\sm)\colon\Mod_A^\sm\to\Mod_{\gamma(A)}^\tensor(\Sc p_{\gl}^\sm)$,
where we abbreviated the commutative global ring spectrum represented by $A$ to 
\[ \gamma(A)\ = \ \CAlg(\gamma^\sm_{\gl})(A)\ .\]

\begin{theorem}\label{thm:monoidal_loc}
  For every commutative orthogonal ring spectrum $A$,
  the lax symmetric monoidal functor
  \[  \Mod(A|\gamma^\sm_{\gl})  \ \colon\     \Mod_A^\sm\ \to\
      \Mod^\tensor_{\gamma(A)}(\Sc p^\sm_{\gl})\]
    is a monoidal localization.
\end{theorem}
\begin{proof}
  We let
  \[ \gamma^\sm_A\ : \ \Mod^\sm_A \ \to \ \M od^\sm_A \]
  be a monoidal localization of the global model structure.
  We write $\M od_A=(\M od^\sm_A)_{\td{1}}$ for the underlying $\infty$-category.
  The following diagram of lax symmetric monoidal functors commutes:
   \[ \xymatrix@C=20mm{ \Mod^\sm_A =\Mod^\tensor_A(\Sp^{O,\sm})
       \ar[r]^-{\Mod(A|\gamma^\sm_{\gl})}\ar[d]_{u_A} &
      \Mod_{\gamma(A)}^\tensor(\Sc p^\sm_{\gl})\ar[d]^{u_{\gamma(A)}}\\
      \Sp^{O,\sm}\ar[r]_-{\gamma_{\gl}^\sm} & \Sc p^\sm_{\gl}
    } \]
  The composite $\gamma_{\gl}^\sm\circ u_A$ inverts global equivalences,
  and the functor $u_{\gamma(A)}$ is conservative.
  So the functor $\Mod(A|\gamma^\sm_{\gl})\colon\Mod_A^\sm \to
  \Mod_{\gamma(A)}^\tensor(\Sc p^\sm_{\gl})$ inverts global equivalences.
  The universal property of the monoidal localization provides a unique
  lax symmetric monoidal functor
  \[\sigma\ \colon\ \M od_A^\sm\ \to\ \Mod^\tensor_{\gamma(A)}(\Sc p_{\gl}^\sm)  \]
  equipped with a monoidal equivalence
  \begin{equation}\label{eq:descend Mod(A)}
    \sigma\circ \gamma^\sm_A \ \sim \ \Mod(A|\gamma^\sm_{\gl})\ .
  \end{equation}
  We define a lax symmetric monoidal functor by
  \[ \rho\ = \ u_{\gamma(A)}\circ \sigma\ \colon \ \M od^\sm_A\ \to \ \Sc p^\sm_{\gl}\ . \]
  The monoidal equivalence 
  \[    \rho\circ \gamma^\sm_A \ =\
      u_{\gamma(A)}\circ\sigma\circ \gamma^\sm_A \ \sim_{\eqref{eq:descend Mod(A)}} \
    u_{\gamma(A)}\circ \Mod(A|\gamma^\sm_{\gl})\ =\ \gamma_{\gl}^\sm\circ u_A \]
  witnesses the commutativity of the following square:
  \begin{equation}\begin{aligned}\label{eq:forget and F}
      \xymatrix@C=12mm{
      \Mod^\sm_A \ar[d]_{\gamma^\sm_A}\ar[r]^-{u_A}
      & \Sp^{O,\sm}\ar[d]^{\gamma^\sm_{\gl}}  \\
      \M od_A^\sm\ar[r]_-{\rho} & \Sc p_{\gl}^\sm }
    \end{aligned}
  \end{equation}
  So the lax symmetric monoidal functor $\rho$ `is' the monoidal descendant
  of the forgetful functor $u_A$.
    
  The forgetful functor $u_A\colon\Mod_A\to\Sp^O$
  is a right Quillen functor for the global model structures.
  Since the underlying functor $\rho_{\td{1}}\colon \M od_A\to \Sc p_{\gl}$ of $\rho$
  is descended from a right Quillen functor,
  it is a right adjoint, see \cite[Theorem 2.1]{mazelgee:adjunctions}.
  In particular, $\rho_{\td{1}}$ preserves limits.
  Also, because the functor $u_A$ detects and preserves global equivalences,
  its derived functor $\rho_{\td{1}}$ is conservative.
 
  Now we show that $\rho_{\td{1}}$ also preserves colimits.
  The global model structure on $\Mod_A$ is stable, hence so is
  the associated $\infty$-category $\M od_A$.
  Since the $\infty$-category $\Sc p_{\gl}$ is also stable,
  and since $\rho_{\td{1}}\colon \M od_A\to \Sc p_{\gl}$ preserves limits,
  it therefore also preserves finite colimits.
  Arbitrary coproducts (i.e., wedges) of orthogonal spectra, and hence also of $A$-modules,
  are homotopical for global equivalences by \cite[Corollary 3.1.37 (i)]{schwede:global}.
  So coproducts in $\Sp^O$ and in $\Mod_A$ represent coproducts
  in $\Sc p_{\gl}$ and $\M od_A$, respectively. Since the functor $u_A\colon\Mod_A\to\Sp^O$
  preserves arbitrary coproducts, so does the descended functor $\rho_{\td{1}}$.
  Since $\rho_{\td{1}}$ preserves finite colimits and arbitrary coproducts,
  it preserves arbitrary colimits.
  
  Lurie shows in \cite{lurie:higher_algebra} that
  the forgetful functor $u_{\gamma(A)}\colon \Mod_{\gamma(A)}(\Sc p^\sm_{\gl})\to\Sc p_{\gl}$
  is conservative and preserves limits and colimits,
  see his Corollary 4.2.3.2 for conservativity, 
  Corollary 4.2.3.3 for limits, and Corollary 4.2.3.5 for colimits,
  using the identification of $\Mod_{\gamma(A)}(\Sc p^\sm_{\gl})$
  with left modules $\LMod_{\gamma(A)}(\Sc p^\sm_{\gl})$ from Corollary 4.5.1.6.
  Since $\rho_{\td{1}}=u_{\gamma(A)}\circ\sigma_{\td{1}}$
  is conservative and preserves limits and colimits,
  we conclude that also the functor $\sigma_{\td{1}}\colon\M od_A\to\Mod_{\gamma(A)}(\Sc p^\sm_{\gl})$
  that underlies $\sigma$ is conservative and preserves limits and colimits.

  By \cite[Theorem 4.3.27 (ii)]{schwede:global},
  the functor $A\sm-\colon\Sp^{O,\sm}\to \Mod^\sm_A$ preserves global equivalences
  between cofibrant orthogonal spectra, and it sends cofibrant orthogonal spectra
  to cofibrant $A$-modules. So the restricted functor
  \[ (A\sm-)|_{\cof}\ \colon \ \Sp^{O,\sm}_{\cof}\ \to\ \Mod^\sm_{A,\cof} \]
  is strong symmetric monoidal and preserves global equivalences.
  The strong symmetric monoidal composite
  \[    \Sp^{O,\sm}_{\cof}\ \xra{(A\sm-)|_{\cof}}\ \Mod^\sm_{A,\cof} \
    \xra{\gamma^\sm_{A,\cof}}\ \M od^\sm_A\]
  thus inverts global equivalences. The modified universal property
  of Remark \ref{rk:monloc_addendum} (iii) provides a unique
  strong symmetric monoidal functor
  \[ L\ \colon\ \Sc p^\sm_{\gl}\ \to\ \M od^\sm_A \]
  equipped with a monoidal equivalence
  \begin{equation}\label{eq:descend L}
    L\circ \gamma^\sm_{\gl,\cof} \ \sim \ \gamma_A^\sm\circ (A\sm -)|_{\cof}\ .
  \end{equation}

  By \cite[Theorem 2.1]{mazelgee:adjunctions} applied to the Quillen pair $(A\sm-,u_A)$,
  their derived functors are adjoint.
  The derived functor of $A\sm-$ is $L_{\td1}$ by \eqref{eq:descend L},
  and that of $u_A$ is $\rho_{\td1}$ by \eqref{eq:forget and F},
  so $(L_{\td{1}},\rho_{\td{1}})$ is an adjoint functor pair.
  We claim that the $\infty$-category $\M od_A$ is generated under colimits by
  the image of the descended left adjoint $L_{\td{1}}\colon \Sc p_{\gl}\to \M od_A$.  
  As the localization of a model category, the $\infty$-category
  $\M od_A$ is cocomplete, see \cite[Theorem 2.5.9]{barnea-harpaz-horel}.
  We showed above that the right adjoint $\rho_{\td{1}}$ is conservative and
  preserves colimits. So $\rho_{\td{1}}$ is monadic by Lurie's
  Barr--Beck theorem \cite[Theorem 4.7.3.5]{lurie:higher_algebra}.
  Hence every object of $\M od_A$ can be presented as a colimit of a diagram
  all of whose values are in the essential image of $L_{\td{1}}$, see 
  \cite[Proposition 4.7.3.14]{lurie:higher_algebra}.
  In particular,  the $\infty$-category $\M od_A$ is generated under colimits by
  the objects in the image of $L_{\td{1}}\colon \Sc p_{\gl}\to \M od_A$.

  Now we will show that $\sigma$ is strong, and not only lax, symmetric monoidal.
  The {\em monoidal mate} of the commutative diagram of lax symmetric monoidal functors
  \[ \xymatrix{
      \Mod^\sm_A\ar[r]^-{u_A}\ar[d]_{\Mod(A|\gamma^\sm_{\gl})}&    \Sp^{O,\sm}\ar[d]^{\gamma^\sm_{\gl}}  \\
      \Mod^\tensor_{\gamma(A)}(\Sc p^\sm_{\gl})\ar[r]_-{u_{\gamma(A)}} &\Sc p^\sm_{\gl}   } \]
  is the composite monoidal transformation
  \begin{align}\label{eq:some_mon_tran}
    (\gamma(A)\sm -)\circ \gamma^\sm_{\gl}\ &\xra{  (\gamma(A)\sm -)\circ \gamma^\sm_{\gl}\circ\text{unit}}
    \
    (\gamma(A)\sm -)\circ \gamma^\sm_{\gl}\circ u_A \circ(A\sm -) \\
      &=\  (\gamma(A)\sm -)\circ u_{\gamma(A)} \circ \Mod(A|\gamma^\sm_{\gl})\circ(A\sm -) \nonumber\\
&\xra{\text{counit}\circ\Mod(A|\gamma^\sm_{\gl})\circ (A\sm -)} \  \Mod(A|\gamma^\sm_{\gl})\circ (A\sm -)  \ . \nonumber
  \end{align}
  The underlying morphism in $\Sc p^\sm_{\gl}$ of the value of  \eqref{eq:some_mon_tran} at
  an orthogonal spectrum $X$ is the monoidal structure morphism
  \[  \gamma_{\gl}(A)\sm \gamma_{\gl}(X)\ \to \ \gamma_{\gl}(A\sm X)\]
  of $\gamma^\sm_{\gl}\colon\Sp^{O,\sm}\to\Sc p^\sm_{\gl}$.
  This monoidal structure morphism is an equivalence on pairs of cofibrant objects,
  by Remark \ref{rk:monloc_addendum} (ii).
  Moreover, for cofibrant orthogonal spectra $X$, the functor $-\sm X$ is fully homotopical
  for global equivalences by \cite[Theorem 4.3.27 (i)]{schwede:global}.
  So by cofibrant replacement, the monoidal structure morphism is an equivalence
  already if only $X$ is cofibrant.
  So when restricted to cofibrant orthogonal spectra,
  the monoidal transformation \eqref{eq:some_mon_tran}
  is a monoidal equivalence, witnessing that the restricted mate diagram 
  \[ \xymatrix@C=15mm{
      \Mod^\sm_{A,\cof}\ar[d]_{\Mod(A|\gamma^\sm_{\gl})_{\cof}}
      &    \Sp^{O,\sm}_{\cof}\ar[d]^{\gamma^\sm_{\gl,\cof}}\ar[l]_-{(A\sm-)|_{\cof}}  \\
      \Mod^\tensor_{\gamma(A)}(\Sc p^\sm_{\gl}) &\Sc p^\sm_{\gl} \ar[l]^-{\gamma(A)\sm-}  } \]
  commutes.
  Here $\Mod(A|\gamma^\sm_{\gl})_{\cof}$ denotes the restriction
  of $\Mod(A|\gamma^\sm_{\gl})$ to cofibrant $A$-modules.
  Since the functors $\gamma^\sm_{\gl,\cof}$
  and $\gamma(A)\sm-$ are strong symmetric monoidal,
  so is their composite,
  and thus also the functor $\Mod(A|\gamma_{\gl}^\sm)_{\cof}\circ(A\sm-)|_{\cof}$.
  The monoidal equivalences
  \[ \sigma\circ L\circ \gamma_{\gl,\cof}^\sm\ \sim \
    \sigma\circ \gamma_{A,\cof}^\sm\circ(A\sm-)|_{\cof}\ \sim \
    \Mod(A|\gamma_{\gl}^\sm)_{\cof}\circ(A\sm-)|_{\cof}\]
  then show that the lax symmetric monoidal functor
  $ \sigma\circ L\circ \gamma_{\gl,\cof}^\sm\colon\Sp^{O,\sm}_{\cof}\to \Mod^\tensor_{\gamma(A)}(\Sc p^\sm_{\gl})$ is strong.
  The first monoidal equivalence is \eqref{eq:descend L} composed with $\sigma$;
  the second is \eqref{eq:descend Mod(A)}, restricted to cofibrant $A$-modules,
  and composed with $(A\sm-)|_{\cof}$.

  The restricted monoidal localization functor
  $\gamma^\sm_{\gl,\cof}\colon\Sp_{\cof}^{O,\sm}\to\Sc p^\sm_{\gl}$
  is strong and essentially surjective on objects, so it detects strongness
  as explained in Remark \ref{rk:monloc_addendum} (ii).
  Hence the lax symmetric monoidal functor
  $\sigma\circ L\colon\Sc p^\sm_{\gl}\to \Mod^\tensor_{\gamma(A)}(\Sc p^\sm_{\gl})$ is strong.
  The monoidal products of both $\M od^\sm_A$ and $\Mod^\tensor_{\gamma(A)}(\Sc p^\sm_{\gl})$
  preserve colimits in each variable.
  Indeed, in $\M od^\sm_A$ we may consider an object of the form $\gamma_A(M)$
  for a cofibrant $A$-module $M$. Then $\gamma_A(M)\tensor-$
  is a total left derived functor of the left Quillen functor $M\sm_A-$,
  and hence a left adjoint. And for $\Mod^\tensor_{\gamma(A)}(\Sc p^\sm_{\gl})$,
  this is \cite[Theorem 3.4.4.2]{lurie:higher_algebra} for $\Oc=\Cc omm$,
  saying that $\Mod^\tensor_{\gamma(A)}(\Sc p^\sm_{\gl})$ is presentably symmetric monoidal.
  Since $\sigma_{\td{1}}$ preserves colimits, and since the $\infty$-category $\M od_A$
  is generated under colimits by the image of the strong 
  symmetric monoidal left adjoint $L\colon \Sc p^\sm_{\gl}\to\M od^\sm_A$,
  we conclude that $\sigma\colon \M od^\sm_A\to \Mod^\tensor_{\gamma(A)}(\Sc p^\sm_{\gl})$
  is strong, as claimed.
  
  Now we show that $\sigma$ is an equivalence of underlying $\infty$-categories.
  To this end we apply \cite[Corollary 4.7.3.16]{lurie:higher_algebra},
  including its addendum, 
  to the functors $\sigma_{\td{1}}$, $u_{\gamma(A)}$
  and $\rho_{\td{1}}=u_{\gamma(A)}\circ\sigma_{\td{1}}$.
  We need to verify the hypotheses of Lurie's result.
  We argued above that $(L_{\td{1}},\rho_{\td{1}})$ is an adjoint functor pair.
  Also, $(\gamma(A)\sm-,u_{\gamma(A)})$ is an adjoint functor pair,
  so hypothesis (1) holds.
  The $\infty$-categories $\M od_A$ and $\Mod_{\gamma(A)}(\Sc p^\sm_{\gl})$
  admit all colimits, and both $\rho_{\td{1}}$ and $u_{\gamma(A)}$
  are conservative and preserve all colimits, verifying hypotheses (2), (3), (4)
  and the hypothesis of the addendum. The hypothesis that is most involved to
  verify is (5); in its equivalent form of 
  \cite[Remark 4.7.3.17]{lurie:higher_algebra}, it asks that the following
  composite natural transformation of endofunctors of $\Sc p_{\gl}$ is an equivalence:
  \begin{align*}
    u_{\gamma(A)}\circ (\gamma(A)\sm -)
    \xra{u_{\gamma(A)}\circ(\gamma(A)\sm-)\circ\text{unit}}
    \ &u_{\gamma(A)}\circ (\gamma(A)\sm- )\circ \rho_{\td{1}}\circ L_{\td{1}}\\
    =
    \quad &u_{\gamma(A)}\circ (\gamma(A)\sm- )\circ u_{\gamma(A)}\circ\sigma_{\td{1}}\circ L_{\td{1}}\\
    \xra{u_{\gamma(A)}\circ \text{counit}\circ\sigma_{\td{1}}\circ L_{\td{1}}}\
      &u_{\gamma(A)}\circ \sigma_{\td{1}}\circ L_{\td{1}} =
        \rho_{\td{1}}\circ L_{\td{1}}\ . 
  \end{align*}
  It suffices to show that this transformation
  becomes an equivalence after precomposition with
  the essentially surjective functor $\gamma_{\gl,\cof}\colon\Sp^O_{\cof}\to\Sc p_{\gl}$.
  Under the identifications \eqref{eq:forget and F} and \eqref{eq:descend L},
  the unit of the derived adjunction $(L_{\td1},\rho_{\td1})$
  corresponds to the image under $\gamma_{\gl}$
  of the 1-categorical adjunction unit $X\to u_A(A\sm X)$,
  see the proof of \cite[Theorem 2.1]{mazelgee:adjunctions}.
  So the composite with $\gamma_{\gl,\cof}$
  is homotopic to the following composite of equivalences:
  \begin{align*}
    u_{\gamma(A)}\circ (\gamma(A)\sm -)\circ\gamma_{\gl,\cof}  \
    _{\eqref{eq:some_mon_tran}}&\sim \
                                 u_{\gamma(A)}\circ \Mod(A|\gamma^\sm_{\gl})_{\td{1}}\circ(A\sm -)|_{\cof}  \\
                               &= \ \gamma_{\gl}\circ u_A\circ(A\sm -)|_{\cof}  \\
    _{\eqref{eq:forget and F}} &\sim \ \rho_{\td{1}}\circ\gamma_A\circ(A\sm -)|_{\cof}  \\
    _{\eqref{eq:descend L}}&\sim \ \rho_{\td{1}}\circ L_{\td{1}}\circ\gamma_{\gl,\cof} \
  \end{align*}
  This completes the verification of the hypothesis of 
  \cite[Corollary 4.7.3.16]{lurie:higher_algebra},
  which then shows that the functor $\sigma_{\td{1}}$ is an equivalence of $\infty$-categories.

We have now shown that $\sigma\colon\M od_A^\sm\to\Mod^\tensor_{\gamma(A)}(\Sc p_{\gl}^\sm)$
is an equivalence of symmetric monoidal $\infty$-categories.
So the functor $\Mod(A|\gamma^\sm_{\gl})$ coincides with the monoidal localization
$\gamma^\sm_A$ up to composition with an equivalence, by \eqref{eq:descend Mod(A)}.
So $\Mod(A|\gamma^\sm_{\gl})$ is itself a monoidal localization.
\end{proof}

In the remainder of this appendix we study the more
refined and highly structured version of commutative global ring spectra
that we call `ultra-commutative'.

\begin{defn}[Ultra-commutative ring spectra]\label{def:com_global_ring}
  The {\em $\infty$-category of ultra-commutative ring spectra}  $\UR_{\gl}$
  is the localization of the 1-category of commutative orthogonal ring spectra
  at the class of morphisms that are global equivalences of underlying orthogonal spectra.
  We denote the localization functor by
  \[ \gamma^{\uc}_{\gl}\ : \  \CAlg(\Sp^{O,\sm})\  \to \ \UR_{\gl} \ .\]
\end{defn}

We shall consistently abuse notation and denote a commutative orthogonal
ring spectrum and the ultra-commutative ring spectrum that it represents,
i.e., the image under the localization functor $\gamma^{\uc}_{\gl}$,
by the same symbol.

\begin{construction}
  We define another `forgetful' functor
  \begin{equation}\label{eq:ucom2CR_gl}
    u_{\gl}\ :\ \UR_{\gl}\ \to\  \CAlg(\Sc p_{\gl}^\sm) 
  \end{equation}
  from ultra-commutative ring spectra to commutative global ring spectra.
  The monoidal localization $\gamma^\sm_{\gl}:\Sp^{O,\sm}\to \Sc p_{\gl}^\sm$ induces a functor
  \[  \CAlg(\gamma_{\gl}^\sm) \ :\  \CAlg(\Sp^{O,\sm})\ \ \to \ \CAlg(\Sc p^\sm_{\gl}) \]
  from the 1-category of commutative orthogonal ring spectra to the
  $\infty$-category of commutative global ring spectra.
  Since the functor $\gamma_{\gl}=(\gamma^\sm_{\gl})_{\td{1}}$ inverts global equivalences
  and the functor $\ev_{\td{1}}\colon\CAlg(\Sc p_{\gl}^\sm)\to\Sc p_{\gl}$ is conservative,
  the functor  $\CAlg(\gamma_{\gl}^\sm)$ sends global equivalences
  of commutative orthogonal ring spectra to equivalences in $\CAlg(\Sc p^\sm_{\gl})$.
  So the universal property of localizations
  provides an essentially unique functor \eqref{eq:ucom2CR_gl}
  equipped with a natural equivalence
  \[  u_{\gl} \circ\gamma^{\uc}_{\gl}\ \sim\ \CAlg(\gamma^\sm_{\gl})\ .     \]
\end{construction}

The forgetful functor \eqref{eq:ucom2CR_gl} is {\em not} an equivalence;
rather, it exhibits the $\infty$-category of ultra-commutative ring spectra
as monadic and comonadic over the $\infty$-category of commutative global ring spectra,
see Theorem \ref{thm:UR_monadic} (iii) below.
This suggests thinking of an ultra-commutative ring spectrum
as extra structure on a commutative global ring spectrum.
The difference between ultra-commutative and
commutative global ring spectra can already be seen at the level of natural operations:
the 0th equivariant homotopy groups of a commutative global ring spectrum
form a global Green functor \cite[Definition 5.1.3]{schwede:global}.
For ultra-commutative ring spectra, this Green functor
comes with additional multiplicative operations that can equivalently be encoded as power operations
or norm maps; depending on the perspective, these operations enhance the global Green functor
to a global power functor \cite[Theorem 5.1.11]{schwede:global}, or a global Tambara functor,
see \cite[Remark 5.1.7]{schwede:global}.
At the level of homotopy operations,
the comonadicity is made explicit in \cite[Theorem 5.2.13]{schwede:global},
and the monadicity is sketched in \cite[Remark 5.2.20]{schwede:global}.

\begin{theorem}\label{thm:UR_monadic}\ 
  \begin{enumerate}[\em (i)]
  \item 
    The functor $\ev_{\td{1}}\circ u_{\gl}\colon \UR_{\gl}\to\Sc p_{\gl}$ is monadic
    and preserves sifted colimits. 
  \item  The $\infty$-category $\UR_{\gl}$ of ultra-commutative ring spectra
    is presentable.
  \item
    The functor $u_{\gl}\colon \UR_{\gl}\to\CAlg(\Sc p_{\gl}^\sm)$ is both monadic and comonadic.
  \end{enumerate}
\end{theorem}
\begin{proof}
  As the localization of a model structure established in \cite[Theorem 5.4.3]{schwede:global},
  the $\infty$-category $\UR_{\gl}$ is complete and cocomplete
  by \cite[Theorem 2.5.9]{barnea-harpaz-horel}.
  We shall show that the functor
  \[ \CAlg(\gamma^\sm_{\gl})\colon \CAlg(\Sp^{O,\sm})\to\CAlg(\Sc p^\sm_{\gl}) \]
  preserves two particular kinds of colimits:
  \begin{enumerate}[(1)]
  \item The functor $\CAlg(\gamma^\sm_{\gl})$ takes pushouts along cofibrations
    between cofibrant commutative orthogonal ring spectra to pushouts in $\CAlg(\Sc p^\sm_{\gl})$.
  \item The functor $\CAlg(\gamma^\sm_{\gl})$
    takes coproducts of cofibrant commutative orthogonal ring spectra to coproducts.
  \end{enumerate}
  Granted this for the moment, we now deduce from (1) and (2) the fact that
  \begin{enumerate}[(1)]\setcounter{enumi}{2}
  \item the functor $u_{\gl}\colon \UR_{\gl}\to\CAlg(\Sc p^\sm_{\gl})$ preserves all colimits.
  \end{enumerate}  
  To prove (3), we exploit that the cofibrations and weak equivalences
  of the global model category structure from \cite[Theorem 5.4.3]{schwede:global}
  make the category $\CAlg(\Sp^{O,\sm})$ of commutative orthogonal ring spectra
  into a {\em category with weak equivalences and cofibrations}
  in the sense of Cisinski \cite[Definition 7.4.12]{cisinski:book}
  that is moreover {\em homotopy cocomplete}
  in the sense of \cite[Definition 7.7.2]{cisinski:book}.
  In this language, property (1)
  and the fact that $\CAlg(\gamma^\sm_{\gl})$ preserves initial objects
  (a special case of (2))
  mean that the functor $\CAlg(\gamma^\sm_{\gl})$ is {\em right exact}
  in the sense of \cite[Definition 7.5.2]{cisinski:book}, with respect
  to the degenerate cofibration structure on the $\infty$-category
  $\CAlg(\Sc p^\sm_{\gl})$, i.e., where all morphisms are cofibrations,
  and the weak equivalences are the equivalences.
  And property (2) means that the functor $\CAlg(\gamma^\sm_{\gl})$ is {\em homotopy cocontinuous}
  in the sense of \cite[Definition 7.7.2]{cisinski:book}.
  Since the weak equivalences in the degenerate cofibration structure on $\CAlg(\Sc p^\sm_{\gl})$
  are already invertible, $\CAlg(\Sc p^\sm_{\gl})$ is its own localization,
  and the left derived functor of $\CAlg(\gamma^\sm_{\gl})$ is nothing but
  the factorization  $u_{\gl}\colon \UR_{\gl}\to\CAlg(\Sc p^\sm_{\gl})$
  through the localization of the source. 
  The dual (opposite) of \cite[Corollary 7.7.5]{cisinski:book} thus shows that $u_{\gl}$ preserves
  all colimits.

  Now we show properties (1) and (2). 
  For (1) we consider a commutative algebra $T$ in a symmetric monoidal $\infty$-category $\Dc$.
  As discussed in Construction \ref{con:Mod(D)},
  the symmetric monoidal $\infty$-category $\Mod_T^\tensor(\Dc)$
  comes with a lax symmetric monoidal forgetful functor
  \[ u_T\ \colon \ \Mod^\tensor_T(\Dc)\ \to \ \Dc\ .  \]
  The induced functor of commutative algebras
  \[ \CAlg(u_T)\ \colon \ \CAlg(\Mod_T^\tensor(\Dc))\ \to \ \CAlg(\Dc)  \]
  takes the initial object of the source to $T$; so it lifts to a functor
  \begin{equation}\label{eq:compare_slices}
    \CAlg(\Mod_T^\tensor(\Dc))\ \xra{\ \sim \ } \ T\backslash\CAlg(\Dc)\ ,  
  \end{equation}
  and this lift is an equivalence by \cite[Corollary 3.4.1.7]{lurie:higher_algebra}
  for $\mathcal O=\text{Comm}$.

  Next we consider a commutative orthogonal ring spectrum $A$.
  As before, we denote the commutative global ring spectrum it represents
  by $\gamma(A)=\CAlg(\gamma^\sm_{\gl})(A)$.
  By Theorem \ref{thm:monoidal_loc}, the lax symmetric monoidal functor
  $\Mod(A|\gamma^\sm_{\gl})$ is a monoidal localization.
  So the composite
  \[ \Mod_{\cof}(\gamma^\sm_{\gl})\ \colon\
    \Mod_{A,\cof}^\sm \xra{\text{incl}} \Mod_A^\sm \xra{\Mod(A|\gamma^\sm_{\gl})}
    \Mod_{\gamma(A)}^\tensor(\Sc p^\sm_{\gl})\]
  is strong symmetric monoidal,  see Remark \ref{rk:monloc_addendum} (ii).
  Since finite coproducts in $\infty$-categories of commutative
  algebras are given by the underlying monoidal product, the induced functor
  on $\infty$-categories of commutative algebras preserves finite coproducts.
  This functor features in the composite:
  \begin{equation} \label{eq:THEcomposite}
        \CAlg(\Mod_{A,\cof}^\sm) \ \xra{\CAlg(\Mod_{\cof}(\gamma^\sm_{\gl}))} \
    \CAlg(\Mod_{\gamma(A)}^\tensor(\Sc p^\sm_{\gl}))\ \xra[\sim]{\eqref{eq:compare_slices}}
    \gamma(A)\bs\CAlg(\Sc p^\sm_{\gl})\ .
      \end{equation}
  The first functor in this composite preserves finite coproducts,
  and the second is an equivalence as a special case of \eqref{eq:compare_slices}.
  So also the composite \eqref{eq:THEcomposite} preserves finite coproducts.
  
  Now we can wrap up and prove (1).
  We consider a pushout square in $\CAlg(\Sp^{O,\sm})$
  \begin{equation}
    \begin{aligned}\label{eq:1-po}
      \xymatrix{A\ar[r]^-j\ar[d]_-i & C\ar[d]^k\\ B\ar[r]_-h & D}   
    \end{aligned}
  \end{equation}
  in which all four commutative orthogonal ring spectra are cofibrant,
  and the morphism $i$ is a cofibration,
  in the global model structure of $\CAlg(\Sp^{O,\sm})$
  established in \cite[Theorem 5.4.3]{schwede:global}.
  By factoring $j$ as a composite of a cofibration followed by
  a weak equivalence in the global model structure of $\CAlg(\Sp^{O,\sm})$
  and appealing to the cube lemma we can assume without loss of generality
  that also the morphism $j$ is a cofibration.
  
  We write $l=h i=k j\colon A\to D$ for the composite.
  By \cite[Theorem 5.4.3 (ii a)]{schwede:global}, the underlying
  morphisms of $A$-modules of $i$, $j$ and $l$ are cofibrations
  in the global model structure on $A$-modules from 
  \cite[Corollary 4.3.29 (i)]{schwede:global}.
  Since $A$ is cofibrant as a module over itself,
  we conclude that all of $A$, $B$, $C$ and $D$ have cofibrant underlying $A$-modules.
  So we can---and will---view the pushout square \eqref{eq:1-po}
  as a pushout square in $\CAlg(\Mod_{A,\cof}^\sm)$.
  Since $A$ is initial in $\CAlg(\Mod_{A,\cof}^\sm)$,
  this means that the morphisms $h\colon B\to D$
  and $k\colon C\to D$ witness $l\colon A\to D$
  as the coproduct of $i\colon A\to B$ and
  $j\colon A\to C$ in the $\infty$-category $\CAlg(\Mod_{A,\cof}^\sm)$.
  Since the composite \eqref{eq:THEcomposite} preserves binary coproducts,
  we conclude that $\gamma(h)$ and $\gamma(k)$ express $\gamma(l)\colon \gamma(A)\to\gamma(D)$
  as the coproduct of $\gamma(i)\colon \gamma(A)\to\gamma(B)$ and
  $\gamma(j)\colon \gamma(A)\to\gamma(C)$
  in the slice $\infty$-category $\gamma(A)\backslash\CAlg(\Sc p^\sm_{\gl})$.
  This precisely means that $\gamma$ takes the pushout square \eqref{eq:1-po} to a pushout
  in the $\infty$-category $\CAlg(\Sc p^\sm_{\gl})$, and we have shown claim (1). 
  
  Now we prove claim (2) by exploiting that infinite coproducts
  are specific filtered colimits of finite coproducts.
  The initial object in both $\CAlg(\Sp^{O,\sm})$ and $\CAlg(\Sc p^\sm_{\gl})$ is given by
  the monoidal unit, the sphere spectrum, with its unique commutative algebra structure.
  So the functor $\CAlg(\gamma^\sm_{\gl})$ preserves initial objects.
  Since $\CAlg(\gamma^\sm_{\gl})$ preserves initial objects and
  pushouts along cofibrations between cofibrant objects,
  it also preserves all finite coproducts of cofibrant objects.
  
  Now we let $\{R_i\}_{i\in I}$ be an $I$-indexed family
  of cofibrant objects in the global model structure of $\CAlg(\Sp^{O,\sm})$.
  We write $P$ for the filtered poset, under inclusion, of finite subsets of $I$.
  We define a functor $\Phi\colon P\to \CAlg(\Sp^{O,\sm})$ by sending a finite subset $S$ of $I$
  to the coproduct $\bigwedge_{i\in S} R_i$, and an inclusion $S\subset T$
  to the canonical morphism $\bigwedge_S R_i \to \bigwedge_T R_i$.
  
  We claim that the underlying $\Sp^O$-valued functor of $\Phi$, i.e., the composite
  \[ P \ \xra{\ \Phi\ } \ \CAlg(\Sp^{O,\sm}) \ \xra{\ev_{\td{1}}} \ \Sp^O  \]
  is Reedy cofibrant, i.e., for every finite subset $S$ of $I$, the latching morphism
  \[ l_S\ : \ \colim_{C\subsetneq S}  \Phi(C)_{\td{1}}\ \to \ \Phi(S)_{\td{1}} \]
  is a flat cofibration of orthogonal spectra; here the colimit in the source of $l_S$ is
  formed over the poset of proper subsets of $S$.
  We prove this by induction over the cardinality of $S$. For $S=\emptyset$
  there are no proper subsets, so the latching morphism is the unique morphism from
  the initial orthogonal spectrum to $\Phi(\emptyset)_{\td{1}}=\mS$, and thus a flat cofibration.
  If $S$ is not empty, we choose an $i\in S$ and we set $S'=S\setminus\{i\}$.
  Because the monoidal product preserves colimits in each variable, 
  the latching object can be presented as the following pushout:
  \[ \xymatrix{\mS\sm (\colim_{C\subsetneq S'} \Phi(C)_{\td{1}})\ar[r]^-{\mS\sm l_{S'}}\ar[d]_{\text{unit}\sm \Id} &
      \mS\sm \Phi(S')_{\td{1}}\ar[d]\\
      R_i\sm (\colim_{C\subsetneq S'} \Phi(C)_{\td{1}})\ar[r]  & 
      \colim_{C\subsetneq S} \Phi(C)_{\td{1}}   } \]
  Under this interpretation, the latching morphism
  $l_S\colon\colim_{C\subsetneq S} \Phi(C)_{\td{1}}\to \Phi(S)_{\td{1}}=R_i\sm \Phi(S')_{\td{1}}$ becomes
  the pushout product of the unit morphism $\mS\to R_i$
  and the latching morphism $l_{S'}\colon\allowbreak\colim_{C\subsetneq S'} \Phi(C)_{\td{1}}\to\Phi(S')_{\td{1}}$.
  The former is a flat cofibration by \cite[Theorem 5.4.3 (ii c)]{schwede:global},
  and $l_{S'}$ is a flat cofibration by induction.
  So the pushout product property lets us conclude that $l_S$ is a flat cofibration.
  This concludes the proof that the underlying $\Sp^O$-valued functor of $\Phi$ is Reedy cofibrant.

  Now we let $\bar\Phi\colon P^\triangleright\to \CAlg(\Sp^{O,\sm})$
  be a colimit of $\Phi$. Then the value of $\bar\Phi$ at the cone point
  is in particular a coproduct in the 1-category $\CAlg(\Sp^{O,\sm})$
  of the family  $\{R_i\}_{i\in I}$.
  We contemplate the following commutative diagram of $\infty$-categories:
  \[ \xymatrix@C=10mm{ P^\triangleright \ar[r]^-{\bar\Phi}\ar@/_1pc/[dr]_-{\CAlg(\gamma^\sm_{\gl})\circ\bar\Phi} &
      \CAlg(\Sp^{O,\sm})\ar[r]^-{\ev_{\td{1}}} \ar[d]^{\CAlg(\gamma^\sm_{\gl})} &
      \Sp^O\ar[d]^{\gamma_{\gl}}\\
      &  \CAlg(\Sc p^{\sm}_{\gl})\ar[r]_-{\ev_{\td{1}}} &\Sc p_{\gl}
    } \]
  The 1-categorical forgetful functor $\ev_{\td{1}}\colon \CAlg(\Sp^{O,\sm})\to \Sp^O$
  preserves filtered colimits, so the composite $\ev_{\td{1}}\circ\bar\Phi$
  is a colimit in the 1-category of orthogonal spectra.
  As a colimit of a Reedy cofibrant diagram, it represents
  the colimit in the associated $\infty$-category, i.e., the composite 
  $\gamma_{\gl}\circ \ev_{\td{1}}\circ\bar\Phi$ is a colimit in the $\infty$-category $\Sc p_{\gl}$
  of global spectra.
  The functor $\ev_{\td{1}}\colon  \CAlg(\Sc p^\sm_{\gl})\to\Sc p_{\gl}$
  creates filtered colimits by \cite[Proposition 3.2.3.1]{lurie:higher_algebra}.
  So $\CAlg(\gamma^\sm_{\gl})\circ \bar\Phi$ is a colimit in $\CAlg(\Sc p^\sm_{\gl})$.
  We have already shown that $\CAlg(\gamma^\sm_{\gl})$ preserves finite coproducts,
  so for every finite subset $S$ of $I$, the object $\CAlg(\gamma^\sm_{\gl})(\Phi(S))$
  is a coproduct  in $\CAlg(\Sc p^\sm_{\gl})$ of $\{\CAlg(\gamma^\sm_{\gl})(R_i)\}_{i\in S}$.
  Hence the value of $\CAlg(\gamma^\sm_{\gl})\circ \bar\Phi$ at the cone point
  is the $P$-filtered colimit of the finite coproducts,
  and hence the  coproduct of the family $\{\CAlg(\gamma^\sm_{\gl})(R_i)\}_{i\in I}$.

  Now we show part (i).
  The weak equivalences in any model category are saturated,
  i.e., a morphism is a weak equivalence if and only if it becomes an equivalence in
  the localization.
  So the functor $\ev_{\td{1}}\circ u_{\gl}\colon \UR_{\gl}\to\Sc p_{\gl}$ is conservative.
  Since this composite is descended from a fully homotopical right Quillen functor
  of model categories, it has a left adjoint by \cite[Theorem 2.1]{mazelgee:adjunctions}.
  We showed above that $u_{\gl}\colon \UR_{\gl}\to\CAlg(\Sc p^\sm_{\gl})$
  preserves all colimits.
  The forgetful functor  $\ev_{\td{1}}\colon \CAlg(\Sc p^\sm_{\gl})\to\Sc p_{\gl}$
  preserves sifted colimits by \cite[Proposition 3.2.3.1]{lurie:higher_algebra}.
  So the composite $\ev_{\td{1}}\circ u_{\gl}\colon \UR_{\gl}\to\Sc p_{\gl}$ preserves sifted colimits.
  In particular, $\ev_{\td{1}}\circ u_{\gl}$ preserves geometric realization,
  i.e., colimits of simplicial objects.
  Since $\ev_{\td{1}}\circ u_{\gl}$ is conservative, has a left adjoint
  and preserves geometric realization, it is monadic by the Lurie--Barr--Beck theorem
  \cite[Theorem 4.7.3.5]{lurie:higher_algebra}.

  (ii) We exploit that the $\infty$-category $\Sc p_{\gl}$
  of global spectra is compactly generated,
  hence $\kappa$-presentable, for some regular cardinal $\kappa$.
  The functor $\ev_{\td{1}}\circ u_{\gl}\colon \UR_{\gl}\to\Sc p_{\gl}$ is monadic by (i).
  So there is a monad $\M$ on the presentable $\infty$-category
  $\Sc p_{\gl}$ such that $\UR_{\gl}$ is equivalent to $\Alg_\M(\Sc p_{\gl})$.
  Since the functor $\ev_{\td{1}}\circ u_{\gl}\colon \UR_{\gl}\to\Sc p_{\gl}$ preserves
  sifted colimits, the underlying endofunctor of $\Sc p_{\gl}$
  of the monad $\M$ preserves sifted colimits.
  So the $\infty$-category of algebras over $\M$,
  and hence also $\UR_{\gl}$,  is $\kappa$-presentable
  by  \cite[Corollary 6.8]{henry-meadows:monads}.

  (iii) The functor $\ev_{\td{1}}\circ u_{\gl}\colon \UR_{\gl}\to\Sc p_{\gl}$ is conservative,
  hence so is the functor $u_{\gl}\colon \UR_{\gl}\allowbreak\to \CAlg(\Sc p^\sm_{\gl})$.
  We argued above that the functor $u_{\gl}$ preserves all colimits;
  so it is in particular accessible.
  The functor $\ev_{\td{1}}\circ u_{\gl}\colon \UR_{\gl}\to\Sc p_{\gl}$ has a left adjoint,
  so it preserves all limits. Since 
  $\ev_{\td{1}}\colon \CAlg(\Sc p^\sm_{\gl})\to\Sc p_{\gl}$
  creates all limits by \cite[Corollary 3.2.2.5]{lurie:higher_algebra},
  we deduce that the functor $u_{\gl}$ preserves all limits.
  The $\infty$-category $\CAlg(\Sc p^\sm_{\gl})$ is
  presentable by \cite[Corollary 3.2.3.5]{lurie:higher_algebra}.
  Since $\UR_{\gl}$ is also presentable by (ii), and $u_{\gl}$ preserves all limits and colimits,
  it has both a left adjoint and a right adjoint
  by Lurie's adjoint functor theorem \cite[Corollary 5.5.2.9]{lurie:htt}.

  Since $u_{\gl}$ is conservative, has a left adjoint,
  and preserves all colimits, $u_{\gl}$ is monadic
  by the Lurie--Barr--Beck theorem \cite[Theorem 4.7.3.5]{lurie:higher_algebra}.
  Since the functor $u_{\gl}^{\op}\colon (\UR_{\gl})^{\op}\to(\CAlg(\Sc p^\sm_{\gl}))^{\op}$
  is conservative, has a left adjoint and preserves all colimits,
  $u_{\gl}^{\op}$ is monadic by the same theorem; thus $u_{\gl}$ is comonadic.
\end{proof}
    
\section{Global classes}\label{app:global_classes}

A {\em global class} is a non-empty class of compact Lie groups that is closed under isomorphism
and passage to closed subgroups.
Every global family in the sense of \cite[Definition 1.4.1]{schwede:global}
is a global class, but global families are additionally assumed to be closed
under quotient groups. The closure property under quotient groups
was used in \cite[Section 4.3]{schwede:global} to establish the $\Fc$-global
stable model structures.
In this appendix, we make no use of these model structures, and we develop
the theory from a higher categorical perspective for the more general global classes,
with particular emphasis on multiplicative features.

\begin{eg}
  An example of a global class that is not a global family is
  the class of compact Lie groups
  isomorphic to a closed subgroup of $U(2)^k$ for some $k\geq 1$;
  this class is even multiplicative, i.e., closed under products.
  It contains the binary tetrahedral group $\tilde A_4$ of order 24,
  because $\tilde A_4$ has a faithful 2-dimensional unitary representation.
  But this class does not contain the alternating group on four letters  $A_4$,
  which is a quotient of $\tilde A_4$.
  Indeed, $A_4$ has no 2-dimensional
  irreducible representations, so every homomorphism $A_4\to U(2)^k$ kills
  the commutator subgroup of $A_4$, and can thus not be injective.
\end{eg}

Let $\Fc$ be a global class of compact Lie groups.
A morphism of orthogonal spectra $f\colon X\to Y$ is an {\em $\Fc$-equivalence}
if the map $\pi_k^G(f)\colon \pi_k^G(X)\to\pi_k^G(Y)$ is an isomorphism for all $k\in\mZ$
and all compact Lie groups $G$ in the class $\Fc$.
Since every global equivalence is in particular an $\Fc$-equivalence,
the $\Fc$-equivalences descend to a well-defined class of morphisms
in the $\infty$-category $\Sc p_{\gl}$ of global spectra.
Then by design, a morphism $f$ in $\Sc p_{\gl}$ is an $\Fc$-equivalence
if and only if for every compact Lie group $G$ in $\Fc$,
the morphism of genuine $G$-spectra $\res_G f$ is an equivalence in $\Sc p_G$.

In the following we write
\[ \Sc p_\Fc \ = \ \td{\Sigma^\infty_+B_{\gl}G\colon G\in\Fc} \]
for the localizing subcategory
of the stable $\infty$-category $\Sc p_{\gl}$
generated by the global spectra $\Sigma^\infty_+B_{\gl}G$ for all $G\in\Fc$.
We show in part (v) of the following theorem that $\Sc p_\Fc$ coincides
with the essential image of the left adjoint
to the Bousfield localization of $\Sc p_{\gl}$ by the class of $\Fc$-equivalences.
So a global spectrum belongs to $\Sc p_\Fc$ if and only if
it is left induced from $\Fc$ in the sense of \cite[Definition 4.5.6]{schwede:global}.
We will therefore refer to the global spectra in $\Sc p_\Fc$ as {\em $\Fc$-left induced}.

We call a functor $f\colon \Cc\to\Dc$ of $\infty$-categories
a {\em Bousfield localization} if it has a fully faithful right adjoint.
Dually, $f$ is a {\em Bousfield colocalization} if it has a fully faithful left adjoint.
An  exact functor of stable $\infty$-categories is a {\em recollement}
if it is simultaneously a Bousfield localization and a Bousfield colocalization,
i.e., it has fully faithful adjoints on both sides.

The most important features of the $\infty$-category of
$\Fc$-left induced spectra are summarized in the following theorem.
For {\em multiplicative} global classes, the subcategory $\Sc p_\Fc$
of $\Fc$-left induced global spectra interacts well with the global smash products,
and Theorem \ref{thm:multiplicative_class} records the additional multiplicative
features in this situation.
When the multiplicative global class is additionally closed under wreath products,
then $\Sc p_\Fc$ also behaves well with respect to ultra-commutative multiplications,
as we explain in Theorem \ref{thm:wreath_closed}.

\begin{theorem}\label{thm:Sp_F properties}
  Let $\Fc$ be a global class of compact Lie groups.
  \begin{enumerate}[\em (i)]
  \item The $\infty$-category $\Sc p_\Fc$ 
    is stable and compactly generated, and hence presentable.
  \item Every $\Fc$-equivalence between $\Fc$-left induced global spectra
    is a global equivalence.
  \item The inclusion $\Sc p_\Fc\to\Sc p_{\gl}$ has a right adjoint
    $\rho_\Fc\colon\Sc p_{\gl}\to \Sc p_\Fc$, and for every global spectrum $X$,
    the adjunction counit $\epsilon_X\colon \rho_\Fc(X)\to X$ is an $\Fc$-equivalence.
  \item Any right adjoint $\rho_\Fc\colon\Sc p_{\gl}\to \Sc p_\Fc$
    to the inclusion itself has a fully faithful right adjoint.
    So $\rho_\Fc$ is a recollement of stable $\infty$-categories.
  \item
    The right adjoint $\rho_\Fc\colon \Sc p_{\gl}\to \Sc p_\Fc$ to the inclusion
    is a localization of $\infty$-categories at the class of $\Fc$-equivalences.
  \item For every $\Fc$-equivalence $f\colon X\to Y$
    and every global spectrum $Z$, the morphism $f\sm Z\colon X\sm Z \to Y\sm Z$
    is an $\Fc$-equivalence.
  \end{enumerate}
\end{theorem}
\begin{proof}
  We prove (i) and (ii) together.
  By design, $\Sc p_\Fc$ is the full subcategory of $\Sc p_{\gl}$
  generated under colimits by the objects $(\Sigma^\infty_+B_{\gl}G)[k]$
  for all $G\in\Fc$ and all $k\in\mZ$.
  All these generating global spectra are compact objects of $\Sc p_{\gl}$
  by \cite[Theorem 4.4.3 (ii)]{schwede:global},
  so $\Sc p_\Fc$ is a compactly generated stable $\infty$-category,
  and hence presentable.
  Since the objects $(\Sigma^\infty_+B_{\gl}G)[k]$
  are compact generators of $\Sc p_\Fc$ and represent the functor $\pi_k^G$,
  by \cite[Theorem 4.4.3 (i)]{schwede:global}, every $\Fc$-equivalence between
  objects in $\Sc p_\Fc$ is an equivalence in $\Sc p_{\gl}$, i.e., a global equivalence.
  
  (iii)
  By design, $\Sc p_\Fc$ is closed under colimits inside $\Sc p_{\gl}$.
  So the inclusion $\Sc p_{\Fc}\to\Sc p_{\gl}$ preserves colimits.
  Since source and target are presentable, the adjoint functor theorem
  \cite[Corollary 5.5.2.9]{lurie:htt} provides a right adjoint
  $\rho_\Fc\colon\Sc p_{\gl}\to\Sc p_\Fc$ to the inclusion.
  We write $\epsilon_X\colon \rho_\Fc(X)\to X$ for the adjunction counit.
  For all $G\in\Fc$ and all integers $k$,
  the upper horizontal map in the following commutative diagram
  is the adjunction bijection:
  \[ \xymatrix@C=12mm{
      h(\Sc p_{\gl})((\Sigma^\infty_+B_{\gl}G)[k],\rho_\Fc(X))\ar[r]^-{(\epsilon_X)_*}_-\iso\ar[d]_\iso
      & h(\Sc p_{\gl})((\Sigma^\infty_+B_{\gl}G)[k],X)  \ar[d]^\iso\\
      \pi_k^G(\rho_\Fc(X))\ar[r]_-{(\epsilon_X)_*}&   \pi_k^G(X)
    }
  \]
  The two vertical maps are isomorphisms because $(\Sigma^\infty_+B_{\gl}G)[k]$
  represents the functor $\pi_k^G$, compare \cite[Theorem 4.4.3 (i)]{schwede:global}.
  So the lower horizontal map is an isomorphism, and thus
  $\epsilon_X$ is an $\Fc$-equivalence, as claimed.
 
  (iv)
  As a right adjoint, $\rho_\Fc$ preserves limits.
  Since source and target are stable, the functor $\rho_\Fc$ preserves finite colimits.
  We claim that $\rho_\Fc$ also preserves infinite sums.
  Indeed, for any family $(X_i)_{i\in I}$
  of global spectra and all $G$ in $\Fc$, the following diagram commutes:
  \[ \xymatrix{
      \bigoplus_I \pi_k^G(\rho_\Fc(X_i))\ar[r]^-\iso\ar[d]_\iso &
      \pi_k^G(\bigoplus_I \rho_\Fc(X_i))\ar[r] &
      \pi_k^G(\rho_\Fc(\bigoplus_I X_i))\ar[d]^\iso\\
      \bigoplus_I \pi_k^G(X_i)\ar[rr]_-\iso&&
      \pi_k^G(\bigoplus_I  X_i)    }
  \]
  The two vertical maps are isomorphisms by (iii).
  The two horizontal isomorphisms exploit that the functor $\pi_k^G$
  takes coproducts of global spectra to direct sums of abelian groups.
  Since this holds for all $G\in\Fc$, the colimit exchange morphism
  $\bigoplus_I \rho_\Fc(X_i)\to\rho_\Fc(\bigoplus_I X_i)$
  is an $\Fc$-equivalence between objects of $\Sc p_\Fc$,
  and hence a global equivalence by (ii).
  This concludes the proof that $\rho_\Fc$ preserves arbitrary sums.
  Since $\rho_\Fc$ preserves finite colimits and sums, it preserves all colimits.
  Since $\Sc p_\Fc$ and $\Sc p_{\gl}$ are presentable, the adjoint functor theorem
  \cite[Corollary 5.5.2.9]{lurie:htt} provides a right adjoint to $\rho_\Fc$.
  The functor $\rho_\Fc$ thus has adjoints on both sides. In such cases,
  one adjoint is fully faithful if and only if the other one is.
  Since the left adjoint to $\rho_\Fc$ is fully faithful, so is its right adjoint.
  
  (v) 
  The commutative square
  \[ \xymatrix{ \rho_\Fc(X)\ar[r]^-{\rho_\Fc(f)}\ar[d]_{\epsilon_X} &\rho_\Fc(Y)\ar[d]^{\epsilon_Y}\\
      X\ar[r]_-f& Y    } \]
  and the fact, proved in (iii), that the adjunction counits are $\Fc$-equivalences,
  show that a morphism $f$ of global spectra is an $\Fc$-equivalence if and only if 
  $\rho_\Fc(f)$ is an $\Fc$-equivalence.
  Since $\rho_\Fc(X)$ and $\rho_\Fc(Y)$ belong to $\Sc p_\Fc$, 
  $\rho_\Fc(f)$ is moreover an $\Fc$-equivalence if and only if it is a global equivalence,
  by (ii).
  We can then apply the dual of \cite[Proposition 5.2.7.12]{lurie:htt}
  to the functor $L=\text{incl}\circ \rho_\Fc\colon\Sc p_{\gl}\to\Sc p_{\gl}$, which yields
  the desired universal property of a localization.

  (vi)
  For any group $G$ in $\Fc$,
  the functor $\res_G^\sm\colon\Sc p^\sm_{\gl}\to \Sc p^\sm_G$ is strong symmetric monoidal,
  so $\res_G(f\sm Z)$ is equivalent in $\Sc p_G$ to $\res_G(f)\sm \res_G(Z)$.
  The morphism $\res_G(f)$ is an equivalence in $\Sc p_G$ because $G\in \Fc$;
  hence also $\res_G(f\sm Z)$ is an equivalence. Since this holds for all $G\in\Fc$,
  the morphism $f\sm Z$ is an $\Fc$-equivalence.
\end{proof}

\begin{eg}\label{eg:Thom_left_induced}
  Let $G$ be a compact Lie group in a global class $\Fc$.
  Then the functor $\res_G\colon\Sc p_{\gl}\to \Sc p_G$
  has a left adjoint
  \[ G\dbs    - \ : \ \Sc p_G\ \to \ \Sc p_{\gl}\ . \]
  This left adjoint takes the suspension $G$-spectrum $\Sigma^\infty_+ G/H$
  to $\Sigma^\infty_+ B_{\gl}H$, for example because both represent the functor
  $\pi_0^H\colon h(\Sc p_{\gl})\to\rAb$.
  Since $G$ belongs to the class $\Fc$, so does its subgroup $H$,
  and thus $G\dbs(\Sigma^\infty_+ G/H)\sim\Sigma^\infty_+ B_{\gl}H$
  lies in the subcategory $\Sc p_\Fc$.
  Since $\Sc p_G$ is generated, as a localizing subcategory,
  by the objects $\Sigma^\infty_+ G/H$, the entire image of the left adjoint $G\dbs-$
  is contained in the subcategory $\Sc p_\Fc$.
  For example, if $V$ and $W$ are representations of a group $G\in\Fc$, then the
  global Thom spectrum
  \[ (B_{\gl}G)^{V-W}\ = \ G\dbs(\mS^{V-W}) \]
  of a virtual $G$-representation is $\Fc$-left induced.
\end{eg}

We alert the reader that in general, the $\infty$-category $\Sc p_\Fc$ of $\Fc$-left induced global
spectra is {\em not} closed inside $\Sc p_{\gl}$ under the global smash product.
This improves if $\Fc$ is a {\em multiplicative} class,
i.e., closed under products, see Theorem \ref{thm:multiplicative_class} (i).

\begin{prop}\label{prop:mult=refl}
  For any global class $\Fc$, the following two conditions are equivalent.
  \begin{enumerate}[\em (a)]
  \item The global class $\Fc$ is {\em multiplicative}, i.e., 
   for all compact Lie groups $G$ and $K$ in $\Fc$, the product $G\times K$
   also belongs to $\Fc$.
  \item The global class $\Fc$ is {\em reflexive},
  i.e., the inclusion of the full subcategory spanned by groups in $\Fc$
  into the category of compact Lie groups and continuous homomorphisms has a left adjoint.
  \end{enumerate}
\end{prop}
\begin{proof}
  (a)$\Longrightarrow$(b)
  We let $K$ be a compact Lie group.
  We let $N$ be the intersection of the kernels of all
  continuous homomorphisms from $K$ to some group in $\Fc$.
  Then $N$ is a closed normal subgroup of $K$, and we claim
  that the quotient homomorphism $K\to K/N$
  is initial among continuous homomorphisms from $K$ to groups in $\Fc$.

  We must first show that the quotient group $K/N$ indeed belongs to $\Fc$.
  Since compact Lie groups satisfy the descending chain condition for closed
  subgroups, there are finitely many $N_1,\dots,N_k$ among the relevant closed normal subgroups
  such that $N=N_1\cap \ldots\cap N_k$.
  We choose continuous homomorphisms $\alpha_i\colon K\to G_i$
  with kernel $N_i$ and $G_i\in\Fc$. Then the homomorphism
  \[ (\alpha_1,\dots,\alpha_k)\ : \ K \ \to \ G_1\times \dots\times G_k \]
  has kernel $N_1\cap\dots\cap N_k=N$, so $K/N$ embeds into 
  $G_1\times \dots\times G_k$. Since the class is multiplicative,
  the product $G_1\times \dots\times G_k$ belongs to $\Fc$, hence so does $K/N$.
  
  Now we can complete the argument. If $\beta\colon K\to G$ is any
  continuous homomorphism to a group in $\Fc$, then $N$ is contained
  in the kernel of $\beta$, so $\beta$ factors uniquely through
  the quotient map $K\to K/N$.
  
  (b)$\Longrightarrow$(a)  
  The argument for global families in \cite[page 437f]{schwede:global}
  can be adapted to global classes, as follows.
  For $G, K\in\Fc$, we let $q\colon G\times K\to L$ be an initial continuous
  homomorphism to a group $L$ that belongs to $\Fc$.
  Then the projections $p_G\colon G\times K\to G$ and $p_K\colon G\times K\to  K$
  factor as $p_G=i \circ q$ and $p_K=j\circ q$ for unique continuous homomorphisms
  $i\colon L\to G$ and $j\colon L\to K$.
  The composite
  \[ G\times K \ \xra{\ q\ }\ L \ \xra{(i,j)}\ G\times K \]
  is the identity, so $q$ is an isomorphism from $G\times K$ onto a closed
  subgroup of $L$. Since $L$ belongs to $\Fc$, so does $G\times K$.
\end{proof}

The following proposition generalizes \cite[Proposition 4.5.8]{schwede:global}
from global families to global classes.

\begin{prop}\label{prop:reflexive}
  Let $\Fc$ be a multiplicative, and hence reflexive, global class.
  \begin{enumerate}[\em (i)]
  \item For every compact Lie group $K$,
    every initial continuous homomorphism $p\colon K\to u K$
    to a group in $\Fc$ is surjective.
  \item A global spectrum $X$ is $\Fc$-left induced if and only
    if for every compact Lie group $K$
    and every initial continuous homomorphism $p\colon K\to u K$
    to a group in $\Fc$, the inflation map
    \[ p^*\ \colon \  \Phi^{u K}_*(X)\ \to \ \Phi^K_*(X) \]
    of geometric fixed point homotopy groups is an isomorphism.
  \end{enumerate}
\end{prop}
\begin{proof}
  (i) The image $p(K)$ of $p\colon K\to u K$ is a closed subgroup of $u K$,
  hence also in the global class $\Fc$. Because $p$ is initial among homomorphisms
  to groups in $\Fc$, there is a unique continuous homomorphism $j\colon u K\to p(K)$
  such that the composite $j\circ p\colon K\to p(K)$
  agrees with $p$ (with restricted codomain).
  Since the composite
  \[ K \ \xra{\ p \ }\ u K \ \xra{\ j \ }\ p(K) \ \xra{\text{incl}}\ u K\]
  equals $p$, initiality shows that $\text{incl}\circ j$ is the identity of $u K$.
  So $p(K)=u K$, and thus $p$ is surjective.
  Part (ii) is proved in the same way as for global families in
  \cite[Proposition 4.5.8]{schwede:global}.
\end{proof}

As we shall now show, for multiplicative global classes,
the $\Fc$-left induced global spectra are monoidally well-behaved.

\begin{theorem}\label{thm:multiplicative_class}
  Let $\Fc$ be a multiplicative global class.
   \begin{enumerate}[\em (i)]
  \item The full subcategory $\Sc p_{\Fc}$ of $\Fc$-left induced global spectra
    is closed in $\Sc p_{\gl}$ under the global smash product.
  \item
    There is a unique symmetric monoidal structure $\Sc p_{\Fc}^\sm$
    on the $\infty$-category $\Sc p_{\Fc}$ equipped with a lift of the
    inclusion to a strong symmetric monoidal functor
    $\text{\em incl}^\sm\colon \Sc p_{\Fc}^\sm\to\Sc p_{\gl}^\sm$.
  \item The functor
    \[  \CAlg(\text{\em incl}^\sm) \ : \ \CAlg(\Sc p_{\Fc}^\sm) \ \to
      \ \CAlg(\Sc p_{\gl}^\sm)  \]
    is fully faithful, and its essential image
    is the subcategory of $\CAlg(\Sc p_{\gl}^\sm)$
    spanned by those commutative global ring spectra whose
    underlying global spectra are $\Fc$-left induced.
  \item The lax symmetric monoidal structure $\rho_\Fc^\sm\colon\Sc p_{\gl}^\sm\to\Sc p_\Fc^\sm$
    induced by $\text{\em incl}^\sm$ on the right adjoint $\rho_\Fc$ of the 
    inclusion is strong symmetric monoidal.
  \item The monoidal structure $\Sc p^\sm_\Fc$ is
    presentably symmetric monoidal,
    and the $\infty$-category $\CAlg(\Sc p_{\Fc}^\sm)$ is presentable.
  \item The pullback square of $\infty$-categories
    \[ 
      \xymatrix@C=15mm{
        \CAlg(\Sc p_{\Fc}^\sm)\ar[d]_{\ev_{\td{1}}}\ar[r]^-{\CAlg(\text{\em incl}^\sm)}&
        \CAlg(\Sc p_{\gl}^\sm)\ar[d]^{\ev_{\td{1}}}\\
        \Sc p_{\Fc}\ar[r]_-{\text{\em incl}} & \Sc p_{\gl}}
    \]
    is vertically left adjointable and horizontally right adjointable.
\end{enumerate}
\end{theorem}
\begin{proof}
  (i)
  We let $G$ and $K$ be two compact Lie groups from the class $\Fc$.
  Then $(\Sigma^\infty_+ B_{\gl}G)\sm (\Sigma^\infty_+ B_{\gl}K)$ is globally
  equivalent to $\Sigma^\infty_+ B_{\gl}(G\times K)$,
  by (1.3.5) and Proposition 4.1.18 of \cite{schwede:global},
  and hence in $\Sc p_{\Fc}$ by the multiplicativity hypothesis on $\Fc$.
  Since $\Sc p_{\Fc}$ is the localizing subcategory of $\Sc p_{\gl}$
  generated by the suspension spectra $\Sigma^\infty_+ B_{\gl}G$ for all
  $G\in \Fc$, and since the global smash product preserves colimits
  in each variable separately, this proves that $\Sc p_{\Fc}\sm \Sc p_{\Fc}$
  is contained in $\Sc p_{\Fc}$.

  (ii) and (iii): 
  Because the subcategory $\Sc p_{\Fc}$ is closed under the global smash product
  by (i), and it contains the monoidal unit $\mS\sim\Sigma^\infty_+ B_{\gl}\{1\}$,
  Proposition 2.2.1.1 and Remark 2.2.1.2 of \cite{lurie:higher_algebra}
  provide the `restricted' symmetric monoidal structure on $\Sc p_{\Fc}$
  along with an extension of the inclusion to a strong symmetric monoidal functor.
  The induced functor
  \[ \CAlg(\text{incl}^\sm)\ : \  \CAlg(\Sc p_{\Fc}^\sm) \ \to\ \CAlg(\Sc p_{\gl}^\sm) \]
  is then fully faithful with the essential image as described in (iii).

  (iv)
  The global sphere spectrum $\mS$ is $\Fc$-left induced for any global class $\Fc$.
  So the adjunction counit $\epsilon_\mS\colon\rho_\Fc(\mS)\to\mS$ is a global equivalence.
  This adjunction counit is the inverse of the unit morphism of the monoidal structure
  on $\rho_\Fc$, so the lax monoidal structure is unital.
  
  For global spectra $X$ and $Y$, the monoidal structure morphism
  $\rho_\Fc(X)\sm\rho_\Fc(Y)\to\rho_\Fc(X\sm Y)$ participates in the following commutative
  diagram in $\Sc p_{\gl}$:
  \[ \xymatrix{\rho_\Fc(X)\sm\rho_\Fc(Y)\ar[r]\ar[d]_{\rho_\Fc(X)\sm \epsilon_Y}&\rho_\Fc(X\sm Y)\ar[d]^{\epsilon_{X\sm Y}}\\
      \rho_\Fc(X)\sm Y\ar[r]_-{\epsilon_X\sm Y} & X\sm Y
    } \]
  Here $\epsilon\colon\text{incl}\circ\rho_\Fc\to \Id_{\Sc p_{\gl}}$ is the adjunction counit.
  The two vertical and the lower horizontal morphisms are $\Fc$-equivalences
  by parts (iii) and (vi) of Theorem \ref{thm:Sp_F properties}. So the upper horizontal morphism
  is an $\Fc$-equivalence between global spectra in $\Sc p_\Fc$, and hence a global equivalence,
  by Theorem \ref{thm:Sp_F properties} (ii).

  (v) The inclusion $\Sc p_\Fc\to\Sc p_{\gl}$ is a left adjoint by
  Theorem \ref{thm:Sp_F properties} (iii).
  So the inclusion preserves and reflects colimits. Since the inclusion
  is strong monoidal and the global smash product preserves
  colimits in each variable separately, also the restricted
  global smash product
  $\sm\colon \Sc p_{\Fc}\times \Sc p_{\Fc}\to \Sc p_{\Fc}$ preserves
  colimits in each variable separately.
  Since $\Sc p_\Fc$ is presentable by Theorem \ref{thm:Sp_F properties} (i),
  we conclude that $\Sc p_\Fc^\sm$ is presentably symmetric monoidal,
  and its $\infty$-category $\CAlg(\Sc p_\Fc^\sm)$ of commutative algebras
  is presentable by \cite[Corollary 3.2.3.5]{lurie:higher_algebra}.    
  
  (vi) The square is horizontally right adjointable by Lemma \ref{lem:CAlg_left_adjoint}.
  Since the vertical functors in the square also have left adjoints,
  horizontal right adjointability is equivalent to
  vertical left adjointability, compare \cite[Lemma A.10]{cnossen-lenz-linskens:partial}.
  So the square is also vertically left adjointable.
\end{proof}

In the following we will identify, via the fully faithful functor
\[ \CAlg(\text{incl}^\sm)\ : \ \CAlg(\Sc p_\Fc^\sm)\ \to \
   \CAlg(\Sc p_{\gl}^\sm) \ ,\]
the $\infty$-category $\CAlg(\Sc p_\Fc^\sm)$ with the full subcategory of $\CAlg(\Sc p_{\gl}^\sm)$
spanned by those commutative global ring spectra
whose underlying global spectrum is $\Fc$-left induced.
This is legitimate by Theorem \ref{thm:multiplicative_class} (iii).

\begin{rk}[Commutative global $\Fc$-approximations]\label{rk:comFapp}
  Theorem \ref{thm:multiplicative_class}
  makes precise that for multiplicative global classes $\Fc$,
  commutative global ring spectra have well-behaved multiplicative
  approximations by $\Fc$-left induced objects.
  Indeed, the fact that the fully faithful functor $\CAlg(\text{incl}^\sm)$
  has a right adjoint (namely the functor $\CAlg(\rho_\Fc^\sm)$)
  means that the full subcategory of $\Fc$-left induced commutative global ring spectra
  is coreflective inside $\CAlg(\Sc p_{\gl}^\sm)$.
  So for every commutative global ring spectrum $R$, there exists a morphism 
  $\epsilon\colon L_\Fc(R)=\CAlg(\rho_\Fc^\sm)(R)\to R$ in $\CAlg(\Sc p_{\gl}^\sm)$
  that enjoys the following equivalent universal properties:
\begin{enumerate}[(a)]
\item
  The pair $(L_\Fc(R),\epsilon)$ represents the restriction of the functor
  $\CAlg(\Sc p_{\gl}^\sm)(-,R)\colon\allowbreak\CAlg(\Sc p_{\gl}^\sm)^{\op}\to\Sc$ to the subcategory
  $\CAlg(\Sc p_{\Fc}^\sm)$.
  In other words, the underlying global spectrum of $L_\Fc(R)$ is $\Fc$-left induced,
  and for every commutative global ring spectrum $S$
  whose underlying global spectrum is $\Fc$-left induced, the
  map of mapping anima
  \[  \CAlg(\Sc p_{\gl}^\sm)(S,\epsilon)\ : \  \CAlg(\Sc p_{\gl}^\sm)(S,L_\Fc(R))\ \to \  \CAlg(\Sc p_{\gl}^\sm)(S,R)\ \]
  is an equivalence.
\item The morphism $\epsilon\colon L_{\Fc}(R)\to R$ is a terminal object
  in the slice $\infty$-category $\CAlg(\Sc p_{\Fc}^\sm)_{/R}$.
\end{enumerate}
Said yet another way: whenever the underlying global spectrum of $S$ is $\Fc$-left induced,
then every morphism $f\colon S\to R$
of commutative global ring spectra factors uniquely through $\epsilon\colon L_\Fc(R)\to R$.

Moreover, the horizontal right adjointability of the square in
Theorem \ref{thm:multiplicative_class} (vi) means that the underlying
morphism $\epsilon_{\td{1}}\colon L_{\Fc}(R)_{\td{1}}\to R_{\td{1}}$
is a universal $\Fc$-approximation of the underlying global spectra,
i.e., a terminal object in the slice $\infty$-category $(\Sc p_{\Fc})_{/R_{\td{1}}}$.
Informally speaking, not only is there a `best approximation' of $R$
by an $\Fc$-left induced object, there is also no ambiguity about whether
the term `best approximation' is to be interpreted
in the $\infty$-category $\CAlg(\Sc p^\sm_{\gl})$
or in the underlying $\infty$-category $\Sc p_{\gl}$.
\end{rk}

In the rest of this appendix we consider multiplicative global classes
that are also closed under wreath products.
As we show in Theorem \ref{thm:wreath_closed}, the closure under wreath products
guarantees that $\Fc$-left induced spectra `interact well' with
ultra-commutative multiplications.
Besides the maximal global family of all compact Lie groups,
two relevant examples for us are the following.

\begin{eg}[Finite groups]
  The global family of finite groups is multiplicative and closed under
  wreath products.
  Any global class that is also closed under wreath products contains
  all symmetric groups, and hence all finite groups.
  So finite groups are minimal among global classes that are closed under wreath products.
\end{eg}

\begin{eg}[Abelian identity component]
  We write $\Ab^\circ$ for the global family
  of compact Lie groups with abelian identity component;
  this family is multiplicative and closed under wreath products,
  and it features in the discussion of the ultra-commutative localization
  of $\Sigma^\infty_+\bP$ by $b_1$, see Theorem \ref{thm:ucom b_1 loc}.
  We claim that if a global class $\Fc$  is closed under wreath products
  and contains a compact Lie group of positive dimension,
  then it contains the global family $\Ab^\circ$.
  So among global classes that are closed under wreath products, $\Ab^\circ$
  is next-to-minimal.
  
  And indeed: if the global class $\Fc$ contains a positive dimensional
  compact Lie group, then it contains the circle group $T=U(1)$.
  If $\Fc$ is closed under wreath products, it also contains $\Sigma_m\wr T$
  for all $m\geq 1$. We let $G$ be any compact Lie group whose identity
  component $G^\circ$ is abelian. Then every irreducible $G$-representation
  is a direct summand of an induced representation $\ind_{G^\circ}^G(\chi)$
  for some continuous homomorphism $\chi\colon G^\circ\to T$.
  So $G$ admits a faithful representation $V$ that is a direct sum
  of representations of the form $\ind_{G^\circ}^G(\chi_i)$
  for suitable homomorphisms $\chi_i\colon G^\circ\to T$.
  If the component group
  $\pi_0(G)=G/G^\circ$ has order $m$, then the representation $\ind_{G^\circ}^G(\chi_i)$
  is classified by a continuous homomorphism to $\Sigma_m\wr T$.
  So altogether, the faithful representation $V$ provides an embedding
  of $G$ into a finite product of copies of  $\Sigma_m\wr T$.
  Since $(\Sigma_m\wr T)^k$ embeds into $\Sigma_{k m}\wr T$,
  the group $G$ belongs to the global class $\Fc$.
\end{eg}

By Theorem \ref{thm:UR_monadic},
the functors $u_{\gl}\colon \UR_{\gl}\to\CAlg(\Sc p_{\gl}^\sm)$
and $\ev_{\td{1}}\circ u_{\gl}\colon \UR_{\gl}\allowbreak\to \Sc p_{\gl}$ are monadic.
In particular, they have left adjoints
$\lambda_{\gl}\colon\CAlg(\Sc p_{\gl}^\sm)\to \UR_{\gl}$
and $\mP_{\gl}\colon \Sc p_{\gl}\to \UR_{\gl}$, respectively.
We refer to $\mP_{\gl}$ as the {\em free ultra-commutative ring spectrum functor}.
We let $\UR_{\Fc}$ denote the full subcategory of $\UR_{\gl}$
spanned by the ultra-commutative ring spectra whose underlying global spectra are
$\Fc$-left induced.

\begin{theorem}\label{thm:wreath_closed}
  Let $\Fc$ be a multiplicative global class that is closed under wreath products.
  \begin{enumerate}[\em (i)]
  \item The free functor $\mP_{\gl}\colon \Sc p_{\gl}\to \UR_{\gl}$
    takes the subcategory $\Sc p_{\Fc}$ of $\Fc$-left induced global
    spectra to the subcategory $\UR_{\Fc}$.
  \item The left adjoint $\lambda_{\gl}\colon\CAlg(\Sc p_{\gl}^\sm)\to \UR_{\gl}$
    to $u_{\gl}$
    takes the subcategory $\CAlg(\Sc p_{\Fc}^\sm)$ to the subcategory $\UR_{\Fc}$.
  \item
    The $\infty$-category $\UR_{\Fc}$ is presentable and
    the restriction $u_\Fc\colon \UR_{\Fc}\to \CAlg(\Sc p_{\Fc}^\sm)$
    of the functor $u_{\gl}$ is bimonadic.
  \item
    The $\infty$-category $\UR_{\Fc}$ is coreflective in $\UR_{\gl}$,
    and the pullback square of $\infty$-categories
    \[ 
      \xymatrix@C=15mm{
  \UR_{\Fc}\ar[d]_{u_{\Fc}} \ar[r]^-{\text{\em incl}}& \UR_{\gl}\ar[d]^{u_{\gl}}\\
        \CAlg(\Sc p_{\Fc}^\sm)\ar[r]_-{\CAlg(\text{\em incl}^\sm)} & \CAlg(\Sc p_{\gl}^\sm)}
    \]
    is vertically left adjointable and horizontally right adjointable.
\end{enumerate}
\end{theorem}
\begin{proof}
  For use in several of the items we show as a first step that
  $\UR_{\Fc}$ is closed under colimits in the ambient $\infty$-category $\UR_{\gl}$.
  We let $F\colon I\to \UR_{\gl}$ be a functor with values in $\UR_{\Fc}$,
  and we let $R$ be a colimit of $F$ in $\UR_{\gl}$.
  Since the functor $u_{\gl}\colon \UR_{\gl}\to\CAlg(\Sc p_{\gl}^\sm)$
  preserves colimits by Theorem \ref{thm:UR_monadic} (iii),
  $u_{\gl}(R)$ is a colimit of $u_{\gl}\circ F$ in $\CAlg(\Sc p_{\gl}^\sm)$.
  The functor $\CAlg(\text{incl}^\sm)\colon\CAlg(\Sc p^\sm_\Fc)\to \CAlg(\Sc p_{\gl}^\sm)$
  is a fully faithful left adjoint with essential image the
  $\Fc$-left induced objects, by parts (iii) and (vi) of Theorem \ref{thm:multiplicative_class}.
  So its essential image is closed under colimits in $\CAlg(\Sc p_{\gl}^\sm)$.
  So the underlying global spectrum of $u_{\gl}(R)$, which is the same
  as the underlying global spectrum of $R$, is $\Fc$-left induced.

  (i) We show that for every group $G$ in $\Fc$ and every $n\geq 0$,
  the global spectrum underlying $\mP_{\gl}((\Sigma^\infty_+ B_{\gl}G)[-n])$ is $\Fc$-left induced.
  The underlying global spectrum of a free ultra-commutative ring spectrum $\mP_{\gl}(X)$
  is the sum of the homogeneous summands $\mP_{\gl}^{[k]}(X)$ for $k\geq 0$,
  where $\mP_{\gl}^{[k]}\colon \Sc p_{\gl}\to\Sc p_{\gl}$ is the total left
  derived functor of the $k$-th symmetric smash power functor
  $Y\mapsto (Y^{\sm k})/\Sigma_k$ on the 1-category $\Sp^O$.
  Moreover,
  \[  \mP_{\gl}^{[k]}((\Sigma^\infty_+ B_{\gl}G)[-n])\ \sim \ (B_{\gl}(\Sigma_k\wr G))^{- \tau_k^n} \]
  is the global Thom spectrum associated to the negative of the $(\Sigma_k\wr G)$-representation
  $\tau_k^n$, the direct sum of $n$ copies of the natural $\Sigma_k$-representation
  on $\mR^k$. Since $G$ belongs to $\Fc$, and $\Fc$ is closed under wreath products,
  the group $\Sigma_k\wr G$ belongs to $\Fc$. So the global Thom spectrum
  $(B_{\gl}(\Sigma_k\wr G))^{- \tau_k^n}$ is $\Fc$-left induced by Example \ref{eg:Thom_left_induced}.
  Hence also the global spectrum underlying $\mP_{\gl}((\Sigma^\infty_+B_{\gl}G)[-n])$
  is $\Fc$-left induced,
  i.e., $\mP_{\gl}((\Sigma^\infty_+ B_{\gl}G)[-n])$ belongs to the subcategory
   $\UR_{\Fc}$.
  
  By definition, $\Sc p_{\Fc}$ is generated as a localizing subcategory of $\Sc p_{\gl}$
  by the suspension spectra $\Sigma^\infty_+ B_{\gl}G$ for all compact Lie groups $G$ in $\Fc$.
  So $\Sc p_{\Fc}$ is generated under colimits in $\Sc p_{\gl}$
  by the objects $(\Sigma^\infty_+ B_{\gl}G)[-n]$ for all $G\in\Fc$ and all $n\geq 0$.
  As a left adjoint, $\mP_{\gl}$ preserves colimits.
  Since $\UR_{\Fc}$ is closed under colimits in $\UR_{\gl}$ by the previous paragraph,
  the functor $\mP_{\gl}$ takes all $\Fc$-left induced global spectra into $\UR_{\Fc}$.

  (ii)  By transitivity of adjoints we have $\mP_{\gl}\sim\lambda_{\gl}\circ\mE_{\gl}$,
  where $\mE_{\gl}\colon \Sc p_{\gl}\to \CAlg(\Sc p_{\gl}^\sm)$
  is left adjoint to the forgetful functor $\ev_{\td{1}}$.
  By the vertical left adjointability in Theorem \ref{thm:multiplicative_class} (vi),
  the functor $\mE_{\gl}$ sends $\Sc p_{\Fc}$ into $\CAlg(\Sc p_\Fc^\sm)$.
  The functor $\mP_{\gl}$ sends $\Sc p_{\Fc}$ into $\UR_{\Fc}$ by part (i); so
  $\lambda_{\gl}$ sends the image of the restricted functor
  $\mE_{\gl}\colon\Sc p_{\Fc}\to\CAlg(\Sc p_\Fc^\sm)$ into $\UR_{\Fc}$.
  Since $\ev_{\td{1}}\colon\CAlg(\Sc p_{\Fc}^\sm)\to\Sc p_{\Fc}$
  is a monadic right adjoint, the $\infty$-category $\CAlg(\Sc p_{\Fc}^\sm)$
  is generated under colimits by $\mE_{\gl}(\Sc p_{\Fc})$.
  Since $\lambda_{\gl}$ preserves colimits and $\UR_{\Fc}$ is closed under colimits in
  $\UR_{\gl}$, the functor $\lambda_{\gl}$ sends all of
  $\CAlg(\Sc p_{\Fc}^\sm)$ into $\UR_{\Fc}$, proving (ii).

  (iii)
  By definition, the functor $u_{\gl}\colon\UR_{\gl}\allowbreak\to\CAlg(\Sc p^\sm_{\gl})$
  takes the subcategory $\UR_{\Fc}$ to the subcategory $\CAlg(\Sc p^\sm_{\Fc})$.
  By (ii), its left adjoint $\lambda_{\gl}$ takes $\CAlg(\Sc p^\sm_{\Fc})$ to $\UR_{\Fc}$.
  The restricted functors
  \[\xymatrix{
      \lambda_\Fc=\lambda_{\gl}|_{ \CAlg(\Sc p^\sm_{\Fc})} :   \CAlg(\Sc p^\sm_{\Fc})\ \ar@<.4ex>[r] &\ \UR_{\Fc} : u_{\gl}|_{\UR_\Fc}= u_{\Fc} \ar@<.4ex>[l]  }\]
  then form another adjoint pair, with respect to the restricted adjunction unit and counit.
  Since $u_{\gl}$ is a monadic right adjoint by Theorem \ref{thm:UR_monadic} (iii),
  and since the restricted functor $u_{\Fc}\colon\UR_{\Fc}\to\CAlg(\Sc p^\sm_{\Fc})$
  is a right adjoint,
  it is monadic by \cite[Proposition 3.26]{henry-meadows:monads}.
    
  Since $\UR_\Fc$ is closed under colimits inside $\UR_{\gl}$,
  and $\CAlg(\Sc p^\sm_\Fc)$ is closed under colimits inside  $\CAlg(\Sc p^\sm_{\gl})$,
  and since $u_{\gl}$ also preserves colimits,
  also the restriction $u_{\Fc}\colon\UR_{\Fc}\to\CAlg(\Sc p^\sm_\Fc)$ preserves colimits.
  Hence the underlying functor
  $u_\Fc\circ\lambda_\Fc\colon\CAlg(\Sc p^\sm_\Fc)\to\CAlg(\Sc p^\sm_\Fc)$
  of the monad for $\UR_\Fc$ preserves all colimits.
  So $u_\Fc\circ\lambda_\Fc$ is in particular $\kappa$-accessible
  for every sufficiently large regular cardinal $\kappa$.
  Since $\CAlg(\Sc p^\sm_\Fc)$ is presentable and the monad
  is $\kappa$-accessible, the $\infty$-category of algebras is again presentable
  by \cite[Corollary 6.8]{henry-meadows:monads}.
  Since $\UR_\Fc$ is equivalent to the $\infty$-category of algebras, it is presentable, too.

  Since source and target of the colimit preserving
  functor $u_{\Fc}\colon\UR_{\Fc}\to\CAlg(\Sc p^\sm_\Fc)$
  are presentable, this functor admits a right adjoint by Lurie's adjoint functor theorem
  \cite[Corollary 5.5.2.9]{lurie:htt}.
  Since we already know that $u_\Fc$ is monadic, the existence of a
  right adjoint makes it comonadic, too, proving (iii).

  (iv) The $\infty$-categories $\UR_\Fc$ and $\UR_{\gl}$ are presentable by part (iii)
  and Theorem \ref{thm:UR_monadic} (ii),
  and $\UR_\Fc$ is closed under colimits inside $\UR_{\gl}$ by the above.
  So the inclusion $\UR_\Fc\to\UR_{\gl}$ is a colimit preserving
  functor between presentable $\infty$-categories, and thus
  a left adjoint by Lurie's adjoint functor theorem \cite[Corollary 5.5.2.9]{lurie:htt}.
  Hence $\UR_\Fc$ is coreflective in $\UR_{\gl}$.

  The design property that $\lambda_\Fc$ is the restriction of $\lambda_{\gl}$
  precisely means that the pullback square in (iv) is vertically left adjointable.
  In this square, both horizontal functors also have right adjoints.
  In such situations, vertical left adjointability is equivalent to
  horizontal right adjointability,
  compare \cite[Lemma A.10]{cnossen-lenz-linskens:partial}.
  So the square is also horizontally right adjointable, proving (iv).
\end{proof}

\Danger If the multiplicative class $\Fc$ is not closed under wreath products,
Theorem \ref{thm:wreath_closed} fails.
The minimal global family $\Fc=\{1\}$ of trivial groups already illustrates this:
the free functor $\mP_{\gl}\colon \Sc p_{\gl}\to \UR_{\gl}$
takes the sphere spectrum (which is $\{1\}$-left induced)
to $\mP_{\gl}(\mS)$, whose underlying global spectrum
$\bigvee_{k\geq 0} \Sigma^\infty_+ B_{\gl}\Sigma_k$ is {\em not} $\{1\}$-left induced.
So the functor $\mP_{\gl}$ does {\em not} take $\Sc p_{\{1\}}$ to $\UR_{\{1\}}$; hence the functor
$\lambda_{\gl}\colon\CAlg(\Sc p_{\gl}^\sm)\to \UR_{\gl}$
does {\em not} take the subcategory $\CAlg(\Sc p_{\{1\}}^\sm)$ to the subcategory $\UR_{\{1\}}$,
either.

\begin{rk}[Ultra-commutative $\Fc$-approximations]\label{rk:approxFucom}
The previous theorem says, loosely speaking, that for multiplicative
global classes $\Fc$ that are also closed under wreath products,
also ultra-commutative ring spectra have well-behaved
approximations by $\Fc$-left induced objects.
This should be compared to Remark \ref{rk:comFapp}, where such approximations
were discussed in the $\infty$-category of commutative global ring spectra,
and the global class was assumed multiplicative, but not necessarily
closed under wreath products.

Indeed, the coreflectiveness of $\UR_{\Fc}$ in Theorem \ref{thm:wreath_closed} (iv)
means that for every ultra-commutative ring spectrum $R$,
there exists a morphism of ultra-commutative ring spectra
$\delta\colon L_\Fc^{\uc}(R)\to R$ that enjoys the following equivalent
universal properties:
\begin{enumerate}[(a)]
\item
  The pair $(L_\Fc^{\uc}(R),\delta)$ represents the restriction of the functor
  $\UR_{\gl}(-,R)\colon\allowbreak(\UR_{\gl})^{\op}\to\Sc$ to the subcategory  $\UR_{\Fc}$.
  In other words, the underlying global spectrum of $ L_\Fc^{\uc}(R)$ is $\Fc$-left induced,
  and for every ultra-commutative ring spectrum $S$
  whose underlying global spectrum is $\Fc$-left induced, the
  map of mapping anima
  \[  \UR_{\gl}(S,\delta)\ : \  \UR_{\gl}(S,L_\Fc^{\uc}(R))\ \to \  \UR_{\gl}(S,R)\ \]
  is an equivalence.
\item The morphism $\delta\colon  L_{\Fc}^{\uc}(R)\to R$ is a terminal object
  in the slice $\infty$-category $(\UR_{\Fc})_{/R}$.
\end{enumerate}
Both properties make precise that whenever
the underlying global spectrum of $S$ is $\Fc$-left induced,
then every morphism $f\colon S\to R$
of ultra-commutative ring spectra factors uniquely through $\delta\colon  L_\Fc^{\uc}(R)\to R$.

Moreover, the horizontal right adjointability of the square in
Theorem \ref{thm:wreath_closed} (iv) means that the underlying
morphism $u_{\gl}(\delta)\colon u_{\gl}( L_{\Fc}^{\uc}(R))\to u_{\gl}(R)$
is a universal $\Fc$-approximation of the underlying commutative
global ring spectrum of $R$ as discussed in Remark \ref{rk:comFapp} above.
So informally speaking,
\[ u_{\gl}( L_{\Fc}^{\uc}(R))\ = \ L_{\Fc}(u_{\gl}(R))\ . \]
And when furthermore combined with
the horizontal right adjointability of the square in
Theorem \ref{thm:multiplicative_class} (vi), we obtain that the morphism
of global spectra underlying $\delta\colon L_{\Fc}^{\uc}(R)\to R$
is a universal $\Fc$-approximation of the underlying global spectrum of $R$.
The essence of Theorem \ref{thm:wreath_closed}
is thus that not only does there exist a
`best approximation of $R$ by an $\Fc$-left induced object',
there is furthermore no ambiguity about whether
the term `best approximation' is interpreted in the $\infty$-category
$\UR_{\gl}$, in the less highly structured $\infty$-category $\CAlg(\Sc p^\sm_{\gl})$,
or in the underlying $\infty$-category $\Sc p_{\gl}$.
\end{rk}

\end{appendix}

\end{document}